\RequirePackage{fix-cm}
\documentclass[12pt]{report}
\usepackage[hyphens]{url}
\usepackage{hyperref}
\usepackage[margin=1.2in]{geometry}
\pdfoutput=1
\usepackage{lmodern}
\usepackage[utf8x]{inputenc}
\usepackage[T1]{fontenc}
\usepackage{amsmath,mathtools,amssymb,etoolbox,enumerate,xspace}
\usepackage{microtype}
\usepackage{caption,subcaption}
\mathtoolsset{mathic,centercolon}
\usepackage[numbered]{bookmark}

\usepackage{chngcntr}
\counterwithout{footnote}{chapter}

\usepackage{xcolor}
\usepackage{cleveref}
\hypersetup{
  colorlinks,
  linkcolor={red!70!black},
  citecolor={green!50!black},
  urlcolor={blue!50!black},
  pdfauthor={Floris van Doorn},
  unicode,
  pdftitle={On the Formalization of Higher Inductive Types and Synthetic Homotopy Theory}}
\usepackage[numbers]{natbib}

\usepackage{amsthm}
\usepackage{aliascnt}
\let\cref\Cref
\mathchardef\texthyphen="2D %

\usepackage{tikz,tikz-cd}
\usetikzlibrary{arrows}
\usetikzlibrary{calc, positioning}
\tikzset{>=stealth',auto,node distance=10mm,text height=1.5ex}
\makeatletter
\protected\def\tikz@nonactivecolon{\ifmmode\mathrel{\mathop\ordinarycolon}\else:\fi} 
\makeatother

\definecolor{keywordcolor}{rgb}{0.7, 0.1, 0.1}   
\definecolor{tacticcolor}{rgb}{0.1, 0.2, 0.6}    
\definecolor{commentcolor}{rgb}{0.4, 0.4, 0.4}   
\definecolor{symbolcolor}{rgb}{0.0, 0.1, 0.6}    
\definecolor{sortcolor}{rgb}{0.1, 0.5, 0.1}      
\definecolor{attributecolor}{rgb}{0.7, 0.1, 0.1} 
\usepackage{listings}

\lstMakeShortInline"

\usepackage{bussproofs}

\newcommand{\arrow}{\ensuremath{\operatorname{\mathsf{arrow}}}}

\renewcommand{\o}{\circ}
\newcommand{\eps}{\ensuremath{\varepsilon}}
\renewcommand{\epsilon}{\varepsilon}
\renewcommand{\phi}{\varphi}
\renewcommand{\C}{\mathbb{C}}  

\newcommand{\br}[1]{\langle#1\rangle}

\makeatletter
\newcommand{\sectionNotes}{\phantomsection\section*{Notes}\addcontentsline{toc}{section}{Notes}\markright{\textsc{\@chapapp{} \thechapter{} Notes}}}
\newcommand{\sectionExercises}[1]{\phantomsection\section*{Exercises}\addcontentsline{toc}{section}{Exercises}\markright{\textsc{\@chapapp{} \thechapter{} Exercises}}}
\makeatother

\newcommand{\jdeq}{\equiv}      
\let\judgeq\jdeq
\newcommand{\defeq}{\vcentcolon\equiv}  

\makeatletter
\def\prd#1{\@ifnextchar\bgroup{\prd@parens{#1}}{\@ifnextchar\sm{\prd@parens{#1}\@eatsm}{\prd@noparens{#1}}}}
\def\prd@parens#1{\@ifnextchar\bgroup%
  {\mathchoice{\@dprd{#1}}{\@tprd{#1}}{\@tprd{#1}}{\@tprd{#1}}\prd@parens}%
  {\@ifnextchar\sm%
    {\mathchoice{\@dprd{#1}}{\@tprd{#1}}{\@tprd{#1}}{\@tprd{#1}}\@eatsm}%
    {\mathchoice{\@dprd{#1}}{\@tprd{#1}}{\@tprd{#1}}{\@tprd{#1}}}}}
\def\@eatsm\sm{\sm@parens}
\def\prd@noparens#1{\mathchoice{\@dprd@noparens{#1}}{\@tprd{#1}}{\@tprd{#1}}{\@tprd{#1}}}
\def\lprd#1{\@ifnextchar\bgroup{\@lprd{#1}\lprd}{\@@lprd{#1}}}
\def\@lprd#1{\mathchoice{{\textstyle\prod}}{\prod}{\prod}{\prod}({\textstyle #1})\;}
\def\@@lprd#1{\mathchoice{{\textstyle\prod}}{\prod}{\prod}{\prod}({\textstyle #1}),\ }
\def\tprd#1{\@tprd{#1}\@ifnextchar\bgroup{\tprd}{}}
\def\@tprd#1{\mathchoice{{\textstyle\prod_{(#1)}}}{\prod_{(#1)}}{\prod_{(#1)}}{\prod_{(#1)}}}
\def\dprd#1{\@dprd{#1}\@ifnextchar\bgroup{\dprd}{}}
\def\@dprd#1{\prod_{(#1)}\,}
\def\@dprd@noparens#1{\prod_{#1}\,}

\def\lam#1{{\lambda}\@lamarg#1:\@endlamarg\@ifnextchar\bgroup{.\,\lam}{.\,}}
\def\@lamarg#1:#2\@endlamarg{\if\relax\detokenize{#2}\relax #1\else\@lamvar{\@lameatcolon#2},#1\@endlamvar\fi}
\def\@lamvar#1,#2\@endlamvar{(#2\,{:}\,#1)}
\def\@lameatcolon#1:{#1}

\def\lamu#1{{\lambda}\@lamuarg#1:\@endlamuarg\@ifnextchar\bgroup{.\,\lamu}{.\,}}
\def\@lamuarg#1:#2\@endlamuarg{#1}

\def\fall#1{\forall (#1)\@ifnextchar\bgroup{.\,\fall}{.\,}}

\def\exis#1{\exists (#1)\@ifnextchar\bgroup{.\,\exis}{.\,}}

\def\sm#1{\@ifnextchar\bgroup{\sm@parens{#1}}{\@ifnextchar\prd{\sm@parens{#1}\@eatprd}{\sm@noparens{#1}}}}
\def\sm@parens#1{\@ifnextchar\bgroup%
  {\mathchoice{\@dsm{#1}}{\@tsm{#1}}{\@tsm{#1}}{\@tsm{#1}}\sm@parens}%
  {\@ifnextchar\prd%
    {\mathchoice{\@dsm{#1}}{\@tsm{#1}}{\@tsm{#1}}{\@tsm{#1}}\@eatprd}%
    {\mathchoice{\@dsm{#1}}{\@tsm{#1}}{\@tsm{#1}}{\@tsm{#1}}}}}
\def\@eatprd\prd{\prd@parens}
\def\sm@noparens#1{\mathchoice{\@dsm@noparens{#1}}{\@tsm{#1}}{\@tsm{#1}}{\@tsm{#1}}}
\def\lsm#1{\@ifnextchar\bgroup{\@lsm{#1}\lsm}{\@@lsm{#1}}}
\def\@lsm#1{\mathchoice{{\textstyle\sum}}{\sum}{\sum}{\sum}({\textstyle #1})\;}
\def\@@lsm#1{\mathchoice{{\textstyle\sum}}{\sum}{\sum}{\sum}({\textstyle #1}),\ }
\def\tsm#1{\@tsm{#1}\@ifnextchar\bgroup{\tsm}{}}
\def\@tsm#1{\mathchoice{{\textstyle\sum_{(#1)}}}{\sum_{(#1)}}{\sum_{(#1)}}{\sum_{(#1)}}}
\def\dsm#1{\@dsm{#1}\@ifnextchar\bgroup{\dsm}{}}
\def\@dsm#1{\sum_{(#1)}\,}
\def\@dsm@noparens#1{\sum_{#1}\,}

\def\wtype#1{\@ifnextchar\bgroup%
  {\mathchoice{\@twtype{#1}}{\@twtype{#1}}{\@twtype{#1}}{\@twtype{#1}}\wtype}%
  {\mathchoice{\@twtype{#1}}{\@twtype{#1}}{\@twtype{#1}}{\@twtype{#1}}}}
\def\lwtype#1{\@ifnextchar\bgroup{\@lwtype{#1}\lwtype}{\@@lwtype{#1}}}
\def\@lwtype#1{\mathchoice{{\textstyle\mathsf{W}}}{\mathsf{W}}{\mathsf{W}}{\mathsf{W}}({\textstyle #1})\;}
\def\@@lwtype#1{\mathchoice{{\textstyle\mathsf{W}}}{\mathsf{W}}{\mathsf{W}}{\mathsf{W}}({\textstyle #1}),\ }
\def\twtype#1{\@twtype{#1}\@ifnextchar\bgroup{\twtype}{}}
\def\@twtype#1{\mathchoice{{\textstyle\mathsf{W}_{(#1)}}}{\mathsf{W}_{(#1)}}{\mathsf{W}_{(#1)}}{\mathsf{W}_{(#1)}}}
\def\dwtype#1{\@dwtype{#1}\@ifnextchar\bgroup{\dwtype}{}}
\def\@dwtype#1{\mathsf{W}_{(#1)}\,}

\def\wtypeh#1{\@ifnextchar\bgroup%
  {\mathchoice{\@lwtypeh{#1}}{\@twtypeh{#1}}{\@twtypeh{#1}}{\@twtypeh{#1}}\wtypeh}%
  {\mathchoice{\@@lwtypeh{#1}}{\@twtypeh{#1}}{\@twtypeh{#1}}{\@twtypeh{#1}}}}
\def\lwtypeh#1{\@ifnextchar\bgroup{\@lwtypeh{#1}\lwtypeh}{\@@lwtypeh{#1}}}
\def\@lwtypeh#1{\mathchoice{{\textstyle\mathsf{W}^h}}{\mathsf{W}^h}{\mathsf{W}^h}{\mathsf{W}^h}({\textstyle #1})\;}
\def\@@lwtypeh#1{\mathchoice{{\textstyle\mathsf{W}^h}}{\mathsf{W}^h}{\mathsf{W}^h}{\mathsf{W}^h}({\textstyle #1}),\ }
\def\twtypeh#1{\@twtypeh{#1}\@ifnextchar\bgroup{\twtypeh}{}}
\def\@twtypeh#1{\mathchoice{{\textstyle\mathsf{W}^h_{(#1)}}}{\mathsf{W}^h_{(#1)}}{\mathsf{W}^h_{(#1)}}{\mathsf{W}^h_{(#1)}}}
\def\dwtypeh#1{\@dwtypeh{#1}\@ifnextchar\bgroup{\dwtypeh}{}}
\def\@dwtypeh#1{\mathsf{W}^h_{(#1)}\,}

\makeatother

\newcommand{\proj}[1]{\ensuremath{\mathsf{pr}_{#1}}\xspace}

\newcommand{\rec}[1]{\mathsf{rec}_{#1}}
\newcommand{\ind}[1]{\mathsf{ind}_{#1}}

\newcommand{\pairr}[1]{{\mathopen{}(#1)\mathclose{}}}

\newcommand{\im}{\ensuremath{\mathsf{im}}} 

\newcommand{\dpath}[4]{#3 =^{#1}_{#2} #4}

\newcommand{\transfib}[3]{\ensuremath{\mathsf{transport}^{#1}(#2,#3)\xspace}}

\newcommand{\mapfunc}[1]{\ensuremath{\mathsf{ap}_{#1}}\xspace} 
\newcommand{\map}[2]{\ensuremath{{#1}\mathopen{}\left({#2}\right)\mathclose{}}\xspace}

\newcommand{\mapdepfunc}[1]{\ensuremath{\mathsf{apd}_{#1}}\xspace} 
\newcommand{\mapdep}[2]{\ensuremath{\mapdepfunc{#1}\mathopen{}\left(#2\right)\mathclose{}}\xspace}
\let\apfunc\mapfunc
\let\ap\map

\let\apd\mapdep

\newcommand{\idfunc}[1][]{\ensuremath{\mathsf{id}_{#1}}\xspace}

\newcommand{\htpy}{\sim}

\newcommand{\eqv}[2]{\ensuremath{#1 \simeq #2}\xspace}

\newcommand{\eqvsym}{\simeq}    

\newcommand{\isequiv}{\ensuremath{\mathsf{isequiv}}}

\newcommand{\hfib}[2]{{\mathsf{fib}}_{#1}(#2)}

\newcommand{\UU}{\ensuremath{\mathcal{U}}\xspace}

\let\type\UU

\newcommand{\set}{\ensuremath{\mathsf{Set}}\xspace}

\newcommand{\prop}{\ensuremath{\mathsf{Prop}}\xspace}

\newcommand{\ua}{\ensuremath{\mathsf{ua}}\xspace} 

\newcommand{\trunc}[2]{\mathopen{}\left\Vert #2\right\Vert_{#1}\mathclose{}}

\newcommand{\tproj}[3][]{\mathopen{}\left|#3\right|_{#2}^{#1}\mathclose{}}
\newcommand{\tprojf}[2][]{|\blank|_{#2}^{#1}}

\newcommand{\emptyt}{\ensuremath{\mathbf{0}}\xspace}

\newcommand{\unit}{\ensuremath{\mathbf{1}}\xspace}

\newcommand{\bool}{\ensuremath{\mathbf{2}}\xspace}
\newcommand{\btrue}{{1_{\bool}}}
\newcommand{\bfalse}{{0_{\bool}}}

\newcommand{\inl}{\ensuremath\inlsym\xspace}
\newcommand{\inr}{\ensuremath\inrsym\xspace}

\newcommand{\seg}{\ensuremath{\mathsf{seg}}\xspace}

\newcommand{\glue}{\mathsf{glue}}

\newcommand{\base}{\ensuremath{\mathsf{base}}\xspace}

\newcommand{\surf}{\ensuremath{\mathsf{surf}}\xspace}

\newcommand{\susp}{\Sigma}
\newcommand{\north}{\mathsf{N}}
\newcommand{\south}{\mathsf{S}}
\newcommand{\merid}{\mathsf{merid}}

\newcommand{\blank}{\mathord{\hspace{1pt}\text{--}\hspace{1pt}}}

\newcommand{\op}{^{\mathrm{op}}}

\newcommand{\N}{\ensuremath{\mathbb{N}}\xspace}
\let\nat\N

\newcommand{\nil}{\mathsf{nil}}
\newcommand{\cons}{\mathsf{cons}}

\newcommand{\Z}{\ensuremath{\mathbb{Z}}\xspace}

\newcommand{\happly}{\mathsf{happly}}

\newcommand{\R}{\ensuremath{\mathbb{R}}\xspace}           

\makeatletter
\def\defthm#1#2#3{%
  \newaliascnt{#1}{thm}
  \newtheorem{#1}[#1]{#2}
  \aliascntresetthe{#1}
  \crefname{#1}{#2}{#3}}

\newtheorem{thm}{Theorem}[section]
\crefname{thm}{Theorem}{Theorems}
\defthm{propn}{Proposition}{Propositions}   
\defthm{cor}{Corollary}{Corollaries}
\defthm{lem}{Lemma}{Lemmas}
\defthm{ax}{Axiom}{Axioms}
\defthm{conj}{Conjecture}{Conjectures}
\theoremstyle{definition}
\defthm{defn}{Definition}{Definitions}
\theoremstyle{remark}
\defthm{rmk}{Remark}{Remarks}
\defthm{ex}{Example}{Examples}
\defthm{exs}{Examples}{Examples}
\defthm{notes}{Notes}{Notes}

\crefformat{section}{\S#2#1#3}
\Crefformat{section}{Section~#2#1#3}
\crefrangeformat{section}{\S\S#3#1#4--#5#2#6}
\Crefrangeformat{section}{Sections~#3#1#4--#5#2#6}
\crefmultiformat{section}{\S\S#2#1#3}{ and~#2#1#3}{, #2#1#3}{ and~#2#1#3}
\Crefmultiformat{section}{Sections~#2#1#3}{ and~#2#1#3}{, #2#1#3}{ and~#2#1#3}
\crefrangemultiformat{section}{\S\S#3#1#4--#5#2#6}{ and~#3#1#4--#5#2#6}{, #3#1#4--#5#2#6}{ and~#3#1#4--#5#2#6}
\Crefrangemultiformat{section}{Sections~#3#1#4--#5#2#6}{ and~#3#1#4--#5#2#6}{, #3#1#4--#5#2#6}{ and~#3#1#4--#5#2#6}

\crefformat{appendix}{Appendix~#2#1#3}
\Crefformat{appendix}{Appendix~#2#1#3}
\crefrangeformat{appendix}{Appendices~#3#1#4--#5#2#6}
\Crefrangeformat{appendix}{Appendices~#3#1#4--#5#2#6}
\crefmultiformat{appendix}{Appendices~#2#1#3}{ and~#2#1#3}{, #2#1#3}{ and~#2#1#3}
\Crefmultiformat{appendix}{Appendices~#2#1#3}{ and~#2#1#3}{, #2#1#3}{ and~#2#1#3}
\crefrangemultiformat{appendix}{Appendices~#3#1#4--#5#2#6}{ and~#3#1#4--#5#2#6}{, #3#1#4--#5#2#6}{ and~#3#1#4--#5#2#6}
\Crefrangemultiformat{appendix}{Appendices~#3#1#4--#5#2#6}{ and~#3#1#4--#5#2#6}{, #3#1#4--#5#2#6}{ and~#3#1#4--#5#2#6}

\crefname{part}{Part}{Parts}

\crefname{figure}{Figure}{Figures}

\let\autoref\cref

\let\c@equation\c@thm
\numberwithin{equation}{section}

\def\noteson{%
\gdef\note##1{\mbox{}\marginpar{\color{blue}\textasteriskcentered\ ##1}}}

\noteson

\newcounter{symindex}

\newenvironment{inductive}
  {\list{}{%
    \leftmargin=1em
    \topsep=1ex
    \parsep=\parskip
    \listparindent=0mm
    \itemindent=0mm
  }\item\relax}
  {\endlist}

\newcommand{\fib}{\ensuremath{\operatorname{\mathsf{fib}}}}
\newcommand{\image}{\ensuremath{\operatorname{\mathsf{im}}}}
\renewcommand{\im}{\image}
\renewcommand{\ker}{\ensuremath{\operatorname{\mathsf{ker}}}}
\renewcommand{\deg}{\ensuremath{\operatorname{\mathsf{deg}}}}
\renewcommand{\max}{\ensuremath{\operatorname{\mathsf{max}}}}
\newcommand{\colim}{\ensuremath{\operatorname{\mathsf{colim}}}}
\newcommand{\quotient}{\ensuremath{\operatorname{\mathsf{quotient}}}}
\newcommand{\groupoidquotient}{\ensuremath{\operatorname{\mathsf{groupoid-quotient}}}}
\newcommand{\twoquotient}{\ensuremath{\operatorname{\mathsf{two-quotient}}}}
\newcommand{\simpletwoquotient}{\ensuremath{\operatorname{\mathsf{simple-two-quotient}}}}
\newcommand{\pushout}{\ensuremath{\operatorname{\mathsf{pushout}}}}
\newcommand{\pt}{\ensuremath{\operatorname{\mathsf{pt}}}}
\newcommand{\mk}{\ensuremath{\operatorname{\mathsf{mk}}}}
\newcommand{\homm}{\ensuremath{\operatorname{\mathsf{hom}}}}
\newcommand{\fin}{\ensuremath{\operatorname{\mathsf{fin}}}}
\newcommand{\words}{\ensuremath{\operatorname{\mathsf{words}}}}
\newcommand{\circlerec}{\ensuremath{\operatorname{\S^1\hspace{-1mm}{.}\mathsf{rec}}}}
\renewcommand{\inl}{\ensuremath{\operatorname{\mathsf{inl}}}}
\renewcommand{\inr}{\ensuremath{\operatorname{\mathsf{inr}}}}
\newcommand{\inm}{\ensuremath{\operatorname{\mathsf{in}}}}
\newcommand{\Vector}{\ensuremath{\operatorname{\mathsf{vector}}}}
\newcommand{\ulift}{\ensuremath{\operatorname{\mathsf{lift}}}}

\newcommand{\leaf}{\ensuremath{\operatorname{\mathsf{leaf}}}}
\newcommand{\node}{\ensuremath{\operatorname{\mathsf{node}}}}
\newcommand{\wtree}{\ensuremath{\operatorname{\mathsf{\omega-tree}}}}
\renewcommand{\nil}{\ensuremath{\operatorname{\mathsf{nil}}}}
\renewcommand{\cons}{\ensuremath{\operatorname{\mathsf{cons}}}}
\newcommand{\spectrum}{\ensuremath{\mathsf{Spectrum}}}
\newcommand{\prespectrum}{\ensuremath{\mathsf{Prespectrum}}}
\newcommand{\group}{\ensuremath{\mathsf{Group}}}
\newcommand{\abgroup}{\ensuremath{\mathsf{AbGroup}}}
\newcommand{\squaret}{\ensuremath{\mathsf{square}}}
\newcommand{\apo}{\ensuremath{\mathsf{apo}}}
\renewcommand{\apd}{\ensuremath{\mathsf{apd}}}
\newcommand{\apdtilde}{\widetilde{\mathsf{apd}}}

\newcommand{\Kloop}{\operatorname{K-\mathsf{loop}}}
\newcommand{\Kelim}{\operatorname{K-\mathsf{elim}}}
\newcommand{\transp}{\ensuremath{\operatorname{\mathsf{transport}}}}
\newcommand{\id}{\ensuremath{\operatorname{\mathsf{id}}}}
\newcommand{\Id}{\ensuremath{\operatorname{\mathsf{Id}}}}
\newcommand{\istrunc}[1]{\operatorname{\mathsf{is-#1-type}}}
\newcommand{\isconn}[1]{\operatorname{\mathsf{is-#1-connected}}}
\newcommand{\refl}{\ensuremath{\operatorname{\mathsf{refl}}}}
\newcommand{\lp}{\ensuremath{\operatorname{\mathsf{loop}}}}
\renewcommand{\surf}{\ensuremath{\operatorname{\mathsf{surf}}}}
\renewcommand{\U}{\UU}
\newcommand{\pbool}{\ensuremath{\S^0}}
\newcommand{\defeqp}{\ensuremath{\vcentcolon=}}

\newcommand{\pU}{\U^*}

\renewcommand{\S}{\mathbb{S}}
\renewcommand{\l}{\lambda}
\newcommand{\sy}{^{-1}}

\newcommand{\pmap}{\to}
\newcommand{\lpmap}{\xrightarrow}

\newcommand{\tr}{\cdot}
\newcommand{\auxl}{\ensuremath{\operatorname{\mathsf{auxl}}}}
\newcommand{\auxr}{\ensuremath{\operatorname{\mathsf{auxr}}}}
\newcommand{\gluel}{\ensuremath{\operatorname{\mathsf{gluel}}}}
\newcommand{\gluer}{\ensuremath{\operatorname{\mathsf{gluer}}}}
\newcommand{\const}{\ensuremath{\mathbf{0}}}
\renewcommand{\op}{^{\mathsf{op}}}

\DeclarePairedDelimiter\angled{\langle}{\rangle} 
\newcommand*{\pair}[2]{\angled{#1,#2}} 

\newcommand*{\inv}{^{-1}}

\newcommand{\alphabar}{\overline{\alpha}}
\newcommand{\rhobar}{\overline{\rho}}
\newcommand{\lambdabar}{\overline{\lambda}}
\newcommand{\gammabar}{\overline{\gamma}}
\newcommand{\zeroh}{\mathsf{z}}
\newcommand{\oneh}{\mathsf{u}}
\newcommand{\two}{\mathsf{b}}
\newcommand{\twist}{\mathsf{tw}}
\newcommand{\smsh}{\wedge}
\newcommand{\mc}{\mathcal}

\newcommand{\sequence}[3][]{(#2,#3)}
\newcommand{\msm}[2]{\Sigma(#1,#2)}
\newenvironment{constr}{%
  \begin{proof}[Construction]%
}{\end{proof}}
\newcommand{\tfcolim}{\mathsf{colim}}
\newcommand{\kshiftequiv}{\mathsf{kshift\underline{~}equiv}}

\usepackage{tocloft}
\newenvironment{xcenter}
 {\par\vspace*{3mm}\setbox0=\hbox\bgroup\ignorespaces}
 {\unskip\egroup\noindent\makebox[\textwidth]{\box0}\par\vspace*{3mm}}

\begin{document}
\pdfbookmark[chapter]{Front Matter}{front}
\pdfbookmark[section]{Title Page}{front}
\title{On the Formalization of Higher Inductive Types and Synthetic Homotopy Theory}
\author{Floris van Doorn}
\pagenumbering{roman}
\begin{titlepage}
  \begin{center}
    \vspace*{20mm}
    \LARGE
    \textbf{On the Formalization of Higher Inductive Types and Synthetic Homotopy Theory}
    \par\vspace*{15mm}\par
    \Large
    Floris van Doorn
    \par\vspace*{15mm}\par
    May 2018\\[15mm] 
    \normalsize
      Dissertation Committee:\\
      Jeremy Avigad\\
      Steve Awodey\\
      Ulrik Buchholtz\\
      Mike Shulman\\[30mm]
Submitted in partial fulfillment of the requirements for the degree of\\
Doctor of Philosophy in Pure and Applied Logic\\
Department of Philosophy\\
Carnegie Mellon University
\end{center}
\end{titlepage}
\setcounter{page}{2}
\pdfbookmark[section]{\contentsname}{toc}
\tableofcontents

\chapter{Introduction}\label{cha:introduction}
\pagenumbering{arabic}

The goal of this dissertation is to present synthetic homotopy theory in the setting
of \emph{homotopy type theory}. We will present various results in this
framework, most notably the construction of the Atiyah-Hirzebruch and Serre
spectral sequences for cohomology, which have been fully formalized in the Lean
proof assistant.

Homotopy type theory, often abbreviated HoTT, is a version of type theory. Type
theory is a language for formal mathematics, in which every object has a
computational interpretation, so that it can also function as a programming
language. It can be used as a foundation of mathematics as an alternative to set
theory.

A key feature of HoTT is that the equality in a space corresponds to the path
spaces; a path between two points $a$ and $b$ is a proof that $a=b$. Two paths
that are not homotopic give different (unequal) proofs of this equality. The
fact that we identify proofs of an equality with a path means that every
construction in HoTT respects paths. 

Many different researchers contributed to the homotopical interpretation of type theory. 
Steve Awodey and Michael Warren gave a model of type theory in abstract homotopy theory~\cite{awodey2009homotopy}.
Benno van den Berg and Richard Garner published a paper addressing the coherence issue~\cite{berg2010models}. 
Independently, Vladimir Voevodsky gave a model of type theory without identity types in simplicial sets 
and formulated the \emph{univalence axiom}, which he proved consistent~\cite{voevodsky2006,voevodsky2009typesystems}.
The univalence axiom states that
homotopy equivalences between two types (spaces) corresponds to equality between
them~\cite{voevodsky2014univalence}. This means that every construction done in
HoTT automatically respects homotopy equivalence, which is a very convenient
property. Also, Voevodsky proved that a consequence of the univalence axiom is
\emph{function extensionality}. This states that two functions are equal when
they are homotopic. 

The fact that all constructions are homotopy invariant also leads to some
challenges. It is not always clear whether we can define a concept of homotopy theory in
homotopy type theory. For example, \emph{singular homology} is a homotopy
invariant notion, but in the construction we use the set of all simplices in a
space, which is not a homotopy invariant notion. In this case, we can define
homology in a different way (see \autoref{sec:spectral-sequence-homology}).
However, for other definitions, such as the Grassmannian manifolds, it is an
open problem whether they can be constructed in homotopy type theory.

A new concept in homotopy type theory is the concept of \emph{higher inductive
types}. These are types that generalize both cell complexes in homotopy theory,
and inductively generated types (like $\N$) in type theory. Higher inductive
types can be used to construct many spaces and operations on spaces often
encountered in homotopy theory.

Type theory is a convenient language for computer proof assistants. These are
programs that allow you to write formal proofs in a specified language, and
then the computer checks whether the proof is correct and complete. There are
many major results formalized in proof assistants, such as the four colour
theorem~\cite{gonthier2005fourcolour}, Feit-Thompson theorem~\cite{gonthier2013oddorder} 
and the Kepler conjecture (Hales' Theorem)~\cite{hales2017kepler}. 
HoTT is a type theory, and it has been implemented in various proof assistants, 
such as Coq~\cite{bauer2016coqhott}, Agda~\cite{hottagda}, cubicaltt~\cite{cubicaltt}, Lean~\cite{vandoorn2017leanhott}
and various experimental proof assistants. One disadvantage of formally verifying
proofs in a proof assistant is that it takes a lot of work spelling out all
details. For example, doing very basic homotopy theory (not using homotopy type
theory) already takes a lot of effort~\cite{zhan2017auto2}. In HoTT this effect
is mitigated, because many homotopical concepts are close to the foundations of
the type theory, making formal proofs only a little more work than a paper proof. 

Various results have been proven and formalized in HoTT, such as the the
Seifert--van Kampen theorem~\cite{favonia2016seifert}, the Blakers--Massey
theorem~\cite{favonia2016blakersmassey} and a development of cellular
cohomology~\cite{buchholtz2018cellular}. Another main result (which has not been
formalized) is the computation of $\pi_4(\S^3)$~\cite{brunerie2016spheres}, which
relies on conjectured properties of the smash product, which we will discuss in
\autoref{sec:smash-product}.

HoTT gives novel proof methods and new insights to homotopy theory. A basic
property of HoTT is \emph{path induction}, which states that when proving
something for a path with one free endpoint, one may assume that the path is the
constant path. This corresponds to the fact that the path space with one fixed
endpoint is contractible. Another technique is the encode-decode method, for
calculating the path space of certain spaces~\cite{licatashulman2013}. Moreover,
the proof of the Blakers--Massey theorem has been translated back to homotopy
theory, resulting in a new proof with novel ideas~\cite{rezk2014blakersmassey}.

Homotopy type theory has models in most model categories~\cite{awodey2009homotopy,berg2010models}, 
which are categorical models for homotopy theory. These models were 
inspired by the groupoid model~\cite{hofmann1998groupoid}. Other models for HoTT
include the simplicial set model~\cite{voevodsky2009typesystems,kapulkin2012simplicialnew,streicher2014simplicial} and the cubical
set model~\cite{bezem2014cubicalsets,cohen2016cubical}. More generally, all Grothendieck
$(\infty,1)$-toposes model HoTT~\cite{cisinski2014models}.\footnote{General Grothendieck $(\infty,1)$-toposes model HoTT with universes \'a la Tarski. 
This notion is weaker than universes \'a la Russell, which are usually considered in HoTT. We explain Russell universes in \autoref{sec:universes}.} Moreover, it is
conjectured that all elementary $(\infty,1)$-toposes form models of HoTT~\cite{shulman2017topos}.

\subsubsection*{Type Theory}
Homotopy type theory is based on Martin-L\"of type theory (also called intuitionistic type theory 
or constructive type theory)~\cite{martinlof1975typetheory,martinlof1984typetheory}. In this type theory there are types, like the
integers $\Z$, vectors $\R^n$; and complex functions $\C\to\C$. There are also
terms, which have a unique type.\footnote{To be more precise: in many type theories there are terms with multiple types, for example due to universe cumulativity, but we will ignore these issues. Moreover, the type theory of Lean has unique typing~\cite{carneiro2018leantheory}.} 
For example the number $-2$ has type $\Z$
(written as $-2:\Z$), the vector $(1,2,3,\ldots,n)$ has type $\R^n$ and we have the
exponential function $\exp:\C\to\C$. One can think of types as sets of objects
(and indeed, there is a model of type theory where the types are exactly sets),
but there are different interpretations, such as the types-as-spaces
interpretation that homotopy type theory provides. The fact that terms have a unique type means that the
$2:\Z$ and the $2:\R$ are different objects. It might be helpful to think of data types in a programming language, in which the \texttt{int} $2$ is stored differently in memory than the \texttt{float} $2$. Of course, the canonical inclusion
$i:\Z\hookrightarrow\R$ does satisfy $i(2)=2$. Type theory has a primitive
notion of computation, so that for example $2+3$ computes to $5$. Every function
that is explicitly defined in type theory therefore describes an algorithm
that can be executed. This means that type theory can be used as a programming
language, and many programming languages make use of a type system. 
The congruence closure of this notion of computation is called 
\emph{definitional equality} or \emph{judgmental equality}, 
and if two terms are judgmentally equal, one can replace one for the other in any term.

There are several methods to construct new types out of existing ones. For
example we can form the function type $A\to B$ for types $A$ and $B$, the
cartesian product type $A\times B$ and the coproduct or sum $A+B$. Propositions
can also be interpreted as types by the \emph{Curry-Howard isomorphism}~\cite{curry1958combinatorylogic,howard1980formulae}, and
under this interpretation $A\times B$ is the conjunction of $A$ and $B$, the sum
$A+B$ is the disjunction and $A\to B$ is the implication. Furthermore, there are
dependent function types $\prd{x:A}P(x)$ and dependent sum types $\sm{x:A}P(x)$,
which correspond to the universal quantification $\forall(x:A), P(x)$ and
existential quantification $\exists(x:A), P(x)$, respectively. So for example
the transitivity of $\le$ on $\N$ can be expressed as $\prd{k,m,n:\N}k \le m \to
m \le n \to k \le n$, and a term of this type is a proof that $\le$ is
transitive. The $P$ in $\prd{x:A}P(x)$ and $\sm{x:A}P(x)$ is called a
\emph{dependent type}, since it is a type depending on a term $x:A$. It has type
$P:A\to\type$, where $\type$ is the universe of (small) types. The dependent
function type $\prd{x:A}P(x)$ consists of functions $f$ that send terms $a:A$
to a term $f(a):P(a)$. Note that the type of $f(a)$ depends on the input $a$.
The dependent sum type $\sm{x:A}P(x)$ consists of dependent pairs $(a,x)$ with
$a:A$ and $x:P(a)$, where the type of $x$ depends on $a$.

Given two terms $a, b : A$, we can form the \emph{identity type} which we write as $a =_A b$ or $a=b$. As a proposition we view $a=_A b$ as the statement that $a$ and $b$ are
equal. In homotopy type theory these identity types correspond to the path space
of the type $A$.

\subsubsection*{Homotopy Type Theory}

There are various versions of dependent type theory with different rules for the
identity type. Some type theories have a \emph{reflection rule}, which states that 
if we have a proof $p:a=b$, then $a$ and $b$ are judgmentally equal. Type theories with this rule are often called \emph{extensional}.
This is a convenient rule, but these type theories have meta-theoretic properties that are often seen as undesirable. 
For example, checking whether a term $t$ has type $A$ is not decidable anymore. 
Since this operation can be viewed as ``checking the correctness of a proof,'' one often wants to work in a type theory with decidable type-checking.

In \emph{intensional} type theory, without the reflection rule, multiple approaches can be taken for the identity type. 
In some versions, there is a rule that any two proofs of the same
equality are themselves equal. This rule, often called \emph{uniqueness of
identity proofs} or \emph{axiom K} states that if $p,q : a = b$, then there is a
proof of $p = q$. In homotopy type theory, this rule is rejected. In the
types-as-spaces interpretation of homotopy type theory, terms of the identity
type $a =_A b$ are interpreted as paths in $A$ from $a$ to $b$. 
We have familiar operations on paths: given two paths $p:a=_Ab$ and $q:b=_Ac$, 
we write $p\cdot q:a=c$ for the concatenation of $p$ and $q$. 
Furthermore, we have the inverse path $p\sy:b=a$ and the constant path $\refl_a:a=a$.
We also have higher paths, the identity type $p =_{a=_Ab} q$ consists of homotopies from path $p$ to $q$. We can form
higher path types between two homotopies, and there are also operations on these higher paths. 
In this way every type comes equipped with the structure of a higher groupoid.

In 2011, higher inductive types were introduced in homotopy type 
theory~\cite{shulman2011intervalimpliesfunext,lumsdaine2011hits,shulman2011HoTThits,shulman2011pi1S1}. With
ordinary inductive types we specify constructors that generate the type, for
example the natural numbers are generated by zero $0:\N$ and the successor
function $\mathsf{succ}:\N\to\N$. Higher inductive types are generated not only
by these ``point constructors'' but also by ``path constructors,'' which specify
the inhabitants of paths or higher paths in the type. For example, the circle
$\S^1$ is generated by a point $\star:\S^1$ and a loop $\ell:\star=\star$. The
rest of the structure of $\S^1$ is built from these constructors. Using higher
inductive types we can construct many other spaces in homotopy theory, 
such as Eilenberg-MacLane spaces and homotopy pushouts. 

As mentioned before in this introduction, we can use HoTT to do homotopy theory.
We think of types as spaces and we think of maps between types as continuous
maps between those spaces. Then we can define usual notions in homotopy theory,
as long as they are homotopy invariant: homotopy equivalences, suspensions,
spheres, etcetera. This is a \emph{synthetic} way to do homotopy theory: many
concepts, such as spaces and paths are uninterpreted constants of the type
theory. This is opposed to \emph{analytic} homotopy theory, where one studies
topological spaces up to homotopy equivalence. This distinction is similar to
the distinction for elementary geometry, which we can do synthetically (points
and lines are undefined concepts) or analytically (we are working in $\R^2$).
Synthetic geometry limits the things one can state or prove, but these proofs
are applicable in every model of the axioms. The same is true for synthetic
homotopy theory: the proofs performed synthetically are true in all models of
HoTT.

In this dissertation I will not be very precise about the exact rules of the type theory we are using. We will
present the constructions and proofs in such a way that they can be performed in the ``HoTT
book''~\cite{hottbook}. Most of the results in this dissertation have been
formalized in the Lean proof assistant~\cite{moura2015lean}. The HoTT mode we
used in Lean has very similar rules to the HoTT book, and the differences are
not relevant for the constructions in this dissertation. A concept closely related to
homotopy type theory is \emph{univalent mathematics}, a term coined by Vladimir
Voevodsky for the development of mathematics where one takes homotopy types as 
primitive objects, and reasons about them using type-theoretic reasoning and the 
univalence axiom. This is pursued in the proof assistant UniMath~\cite{unimath}. There are also
radically different type theories which are studied in homotopy type
theory. These are called ``cubical type theories'' because they all have a
primitive notion of cubes. Examples include the cubical type theory described in~\cite{cohen2016cubical}, 
which was implemented in the proof assistant cubicaltt~\cite{cubicaltt}, 
and computational higher-dimensional type theory~\cite{angiuli2017computational}, 
on which the proof assistant RedPRL is based~\cite{redprl}. 
These type theories are extensions of the type theory presented in
the HoTT book, which we will call book-HoTT. In book-HoTT the univalence axiom
is an axiom: an uninterpreted constant of a certain type. This breaks the
computational behavior of the type theory. For example not every closed term of
type $\N$ computes to either $0$ or the successor of another number. These
cubical type theories add primitive concepts to the theory to make the
univalence axiom provable, and therefore all terms in these system do compute.

We will often want to compare homotopy type theory with ordinary homotopy
theory. We will use the adverb ``classically'' to refer to the concepts and
theorems in homotopy theory that do not involve HoTT.\footnote{This use of classically has nothing to do with the word classical in ``classical logic,'' involving the law of excluded middle or the axiom of choice. In homotopy type theory one can consistently assume the law of excluded middle or the axiom of choice, formulated in a precise way so that it corresponds to what it usually means. However, doing so removes the computational content of all notions defined using it.} 
Conversely, we will say that something is provable in HoTT if we can prove it in book-HoTT.

\subsubsection*{Contents}

In \autoref{cha:preliminaries} we review the basic concepts in homotopy type
theory. For a more detailed and thorough exposition, we refer to~\cite{hottbook}. 
Alternative introductions can be found in~\cite{favonia2017thesis} and~\cite{brunerie2016spheres}. In
\autoref{sec:martin-lof-type} we introduce the basic concepts of type theory:
functions, pairs, universes, and inductive types such as the identity type. In
\autoref{sec:homotopy-type-theory} we will introduce the basics of homotopy type
theory. In particular we will formally state the univalence axiom and present
higher inductive types. In \autoref{sec:lean} we will discuss the Lean proof
assistant in more detail.

In \autoref{cha:high-induct-types} we will study higher inductive types
internally in HoTT. The main problem we will focus on is the interdefinability
of higher inductive types. In particular, we try to construct various higher
inductive types from the homotopy pushout. We will define the propositional
truncation in \autoref{sec:prop-trunc}, nonrecursive higher inductive types with
2-path constructors in \autoref{sec:non-recursive-2} and work towards defining
certain \emph{localizations} in \autoref{sec:colimits}.

In \autoref{cha:homotopy-theory} we present some synthetic homotopy theory in
HoTT. In \autoref{sec:computing-pi3s2} we will describe the formalization of the
long exact sequence of homotopy groups and its application to compute
$\pi_3(\S^2)$. Although this construction has been described before in HoTT
in~\cite[Section 8.4]{hottbook} and~\cite[Section 2.5.1]{brunerie2016spheres},
no formally verified proof has been given before. In
\autoref{sec:eilenb-macl-spac} we will study Eilenberg-MacLane spaces, which are
spaces with only one nontrivial homotopy group. Eilenberg-MacLane spaces have
been defined in HoTT before~\cite{licata2014em}. Here we prove the (classically
known) results that Eilenberg-MacLane spaces are unique, and give an equivalence
of categories between the category of (abelian) groups and an appropriate class of pointed
types. In \autoref{sec:smash-product} we will discuss the smash product. The
ultimate goal is to prove that the smash product forms a 1-coherent symmetric monoidal
product on pointed types, and we will give one approach towards proving this
using a Yoneda-style argument.

In \autoref{cha:serre-spectr-sequ} we develop the theory of spectral sequences
in HoTT. We give the construction of a spectral sequence from an exact couple
(in \autoref{sec:exact-couples}) and show how to construct an exact couple from
a tower of spectra (in \autoref{sec:spectra}). We construct the
classically-known Atiyah-Hirzebruch and Serre spectral sequences for cohomology
(in \autoref{sec:spectral-sequence-cohomology}), and give some ideas towards
doing the same for their counterparts in homology (in
\autoref{sec:spectral-sequence-homology}). 

\chapter{Preliminaries}\label{cha:preliminaries}

In this chapter we will give a brief overview of type theory and homotopy type
theory. We cannot cover all the subtleties, so readers new to (homotopy) type
theory should consult the homotopy type theory book~\cite{hottbook}.

In \autoref{sec:lean} we will discuss the proof assistant \emph{Lean}. All main
results in this dissertation have been formalized in Lean. 

\section{Martin-L\"of Type Theory}\label{sec:martin-lof-type}

As mentioned in the introduction, homotopy type theory is based on a system called \emph{Martin-L\"of type theory} or \emph{intuitionistic type theory}. 
There are types and there are terms, which have a unique type. There is a notion of computation. Two terms $t$ and $s$ are considered \emph{judgmentally equal} or \emph{definitionally equal}, 
denoted $t \equiv s$ if $t$ and $s$ compute to the same term.

We are working in dependent type theory, which means that types can depend on terms. 
For example, there is a type of vectors of length $n:\N$ in type $A$, denoted $\Vector_A(n)$. 
In this case $\Vector_A$ is a dependent type over $\N$. 
An example term in this type family is $(5,6,7,8) : \Vector_\N(4)$. 
When we say that a term has a unique type, we mean that it has a unique type up to definitional equality. 
In our example, we also have that $(5,6,7,8) : \Vector_\N(2+2)$, because $2+2\equiv 4$. 
More generally, if we have two definitionally equal types $A\equiv B$ and if $t : A$, then $t : B$. 
Logically (under the types-as-propositions interpretation) dependent types are predicates. 
We will explain the topological interpretation of dependent types at the end of \autoref{sec:pair-types}.

In the remainder of this section we will discuss the type formers of Martin-L\"of type theory more closely.

\subsection{Function Types}\label{sec:function-types}
Given a type $A$ and a family of types $B$ depending on $A$, we can form the \emph{dependent function type} (also called \emph{product type} or \emph{pi type}) $$(x : A) \to B(x)\qquad\text{or}\qquad\prd{x:A}B(x).$$
We will use the former notation in this document. A term $f:(x : A)\to B(x)$ is a function that sends each element $a : A$ to an element\footnote{Formally, $B(a)$ is the term $B(x)$ where we substitute $a$ for $x$. In \autoref{sec:universes} we will see that we can treat $B$ as a function into a universe, and that alternatively we can view $B(x)$ and $B(a)$ as function applications.} $f(a):B(a)$. We also use the notation $fa$ or $f\ a$ for $f(a)$. Note that the type of $f(x)$ depends on $x$. We can form functions using \emph{lambda-abstraction}. Given a term $t(x) : B(x)$, we can form the term $\lam{x}t(x): (x : A) \to B(x)$, which is the function $x\mapsto t(x)$, i.e. the function that sends $x$ to $t(x)$. We get the computation rule 
$$(\lam{x}t(x))a\equiv t(a)$$
for $a : A$, which is called the \emph{beta-rule} or \emph{beta-reduction}. We also have an \emph{eta-rule}, which states that every function is a lambda abstraction. This means that for $f : (x : A) \to B(x)$ we have 
$$f\equiv \lam{x}f(x).$$
We will often define functions by writing $f(x)\defeq t$ (where $x$ may occur in $t$), which formally means that we define $f$ as $\lam{x}t$.

An important special case occurs when $B$ does not depend on $A$. In this case the dependent function type $(x : A) \to B$ is written as $A \to B$, which is the type of functions from type $A$ to type $B$.

Logically, the type $A \to B$ is interpreted as the implication $A \Rightarrow B$ and the type $(x : A) \to B(x)$ is interpreted as the universal quantification $\forall(x : A), B(x)$. Topologically, a function $f : A \to B$ corresponds to a continuous map from $A$ to $B$. The type $A \to B$ is the mapping space from $A$ to $B$. We will explain the topological interpretation of $(x : A) \to B(x)$ at the end of \autoref{sec:pair-types}.

We can define the identity function 
$$\idfunc\equiv\idfunc[A]\defeq \lam{x:A}x:A\to A$$ 
and the composition of functions: if $f : A \to B$ and $g : B \to C$, then $g \o f\defeq \lam{x}g(f(x)):A\to C$. Given $b:B$, we also have a constant function $\const_b\defeq \lam{x}b:A \to B$.

We will often write some arguments of a function implicitly. Such arguments are written with curly braces in the type. For example, given a dependent type $C$ over $\N$, we write 
$$g : \{n : \N\} \to C(n) \to C(n+1)$$
to emphasize that the first argument of $g$ is implicit. In this case, for $c : C(n)$ we will write $g(c)$ for $g$ applied (implicitly) to $n$ and applied to $c$. The curly braces are only to indicate how we write function application for functions with this type, for all other purposes the types $\{x : A\} \to B(x)$ and $(x : A) \to B(x)$ are the same.

\subsection{Pair Types}\label{sec:pair-types}
Given a type family $B$ depending on a type $A$, we can form the \emph{dependent pair type} (also called \emph{dependent sum type} or \emph{sigma type}) 
$$(x : A) \times B(x)\qquad\text{or}\qquad\sm{x:A}B(x).$$
We will use the former notation in this document. A term of type $(x : A)\times B(x)$ is a pair consisting of an element $a : A$ and an element $b : B(a)$. Given $a : A$ and $b : B(a)$, we can form the term $(a,b):(x : A)\times B(x)$, and we have projections 
$$p_1:(x : A)\times B(x)\to A\quad \text{ and }\quad p_2:(z : (x : A)\times B(x))\to B(p_1(z)).$$
We will sometimes write $x.i$ for $p_i(x)$. There are beta rules $p_1(a,b)\equiv a$ and $p_2(a,b)\equiv b$ and an eta rule stating that for any $z:(x : A)\times B(x)$ we have $z\equiv (p_1z, p_2z)$. In Lean, there is no eta rule for dependent pair types, but instead there is an induction principle, similar to those of inductive types (see \autoref{sec:inductive-types}). 

If $B(x)$ does not depend on $x$, we write $(x : A) \times B$ simply as $A\times B$. In this case we retrieve the usual cartesian product of $A$ and $B$.

Logically we can think of $A\times B$ as the conjunction of $A$ and $B$, as
described above. Furthermore, we can think of $(x : A)\times B(x)$ as a
proof-relevant version of the existential quantifier $\exists(x : A). B(x)$. It
is proof-relevant in the sense that from a proof of $(x : A)\times B(x)$ we can
extract a witness $a : A$ such that $B(a)$ holds. In \autoref{sec:truncatedness}
we will define an existential quantifier from which the witness cannot be
extracted.

Topologically, we think of $A\times B$ as the product space of $A$ and $B$. 
The map $p_1 : (x : A) \times B(x) \to A$ corresponds to a \emph{fibration}. A fibration is a map that has the homotopy lifting property with respect to any space, which is given by transport, to be defined in \autoref{sec:paths}. Under this interpretation, $(x : A) \times B(x)$ is the total space of the fibration $p_1$, and $B(a)$ is the \emph{fiber} of $p_1$ at point $a$. The type $(x : A)\to B(x)$ is the type of sections of $p_1$. These observations are usually summarized as ``dependent types correspond to fibrations.'' We will often call dependent functions \emph{sections}.

\subsection{Universes}\label{sec:universes}
In our discussions below we need one or more \emph{universes} in our type theory. There are different styles of universes in type theory~\cite{martinlof1984typetheory}, we will describe the universes \'a la Russell. A universe $\type$ is a type that has types as its terms. That is to say, if $A:\type$, then $A$ is a type. It is closed under all type-forming operations. For example, for pi-types this means that if $A : \type$ and for $a:A$ we have $B(a):\type$, then
$$(a : A) \to B(a):\type.$$
We can now interpret dependent types in $\type$, such as $B$ above, as functions $B:A\to\type$.

In the proofs in this document we can often get away with assuming only a single universe. However, it is useful to have the property that all types have a type themselves, and we cannot do that with a single universe $\type$, because positing $\type:\type$ is inconsistent~\cite{girard1972paradox}. Instead, we will assume that we have a tower of universes $$\type_0:\type_1:\type_2:\cdots$$
such that for every type $A$ there is an $i$ such that $A:\type_i$. In this case every dependent type can be interpreted as a function $A\to\type_i$ for some $i$. As is customary, we usually omit writing universe levels explicitly, and we will perform constructions polymorphic over all universes. For example, if we write $$\idfunc : \{A : \type\} \to A \to A,$$ we really mean that for any universe level $i$ we have $$\idfunc[i] : \{A : \type_i\} \to A \to A.$$

One rule that is sometimes assumed is \emph{universe cumulativity}, which states that if $A : \type_i$, then $A : \type_j$ for $j\geq i$. 
This can be problematic, and lead to violation of nice properties of the type theory, such as subject reduction or canonicity~\cite{luo2012universes}. 
In this document (and in Lean), we do not assume universe cumulativity. 
Instead, using inductive types (see \autoref{sec:inductive-types}) we can construct for $A : \type_i$ a new type $\ulift A:\type_j$ for $j\geq i$ such that $A\simeq \ulift A$.

\subsection{Inductive Types}\label{sec:inductive-types}

\emph{Inductive types} are types that are inductively generated by some
\emph{constructors}. A simple example is $\N$, which is inductively generated by
$0$ and the successor function $S\defeq\lam{x}x+1$. In this section we will
discuss some inductive types that we will need in this dissertation. We will
talk about the empty type, the unit type, the booleans, the natural numbers and
the sum type. The dependent pair type (\autoref{sec:pair-types}) is also an inductive type. 

\subsubsection*{The empty type}
The empty type $\emptyt:\type_0$ is a type without inhabitants. There are no constructors, and we have as induction principle that if $P:\emptyt\to\U_i$, then $$\ind{\emptyt}:(x : \emptyt) \to P(x).$$ This conveys that $\emptyt$ indeed has no inhabitants, because if we view $P$ as a predicate, we can prove anything about all inhabitants of $\emptyt$. We can define negation $\neg A\defeq A \to \emptyt$.

\subsubsection*{The unit type}
The unit type $\unit:\type_0$ is a type with exactly one inhabitant $\star:\unit$. The induction principle states that if $P:\unit\to\U_i$, then $$\ind{\unit}:P(\star) \to (x : \unit) \to P(x).$$ This states that $\star$ is the only inhabitant of $\unit$, because if we can prove something for $\star$, then it holds for all inhabitants of $\unit$. There is a computation rule $$\ind{\unit}(p,\star)\equiv p.$$

\subsubsection*{The booleans}
The type of booleans $\bool:\type_0$ has exactly two inhabitants $\btrue,\bfalse:\bool$. Its induction principle states that if $P:\bool\to\U_i$, then $$\ind{\bool}:P(\btrue) \to P(\bfalse) \to (x : \bool) \to P(x).$$
The computation rules are 
$$\ind{\bool}(p_\btrue,p_\bfalse,\btrue)\equiv p_\btrue\qquad\text{and}\qquad 
\ind{\bool}(p_\btrue,p_\bfalse,\bfalse)\equiv p_\bfalse.$$

\subsubsection*{The natural numbers}
A more interesting type is the type of natural numbers $\N:\type_0$. It has a constructor $0:\N$ and a unary constructor $S:\N \to \N$, and it is freely generated by these constructors. This means that if $P:\N\to\U_i$ and if we have $p_0:P(0)$ and $p_S : (k : \N) \to P(k) \to P(S\ k)$, then $$\ind{\N}(p_0,p_S):(n : \N) \to P(n).$$
If we view $P$ as a predicate, this is the usual induction principle for $\N$: to prove something for all numbers we need to prove it for $0$ and we need to prove it for $k+1$ assuming it holds for $k$, for an arbitrary $k$. However, this induction principle also allows us to define (dependent) functions from $\N$. These functions satisfy the computation rules
$$\ind{\nat}(p_0,p_S,0)\equiv p_0\qquad\text{and}\qquad 
\ind{\nat}(p_0,p_S,S\ n)\equiv p_S(n,\ind{\nat}(p_0,p_S,n)).$$
Often, we will want to give a name $f$ to $\ind{\nat}(p_0,p_S)$, and we will instead denote the recursive definition of $f$ using pattern matching notation:
$$f(0)\defeq p_0\qquad\text{and}\qquad f(S\ n)\defeq p_S(n,f(n)).$$
For example, we can define addition and multiplication $+,\cdot:\N\to\N\to\N$ recursively (in the second argument) as
\begin{align*}
  n+0&\defeq n & n\cdot 0&\defeq 0\\
  n+(S\ m)&\defeq S(n + m) & n \cdot (S\ m)&\defeq n \cdot m + n.
\end{align*}
Note that $n+1\equiv S\ n$, and we will often write $n+1$ instead of $S\ n$ from now on.

\subsubsection*{The sum type}
Given two types $A$ and $B$, we can form the \emph{sum type} or \emph{coproduct} $A+B$ with constructors $\inl:A\to A+B$ and $\inr : B \to A + B$. 
The induction principle states that for $P:A+B\to\type$ with maps $p_{\inl} : (a : A) \to P(\inl a)$ and $p_{\inr} : (b : B) \to P(\inr b)$ we get a section 
$$\ind{{+}}(p_{\inl},p_{\inr}):(x : A + B) \to P(x)$$
with computation rules
$$\ind{{+}}(p_{\inl},p_{\inr},\inl(a))\equiv p_{\inl}(a)\qquad\text{and}\qquad
\ind{{+}}(p_{\inl},p_{\inr},\inr(b))\equiv p_{\inr}(b)$$

Logically, the type $A+B$ is the proof-relevant disjunction of $A$ and $B$. It is
proof-relevant in the sense that a proof of $A + B$ is of the form
$\inl a$ or $\inr b$. Therefore, a proof comes with a proof of either $A$ or $B$. In
\autoref{sec:truncatedness} we will see a disjunction that does not have this
property.

\subsubsection*{General Inductive Types}

In \autoref{sec:inductive-types} we saw various instances of inductive types.
Also the sigma-types from \autoref{sec:pair-types} (without eta rule) are an instance of an
inductive type. We will now explain inductive types and families of inductive
types in general. For a more detailed description, see~\cite[Section ``Inductive Types'']{carneiro2018leantheory}.

When defining an inductive type, we have to list its constructors. For example, we could define the sum type as follows. Given $A\ B : \type$, we define 

\begin{inductive}
  \texttt{inductive} $A + B : \type \defeqp$ \\
  $\bullet\ \inl : A \to A + B;$\\
  $\bullet\ \inr : B \to A + B.$
\end{inductive}
This defines the type $A+B$ with constructors $\inl$ and $\inr$ of the specified
type. Each constructor must have as target the inductive type currently being
defined (in this case $A+B$).\footnote{For \emph{higher inductive types}
(\autoref{sec:high-induct-types}) the conclusion can also be a (higher) path in
the type currently being defined.} Constructors can be \emph{recursive},
meaning that the type being defined can occur in the domain of a constructor.
For example, here is the type of $\omega$-branching trees with leaves labeled by
a type $C$.
\begin{inductive}
  \texttt{inductive} $\wtree_C : \type \defeqp$ \\
  $\bullet\ \leaf : C\to\wtree_C;$\\
  $\bullet\ \node : (\N \to \wtree_C)\to\wtree_C.$
\end{inductive}
A restriction on recursive constructors is that the inductive type being defined can only occur in strictly positive positions, that is as the target of one of the arguments of the constructor.

Every inductive type has an induction principle. We can algorithmically find the type of the induction principle from the constructors. The first argument of the induction principle (often left implicit) is the \emph{motive}, which is an arbitrary type family over the inductive type being defined, for $\wtree_C$ this has type $P:\wtree_C\to\type$. Then for every constructor $c$ there is an argument that mimics the type of $c$ and has as target $P(c(\cdots))$. 
For $\wtree_C$ these arguments have type $p_{\leaf}:(c:C)\to P(\leaf c)$ and 
$$p_{\node}:(f:\N\to\wtree_C)\to ((n : \N) \to P(f\ n))\to P(\node f).$$
Note that for each recursive argument $f$ of the constructor we assume an \emph{induction hypothesis} of type $P(f(\cdots))$. The induction principle then gives a section of $P$. So for example we get 
$$\ind{\wtree_C}(p_{\leaf},p_{\node}):(x : \wtree_C) \to P(x).$$
Finally, the computation rules states that if the induction principle acts on a constructor, then it will reduce to the argument corresponding to that constructor. For $\wtree_C$ this means (abbreviating $s\defeq \ind{\wtree_C}(p_{\leaf},p_{\node})$)
$$s(\leaf(c))\equiv p_{\leaf}(c)\qquad\text{and}\qquad
s(\node(f))\equiv p_{\node}(f,\lam{n}s(f\ n)),$$
where applying $s$ to the recursive constructor leads to a recursive call of $s$.

One important generalization of inductive types are families of inductive types. In this case, a family of types $P$ is being defined simultaneously indexed over some type $I$. In this case, constructors must have as target $P(t)$ where $t$ is a term of type $I$ formed by the (nonrecursive) arguments of the constructor. An example of an inductive family of types is the type of vectors in $A$ of some length $n:\N$.
\begin{inductive}
  \texttt{inductive} $\Vector_A : \N \to \type \defeqp$ \\
  $\bullet\ \nil : \Vector_A(0);$\\
  $\bullet\ \cons : \{n : \N\} \to A \to \Vector_A(n) \to \Vector_A(n+1).$
\end{inductive}
Note that the parameter $A$ remains fixed in the definition of $\Vector_A(n)$, while the index $n:\N$ is not: the constructor $\cons$ constructs a vector of length $n+1$ from a vector of length $n$.
The induction principle can again be extracted algorithmically. It is important that the motive also quantifies over all indices of the inductive family. For vectors it states that given a motive
$$P:\{n : \N\} \to \Vector_A(n)\to \type$$ 
and induction steps
\begin{align*}
  p_{\nil}&:P(\nil)\\
  p_{\cons}&:(n : \N) \to (a : A) \to (x : \Vector_A(n))\to P(x) \to P(\cons(a,x)),
\end{align*}
we get a section 
$$\ind{\Vector}(p_{\nil},p_{\cons}):\{n : \N\} \to (x : \Vector_A(n)) \to P(x)$$
with the expected computation rules.

A very important inductive family of types is the \emph{identity type}.\footnote{also called \emph{path type}, \emph{identification type} or \emph{equality type}.} This is a family of types with parameters $A:\type$ and $a:A$ and is defined as
\begin{inductive}
  \texttt{inductive} $\Id_A(a,{-}) : A \to \type \defeqp$ \\
  $\bullet\ \refl_a : \Id_A(a,a).$
\end{inductive}
We also denote the type $\Id_A(a_1,a_2)$ by $a_1=_A a_2$ or $a_1=a_2$ and $\refl_a$ by $\refl$, $1_a$ or $1$. Its induction principle states that for a family $P:(a' : A) \to a = a' \to \type$ and a term $p_{\refl}: P(a,1_a)$ we find a section
$$\ind{=}(p_{\refl}):(a' : A) \to (p : a = a') \to P(a',p).$$
In words: we may assume that a path with free right endpoint (that is, the right hand side of the equality is a variable) is reflexivity. 

Logically, the identity type corresponds to equality. Under this interpretation, a term of type $a_1=a_2$ is a proof that $a_1$ and $a_2$ are equal. Homotopically, the identity type corresponds to the path space of $A$, and we will explore this interpretation more in \autoref{sec:paths}.


\section{Homotopy Type Theory}\label{sec:homotopy-type-theory}

We will now discuss in more detail the homotopical interpretation of types, and the basic concepts of homotopy type theory.

\subsection{Paths}\label{sec:paths}

Elements of an identity type form paths in the space. We can define the usual operations on paths.

Given a path $p:a=_Ab$, we can define the \emph{inverse} $p\sy:b=_Aa$. We can do this by path induction. Define the family $$P\defeq \lam{x:A}{q:a=_Ax}x=_Aa:(x:A)\to a=_Ax \to \type.$$
We now have $\refl_a:P(a,p)\equiv a =_A a$, and therefore we get 
$$p\sy\defeq\ind{=}(\refl_a,b,p):b=_Aa.$$
The computation rule gives that $\refl_a\sy\equiv \refl_a$.

We can explain the proof in words more intuitively. Path induction states that we may assume that a path with a free endpoint is reflexivity. Since $p$ has a free endpoint ($b$ is a variable), we may assume that $b\equiv a$ and $p\equiv\refl_a$. In this case, we can define $$p\sy\equiv \refl_a\sy\defeq \refl_a:a=a.$$
The map path inversion we have defined this way has type 
$$\{a\ b : A\} \to a = b \to b = a.$$
We can also define path \emph{concatenation}. Given $p:a=_Ab$ and $q:b=_Ac$, we define $p\tr q:a=_Ac$ again by path induction. We will only give the intuitive argument and leave the formal proof to the reader. Since $q$ has free endpoint $c$, we may assume that $c\equiv b$ and $q\equiv\refl_b$. In this case, we define $p\tr \refl_b\defeq p:a=b$.

We can also define higher paths. For example, given $p:a=b$ and $q:b=c$ and $r:c=d$, we have a path
$$p\tr(q \tr r)=(p\tr q)\tr r,$$
which is the \emph{associativity} of path concatenation. We can prove this by path induction on $r$: if $r$ is reflexivity, then both sides reduce to $p\tr q$.

By using path induction, we can also prove the following equalities:
\begin{align*}
 p\cdot 1 &= p & p\cdot p\sy &= 1\\
 1\cdot p &= p & p\sy\cdot p &= 1.
\end{align*}
It is trickier to prove the Eckmann-Hilton property of equality, which states that given $a:A$ and $p,q:\refl_a=\refl_a$, we have $p\cdot q=q\cdot p$. 
The problem is that cannot apply path induction to $p$ or $q$ directly. We omit the proof here and refer to~\cite[Theorem 2.1.6]{hottbook}.

Given a map $f:A\to B$, we can prove that $f$ respects paths. Given a path
$p:a=_Aa'$, we define $\apfunc{f}(p):f(a)=_Bf(b)$ by path induction: for
reflexivity we define $\apfunc{f}(\refl_a)\defeq \refl_{f(a)}$. We will sometimes
abuse notation and write $f(p)$ for $\apfunc{f}(p)$. From a logical perspective this just states that functions respect equality, 
but from a homotopical perspective, this states that functions respect paths, which is in line with our intuition that all functions are continuous in HoTT.

We can compute what $\apfunc{}$ does when our map is the identity map, a constant map or a composition of maps:
\begin{align*}
  \apfunc{\id_A}(p)&=p&\apfunc{\const_b}(p)&=\refl_b&\apfunc{g\o f}(p)=\apfunc{g}(\apfunc{f}(p)).
 \end{align*}
 All three of these properties are easily proven by path induction. Also, we can compute $\apfunc{}$ when we apply it to inverses or concatenations of paths:
 \begin{align*}
  \apfunc{f}(p\tr q)&=\apfunc{f}(p)\tr\apfunc{f}(q)&\apfunc{f}(p\sy)&=(\apfunc{f}(p))\sy.
 \end{align*}

 Given a dependent type $P:A\to\type$ and a path $p:a=_Aa'$, we can define the
 \emph{transport} function $\transp^P(p): P(a)\to P(a')$. We define it by path induction; for reflexivity we define $\transp^P(\refl_a)\defeq\id_{P(a)}$. When $P$ is known from context we will write $p_*(b)$ for $\transp^P(p,b)$.

By path induction we can prove basic equalities about transports. We have

\begin{align*}
  \transp^P(p \tr q,x) &= \transp^P(q,\transp^P(p,x))\\
  \transp^{\lam{a}B}(p,x) &= x\\
  \transp^{P\o f}(p,x) &= \transp^P(\apfunc{f}(p),x)\\
  f_{a'}(\transp^P(p,x)) &= \transp^Q(p,f_a(x))&&\text{for $f:(a:A)\to P(a)\to Q(a)$.}
\end{align*}

\subsection{Equivalences}\label{sec:equivalences} 

In this section we talk about maps between types that have an inverse in a suitable way. Before we can give the definition, we need to define homotopy.

Given two dependent maps $f,g:(a:A)\to B(a)$, a \emph{homotopy} $h:f\sim g$ is a proof that $f$ and $g$ are pointwise equal: 
$$(f\sim g)\defeq (a:A)\to f(a)=_{B(a)}g(a).$$
Recall that all maps are considered continuous, so this actually gives a continuous deformation of $f$ to $g$, which is exactly what a homotopy is in topology. 

\begin{defn}Suppose given a function $f:A\to B$.
  \begin{itemize}
    \item A \emph{left-inverse} of $f$ is an inhabitant of $(g:B\to A)\times g\o f\sim\id_A$.
    \item Similarly, a \emph{right-inverse} of $f$ is an inhabitant of $(h:B\to A)\times f\o h\sim\id_B$.
    \item We say that $f$ is an \emph{equivalence} or $\isequiv(f)$ if $f$ has both a left and a right inverse. We will denote its left-inverse by $f\sy$. We can then show that $f\sy$ is also a right inverse of $f$.
    \item The type of equivalences between $A$ and $B$ is $(A\simeq B)\defeq(f:A\to B)\to \isequiv(f)$. Given an element $f:A\simeq B$, we will also use $f$ to denote the underlying map $A\to B$.
  \end{itemize}
\end{defn}
It is easy to show that the identity map $\id_A:A\simeq A$ is an equivalence. Moreover, if $g:B\to C$ and $f:A\to B$ are both equivalences, then $g\o f$ and $f\sy$ are also equivalences. This shows that equivalences are reflexive, symmetric and transitive. 

A very important property is that any two inhabitants of $\isequiv(f)$ are equal: if $p,q:\isequiv(f)$, then $p=q$. We will not prove this here, but it is shown in~\cite[Theorem 4.3.2]{hottbook}. This property is the reason that we define the notion of equivalences this way. 
If we would define $\isequiv(f)$ by requiring a map that is both a left \emph{and} a right inverse of $f$, then this property would not hold.

Given two equivalences $f,f':A\simeq B$, it does not matter whether we compare them as functions or equivalences:
$$(f=_{A\simeq B}f')\simeq (f=_{A\to B}f') \simeq (f\sim f').$$

By path induction we also get a map $(A=_{\type}B)\to(A\simeq B)$, because if the path $p:A=B$ is reflexivity, we can just take $\id_A:A\simeq A$ as our equivalence. In plain Martin L\"of type theory one cannot characterize what the type $A=B$ is. This is where the univalence axiom comes in. The \emph{univalence axiom} states that the map
$$(A=B)\to (A\simeq B)$$
is an equivalence. In particular this means that we get a map in the other direction: given an equivalence $e:A\simeq B$, we get an equality $\ua(e):A=B$.

\subsection{More on paths}\label{sec:more-paths} 
In this section we will discuss dependent paths, or pathovers; higher paths, such as squares and cubes; and paths in type formers.

\subsubsection*{Pathovers}
We will often need to relate elements in two different fibers of a dependent type. Suppose we have a family $P:A\to\type$ with $x:P(a)$ and $x':P(a')$. If we have a path $p:a=a'$, we can form the type $x =_p^P x'$ of \emph{dependent paths} or \emph{pathovers} over $p$. There are four equivalent ways to define this:
\begin{enumerate}
  \item \label{item:pathover-tr} We can define $(x =_p^P x')\defeq(\transp^P(p,x)=x')$
  \item We can define $(x =_p^P x')\defeq(x=\transp^P(p\sy,x'))$
  \item We can define $(x =_p^P x')$ by path induction on $p$. If $p\defeq\refl_a$, we define 
  $(x =_p^P x')\equiv(x =_{\refl_a}^P x')\defeq(x=_{P(a)}x')$.
  \item \label{item:pathover-ind} We can define $(x =_p^P x')$ by a family of inductive types. For fixed $A:\type$ and $P:A\to\type$ and $a:A$ and $x:P(a)$ we have the following family:
  \begin{inductive}
    \texttt{inductive} $x =_{(-)}^P {(-)} : \{a' : A\} \to a = a' \to P(a') \to \type \defeqp$ \\
    $\bullet\ \refl : x =_{\refl_a}^P x.$
  \end{inductive}  
\end{enumerate}
It does not matter which of these definitions we pick, because we can prove that all of them are equivalent.\footnote{In Lean, we chose option \ref{item:pathover-ind}. Option \ref{item:pathover-tr} would probably be slightly more convenient to work with, because then this characterization becomes a definitional equality. In practice it will not matter much, though.}

We have the following equivalences between pathovers:
\begin{align*}
  (x=^{\lam{a}B}_px') &\simeq (x =_B x') & (x =^{P \o f}_p x') \simeq (x =^P_{\apfunc{f}(p)}x').
\end{align*}

We can do operations on pathovers, similar to the operations on paths. We have concatenation and inversion, and we will abuse notation and denote them with the same notation.
\begin{align*}
({-})\cdot ({-})&: x_1 =^P_{p} x_2 \to x_2 =^P_{q} x_3 \to x_1 =^P_{p\tr q} x_3\\
({-})\sy&: x_1 =^P_{p} x_2 \to x_2 =^P_{p\sy} x_1.
\end{align*}
We have a dependent version of $\apfunc{}$. Given a dependent map $f:(a:A)\to P(a)$, we get
$$\apd_f: (p : a = a') \to f(a) =^P_p f(a').$$
A variant to $\apd$ is the following. Given $f:A\to B$, a family $P:B\to\type$ and a section $g:(a:A)\to P(f(a))$, we define
\begin{align}\apdtilde_g: (p : a = a') \to g(a) =^P_{\apfunc{f}(p)} g(a').\label{eq:apdtilde}\end{align}
The difference between $\apd$ and $\apdtilde$ is over which path they lie.

Furthermore, if we have a map $f:A\to B$ and two families $P:A\to\type$ and $Q:B\to\type$ and a fiberwise map $g:(a:A)\to P(a)\to Q(f(a))$, then we get a fiberwise version of $\apfunc{}$:
\begin{equation}\apo_g: x =^P_p x'\to g_a(x) =^Q_{\apfunc{f}(p)} g_{a'}(x').\label{eq:apo}\end{equation}

\subsubsection*{Squares}
For higher paths, it is convenient to define a separate notion of a square in a type:
\begin{center}
  \begin{tikzpicture}[node distance=10mm]
  \node (tl) at (0,0) {$a_{00}$};
  \node[right = of tl] (tr) {$a_{20}$};
  \node[below = of tl] (bl) {$a_{02}$};
  \node (br) at (tr |- bl)  {$a_{22}$};
  \path[every node/.style={font=\sffamily\small}]
  (tl) edge[double equal sign distance] node [above] {$p_{10}$} (tr)
       edge[double equal sign distance] node [right] {$p_{01}$} (bl)
  (bl) edge[double equal sign distance] node [above] {$p_{12}$} (br)
  (tr) edge[double equal sign distance] node [right] {$p_{21}$} (br);
  \end{tikzpicture}
\end{center}
Suppose given four paths as in the diagram above, that is
\begin{align*}
  p_{10}&:a_{00}=a_{20}&p_{01}&:a_{00}=a_{02}\\
  p_{12}&:a_{02}=a_{22}&p_{21}&:a_{20}=a_{22}.
\end{align*}
We have a type of squares $\squaret(p_{10},p_{12},p_{01},p_{21})$, which we can define in either of the two following equivalent ways
\begin{enumerate}
\item We can define $\squaret(p_{10},p_{12},p_{01},p_{21})\defeq (p_{10}\tr p_{21} = p_{01} \tr p_{12})$.
\item $\squaret(p_{10},p_{12},p_{01},p_{21})$ is defined as an inductive family of types. For a fixed $a_{00}:A$ we define the family
\begin{inductive}
  \texttt{inductive} $\squaret({-},{-},{-},{-}) : \{a_{20}\ a_{02}\ a_{22} : A\} \to a_{00} = a_{20} \to a_{02} = a_{22} \to a_{00} = a_{02} \to a_{20} = a_{22} \to \type \defeqp$ \\
  $\bullet\ \refl : \squaret(\refl_{a_{00}},\refl_{a_{00}},\refl_{a_{00}},\refl_{a_{00}}).$
\end{inductive}  
\end{enumerate}
We will usually write squares using diagrams as above. There are various operations on squares. For example, we can horizontally concatenate them. If we can fill each of the individual squares below, we can fill the outer rectangle (which has as top $p_{10}\cdot p_{30}$ and as bottom $p_{12}\cdot p_{32}$).
\begin{center}
  \begin{tikzpicture}[node distance=10mm]
  \node (tl) at (0,0) {$a_{00}$};
  \node[below = of tl] (bl) {$a_{02}$};
  \node[right = of tl] (tr) {$a_{20}$};
  \node (br) at (tr |- bl)  {$a_{22}$};
  \node[right = of tr] (tr2) {$a_{40}$};
  \node (br2) at (tr2 |- bl)  {$a_{42}$};
  \path[every node/.style={font=\sffamily\small}]
  (tl) edge[double equal sign distance] node [above] {$p_{10}$} (tr)
       edge[double equal sign distance] node [right] {$p_{01}$} (bl)
  (bl) edge[double equal sign distance] node [above] {$p_{12}$} (br)
  (tr) edge[double equal sign distance] node [above] {$p_{30}$} (tr2)
       edge[double equal sign distance] node [right] {$p_{21}$} (br)
  (br) edge[double equal sign distance] node [above] {$p_{32}$} (br2)
  (tr2) edge[double equal sign distance] node [right] {$p_{41}$} (br2);
  \end{tikzpicture}
\end{center}
We can also vertically concatenate squares, and horizontally or vertically invert squares.

Given a homotopy $h:f\sim g$ between nondependent functions $f,g:A\to B$ and a path $p:a=_Aa'$, we get the following \emph{naturality square}.
\begin{center}
  \begin{tikzpicture}[node distance=10mm]
  \node (tl) at (0,0) {$f(a)$};
  \node[right = of tl] (tr) {$g(a)$};
  \node[below = of tl] (bl) {$f(a')$};
  \node (br) at (tr |- bl)  {$g(a')$};
  \path[every node/.style={font=\sffamily\small}]
  (tl) edge[double equal sign distance] node [above] {$h(a)$} (tr)
       edge[double equal sign distance] node [left] {$\apfunc{f}(p)$} (bl)
  (bl) edge[double equal sign distance] node [above] {$h(a')$} (br)
  (tr) edge[double equal sign distance] node [right] {$\apfunc{g}(p)$} (br);
  \end{tikzpicture}
\end{center}

\subsubsection*{Squareovers and cubes}
Going up further, we have the type of \emph{squareovers}. A squareover is a square in a dependent type over a square. Suppose that we have a dependent type $P:A\to\type$, a square $s$ in $A$ and a dependent path over each of the sides of the square, as in the following diagram.
\begin{center}
  \begin{tikzpicture}[node distance=10mm]
  \node (tl) at (0,0) {$x_{00}$};
  \node[right = of tl] (tr) {$x_{20}$};
  \node[below = of tl] (bl) {$x_{02}$};
  \node (br) at (tr |- bl)  {$x_{22}$};
  \path[every node/.style={font=\sffamily\small}]
  (tl) edge[double equal sign distance] node [above] {$q_{10}$} (tr)
       edge[double equal sign distance] node [right] {$q_{01}$} (bl)
  (bl) edge[double equal sign distance] node (t) [above] {$q_{12}$} (br)
  (tr) edge[double equal sign distance] node [right] {$q_{21}$} (br);
  \node[below = of bl] (tl) {$a_{00}$};
  \node (tr) at (br |- tl) {$a_{20}$};
  \node[below = of tl] (bl) {$a_{02}$};
  \node (br) at (tr |- bl)  {$a_{22}$};
  \node (m) at ($(tl)!0.5!(br)$) {$s$};
  \path[every node/.style={font=\sffamily\small}]
  (tl) edge[double equal sign distance] node (b) [above] {$p_{10}$} (tr)
       edge[double equal sign distance] node [left] {$p_{01}$} (bl)
  (bl) edge[double equal sign distance] node [above] {$p_{12}$} (br)
  (tr) edge[double equal sign distance] node [right] {$p_{21}$} (br)
  (t)  edge[->, shorten <= 3mm] (b);
  \end{tikzpicture}
\end{center}
We have the type of \emph{squareovers} or \emph{dependent squares}, which fill the top square and lie over the bottom square. We can again define this using multiple methods, but the most convenient method here is to define it as an inductive family. We take as parameters the type $A$, the family $P$ and the points $a_{00}$ and $x_{00}$ and let all the other arguments be indices. We have a ``reflexivity squareover'' when the square $s$ is the reflexivity square and each of the four pathovers are reflexivity pathovers. 

We can also define a type of cubes. Given six squares in a type with twelve paths as sides, fitting together in a cube, we can define the type of fillers of the cube. This is again done using a family of inductive types, where we give a cube filler when all the six sides are reflexivity squares. Of course, we could continue by defining cubeovers and 4-cubes, but we will not need them in this dissertation.

\subsection*{Paths in type formers}

In each of the type formers of \autoref{sec:martin-lof-type} we can compute what the paths in that type are, and what the operations of paths are in that type.

As a simple example, consider the cartesian product type $A\times B$. A path in the cartesian product is just a pair of paths.
$$(x=_{A\times B}y)\simeq(p_1x=_Ap_1y)\times (p_2x=_Bp_2y)$$
In particular, given paths $r:p_1x=p_1y$ and $s:p_2x=p_2y$, we get a path 
$x=y$, which we will denote $(r,s)$. Given maps $f:A\to A'$ and $g:B\to B'$, we get the map $f\times g:A\times B\to A'\times B'$ and we can compute
$$\apfunc{f\times g}(r,s)=(\apfunc{f}(r),\apfunc{g}(s))$$. 
Given families $P,Q:A\to\type$, we can compute transport:
$$\transp^{\lam{a}P(a)\times Q(a)}(p,(x,y))=(\transp^P(p,x),\transp^Q(p,y))$$
Pathovers in a family of cartesian products are also pairs of pathovers:
$$(x,y)=^{\lam{a}P(a)\times Q(a)}_p(x',y')\simeq(x=^P_px')\times (y=^Q_p y').$$

In sigma-types the relations are a bit more difficult, since the second component depends on the first. In the type $(a:A)\times B(a)$ paths are pairs of a path and a path over that path:
$$(x=_{(a:A) \times B(a)}y)\simeq(r:p_1x=_Ap_1y)\times (p_2x=^B_rp_2y)$$
We will also denote in this case the map from right to left by $({-},{-})$. Given a map $f:A\to A'$ and a fiberwise map $g:(a:A)\to B(a)\to B'(f(a))$, we get a functorial action of the sigma type: $f\times g: ((a:A)\times B(a))\to ((a':A')\times B(a'))$. In this case, we can compute
$$\apfunc{f\times g}(r,s)=(\apfunc{f}(r),\apo_g(s)),$$
where $\apo$ is defined in \eqref{eq:apo}.
We leave the rule for transports as an exercise to the reader, but the rule for pathovers in a family of sigma-types is the following. For 
$B:A\to\type$ and $C:(a:A)\to B(a)\to\type$ we get:\footnote{We could define a new notion ``path over a pathover,'' but the rule given here suffices for all the cases we considered.}
$$((a,b)=^{\lam{a}(b:B(a))\times C(a,b)}_p(a',b'))\simeq(q:a=^P_pa')\times (y=^{\lam{x:(a{:}A)\times B(a)}Q(p_1x,p_2x)}_{(p,q)} y').$$

For dependent function types the situation is a bit more complicated. Given $f,g:(a:A)\to
B(a)$, by path induction we get a map
$$\happly:(f=g)\to f\sim g.$$ 
However, we cannot show in plain Martin-L\"of type theory that this map gives
rise to an equivalence. In homotopy type theory we can use the univalence axiom
(see \autoref{sec:equivalences}) to show that $\happly$ is an equivalence. We
skip the proof here, but refer the reader to~\cite[Section 4.9]{hottbook}. Using univalence we can also prove the other properties. The general rule for pathovers in a dependent function type is complicated, but two important special cases are the following. In the first case, the domain does not depend on the path. We have types $A$ and $B$ and a family $C:A\to B\to \type$ and then we can prove:
$$(f=^{\lam{a}(b:B)\to C(a,b)}_pg)\simeq(b:B)\to f(b)=^{C({-},b)}_{p} g(b).$$
The second case is for nondependent functions. Given a type $A$ and two families $B,C:A\to\type$, we have
$$(f=^{\lam{a}B(a)\to C(a)}_pg)\simeq(b:B(a))\to f(b)=^{C}_{p} g(p_*(b)).$$

We characterized paths in the universe in \autoref{sec:equivalences} using the univalence axiom. We will not need to do much path algebra in inductive types, except for the identity type, pathover type and square type. 
A pathover in a family of identity types is a square. Suppose given types $A$ and $B$ and functions $f, g : A \to B$, a path $p:a=_Aa'$ and paths $q:f(a)=g(a)$ and $r:f(a')=g(a')$. Then the pathover type becomes equivalent to the square type shown below.
$$(q=^{\lam{a}f(a)=g(a)}_{p}r)\simeq 
\begin{tikzpicture}[node distance=10mm,baseline=(l.base)]
  \node (tl) at (0,0) {$f(a)$};
  \node[right = of tl] (tr) {$g(a)$};
  \node[below = of tl] (bl) {$f(a')$};
  \node (br) at (tr |- bl)  {$g(a')$};
  \path[every node/.style={font=\sffamily\small}]
  (tl) edge[double equal sign distance] node [above] {$q$} (tr) edge[double
       equal sign distance] node [right] (l) {$\apfunc{f}(p)$} (bl) (bl)
       edge[double equal sign distance] node [above] {$r$} (br) (tr) edge[double
       equal sign distance] node [right] {$\apfunc{g}(p)$} (br);
\end{tikzpicture}$$ 
       
We also sometimes encounter a pathover in a dependent family of
pathovers. In that case we get a squareover. Suppose we are given
functions $f,g:A\to B$, and a homotopy $h:f\sim g$, a dependent family
$C:B\to\type$ and sections $c:(a:A)\to C(f(a))$ and $c':(a:A)\to
C(g(a))$. We want to characterize a pathover in the family
$P\defeq\lam{a}c(a)=^C_{h(a)}c'(a):A\to\type$. If we are also given a
path $p:a=_Aa'$ and two pathovers $q:c(a)=^C_{h(a)}c'(a)$ and
$q':c(a')=^C_{h(a')}c'(a')$, then the pathover $q=^P_pq'$ is equivalent to the
following squareover, where $\apdtilde$ is defined in \eqref{eq:apdtilde}, and the bottom square is a naturality square.

\begin{center}
  \begin{tikzpicture}[node distance=10mm]
  \node (tl) at (0,0) {$c(a)$};
  \node[right = of tl] (tr) {$c'(a)$};
  \node[below = of tl] (bl) {$c(a')$};
  \node (br) at (tr |- bl)  {$c'(a')$};
  \path[every node/.style={font=\sffamily\small}]
  (tl) edge[double equal sign distance] node [above] {$q$} (tr)
       edge[double equal sign distance] node [left] {$\apdtilde_b(p)$} (bl)
  (bl) edge[double equal sign distance] node (t) [above] {$q'$} (br)
  (tr) edge[double equal sign distance] node [right] {$\apdtilde_{b'}(p)$} (br);
  \node[below = of bl] (tl) {$f(a)$};
  \node (tr) at (br |- tl) {$g(a)$};
  \node[below = of tl] (bl) {$f(a')$};
  \node (br) at (tr |- bl)  {$g(a')$};
  \node (m) at ($(tl)!0.5!(br)$) {nat.};
  \path[every node/.style={font=\sffamily\small}]
  (tl) edge[double equal sign distance] node (b) [above] {$h(a)$} (tr)
       edge[double equal sign distance] node [left] {$\apfunc{f}(p)$} (bl)
  (bl) edge[double equal sign distance] node [below] {$h(a')$} (br)
  (tr) edge[double equal sign distance] node [right] {$\apfunc{g}(p)$} (br)
  (t)  edge[->, shorten <= 3mm] (b);
  \end{tikzpicture}
\end{center}
Lastly, we will mention that a pathover in a family of squares is a cube, but we will not explain the details here.

\subsection{Truncated Types}\label{sec:truncatedness}

In HoTT we can define iterated path spaces in any type. In certain types, if we iterate path spaces enough times, these path spaces do not contain any information. These types are called \emph{truncated}. The notion of an $n$-truncated type, was introduced in 2009 by Vladimir Voevodsky under the name ``a type of h-level $n+2$.''

We define the notion that $A$ is $n$-truncated, or that $A$ is an \emph{$n$-type} or $\istrunc{n}(A)$ recursively for $n\geq-2$. We say that a type $A$ is $(-2)$-truncated or \emph{contractible} if it has exactly one inhabitant, i.e. if we can prove
$$(a_0:A)\times (a:A)\to a=a_0.$$
A type $A$ is $(n+1)$-truncated if for all $a\ a':A$ the type $a=_Aa'$ is $n$-truncated. 

We can show that $\unit$ is contractible and that every contractible type is equivalent to $\unit$.

The $(-1)$-truncated types are called \emph{mere propositions} or \emph{propositions} for short. A type $A$ is a proposition precisely when any two of its inhabitants are equal, i.e. if we can prove 
$$(a\ a':A)\to a=a'.$$
We call these types propositions because these types correspond to truth values, and do not contain any further information. In particular, if a proposition is inhabited, then it is contractible.
It is easy to see that $\emptyt$ and $\unit$ are mere propositions, and in \autoref{sec:equivalences} we saw that the statement $\isequiv(f)$ is a mere proposition. 

One level up, the $0$-types are called \emph{sets}. These are the types for which uniqueness of identity proofs holds. Examples of sets are $\N$ and $\bool$.

On the next level we have the $1$-types or \emph{groupoids}. Below we list some properties of truncated types, see~\cite[Section 7.1]{hottbook} for their proofs.
\begin{lem}\mbox{}
\begin{itemize}
\item If $A$ is $n$-truncated, then $A$ is $m$-truncated for all $m\geq n$.
\item If $A$ is $n$-truncated and $A\simeq B$, then $B$ is $n$-truncated.
\item If $A$ and $B$ are $n$-truncated types, then $A\times B$ and $A\simeq B$ are $n$-truncated. 
      If $n\geq0$, then $A+B$ is also $n$-truncated.
\item If $B:A\to\type$ is a family of $n$-truncated types (i.e. $(a:A)\to\istrunc{n}(B(a))$), then $(a:A)\to B(a)$ is $n$-truncated. If moreover $A$ is also $n$-truncated, then $(a:A)\times B(a)$ is also $n$-truncated.
\item Given $a_0:A$, the type $(a:A)\times (a_0=a)$ is contractible.
\item The type $\istrunc{n}A$ is a mere proposition.
\end{itemize}
\end{lem}

We define the \emph{subuniverse of $n$-types} as $\type_{\leq n}\defeq (X:\type)\times\istrunc{n}(X)$. For $X:\type_{\leq n}$ we will also write $X$ for the underlying type of $X$. We write $\prop\defeq \type_{\leq -1}$ and $\set\defeq\type_{\leq0}$.

We can do set-level mathematics in the subuniverse of sets. For example, we can define a \emph{group} to be a set with operations satisfying the following axiomatization:\footnote{From these equalities the fact that $e$ is a left-identity and $i$ is a left-inverse can be derived.}
\begin{align*}
  \group&\defeq(G:\set)\times (m:G\to G \to G)\times (i:G\to G)\times (e:G)\times ((x\ y\ z: G) \to {}\\
  &\mathrel{\hphantom{\defeq}} m(x,m(y,z))=m(m(x,y),z)\times m(x,e)=x\times m(x,i(x))=e).
\end{align*}
A group $G$ is \emph{abelian} if it moreover satisfies $m(x,y)=m(y,x)$ for all $x,y:G$. This gives the usual notion of groups, and we can perform all basic group theory in this setting.

\subsubsection*{Truncations}

We can turn every type $A$ into an $n$-type $\|A\|_n$ in a universal way, which is called the \emph{$n$-truncation} of $A$. It comes with a map $|{-}|_n:A\to\|A\|_n$ and has the following induction principle. Suppose given $P:\|A\|_n\to\type$ such that $P(x)$ is $n$-truncated for all $x:\|A\|_n$. If we are given a dependent map $f:(a:A)\to P(|a|_n)$, we get a section
$$\ind{\|{-}\|}(f):(x:\|A\|_n)\to P(x)$$
such that $\ind{\|{-}\|}(f,|a|_n)\equiv f(a)$.

We will now state some properties of the $n$-truncation, for the proofs we refer to~\cite[Section 7.3]{hottbook}.
\begin{lem}\mbox{}
  \begin{itemize}
  \item The truncation is functorial. Given $f:A\to B$, we get a map $\|f\|_n : \|A\|_n \to \|B\|_n$. This map respects composition and identities: $\|g\o f\|_n\sim \|g\|_n \o \|f\|_n$ and $\|\id_A\|_n\sim \id_{\|A\|_n}$.
  \item $A$ is an $n$-type iff $|{-}|_n : A \to \|A\|_n$ is an equivalence.
  \item The equality type in the truncation is truncated equality, but shifted: $$(|a|_{n+1} =_{\|A\|_{n+1}} = |a'|_{n+1})\simeq \|a =_A a'\|_n.$$
  \item Truncating twice is the same as truncating once: $$\|\|A\|_n\|_k\simeq \|A\|_{\min(n,k)}.$$
  \end{itemize}
\end{lem}

In particular the \emph{propositional truncation} $\|A\|\defeq\|A\|_{-1}$ of $A$ is a proposition stating that $A$ is \emph{merely inhabited}~\cite{awodey2004Propositions}. 
We can use it to define \emph{proof irrelevant} versions of the disjunction or existential quantifier. We have the \emph{mere disjunction} 
\begin{align*}(P\vee Q)&\defeq \|P+Q\|
  \intertext{and the \emph{mere existential}}
(\exists(x:A).P(x))&\defeq \|(x:A)\times P(x)\|.
\end{align*}
We say that there \emph{merely exists} $x:A$ such that $P(x)$ holds if $\exists(x:A).P(x)$ is inhabited, to contrast with constructing an element in the untrucated dependent pair type. If we construct an element of $(x:A)\times P(x)$, we will sometimes say that there \emph{purely exists} an $x$ such that $P(x)$ holds, but often we will drop the adverb \emph{purely}.

\subsubsection*{Connected types}
A type is truncated if the type contains no interesting information in a high enough dimension. Dually, a type is \emph{connected} if it contains no interesting information in a low enough dimension.

We say that a type $A$ is \emph{$n$-connected} for $n\geq-2$ if $\|A\|_n$ is contractible. From the definition we see that every type is $(-2)$-connected. A type is $(-1)$-connected precisely when it is merely inhabited. A type is called 0-connected 
when $A$ has exactly one \emph{connected component}. A 1-connected type is called \emph{simply connected}.

\subsubsection*{Fibers}

We can extend the notion of truncated types and connected types to functions. Given a function $f:A\to B$ and a point $b:B$, we define the \emph{fiber} of $f$ at $b$ to be 
$$\fib_f(b)\defeq(a:A)\times f(a)=b.$$ 
The fiber of the projection $p_1:((a:A)\times B(a))\to A$ at $a:A$ is equivalent to $B(a)$, which explains the terminology that $B(a)$ is the fiber of $B$ over $a$. 


We say that a function $f:A\to B$ is $n$-truncated ($n$-connected) when for all $b:B$ the type $\fib_f(b)$ is $n$-truncated ($n$-connected). The function $f$ is $(-2)$-truncated precisely when it is an equivalence. The function $f$ is $(-1)$-truncated, or an \emph{embedding}, if for all $a\ a':A$ the map $\apfunc{f}:a=_Aa'\to f(a)=_Bf(a')$ is an equivalence. A map $f:A\to B$ between sets is an embedding iff it is \emph{injective}, i.e. if we have a map $f(a)=f(a')\to a=a'$ for all $a\ a':A$. On the other hand, a $(-1)$-connected map is called a \emph{surjection}, which means that for every $b:B$ there merely exists an $a:A$ such that $f(a)=b$.

Every map can be factorized as an $n$-connected map followed by an $n$-truncated map in a unique way, which means that these classes form an \emph{orthogonal factorization system}~\cite{rijke2017modalities}.

Similar to the universe of $n$-truncated types, we have a universe of $n$-connected types: $$\type_{>n}\defeq (A:\type)\times\isconn{n}(A).$$

\subsection{Pointed Types}\label{sec:pointed}
A lot of homotopy theory is done in the $(\infty,1)$-category of pointed types
where the morphisms are maps that preserve the basepoints of the types. Below are the
basic definitions for pointed types.
\begin{defn}\label{def:pointed-types-basic}\mbox{}
  \begin{enumerate}
    \item A type $A$ is \emph{pointed} if $A$ has a distinguished basepoint
    $a_0:A$. For example, $\unit$ is pointed by $\star$ and $\bool$ is pointed
    with $\bfalse$. We will also write $\pbool$ for the pointed type $\bool$. $A\times B$ is pointed if both $A$ and $B$ are
    pointed,\footnote{More formally, we have to specify the basepoint of $A\times
    B$, because being pointed is structure on a type, not a property of the
    type, but there is only one choice of basepoint in this example and other 
    examples where we leave the basepoint implicit.} $(a:A)\to B(a)$ is
    pointed if $B$ is a family of pointed types, and $(a:A)\times B(a)$ is
    pointed if $A$ is pointed and $B(a_0)$ is pointed.
    \item The type of \emph{pointed types} is $\type^*\defeq(A:\type)\times A$.
    Given a pointed type $A:\type^*$, we will also write $A$ for its underlying
    type.
    \item Given two pointed types $A,B:\type^*$, a \emph{pointed map} $f:A\to^*B$
    is a pair consisting of a map $f:A\to B$ and a path $f_0$ stating that $f$
    preserves the basepoint, that is $f_0:f(a_0)=b_0$. The type $A\to^*B$ is
    pointed with basepoint
    $\const\equiv\const_{A,B}\defeq(\lam{a}b_0,\refl_{b_0})$.
    \item We have an identity pointed map $\id\equiv\id_A:A\to^*A$ defined as
    $(\lam{x}x,\refl_{a_0})$ and if $g:B\to^*C$ and $f:A\to^*B$ we have a
    composite $g\o f:A\to^* C$ defined as $(\lam{x}g(f(x)),\mapfunc{g}(f_0)\cdot
    g_0)$.
    \item More generally, Given a pointed type $A:\type^*$ and a family of types
    $B:A\to\type$ with a basepoint $b_0:B(a_0)$, a \emph{pointed dependent map}
    $f:(a:A)\to^*B(a)$ is a pair consisting of a dependent map $f:(a:A)\to B(a)$
    and a path $f_0:f(a_0)=b_0$. If we require that $B$ is a family of pointed
    types, i.e. $B:A\to\type^*$, then $(a:A)\to^* B(a)$ is pointed with
    basepoint $(\lam{a}b_0(a),\refl_{b_0(a_0)}).$
    \item Given two pointed dependent maps $f,g:(a:A)\to^* B(a)$, a \emph{pointed
    homotopy} $h:f\sim^* g$ is a pointed dependent map $(a:A)\to^* f(a)=g(a)$.
    This is well-defined, since the type $f(a_0)=g(a_0)$ is pointed by $f_0\cdot g_0\sy$.
    Expanding the definition, this means that $h$ is a pair of a homotopy
    $h:f\sim g$ and a 2-path stating that $h$ relates the basepoint-preserving
    paths of $f$ and $g$. This means that we have $h_0:h(a_0)=f_0\cdot g_0\sy$, or equivalently,
    $h_0:h(a_0)\cdot g_0=f_0$. We say that a diagram of pointed types commutes
    if there are pointed homotopies between the corresponding composites of
    pointed maps.
    \item A pointed map $e:A\to^* B$ is a \emph{pointed equivalence} if it has a
    left-inverse and a right-inverse. That is, there is $\ell:B\to^* A$ such that
    $\ell\o e\sim^* \id_A$ and $r:B\to^* A$ such that $e\o r\sim^* \id_B$. The
    type of pointed equivalences between $A$ and $B$ is denoted $A\simeq^* B$.
    The identity map is a pointed equivalence and pointed equivalences are
    closed under composition.
    \item Given $A:\type^*$, we define its \emph{loop space} $\Omega A:\type^*\defeq
    (a_0=a_0,\refl_{a_0})$. We define the \emph{iterated loop space} $\Omega^nA$ by iteration as $\Omega^0A\defeq A$ and $\Omega^{n+1}A\defeq \Omega(\Omega^nA)$.
    \item We define the $n$-th homotopy group of $A$ as the set-truncation of the iterated loop space, i.e. $\pi_n(A)\defeq\|\Omega^nA\|_0$. This is a group for $n\geq1$ that is abelian for $n\geq 2$.
    \item Given a pointed map $f:A\to^* B$, we define the \emph{pointed fiber of
    $f$} $\fib_f:\type^*$ as $\fib_f(b_0)\equiv(x:A)\times f(a)=b_0$ with basepoint $(a_0,f_0)$.
    There is a pointed map $p_1:\fib_f\to^*A$ defined as
    $(\lam{x}p_1(x),\refl_{a_0})$.
  \end{enumerate}
\end{defn}
Here are some basic properties of pointed types. We omit the proofs.
\begin{lem}\label{lem:pointed-types-basic}\mbox{}
  \begin{enumerate}
    \item Suppose given a pointed map $f:A\to^* B$. The type of proofs that $f$ is an
    equivalence is equivalent to the type that $f$ is a pointed equivalence. In
    particular, being a pointed equivalence is a property. Also, we can define a
    pointed equivalence $X\simeq^* Y$ by giving a map $e:X\to Y$ that is both
    an equivalence and pointed.
    \item Suppose given $A,B:\type^*$. Univalence implies \emph{univalence for pointed
    types}: the canonical map $(A = B)\to (A\simeq^* B)$ is an equivalence.
    \item Suppose given pointed maps $f,g:(a:A)\to^* B(a)$. Function extensionality
    implies \emph{function extensionality for pointed maps:} the canonical
    map $(f = g) \to (f \sim^* g)$ is an equivalence. 
    \item We have the usual categorical laws: 
    \begin{align*}
      f\o\id &\sim^* f&\id\o f&\sim^* f&(h\o g)\o f&\sim^* h\o (g\o f)\\
      f\o\const &\sim^* \const&\const\o f&\sim^*\const 
    \end{align*}
    The two homotopies showing $\const\o\id\sim^*\const$ are equal. This is also
    true for the two homotopies of $\id\circ\const\sim^*\const$ and of $\id\o\id\sim^*\id$
    and of $\const\o \const\sim^*\const$.
    \item 
    We can form iterated pointed maps $(A\to^* B \to^* C)\defeq (A\to^*
    (B\to^* C))$. To show that such a map preserves the basepoint, we need to
    give an equality between pointed maps, or equivalently, we can give a
    pointed homotopy between pointed maps. For example, the above homotopies
    involving $\const$ imply that we have precomposition and postcomposition maps.
    For $f:A\to^* B$ we have a pointed map $({-})\o f:(B\to^* C)\to^* A \to^* C$
    and for $g:B\to^* C$ we have a pointed map $g\o({-}):(A\to^* B)\to^* A \to^*
    C$. We will also write $f\to C$ resp. $A \to g$ for these maps. 
    Precomposition and postcomposition commute, which means that the
    following square commutes.
    \begin{center}
      \begin{tikzpicture}[node distance=10mm]
      \node (tl) at (0,0) {$(A\to^* B)$};
      \node[right = of tl] (tr) {$(A\to^* B)$};
      \node[below = of tl] (bl) {$(A'\to^* B)$};
      \node (br) at (tr |- bl)  {$(A'\to^* B')$};
      \path[every node/.style={font=\sffamily\small}]
      (tl) edge[->] node [above] {$g\o({-})$} (tr)
           edge[->] node [right] {$({-})\o f$} (bl)
      (bl) edge[->] node [above] {$g\o({-})$} (br)
      (tr) edge[->] node [right] {$({-})\o f$} (br);
      \end{tikzpicture}
    \end{center}
    Moreover, if $f$ or $g$ are constant, then these maps are pointed homotopic
    to constant maps, which gives a pointed map 
    $$({-})\o({-}):(B\to^* C)\to^*(A\to^* B)\to^* A\to^* C.$$ 
    \item\label{item:fiber-composition} 
    There are also dependent versions of these composition maps. In particular,
    if $g:(a:A)\to B(a)\to^* C(a)$, then we have a map 
    $$g\o({-}):((a:A)\to^* B(a))\to^* (a:A)\to^* C(a).$$
    We have an equivalence $$\fib_{g\o({-})}\simeq^* ((a:A)\to^* \fib_{ga}).$$
    \item $\Omega$ and $\Omega^n$ are pointed functors. For $\Omega$ this means
    that given a pointed map $f:A\to^* B$, we can define $\Omega f:\Omega A \to^*
    \Omega B$, with pointed homotopies $\Omega(g\o f)\sim^* \Omega g \o \Omega
    f$ and $\Omega\id\sim^* \id$ and $\Omega\unit\simeq^*\unit$. This also
    implies that $\Omega\const\sim^*\const$ and that if $e:A\simeq^* B$ then
    $\Omega e:\Omega A\simeq^* \Omega B$.
    \item \label{item:pointed-function-extensionality} 
    There is a pointed version of function extensionality for pointed types. 
    If $B$ is a family of pointed types, we have a pointed equivalence
    $$e_B:\Omega((a:A)\to^* B(a))\simeq^* ((a:A)\to^* \Omega B(a)).$$
    This equivalence is natural in $B$. This means that given a fiberwise
    pointed map $f:(a:A)\to B(a)\to^* C(a)$, the following square commutes.
    \begin{center}
      \begin{tikzpicture}
      \node (tl) at (0,0) {$\Omega((a:A)\to^* B(a))$};
      \node[right = of tl] (tr) {$((a:A)\to^* \Omega B(a))$};
      \node[below = of tl] (bl) {$\Omega((a:A)\to^* C(a))$};
      \node (br) at (tr |- bl)  {$((a:A)\to^* \Omega C(a))$};
      \path[every node/.style={font=\sffamily\small}]
      (tl) edge[->] node [above] {$e_B$} (tr)
           edge[->] node [right] {$\Omega(f\o({-}))$} (bl)
      (bl) edge[->] node [above] {$e_C$} (br)
      (tr) edge[->] node [right] {$\Omega f \o ({-})$} (br);
      \end{tikzpicture}
    \end{center}
    \item The fiber of a pointed map is functorial. This means that given a commuting square, we get a pointed map from the fiber of the top map to the fiber of the bottom map.
    \begin{center}
      \begin{tikzpicture}
      \node (tl) at (0,0) {$\fib_f$};
      \node[right = of tl] (t) {$A$};
      \node[right = of t] (tr) {$B$};
      \node[below = of tl] (bl) {$\fib_{f'}$};
      \node (b)  at (t  |- bl)  {$A'$};
      \node (br) at (tr |- bl)  {$B'$};
      \path[every node/.style={font=\sffamily\small}]
      (tl) edge[->] node [above] {$p_1$} (t)
           edge[->, dashed] (bl)
      (t)  edge[->] node [above] {$f$} (tr)
           edge[->] node [right] {$g$} (b)
      (bl) edge[->] node [above] {$p_1$} (b)
      (b)  edge[->] node [above] {$f'$} (br)
      (tr) edge[->] node [right] {$h$} (br);
      \end{tikzpicture}
    \end{center}
    Moreover, if the left and the right sides of the squares are equivalences, then the functorial action is an equivalence. Lastly, $p_1$ is natural, which means that the left square commutes.
    \item Given a pointed map $f:A\to^*B$, we have a equivalence $\Omega\fib_f\simeq^*\fib_{\Omega f}$ that is natural in $f$. This means that if we have a commuting square with top $f$ and bottom $f'$, then the following square commutes (the left and the right side come from the functorial action of $\fib$).
    \begin{center}
      \begin{tikzpicture}
      \node (tl) at (0,0) {$\Omega\fib_f$};
      \node[right = of tl] (tr) {$\fib_{\Omega f}$};
      \node[below = of tl] (bl) {$\Omega\fib_{f'}$};
      \node (br) at (tr |- bl)  {$\fib_{\Omega f'}$};
      \path[every node/.style={font=\sffamily\small}]
      (tl) edge[->] node [above] {$\sim$} (tr)
           edge[->] (bl)
      (bl) edge[->] node [above] {$\sim$} (br)
      (tr) edge[->] (br);
      \end{tikzpicture}
    \end{center}
    \item We have a pointed equivalence $(\pbool\to^* X)\simeq^* X$ natural in $X$.
    \item A pointed type $A$ is $n$-connected iff $\pi_k(A)$ is trivial (contractible) for all $k\leq n$. If a type $A$ is $n$-truncated, then $\pi_k(A)$ is trivial for $k>n$ (however, the converse is not true in general).
  \end{enumerate}
\end{lem}

\subsection{Higher Inductive Types} \label{sec:high-induct-types}

Higher inductive types are a generalization of inductive types where we specify not only the generating points in the type by constructors, but also the generating paths and higher paths. 
The idea is that the type together with its (higher) path spaces are freely generated by these constructors.

A simple example is the \emph{interval}. The interval $I$ is generated by two points $0,1:I$ and a path $\seg:0=1$. Using a syntax similar to that of inductive types, we could write
\begin{inductive}
  \texttt{HIT} $I:\type \defeqp$ \\
  $\bullet\ 0, 1 : I$; \\ 
  $\bullet\ \seg : 0=_I1$.
\end{inductive}
Note that this is not an inductive type, since the last constructor does not specify an element in $I$, but an element in the path space of $I$. We get an induction principle for higher inductive types, similar to the induction principle for inductive types. We first give a special case, the nondependent induction principle, also called the \emph{recursion principle}. For the interval this states the following. Given a type $X$, if we have points $x_0\ x_1:X$ and a path $p:x_0=_Xx_1$, then we get a map $\rec{I}(x_0,x_1,p):I\to X$. On the points this has the expected computation rules: 
$$\rec{I}(x_0,x_1,p,0)\equiv x_0\qquad\text{and}\qquad\rec{I}(x_0,x_1,p,1)\equiv x_1.$$
We want a similar computation rule on paths. We can apply the induction principle to $\seg$ using $\apfunc{}$. The resulting computation rule is
$$\apfunc{\rec{I}(x_0,x_1,p)}(\seg)=p.$$
Note that for this case we postulate a member of the identity type instead of making this a definitional equality. There are various reasons for this. Firstly, in this type theory, there is no justification for this equality to be definitional. There are various ways to define $\apfunc{}$, and there is no good reason for the computation rules to favor this definition. Secondly, in the early proof assistants for HoTT there was no support for definitional computation rules on path constructors, but there was a trick to get it for the point constructors~\cite{licata2011trick}. In fact, calling this rule a ``computation rule'' is not quite accurate, since there is no computation going on. We will still keep using this terminology, so that we have the same terminology as for inductive types. In the cubical type theories mentioned in the introduction we can make these terms reduce judgmentally, making them convenient for working with higher inductive types. 

The induction principle for the interval is the following. Suppose given a family $P:I\to\type$ with elements $x_0:P(0)$ and $x_1:P(1)$. We need to relate $x_0$ and $x_1$ in some way, but we cannot ask that they are equal, since they live in different types. Instead, we require a pathover $p:x_0=^P_{\seg}x_1$. In this case we get a dependent map
$\ind{I}(x_0,x_1,p):(i:I)\to P(i)$ with computation rules on points
$$\ind{I}(x_0,x_1,p,0)\equiv x_0\qquad\text{and}\qquad\ind{I}(x_0,x_1,p,1)\equiv x_1.$$
For the computation rule on paths, we need to use $\apd$ to apply the induction principle to $\seg$, and we get
$$\apd_{\ind{I}(x_0,x_1,p)}(\seg)=p.$$

A more interesting example of a higher inductive type is the \emph{(graph) quotient} which we will call a \emph{quotient} in this dissertation. Given $A : \U$ and $R : A \to A \to \U$, the quotient is the following higher inductive type.
\begin{inductive}
\texttt{HIT} $\quotient_A(R) \defeqp$ \\
$\bullet\ i : A \to \quotient_A(R)$; \\
$\bullet\ \glue : (a\ a' : A) \to R(a,a')\to i(a) = i(a')$.
\end{inductive}
We will sometimes use the notation $[{-}]_0$ for $i$ and $[{-}]_1$ for $\glue$.

A very similar higher inductive type is the \emph{homotopy pushout}, or \emph{pushout} for short. Given two maps $f:A\to B$ and $g:A\to C$, their pushout is the following HIT.
\begin{inductive}
\texttt{HIT} $\pushout(f,g) \defeqp$ \\
$\bullet\ \inl : B \to \pushout(f,g)$ \\
$\bullet\ \inr : C \to \pushout(f,g)$ \\
$\bullet\ \glue : (a : A) \to \inl(f(a)) = \inr(g(a))$
\end{inductive}
We denote $\pushout(f,g)$ by $B+_AC$ if $f$ and $g$ are clear from the context. In this section, we will define other higher inductive types in terms of the pushout. However, we could also start with the quotient, by the following lemma.

\begin{lem}
  The pushout and quotient are interdefinable in MLTT.
\end{lem}
\begin{proof}
  We will only give the definitions of the pushout and the quotient in terms of the other. 
  Showing that these definitions are correct is easy, and we omit it here.

  If we have quotients, we can define the pushout of $f: A \to B$ and $g:B\to C$ as the quotient of
  $B+C$ under the relation $R:B+C\to B+C\to\U$, which is inductively generated by
  $\mk : (a : A) \to R (f(a),g(a))$.

  On the other hand, if we have pushouts, we can define the quotient of $A$ under $R$ as
  follows. Let $T\defeq(a\ a' : A)\times R(a,a')$ be the total space of $R$. Then the quotient of
  $A$ under $R$ is the pushout of $f\defeq\pair{\pi_1}{\pi_2} : T+T \to A$ and
  $g\defeq\pair\id\id:T+T\to T$.
\end{proof}

Many higher inductive types can be defined in terms of the homotopy pushout (or equivalently, the quotient):
\begin{itemize}
\item The \emph{cofiber} of a map $f : A \to B$ is defined as $C_f\defeq B+_A\unit$.
The maps are $f$ and $!$.
\item The \emph{suspension} $\susp A$ of type $A$ is defined as $\susp A \defeq
  \unit+_A\unit$, i.e. as the cofiber of the map $A \to \unit$. The points are
  called $\north$ and $\south$ and $\glue$ is called $\merid$.
\item The \emph{wedge sum} of a family of pointed types $A : I \to \pU$ is defined as the
  cofiber of the map $I \to (i : I) \times A(i)$, which sends $i$ to the pair
  $(i,\pt_{A(i)})$. The binary wedge $A\vee B$ of two pointed types $A\ B:\pU$ can
  equivalently be described as the pushout of $A+_\unit B$ where the maps come
  from the basepoints of $A$ and $B$.
\item The \emph{smash product} $A\wedge B$ of $A$ and $B$ can be defined as the cofiber
  of the map $A\vee B\to A\times B$, which sends $\inl(a)$ to $(a,b_0)$ and $\inr(b)$
  to $(a_0,b)$ and $\glue(\star)$ to $\refl_{(a_0,b_0)}$. We will discuss the smash product in~\cref{sec:smash-product}.
\item The \emph{$n$-sphere} $\S^n$ is defined inductively for $n\geq0$: $\S^0\defeq\bool$ and
  $\S^{n+1}\defeq\susp\S^n$. The $n$-sphere is pointed with point $\north$ for $n\geq1$ and with $\bfalse$ for $n=0$. 
  We could also start counting at $n=-1$, defining $\S^{-1}=\emptyt$, but we often only want to consider the pointed spheres.
\end{itemize}

Another higher inductive type that we will study is the \emph{sequential colimit} or \emph{colimit} for short. This is the following HIT for $A:\N\to\type$ and $f:(n:\N)\to A(n)\to A(n+1)$:
\begin{inductive}
  \texttt{HIT} $\colim(A,f) \defeqp$ \\
  $\bullet\ \iota : (n : \N) \to A(n) \to \colim(A,f)$; \\
  $\bullet\ \kappa : (n : \N) \to (a : A(n)) \to i_{n+1}(f_n(a))=i_n(a)$.
\end{inductive}

We can define $\colim(A,f)$ using quotients, namely as $\quotient(B,R)$ where
  $B=(n : \N)\times A(n)$ is the total space of $A$ and $R:B\to B\to\U$ is inductively generated by
  $\mk:(n : \N)\to(a : A(n))\to R(f_n(a),a)$. We will discuss the colimit more in \autoref{sec:colimits}

We will use the following properties of these higher inductive types. For the proof we refer to~\cite[Chapter 8]{hottbook}
\begin{lem}\mbox{}
\begin{itemize}
\item If $A$ is $n$-connected, then $\Sigma A$ is $(n+1)$-connected.
\item The suspension is left-adjoint to the loop space: $\Sigma\dashv\Omega$. That means that for any two pointed types $A$ and $B$ there is a pointed equivalence
$$(\Sigma A\to^* B)\simeq^*(A\to^*\Omega B)$$
that is natural in $A$ and $B$.
\item We have the following equivalence: $\Omega \S^1\simeq \Z$. Therefore $\S^1$ is1-truncated and $\pi_1(\S^1)\simeq \Z$.
\end{itemize}
\end{lem}
In particular, by the above lemma we know that $\S^n$ is $(n-1)$-connected, and hence that $\pi_k(\S^n)$ is trivial for $k<n$.

Another higher inductive type is the \emph{torus}, which is the following higher inductive type
\begin{inductive}
  \texttt{HIT} $T^2 \defeqp$ \\
  $\bullet\ \star : T^2$; \\
  $\bullet\ \ell_1\ \ell_2 : \star = \star$;\\
  $\bullet\ \ell_1 \cdot \ell_2 = \ell_2 \cdot \ell_1$.
\end{inductive}
The last constructor of the torus is a 2-path constructor. In general, HITs can have as constructor any higher path. We say that a HIT is an \emph{$n$-HIT} if its highest path constructor has dimension $n$. So the torus is a 2-HIT and all the other HITs we have seen are 1-HITs.

Higher inductive types can also have \emph{recursive} constructors. If a higher inductive type has at least one recursive constructor, we will call it a recursive HIT. For example, we can encode the propositional truncation as a HIT with a recursive path constructor:
\begin{inductive}
  \texttt{HIT} $\|A\| \defeqp$ \\
  $\bullet\ |{-}|:A\to\|A\|$; \\
  $\bullet\ (x_1\ x_2 : \|A\|) \to x_1 = x_2$.
\end{inductive}
Higher truncations can also be encoded using HITs~\cite[Section 7.3]{hottbook}.

\section{Lean}\label{sec:lean}

Lean~\cite{moura2015lean} is an interactive theorem prover that is mainly developed at Microsoft
Research and Carnegie Mellon University.\footnote{The contents of this section are based on~\cite{vandoorn2017leanhott}, which was written with Jakob von Raumer and Ulrik Buchholtz.} 
The project was started in 2013 by Leonardo de Moura to
bridge the gap between interactive theorem proving and automated theorem proving. Lean is an
open-source program released under the Apache License 2.0.

In its short history, Lean has undergone several major changes. The second version (Lean~2) supports
two kernel modes. The standard mode is for proof irrelevant reasoning, in which "Prop", the bottom
universe, contains types whose objects are considered to be judgmentally equal. This is incompatible
with homotopy type theory, so there is a second HoTT mode without "Prop". In 2016, the third major
version of Lean (Lean~3) was released~\cite{ebner2017metaprogramming}.
In this version, many components of Lean have been rewritten. Of note, the unification procedure
has been restricted, since the full higher-order unification that is available in Lean~2 can lead
to timeouts and error messages that are unrelated to the actual mistakes. Due to certain design
decisions, such as proof erasure in the virtual machine and a function definition package that
requires axiom K~\cite{goguen2006eliminating}, the homotopy type theory mode is currently not natively
supported in Lean~3. However, a trick found by Gabriel Ebner allows us to build a homotopy type theory library in Lean~3.
In this library, we do not use singleton elimination, which is the feature of "Prop" that is inconsistent with univalence. 
Singleton elimination is the property that some "Prop"-valued inductive types can eliminate to all universe levels. 
Gabriel Ebner also wrote a piece of code that no definition in this library uses singleton elimination in its definition.
Porting the HoTT library from Lean~2 to Lean~3 is a lot of work, because of the changes in the elaborator and in the syntax. 
All major results in this dissertation are only formalized in Lean~2 and not yet in Lean~3.
The HoTT~3 library can be found at \url{https://github.com/gebner/hott3}.

The HoTT kernel of Lean~2 provides the following primitive notions:
\begin{itemize}
\item \emph{Type universes} "Type.{u} : Type.{u + 1}" for each universe level $u \in \mathbb{N}$.
In Lean, this chain of universes is non-cumulative, and all universes are predicative.
\item \emph{Function types} "A → B : Type.{max u v}" for types "A : Type.{u}" and "B : Type.{v}" as well as
\emph{dependent function types} "Πa, B a : Type.{max u v}" for each type "A : Type.{u}" and type family
"B : A → Type.{v}". These come with the usual "β" and "η" rules.
\item \emph{inductive types} and \emph{inductive type families}, as proposed by
Peter Dybjer~\cite{dybjer1994inductive}.
Every inductive definition adds its constructors and dependent recursors to the environment.
Pattern matching is \emph{not} part of the kernel
\item two kinds of \emph{higher inductive types}:
"n"-truncation and (typal) quotients.
\end{itemize}

Outside the kernel, Lean's elaborator uses backtracking search to infer implicit information. It
does the following simultaneously.
\begin{itemize}
\item The elaborator fills in \emph{implicit arguments} that can be inferred from the context,
  such as the type of the term to be constructed and the given explicit arguments. Users mark
  implicit arguments with curly braces. For example, the type of equality is
  "eq : Π{A : Type}, A → A → Type", which allows the user to write "eq a₁ a₂" or "a₁ = a₂" instead
  of "@eq A a₁ a₂". The symbol "@" allows the user to fill in implicit arguments explicitly. The
  elaborator supports both first-order unification and higher-order unification.
\item We can mark functions as \emph{coercions}, which are then ``silently'' applied when needed.
  For example, we have the type of equivalences "A ≃ B", which is a structure consisting of a function
  "A → B" with a proof that the function is an equivalence. The map "(A ≃ B) → (A → B)" is marked as
  a coercion. This means that we can write "f a" for "f : A ≃ B" and "a : A", and the coercion is
  inserted automatically.
\item Lean was designed with \emph{type classes} in mind, which can provide canonical inhabitants of
  certain types. This is especially useful for algebraic structures and for type properties like
  truncatedness and connectedness.  Type class instances can refer to other type classes, so that we
  can chain them together. This makes it possible for Lean to automatically infer why types are
  "n"-truncated if our reasoning requires this, for example when we are eliminating out of a
  truncated type. For example we show that the type of functors between categories "C" and "D" is
  equivalent to an iterated sigma type.
  \begin{lstlisting}[gobble=2]
  (Σ (F₀ : C → D) (F₁ : Π {a b}, hom a b → hom (F₀ a) (F₀ b)),
  (Π (a), F₁ (ID a) = ID (F₀ a)) ×
  (Π {a b c} (g : hom b c) (f : hom a b),
    F₁ (g ∘ f) = F₁ g ∘ F₁ f)) ≃ functor C D
  \end{lstlisting}
  Note the use of coercions here: "F₀ : C → D" really means a function from the objects of "C" to
  the objects of "D". From this equivalence, Lean's type class inference can automatically infer
  that "functor C D" is a set if the objects of "D" form a set. Type class inference will repeatedly
  apply the rules when sigma-types and pi-types are sets, and use the facts that hom-sets are sets
  and that equalities in sets are sets (in total 20 rules are applied for this example).
\item Instead of giving constructions by explicit terms, we can also make use of
Lean's \emph{tactics}, which give us an alternative way to construct terms step by
step. This is especially useful if the proof term is large, or if the elaboration relies heavily on
higher-order unification.
\item We can define custom syntax, including syntax with binding.
In the following example we declare two custom notations.
\begin{lstlisting}
infix ⬝ := concat
notation `Σ` binders `, ` r:(scoped P, sigma P) := r
\end{lstlisting}
The first line allows us to write "p ⬝ q" for path concatenation "concat p q". The second line
allows us to write "Σ x, P x" instead of "sigma P". This notation can also be chained:
"Σ (A : Type) (a : A), a = a" means "sigma (λ(A : Type), sigma (λ(a : A), a = a))".
\end{itemize}

All main results in this dissertation have been formalized in Lean. Some corollaries or examples have not been formalized, in which case we will explicitly mention this. 
The formalizations are separated in two Github repositories: the Lean-HoTT library\footnote{\url{https://github.com/leanprover/lean2/blob/master/hott/hott.md}} 
and the ``spectral'' repository, which was originally a repository to formalize spectral sequences, but now also contain many other results in synthetic homotopy theory.\footnote{\url{https://github.com/cmu-phil/Spectral/}}

Below is a table with the locations of the formal results in the libraries.
\begin{xcenter}
  {\footnotesize 
\begin{tabular}{|c|c|c|} \hline
  \textbf{Theorem} & \textbf{File} & \textbf{Name} \\ \hline
\autoref{thm:main} & \texttt{hott/hit/prop\_trunc.hlean} & \texttt{ptrunc\_equiv\_trunc} \\ 
\autoref{thm:simple-two-quotient-def} & \texttt{hott/hit/two\_quotient.hlean} & \texttt{simple\_two\_quotient.rec} \\ 
\autoref{thm:colim_sm} & \texttt{Spectral/colimit/seq\_colim.hlean} & \texttt{sigma\_seq\_colim\_over\_equiv} \\
\autoref{cor:eq_colim} & \texttt{Spectral/colimit/seq\_colim.hlean} & \texttt{seq\_colim\_eq\_equiv} \\
\autoref{thm:les-homotopy} & \texttt{hott/homotopy/LES\_of\_homotopy\_groups.hlean} & \texttt{is\_exact\_LES\_of\_homotopy\_groups} \\
\autoref{cor:homotopy-group-spheres-1} & \texttt{hott/homotopy/sphere2.hlean} & \texttt{$\pi$2S2} and \texttt{$\pi$nS3\_eq\_$\pi$nS2} \\
\autoref{cor:homotopy-group-spheres-2} & \texttt{hott/homotopy/sphere2.hlean} & \texttt{$\pi$nSn} and \texttt{$\pi$3S2} \\
\autoref{thm:EM-equiv-categories} & \texttt{Spectral/homotopy/EM.hlean} & \texttt{AbGrp\_equivalence\_cptruncconntype${}^\prime$} \\
\autoref{thm:smash-adjoint} & \texttt{Spectral/homotopy/smash\_adjoint.hlean} & \texttt{smash\_adjoint\_pmap} \\
\autoref{thm:exact-couple-convergence} & \texttt{Spectral/algebra/spectral\_sequence.hlean} & \texttt{is\_built\_from\_infpage} \\
\autoref{thm:spectral-sequence-spectrum} & \texttt{Spectral/algebra/spectral\_sequence.hlean} & \texttt{converges\_to\_sequence} \\
\autoref{thm:atiyah-hirzebruch-reduced} & \texttt{Spectral/cohomology/serre.hlean} & \texttt{atiyah\_hirzebruch\_convergence} \\
\autoref{thm:serre-spectral-sequence} & \texttt{Spectral/cohomology/serre.hlean} & \texttt{serre\_convergence} \\
\hline
\end{tabular}
  }
\end{xcenter}

\chapter{Higher Inductive Types}\label{cha:high-induct-types}

In this chapter we will study properties of Higher Inductive Types (HITs), which
we introduced in \Cref{sec:high-induct-types}. There is no uniformly accepted
scheme of which HITs are allowed, and the semantics of HITs is a topic of
current research. There are semantic interpretations of a large class of Higher
Inductive Types~\cite{lumsdaine2017HITsemantics}, but there are still open
questions. Firstly, a general scheme for higher inductive types is unknown,
although~\cite{altenkirch2016qiits} is a step in the right direction. Secondly,
it is unknown whether universes can be closed under higher inductive types. 
This is unknown even in the case for homotopy pushouts. In this chapter, we do
not study the semantics of higher inductive types. Instead, we will work
internally in a type theory that has some specific HITs, and construct other
HITs from the ones we started with.

In particular, we are interested in the case where we start with the quotient, or equivalently, the homotopy pushout.

One HIT from \Cref{sec:high-induct-types} that we have not yet defined using
quotients is the $n$-truncation. In \Cref{sec:prop-trunc} we will define the
propositional truncation using quotients. A construction of the $n$-truncations is given by the join
construction~\cite{rijke2017join}. This shows that we can define certain
recursive HITs using quotients. We will make a start on defining a bigger class of recursive HITs using quotients in \cref{sec:colimits}. 

Another class of HITs we want to construct is HITs with higher path constructors. We construct
nonrecursive 2-HITs in \Cref{sec:non-recursive-2}, using a method very similar to the hubs and
spokes method~\cite[Section 6.7]{hottbook}.

One might wonder after these examples whether all HITs can be reduced to
quotients. This turns out to be false. In~\cite[Section
9]{lumsdaine2017HITsemantics} the authors describe a specific recursive 1-HITs
that cannot be reduced to quotients. Still, it is worthwhile to see which
higher inductive types can be constructed from quotients, for example if one is
interested in a model of HoTT with homotopy pushouts, but without the extra
structure to model all HITs.

\section{Propositional Truncation}\label{sec:prop-trunc}

In this section we will construct the propositional truncation from quotients.\footnote{The contents of this section have been published in~\cite{vandoorn2016proptrunc}. However, \autoref{cor:prop-trunc-univ-set} is new.}

Given a type $A$, define $\{A\}$ as the quotient of $A$ by the indiscrete relation
$R\defeq\l(a\ a':A),\unit$. We will call the type $\{A\}$ the \emph{one-step truncation}, since
repeating it will give the propositional truncation. We will denote its point constructor by $f : A \to \{A\}$ and its path constructor by $e : (x\ y : A) \to f(x) = f(y)$. We call a function $g : A \to B$ \emph{weakly constant} if $(x\ y : A) \to g(x) = g(y)$ is inhabited. Note that maps $\{A\}\to B$ correspond exactly to weakly constant maps $A\to B$.

Given a type $A$, we define a sequence
$\{A\}_{-}:\N\to\U$ by
\begin{align}
\begin{aligned}
\{A\}_0&:\equiv A\\ \{A\}_{n+1}&:\equiv \{\{A\}_n\}
\end{aligned}
\label{eq:An}
\end{align}
We have map $f_n:\equiv f : \{A\}_n\to \{A\}_{n+1}$, which is the constructor of the one-step
truncation. This gives the sequence
\begin{equation}
A\xrightarrow{f}\{A\}\xrightarrow{f}\{\{A\}\}\xrightarrow{f}\cdots \label{eq:prop-sequence}
\end{equation}
We define $\{A\}_\infty=\colim(\{A\}_{-},f_{-})$. We will prove that $\{A\}_\infty$ is the
propositional truncation of $A$, in the sense that the construction $A\mapsto \{A\}_\infty$ has the
same formation, introduction, elimination and computation rules for the propositional truncation.

We have already shown the formation rule of the propositional truncation (note that $\{A\}_\infty$
lives in the same universe as $A$).

We also easily get the point constructor of the propositional truncation, because that is just the
map $i_0:A\to \{A\}_\infty$. The path constructor $(x, y : \{A\}_\infty) \to x = y$, i.e. the
statement that $\{A\}_\infty$ is a mere proposition, is harder to define. We will postpone this
until after we have defined the elimination and computation rules.

The elimination principle --- or induction principle --- for the propositional truncation is the
following statement. Suppose we are given a family of propositions $P : \{A\}_\infty \to \prop$ with
a section $h : (a : A) \to P(i_0(a))$. We then have to construct a map
$k : (x : \{A\}_\infty) \to P(x)$. To construct $k$, take an $x : \{A\}_\infty$. Since $x$ is in a
colimit, we can apply induction on $x$. Notice that we construct an element in $P(x)$, which is a
mere proposition, so we only have to define $k$ on the point constructors. This means that we can
assume that $x\equiv i_n(a)$ for some $n : \N$ and $a:\{A\}_n$. Now we apply induction on $n$.

If $n\equiv0$, then we can choose $k(i_0(a)):\equiv h(a):P(i_0(a))$.

If $n\equiv \ell+1$ for some $\ell:\N$, we know that $a:\{\{A\}_\ell\}$, so we can induct on
$a$. The path constructor of this induction is again automatic. For the point constructor, we can
assume that $a\equiv f(b)$. In this case we need to define $k(i_{\ell+1}(f(b))) :
P(i_{\ell+1}(f(b)))$. By induction hypothesis, we have an element $y : P(i_\ell(b))$. Now we can
transport $x$ along the equality $(g_\ell(b))\sy : i_\ell(b)=i_{\ell+1}(f(b))$. This gives the
desired element in $P(i_{\ell+1}(f(b)))$.

We can write the proof in pattern matching notation:
\begin{itemize}
\item $k(i_0(a)):\equiv h(a)$
\item $k(i_{n+1}(f_n(a))):\equiv(g_n(b))_*\sy(k(i_n(b)))$
\end{itemize}
The definition $k\ (i_{0}\ a) :\equiv h\ a$ is also the judgmental computation rule for the point constructors of the propositional truncation.

For the remainder of this section we will prove that $\{A\}_\infty$ is a mere proposition. We will
need the following two lemmas.

\begin{lem}\label{lem:pieq}
  Let $X$ be a type with $x : X$. Then the type $(y:X) \to x=y$ is a mere proposition.
\end{lem}
\begin{proof}
  To prove that $(y:X) \to x=y$ is a mere proposition, we assume that it is inhabited and show that
  it is contractible. Let $f : (y:X) \to x=y$. From this, we conclude that $X$ is contractible
  with center $x$. Now given any $g : (y:X) \to x=y$, we know that $f$ and $g$ are pointwise
  equal, because their codomain is contractible. By function extensionality we conclude that $f=g$,
  finishing the proof.
\end{proof}
\begin{lem}\label{lem:apconstant}
  If $g:X\to Y$ is weakly constant, then for every $x, x' : X$, the function $\text{ap}_g:x=x'\to
  g(x)=g(x')$ is weakly constant. That is, $\ap gp=\ap gq$ for all $p,q:x=x'$.
\end{lem}
\begin{proof}
  Let $q : (x, y : X) \to g(x)=g(y)$ be the proof that $g$ is weakly constant, and fix $x : X$. We
  first prove that for all $y : X$ and $p : x = y$ we have
  \begin{equation}\label{eq:apconstanteq}
    \ap{g}{p} = q(x,x)^{-1}\ \cdot\ q(x,y).
  \end{equation}
  This follows from path induction, because if $p$ is reflexivity, then $\ap{g}{\refl_x} \equiv
  \refl_{g(x)} = q(x,x)^{-1}\ \cdot\ q(x,x)$. The right hand side of \eqref{eq:apconstanteq} does not
  depend on $p$, hence $\text{ap}_g$ is weakly constant.
\end{proof}

To prove that $\{A\}_\infty$ is a mere proposition, we need to show
$(x, y : \{A\}_\infty) \to x=y$. Since $(y : \{A\}_\infty) \to x=y$ is a mere proposition, we can use the
induction principle for the propositional truncation on $x$, which we have just proven for
$\{A\}_\infty$. This means we only have to show that for all $a : A$ we have
$(y : \{A\}_\infty) \to i_0(a)=y$. We do not know that $i_0(a)=y$ is a mere proposition,\footnote{Of
  course, we do know that it is a mere proposition after we have finished the proof that
  $\{A\}_\infty$ is a mere proposition.} so we will just use the regular induction principle for
colimits on $y$. We then have to construct two inhabitants of the following two types:
\begin{enumerate}
\item For the point constructor we need $p(a,b) : i_0(a) = i_n(b)$ for all $a : A$ and $b :
  \{A\}_n$.
\item We have to show that $p$ respects path constructors:
  \begin{equation}
    p(a,f(b))\ \cdot\ g(b) = p(a,b).\label{eq:coh}
  \end{equation}
\end{enumerate}

We have a map $f^n:A \to \{A\}_n$ defined by induction on $n$, which repeatedly applies $f$. We also
have a path $g^n(a): i_n(f^n(a)) = i_0(a)$, which is a concatenation of instances of $g$.

\begin{figure}\begin{center}
\begin{tikzpicture}[node distance=3cm,font=\footnotesize,
  thick,main node/.style={font=\large}]

  \node[main node,label=left:$a$] (a) at (0,0) {$\bullet$}; \node[main node,label=left:$f^n(a)$]
  (fa) [above of=a] {$\bullet$}; \node[main node,label=left:$f^{n+1}(a)$] (ffa) [above of=fa]
       {$\bullet$}; \tikzset{node distance=2.5cm}; \node[main node,label=right:$b$] (b) [right
         of=fa] {$\bullet$}; \node[main node,label=right:$f(b)$] (fb) [right of=ffa] {$\bullet$};
       \tikzset{node distance=2cm}; \node (Am1) [right of=fb] {$\{A\}_{n + 1}$}; \node (Am) [right
         of=b] {$\{A\}_n$}; \node (An) [right of=a] {$A$}; \draw ($(ffa)!0.4!(fb)$) ellipse (2.7cm
       and 1cm); \draw ($(fa)!0.4!(b)$) ellipse (2.4cm and 1cm); \draw (a) circle (1cm); \path (fa) edge node [left] {$g^n$} (a) edge node [right]
       {$g$} (ffa) (a) edge [bend left=60] node [above left] {$g^{n+1}$} (ffa) (fb) edge [bend
         right=20] node [above] {$e$} (ffa) edge node [left] {$g$} (b);
\end{tikzpicture}
\caption{The definition of $p$. The applications of $i$ and the arguments of the paths are
  implicit.}
\label{fig:defp}
\end{center}\end{figure}
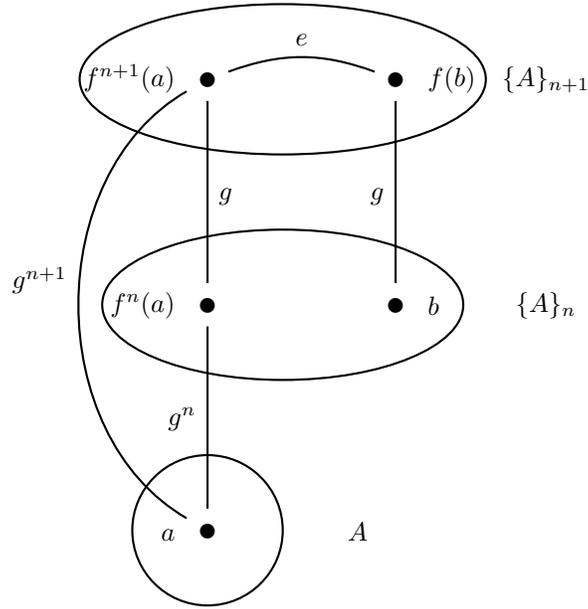

We can now define $p(a,b)$ as displayed in \Cref{fig:defp}, which is the concatenation
\begin{align*}
i_0(a) &= i_{n+1}(f^{n+1}(a)) &&\text{(using $g^{n+1}$)}\\ &\equiv i_{n+1}(f(f^n(a)))\\ &=
i_{n+1}(f(b)) &&\text{(using $e$)}\\ &= i_n(b) &&\text{(using $g$)}
\end{align*}

Note that by definition $g^{n+1}(a) \equiv g(f^n(a))\ \cdot\ g^n(a)$, so the triangle on the left of
\Cref{fig:defp} is a definitional equality.

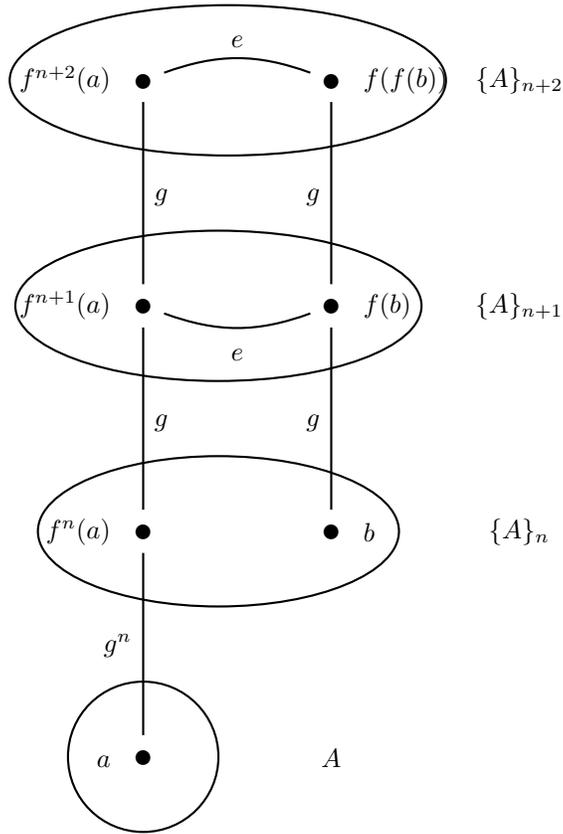
\begin{figure}\begin{center}
\begin{tikzpicture}[node distance=3cm,font=\footnotesize,
  thick,main node/.style={font=\large}]

  \node[main node,label=left:$a$] (a) at (0,0) {$\bullet$}; \node[main node,label=left:$f^n(a)$]
  (fa) [above of=a] {$\bullet$}; \node[main node,label=left:$f^{n+1}(a)$] (ffa) [above of=fa]
       {$\bullet$}; \node[main node,label=left:$f^{n+2}(a)$] (fffa) [above of=ffa]{$\bullet$};
       \tikzset{node distance=2.5cm}; \node[main node,label=right:$b$] (b) [right of=fa]
               {$\bullet$}; \node[main node,label=right:$f(b)$] (fb) [right of=ffa] {$\bullet$};
               \node[main node,label=right:$f(f(b))$] (ffb) [right of=fffa]{$\bullet$}; \node (Am2)
                    [right of=ffb]{$\{A\}_{n + 2}$}; \node (Am1) [right of=fb] {$\{A\}_{n + 1}$};
                    \node (Am) [right of=b] {$\{A\}_n$}; \node (An) [right of=a] {$A$}; \draw
                    ($(fffa)!0.45!(ffb)$) ellipse (2.9cm and 1cm); \draw ($(ffa)!0.4!(fb)$) ellipse
                    (2.7cm and 1cm); \draw ($(fa)!0.4!(b)$) ellipse (2.4cm and 1cm); \draw (a)
                    circle (1cm); \path (fa) edge node
                    [left] {$g^n$} (a) edge node [right] {$g$} (ffa) (fb) edge [bend left=20] node
                    [below] {$e$} (ffa) edge node [left] {$g$} (b) edge node [left] {$g$} (ffb)
                    (fffa) edge node [right] {$g$} (ffa) edge [bend left=20] node [above] {$e$}
                    (ffb);
\end{tikzpicture}
\caption{The coherence condition for $p$. The applications of $i$ and the arguments of the paths are
  implicit.}
\label{fig:cohp}
\end{center}\end{figure}

Now we have to show that this definition of $p$ respects the path constructor of the colimit, which
means that we need to show~\eqref{eq:coh}. This is displayed in \Cref{fig:cohp}.  We only need to
fill the square in \Cref{fig:cohp}. To do this, we first need to generalize the statement,
because we want to apply path induction. Note that if we give the applications of $i$ explicitly,
the bottom and the top of this square are $$\ap i{e(f^{n+1}(a),f(b))}$$ and $$\ap
i{e(f^{n+2}(a),f(f(b)))},$$ respectively. This means we can apply the following lemma to prove this
equality.

\begin{figure}\begin{center}
\begin{tikzpicture}[node distance=3cm,font=\footnotesize,
  thick,main node/.style={font=\large, 
  }]

  \node[main node,label=left:$i(x)$] (x) at (0,0) {$\bullet$}; \node[main node,label=left:$i(f(x))$]
  (fx) [above of=x] {$\bullet$}; \tikzset{node distance=2.5cm}; \node[main node,label=right:$i(y)$]
  (y) [right of=x] {$\bullet$}; \node[main node,label=right:$i(f(y))$] (fy) [right of=fx]
       {$\bullet$}; \node (Ak1) [right of=fy] {$\{A\}_{n + 1}$}; \node (Ak) [right of=y]
       {$\{A\}_n$}; \draw ($(fx)!0.5!(fy)$) ellipse (2.8cm and 1cm); \draw ($(x)!0.5!(y)$) ellipse
       (2.6cm and 1cm); \path (x) edge [bend right=20]
       node [below] {$\ap ip$} (y) edge node [right] {$g$} (fx) (fy) edge [bend right=20] node
       [above] {$\ap i{p'}$} (fx) edge [bend left=20] node [below] {$\ap i{\ap fp}$} (fx) edge node
       [left] {$g$} (y);
\end{tikzpicture}
\caption{The situation in \Cref{lem:cohplemma}.}
\label{fig:cohplemma}
\end{center}\end{figure}

\begin{lem}\label{lem:cohplemma}
  Suppose we are given $x\ y : \{A\}_n$, $p:x=y$ and $p' : f(x) = f(y)$. Then we can fill the outer
  square in \Cref{fig:cohplemma}, i.e.
    $$g(x)\ \cdot\ \ap ip = \ap i{p'}\ \cdot\ g(y).$$
\end{lem}
\begin{proof}
We can fill the inner square of the diagram by induction on $p$, because if $p$ is reflexivity, then
the inner square reduces to
  $$g(x)\ \cdot\ \refl_{i(x)}=\refl_{i(f(x))}\ \cdot\ g(x).$$ To show that the two paths in the top
are equal, first note that $i_{k} : \{A\}_{k} \to \{A\}_\infty$ is weakly constant. To see this,
look at \Cref{fig:defp}. The path from $f^n(a)$ to $b$ in that figure gives a proof of
$i_n(f^n(a))=i_n(b)$ that does not use the form of $f^n(a)$, so we also have $i_k(u)=i_k(v)$ for
$u,v :\{A\}_k$. Since $i_{n+1}$ is weakly constant, by \Cref{lem:apconstant} the
function $$\text{ap}_{i_{n+1}} : f(x)=f(y)\to i_{n+1}(f(x))=i_{n+1}(f(y))$$ is also weakly
constant. This means that the two paths in the top are equal, proving the Lemma.
\end{proof}
We have now given the proof of the following theorem:

\begin{thm}\label{thm:main}
The map $A\mapsto \{A\}_\infty$ satisfies all the properties of the propositional truncation $\|{-}\|$, including the universe level and judgmental computation rule on point constructors.
\end{thm}

We will mention two corollaries of this result. An alternate proof of the first one is given in~\cite{kraus2014anonymousexistence}.
\begin{cor}\label{c:hstable}
  Given a weakly constant function $h : A \to A$, there is a function $\|A\|\to A$.
\end{cor}
\begin{proof}
  The weakly constant function $h$ gives a function $\tilde h : \{A\}\to A$. The HIT $\{{-}\}$ is
  functorial (just like all other HITs), so by its functorial action we get a map $\{\tilde
  h\}:\{\{A\}\}\to\{A\}$, which we can compose with $\tilde h$ to get a map $\{\{A\}\}\to A$. By
  induction on $n$ we get a map $k_n : \{A\}_n \to A$. Formally, we define
  \begin{align*}
  k_0(a)&:\equiv a\\ k_{n+1}(x)&:\equiv \tilde h(\{k_n\}(x))
  \end{align*}
  However, this sequence of maps does not form a cocone, because the triangles do not commute. (For
  example for the first triangle we have to show $h(a)=a$ for all $a$.) But we can easily modify the
  definition by postcomposing with $h$. Define $h_n:\equiv h\circ k_n : \{A\}_n\to A$. Now we get a
  cocone; all triangles commute because $h$ is weakly constant. By the universal property of the sequential colimit we get
  a map $\|A\|\to A$.
\end{proof}

We can also construct maps out of the propositional truncation into a set by giving a weakly constant function. 
An alternate proof was given in~\cite{kraus2014universalproperty}.
\begin{cor}\label{cor:prop-trunc-univ-set}
  Suppose given a weakly constant function $g:A\to B$ where $B$ is a set. Then there is a map $\tilde g:\|A\|\to B$ such that $\tilde g(|a|)=g(a)$.
\end{cor}
\begin{proof}
  First note that given any map $h:\{X\}\to B$, we get a map $h':\{\{X\}\}\to B$ such that $h'\o f\sim h$. 
  Namely, on point constructors we define $h'(f(x))\defeq h(x)$ for $x:\{X\}$. 
  Now given $x',y':\{X\}$, we want to define $h'$ on $e(x',y')$. We perform induction on both $x$ and $y$. 
  In the case that both $x$ and $y$ are point constructors, $x'\equiv f(x)$ and $y'\equiv f(y)$ we can define 
  $$\apfunc{h'}(e(x',y'))\defeqp\apfunc{h}(e(x,y)):h'(f(x'))\equiv h(f(x))=h(f(y))\equiv h'(f(y')).$$
  In the other three cases, we are constructing a 2-path (or 3-path) in $B$, which is automatically filled because $B$ is a set. 
  This finishes the construction of $h'$, which satisfies $h'\o f\sim h$ by definition.

  Now we can define a cocone $g_n:\{A\}_n\to B$ as follows. $g_0$ and $g_1$ are given by $g$. 
  We now define $g_{n+2}\defeq g_{n+1}'$. These $g$'s form a cocone because $g_{n+2}'\o f\sim g_{n+1}$. 
  This gives a map $\tilde g:\|A\|\to B$ such that $\tilde g(|a|)=g(a)$.
\end{proof}

An alternative construction of the propositional truncation using non-recursive HITs has been given in~\cite{kraus2016hits}. All results in this section have been fully formalized.

\section{Non-recursive 2-HITs}\label{sec:non-recursive-2}

We can also define nonrecursive 2-HITs using quotients.\footnote{A summary of this section also appeared in~\cite{vandoorn2017leanhott}.} There are various 2-HITs we would like to construct, such as the torus (as formulated in \autoref{sec:high-induct-types}), groupoid quotients, and Eilenberg-MacLane spaces $K(G,1)$. The construction of 2-HITS uses a method similar to the hubs-and-spokes method described in~\cite[Sect.~6.7]{hottbook}.

The idea behind the hubs-and-spokes method is that for any path $p : x =_A x$ we can define a map $f : \S^1 \to A$
with $\apfunc{f}(\lp) = p$ by circle induction. Then we can prove the equivalence
$$(p = 1) \simeq (x_0 : A) \times (z : \S^1) \to f(z) = x_0.$$ 
This equivalence informally
states that filling in a loop is the same as adding a new point $x_0$, the \emph{hub}, and
\emph{spokes} $f(z) = x_0$ for every $z : \S^1$, similar to the spokes in a wheel.
This means that in a higher inductive type, we can replace a 2-path constructor $p = 1_x$ by
a new point constructor $x_0 : A$ and a family of 1-path constructors $(z : \S^1) \to f(z) = x_0$. 2-path constructors of the form $p = q$ can be replaced by the equivalent path constructor $p \cdot q\sy = 1$.

This construction reduces certain 2-HITs to 1-HITs. However, this reduction is not a quotient, since this family of path constructors refers to other path constructors (in the definition of $f$), which is not allowed in quotients. If we use quotients, we need to take the quotient twice. We first define a quotient with only the 1-paths (and the hubs), and then use another quotient to add the spokes. In this section we will describe this construction of 2-HITs from quotients. 

To be more formal, let us first prove a slightly more general version of the above equivalence.
\begin{lem}\label{lem:twopathlemma} 
  Given a path $p : a =_A a$ and $f : A \to B$, we have an equivalence
  $$e:((b_0 : B) \times (z : \S^1) \to f(\circlerec(a,p,z)) = b_0) \simeq (\apfunc{f}(p) = 1),$$
  where $\circlerec : (y : P)\to y = y \to \S^1 \to P$ is the nondependent eliminator of the circle $\S^1$. 
\end{lem}
\begin{proof}
This follows from the following chain of equivalences.
\begin{align*}
  &\mathrel{\hphantom{\simeq}}(b_0 : B) \times (z : \S^1) \to f(\circlerec(a,p,z)) = b_0 \\
  &\simeq (b_0 : B) \times (q : f(a) = b_0) \times q =_{\lp}^{f(\circlerec(a,p,{-}))=b_0} q\\
  &\simeq 1_{f(a)} =_{\lp}^{f(\circlerec(a,p,{-}))=f(a)} 1_{f(a)}\\
  &\simeq \apfunc{f\circ\circlerec(a,p)}(\lp) = \apfunc{\const_{f(a)}}(\lp)\\
  &\simeq \apfunc{f}(p) = 1.\qedhere
\end{align*}
\end{proof}

More formally, for $A:\type$ and $R:A\to A\to\type$ we will define words in $R$ to be the following inductive family of types:
\begin{inductive}
\texttt{inductive} $\words_R : A \to A \to \type \defeqp$ \\
$\bullet\ [{-}] : \{a\ a' : A\} \to R(a,a') \to \words_R(a,a');$\\
$\bullet\ \langle{-}\rangle : \{a\ a' : A\} \to a = a' \to \words_R(a,a');$ \\
$\bullet\ {-}\sy:\{a\ a' : A\} \to \words_R(a,a') \to \words_R(a',a);$; \\
$\bullet\ {-} \cdot {-}: \{a_1\ a_2\ a_3 : A\} \to \words_R(a_1,a_2) \to \words_R(a_2,a_3) \to \words_R(a_1,a_3).$
\end{inductive}
A \emph{specification for a \emph(nonrecursive\emph) 2-HIT} consists of a type $A$ and two families
$R : A \to A \to \type$ and $S : \{a\ a' : A\} \to \words_R(a,a') \to \words_R(a,a') \to \type$.
Using this, we define the 2-HIT $\twoquotient(A,R,S)$ with constructors
\begin{inductive}
\texttt{HIT} $\twoquotient(A,R,S) \defeqp$ \\
$\bullet\ [{-}]_0 : A \to \twoquotient(A,R,S);$\\
$\bullet\ [{-}]_1 : \{a\ a' : A\} \to R(a,a') \to [a]_0=[a']_0;$ \\
$\bullet\ [{-}]_2:\{a\ a' : A\} \to \{t\ t' :\words_R(a,a')\} \to S(t,t') \to
\overline{[t]_1}=\overline{[t']_1}$.
\end{inductive}
where $\overline{[t]_1}$ is the action of $[{-}]_1$ on words in $R$. So if $t : \words_R(a,a')$, then $\overline{[t]_1}:[a]_0=[a']_0$ is defined by recursion over $t$. For example, the recursive steps for concatenation of words is 
$$\overline{[t_1 \cdot t_2]_1} \defeq \overline{[t_1]_1} \cdot \overline{[t_2]_1}.$$

Before we define $\twoquotient(A,R,S)$, we first define a special case with only reflexivities on the right hand side of 2-path constructors. This is the following HIT, where $Q$ has type
$\{a : A\} \to \words_R(a,a) \to \type$.
\begin{inductive}
\texttt{HIT} $\simpletwoquotient(A,R,Q) \defeqp$ \\
$\bullet\ [{-}]_0 : A \to \simpletwoquotient(A,R,Q);$\\
$\bullet\ [{-}]_1 : \{a\ a' : A\} \to R(a,a') \to [a]_0=[a']_0;$ \\
$\bullet\ [{-}]_2:\{a : A\} \to \{t :\words_R(a,a)\} \to Q(t) \to
\overline{[t]_1}=1$.
\end{inductive}
To define this, we first define a new type where we add a \emph{hub} to $A$ for every path specified by $Q$.
$$B \defeq A + (a : A) \times (t : \words_R(a,a)) \times Q t.$$
Then we quotient this as specified by $R$, to obtain the 1-paths.
$$C \defeq \quotient_B(R_B),$$
where the inductive family of types $R_B$ is defined as follows.
\begin{inductive}
\texttt{inductive} $R_B : B \to B \to \type \defeqp$ \\
$\bullet\ \{a\ a' : A\} \to R(a, a') \to R_B(\inl a,\inl a').$
\end{inductive}
We now define $D\defeq\simpletwoquotient(A,R,Q') \defeq \quotient_C(R_C)$ where $R_C$ is defined as the following inductive family of types (we write $u_t\defeq\circlerec([\inl a]_0,\overline{[t]_1}):\S^1\to C$)
\begin{inductive}
\texttt{inductive} $R_C : C \to C \to \type \defeqp$ \\
$\bullet\ \{a : A\} \to \{t : \words_R(a,a)\} \to (q : Q(t)) \to (x : \S^1) \to R_C(u_t(x),[\inr (a,t,q)]_0).$
\end{inductive}
We will now define the expected constructors, eliminators, and computation rules for this two-quotient.
\begin{thm}\label{thm:simple-two-quotient-def}
  The type $D\defeq\simpletwoquotient(A,R,Q)$ is the HIT as specified above. This means that
  \begin{itemize}
    \item There is a 0-path constructor $\br{{-}}_0 : A \to D$;
    \item There is a 1-path constructor $\br{{-}}_1 : \{a\ a' : A\} \to R(a,a') \to \br{a}_0=\br{a'}_0$;
    \item There is a 2-path constructor $\br{{-}}_2 : \{a : A\} \to \{t : \words_R(a,a)\} \to Q(t) \to \overline{\br{t}_1}=1$;
    \item There is an induction principle that states the following: given a family $P: D \to \type$ with $s_0 : (a : A) \to P\br{a}_0$ and 
    $$s_1 : \{a\ a' : A\} \to (r : R(a,a')) \to s_0(a) =_{\br{r}_1}^P s_0(a')$$
    $$s_2 : \{a : A\} \to \{t : \words_R(a,a)\} \to (q : Q(t)) \to 
    \overline{s_1}(t)=_{\br{q}_2}1,$$
    then $P$ has a section $f : (d : D) \to P(d)$ that computes on the point and 1-path constructors: $f\br{a}_0\equiv s_0(a)$ and $\apd_f\br{r}_1=s_1(r)$.
    \item There is a recursion principle that states the following: given $P : \type$ with $p_0 : A \to P$ and 
    $$p_1 : \{a\ a' : A\} \to R(a,a') \to p_0(a) = p_0(a')$$
    $$p_2 : \{a : A\} \to \{t : \words_R(a,a)\} \to Q(t) \to \overline{p_1}(t)=1,$$
    then there is a map $g : D \to P$ that computes on the point 1-path and 2-path constructors. This means that $g\br{a}_0\equiv p_0(a)$ for $a : A$, and that there is a path
    $\iota_1:\apfunc{g}\br{r}_1=p_1(r)$ for $r : R(a,a')$ and a filler of the following square for $q : Q(t)$.
    \begin{center}
    \begin{tikzpicture}[thick, node distance=3cm]
      \node (tl)       at (0,0) {$\apfunc{g}\overline{\br{t}_1}$};
      \node[right of = tl] (tr) {$1$};
      \node (bl)      at (0,-2) {$\overline{p_1}(t)$};
      \node (br) at (tr |- bl)  {$1$};
      \path[every node/.style={font=\sffamily\small}]
      (tl) edge[double equal sign distance] node {$\apfunc{\apfunc{g}}\br{q}_2$} (tr)
           edge[double equal sign distance] node[left] {$\overline{\iota_1}$} (bl)
      (tr) edge[double equal sign distance] node {$1$} (br)
      (bl) edge[double equal sign distance] node {$p_2(q)$} (br);
    \end{tikzpicture}
  \end{center}
  \end{itemize}
\end{thm}
\begin{rmk}\mbox{}
  \begin{itemize}
    \item We do not prove a computation rule for the induction principle on 2-paths. Although we strongly expect this to be true, it will involve an elaborate computation. This computation rule is not necessary to define $\simpletwoquotient(A,R,Q)$ up to equivalence. If we had another type with these exact constructors, eliminators and computation rules, we can prove that it is equivalent to this one. Furthermore, in many examples of two-quotients we will 1-truncate the result, such as for Eilenberg-MacLane spaces and groupoid quotients (see \autoref{sec:eilenb-macl-spac}). After the 1-truncation, the computation rules on the 2-paths are automatic, since these 3-paths can be constructed just from the assumption that the type family is truncated.
    \item We do not define the recursion principle as a special case of the induction principle. We can define it is a much simpler way, so that we can compute its action on 2-paths more easily.
    \item We use overlines to denote elimination out of the inductive type $\words_R$. The exact type and definition of the overline depends on the type of the object we overline. 
    For example 
    $$\overline{\br{{-}}_1}: \words_R(a,a')\to\br{a}_0=\br{a'}_0$$ 
    is defined recursively by path concatenation and path inversion. In contrast
    $$\overline{s_1}: (t : \words_R(a,a'))\to s_0(a)=_{\overline{\br{t}_1}}^P s_0(a')$$ 
    is defined recursively by pathover concatenation and pathover inversion and 
    $$\overline{\iota_1}: (t : \words_R(a,a'))\to \apfunc{g}\overline{\br{t}_1}=\overline{p_1}(t)$$ is defined recursively by horizontal concatenation and horizontal inversion and by using the rules 
    $\apfunc{g}(p \cdot q)=\apfunc{g}(p)\cdot\apfunc{g}(q)$ and $\apfunc{g}(p\sy)=(\apfunc{g}(p))\sy$.
  \end{itemize}
\end{rmk}
\begin{proof}
  \textbf{Constructors.}\\
  We define for $a : A$ the point constructor 
  $$\br{a}_0\defeq [[\inl a]_0]_0:D$$ 
  and for $r : R(a,a')$ the 1-path constructor
  $$\br{r}_1\defeq\apfunc{[{-}]_0}[r]_1:\br{a}_0=\br{a'}_0$$ 
  from the path constructors of $C$.


  The 2-path constructor $\br{q}_2:\overline{\br{t}_1}=1$ for $q : Q(t)$ is
  defined as the concatenation
  $\overline{\br{t}_1}=\apfunc{[{-}]_0}\overline{[t]_1}=1.$ Here the first
  equality is by a general lemma about $\words_R$ that states that
  $\overline{\apfunc{f}(h(r))}=\apfunc{f}(\overline{h(r)})$. The second path
  uses \autoref{lem:twopathlemma} and is defined as
  $e([[\inr(a,t,q)]_0]_0,[q,{-}])$, where 
  $$[q,{-}]:[\circlerec([\inl a]_0,\overline{[t]_1},x)]_0\equiv [u_t(x)]_0=[[\inr(a,t,q)]_0]_0$$ 
  is the path constructor $[q,x]_1$ of $D$.\\ \mbox{} \\
  \textbf{Induction Principle.}\\
  For the induction principle, suppose given $P$, $s_0$, $s_1$ and $s_2$ as in the theorem statement. We first define $f_0 : (c : C) \to P[c]_0$ by induction on $c : C$. 
  We define $$f_0[\inl a]\defeq s_0(a)$$ and (denoting $b\defeq\base:\S^1$)
  $$f_0[\inr(a,t,q)]\defeq \transp^P([q,b]_1,s_0(a)).$$
  For the path constructor, we need to construct for $r : R(a,a')$ the pathover 
  $$\apd_{f_0}[r]_1:s_0(a)=_{[r]_1}^{P[{-}]_0}s_0(a').$$
  Here we can use $s_1$, and then apply the equivalence 
  $$y=_{\apfunc{f}(p)}^P y' \simeq y =_p^{P \circ f} y'.$$
  Note that this equivalence holds by reflexivity in a cubical type theory. In the remainder of this proof we will denote any occurrence of this and similar by a tilde for readability. So we define
  $$\apd_{f_0}[r]_1\defeqp\widetilde{s_1}(r).$$
  This defines $f_0$, which is $f$ applied to the point constructor of $D$, that is, $f[c]_0\defeq f_0(c)$. Now we need to define for $x:\S^1$ the pathover
  $$\apfunc{f}[q,x]_1:f_0(u_t(x))=_{[q,x]_1}^P 
  ([q,b]_1)_*(s_0a).$$
  We will fill this pathover by induction to $x$. For $x\equiv b$ we can constructor the resulting pathover by 
  $$1_{[q,b]_1}:\apfunc{f}[q,b]_1:s_0a=_{[q,b]_1}^P ([q,b]_1)_*(s_0a),$$
  where in general $1_p:y=_p^P p_*(y)$ can be easily defined by induction on $p$. When $x$ varies along loop, we need to construct a pathover between two pathovers and this corresponds to the following squareover. The bottom square is a square in $D$, namely the naturality square of $$[(q,{-})]_1:(x : \S^1) \to [u_t(x)]_0=[[\inr(a,t,q)]_0]_0$$ applied to the path $\lp$, and the top square is the squareover we need to fill.
  \begin{center}
    \begin{tikzpicture}[thick, node distance=3cm]
      \node (tl)       at (0,0) {$s_0(a)$};
      \node[right of = tl] (tr) {$[q,b]_{1*}(s_0(a))$};
      \node (bl)      at (0,-2) {$s_0(a)$};
      \node (br) at (tr |- bl)  {$[q,b]_{1*}(s_0(a))$};
      \path[every node/.style={font=\sffamily\small}]
      (tl) edge[double equal sign distance] node {$1_{[q,b]_1}$} (tr)
           edge[double equal sign distance] node[left]{$\apdtilde_{f_0\circ u_t}(\lp)$} (bl)
      (tr) edge[double equal sign distance] node {$\apdtilde_{\const_{[q,b]_{1*}(s_0(a))}}(\lp)$} (br)
      (bl) edge[double equal sign distance] node (b) {$1_{[q,b]_1}$} (br);
      \node (tl)       at (0,-4.5) {$\br{a}_0$};
      \node[right of = tl] (tr) {$[[\inr(a,t,q)]_0]_0$};
      \node (bl)      at (0,-6.5) {$\br{a}_0$};
      \node (br) at (tr |- bl)  {$[[\inr(a,t,q)]_0]_0$};
      \path[every node/.style={font=\sffamily\small}]
      (tl) edge[double equal sign distance] node (t) {$[q,b]_1$} (tr)
           edge[double equal sign distance] node[left]{$\apfunc{[u_t({-})]_0}(\lp)$} (bl)
      (tr) edge[double equal sign distance] node {$\apfunc{\const_{[[\inr(q)]_0]_0}}(\lp)$} (br)
      (bl) edge[double equal sign distance] node {$[q,b]_1$} (br)
      (1.3,-2.5)  edge[->] (1.3,-3.5);
    \end{tikzpicture}
  \end{center}
  We will first focus on the left side of the squareover. We compute
  \begin{align*}
    \apdtilde_{f_0\circ u_t}(\lp)
    &= \apdtilde_{f_0}(\apfunc{u_t}(\lp))\\
    &= \apdtilde_{f_0}(\overline{[t]_1})\\
    &= \widetilde{\overline{\widetilde{s_1}}(t)}\\
    &= \widetilde{\overline{s_1}(t)}\\
    &= \widetilde{(\br{q}_2\sy)_*1} && \text{(using $s_2$)} \\
    &\equiv\vcentcolon \widetilde{1}.
  \end{align*}
  Here with $\widetilde{1}$ we mean the pathover $1:s_0(a) =_{1_{\br{a}_0}}^{P}s_0(a)$ but transported along the path
  $$1_{\br{a}_0} \stackrel{\br{q}_2}= \overline{\br{t}_1} = \apfunc{[{-}]_0}\overline{[t]_1} = 
  \apfunc{[{-}]_0}(\apfunc{u_t}(\lp))=\apfunc{[u_t({-})]_0}(\lp).$$
  By unfolding the definition of $\br{q}_2$ this can be simplified to the following concatenation:
  $$1_{\br{a}_0} \stackrel{e}= \apfunc{[{-}]_0}\overline{[t]_1} = 
  \apfunc{[{-}]_0}(\apfunc{u_t}(\lp))=\apfunc{[u_t({-})]_0}(\lp).$$
  The right side of the squareover is easier to manipulate:
  $$\apdtilde_{\const_{[q,b]_{1*}(s_0(a))}}(\lp)=\widetilde{1},$$
  where in this case we mean the pathover $1:s_0(a) =_{1_{\br{a}_0}}^{P}s_0(a)$ transported along the path
  $$1_{\br{a}_0} = \apfunc{\const_{[[\inr(q)]_0]_0}}(\lp).$$
  Now in both the left and the right side these transports only act on the path they lie over. This means that we can ``push them down'' to the base square. 

  After we do that, we have a vertically degenerate squareover, and we only have to show that the square over which it lies is also vertically degenerate, which is a straightforward calculation. 

  This finishes the definition of $f$. The computation rule $f\br{a}_0\equiv s_0(a)$ follows directly from the computation rule for the quotient. Furthermore, we have 
  $$\apd_f\br{r}_1\equiv\apd_f\mapfunc{[{-}]_0}[r]_1=\apdtilde_{f\circ[{-}]_0}[r]_1\equiv\apdtilde_{f_0}[r]_1=s_1(r).$$
  \textbf{Recursion Principle.}\\
  For the recursion principle, suppose given $P, p_0, p_1, p_2$ as in the theorem statement. We first define $g_0 : C \to P$ by 
  \begin{align*}
    g_0[\inl a]_0&\defeq p_0(a)\\
    g_0[\inr (a,t,q)]_0&\defeq p_0(a)\\
    \apfunc{g_0}[r]_1&\defeqp p_1(r).
  \end{align*}
  We define $g:D\to P$ by $g[c]_0\defeq g_0(c)$ and then we need to define 
  $\apfunc{g}[q,x]_1 : g_0(u_t(x)) = p_0(a)$, which we do by induction to $x$. For $x\equiv b$, this can be done by reflexivity, so $\apfunc{g}[q,b]_1\defeqp 1_{p_0(a)}$. When $x$ varies over $\lp$, we need to fill the following square.
  \begin{center}
    \begin{tikzpicture}[thick, node distance=3cm]
      \node (tl)       at (0,0) {$p_0(a)$};
      \node[right of = tl] (tr) {$p_0(a)$};
      \node (bl)      at (0,-2) {$p_0(a)$};
      \node (br) at (tr |- bl)  {$p_0(a)$};
      \path[every node/.style={font=\sffamily\small}]
      (tl) edge[double equal sign distance] node {$1$} (tr)
           edge[double equal sign distance] node[left] {$\apfunc{g_0\circ u_t}(\lp)$} (bl)
      (tr) edge[double equal sign distance] node {$\apfunc{\const_{p_0(a)}}(\lp)$} (br)
      (bl) edge[double equal sign distance] node {$1$} (br);
    \end{tikzpicture}
  \end{center}
  This can be done by the following calculation.
  $$\apfunc{g_0\circ u_t}(\lp)=\apfunc{g_0}\apfunc{u_t}(\lp)=\apfunc{g_0}\overline{[t]_1}=
  \overline{p_1}(t)\stackrel{p_2}=1=\apfunc{\const_{p_0(a)}}(\lp).$$
  This completes the definition of $g$. The computation rule $g\br{a}_0\equiv p_0(a)$ follows from the computation rule for quotients on points. We can define the computation rule on paths as the composite
  $$\iota_1 : \apfunc{g}\br{r}_1 \equiv \apfunc{g}\apfunc{[{-}]_0}[r]_1 = \apfunc{g_0}[r]_1 = p_1(r).$$
  The fact that $g$ has the correct computation rule for 2-paths requires some complicated path algebra, which we will omit here.
\end{proof}

We can now define the general version of the 2-quotient, $\twoquotient(A,R,S)$, to be equal to
$\simpletwoquotient(A,R,Q)$ where $Q$ is the inductive family
\begin{inductive}
\texttt{inductive} $Q : \{a : A\} \to \words_R(a, a) \to \type \defeqp$ \\
$\bullet\ (a\ a' : A) \to (t\ t' : \words_R(a,a')) \to (s : S(t,t')) \to Q(t \cdot t\sy).$
\end{inductive}
We then show that $\twoquotient(A,R,S)$ and $\|\twoquotient(A,R,S)\|_n$ have the right elimination
principles and computation rules (it requires some work to show that the eliminator of the truncated 2-quotient has the right computation rules on 2-paths). 

This allows us to define all nonrecursive HITs with point, 1-path and 2-path constructors.
For example, we define the torus $T^2 := \twoquotient(\unit,R,S)$ where
$R(⋆,⋆) = \bool$ (giving two path constructors $p$ and $q$ from the basepoint to itself) and
$Q$ is generated by the constructor
$s_0 : S (\bfalse \cdot \btrue) (\btrue \cdot \bfalse)$, which determines a path $p \cdot q = q \cdot p$.
We also define the \emph{groupoid quotient}: For a groupoid $G$ we define its quotient as
$\|\twoquotient(G, \homm_G, S)\|_1$ where:
\begin{inductive}
\texttt{inductive} $S \defeqp$ \\
$\bullet\ (a\ b\ c : G) \to (g : \homm(b,c)) \to (f : \homm(a,b)) \to S (g \circ f) (f \cdot g)$
\end{inductive}
If $G$ is just a group (considered as a groupoid with a single object), then the groupoid quotient
of $G$ is exactly the Eilenberg-MacLane space $K(G,1)$. For more information, see \autoref{sec:eilenb-macl-spac}.

\section{Colimits}\label{sec:colimits}
\lstDeleteShortInline"

We can ask whether we can use the construction of \Cref{sec:prop-trunc} can be generalized to
construct other higher inductive types.\footnote{The work in this section is joint work with Egbert Rijke and Kristina Sojakova.} 
The general idea is that we can construct a recursive higher
inductive type as a sequential colimit of repeatedly applying a nonrecursive version of the
HIT. This does not work in general: if a constructor is infinitary, there is no reason why the type after $\omega$ many steps is the desired type. However, this does work for a general class of higher inductive types, the \emph{$\omega$-compact localizations}. In this section we will show various properties of colimits that are used in the proof of this fact. The full proof will appear in an upcoming preprint.

\begin{defn}
  Suppose given a type $A$, families $P, Q : A \to \type$ and $F : \{a : A\} \to P(a) \to Q(a)$.

  A type $X$ is \emph{$F$-local} if for all $a : A$ the map
  $$\psi_X(a)\defeq \lam{f}f \o F(a) : (Q(a) \to X) \to (P(a) \to X)$$
  is an equivalence.


  The \emph{$F$-localization} $L_FX$ or $LX$ of $X$ turns $X$ into a $F$-local type in a universal
  way. This means there is a map $\ell_X : X \to LX$ such that for any $F$-local type $Y$ there is an
  equivalence of maps $(LX \to Y) \to (X \to Y)$ given by precomposition with $\ell_X$. $L_FX$ can
  be given as a higher inductive type with the following constructors:
  \begin{lstlisting}[gobble=4]
    HIT L F X : Type :=
    | incl : X → L X
    | rinv : Π{a} (f : P a → L X), Q a → L X
    | isri : Π{a} (f : P a → L X) (x : P a), rinv f (F x) = f x
    | linv : Π{a} (f : P a → L X), Q a → L X
    | isli : Π{a} (f : Q a → L X) (x : Q a), linv (f ∘ F) x = f x.
  \end{lstlisting}
\end{defn}

For a sequence $(A_n,f_n)_n$ we denote the colimit by $\colim(A)$ or $A_\infty$. Also, for any type $X$, we can define a new sequence $(X\to A_n, f_n \o ({-}))_n$. Note that there is a canonical map
$$\xi_X : \colim(X\to A_n)\to (X\to A_\infty).$$
It is defined by $\xi_X(i_n(f))\defeq i_n\o f$ and $\xi_X(\kappa(f))\defeq \kappa_n\o f$, where $\kappa$ is the path constructor of the colimit.

\begin{defn}
  A type $X$ is said to be \emph{$\omega$-compact} if the map $\xi_X$ is an equivalence for all sequences $(A_n,f_n)_n$.
\end{defn}

Examples of $\omega$-compact types are the finite types. Moreover, the $\omega$-compact types are closed under dependent pair types and pushouts. A non-example of an $\omega$-compact type is $\N$. We will omit the details here.

\begin{thm}\label{thm:localization}
  Assume that for all $a : A$ the types $P(a)$ and $Q(a)$ are $\omega$-compact. Then we can
  construct the $F$-localization in MLTT$+$quotients.
\end{thm}

We will not prove this theorem here, but defer it to an upcoming preprint. However, we will develop machinery here that is crucial to prove this theorem. In particular we prove that sigma-types commute with sequential colimits.

\subsubsection*{Type Sequences}

\begin{defn}
A \emph{type sequence} $\sequence{A}{f}$ consists of a diagram of the form
\begin{equation*}
\begin{tikzcd}
A_0 \arrow[r,"f_0"] & A_1 \arrow[r,"f_1"] & A_2 \arrow[r,"f_2"] & \cdots
\end{tikzcd}
\end{equation*}
Thus, the type of all sequences of types is
\begin{equation*}
\mathrm{Seq} \defeq (A:\nat\to\type)\times(n:\nat) \to A_n\to A_{n+1}
\end{equation*}
\end{defn}

Recall that the relation $\leq$ on the natural numbers is defined as an inductive family of types $\leq\mathop{:}\N\to\N\to\UU$ with
\begin{align*}
r & : (n:\N) \to n\leq n \\
s & : (n,m:\N)\to n\leq m \to n\leq m+1.
\end{align*}
It follows that $n\leq m$ is a a mere proposition for each $n,m:\N$.

\begin{defn}
Let $\sequence{A}{f}$ be a type sequence. For any $n,m:\nat$, we define
\begin{equation*}
f^{n\leq m} : A_n\to A_m.
\end{equation*}
where we leave the proof that $n\leq m$ implicit.
\end{defn}

\begin{proof}[Construction]
We define $f^{n\leq m}$ by induction on the proof that $n\leq m$ by taking
\begin{align*}
f^{n\leq n} & \defeq \idfunc[A_n] \\
f^{n\leq m+1} & \defeq f_m\circ f^{n\leq m} \qedhere
\end{align*}
\end{proof}

\begin{defn}
Let $\sequence{A}{f}$ be a type sequence. For any $n,k:\nat$, we define
$f_n^k:A_n\to A_{n+k}$ to be $f^{n\leq n+k}(p)$, where $p$ is the canonical proof that $n\leq n+k$.
\end{defn}

\begin{defn}
A \emph{sequence $\sequence{B}{g}$ of types over $\sequence{A}{f}$} consists of a diagram of the form
\begin{equation*}
\begin{tikzcd}
B_{0} \arrow[r,"g_0"] \arrow[d,->>] & B_{1} \arrow[r,"g_1"] \arrow[d,->>] & B_{2} \arrow[r,"g_2"] \arrow[d,->>] & \cdots \\
A_0 \arrow[r,"f_0"] & A_1 \arrow[r,"f_1"] & A_2 \arrow[r,"f_2"] & \cdots
\end{tikzcd}
\end{equation*}
where each $g_n$ has type $(a:A_n)\to B_n(x)\to B_{n+1}(f_n(a))$, implicitly rendering the
squares commutative.

We say that a sequence $\sequence{B}{g}$ over $\sequence{A}{f}$ is \emph{equifibered} if each $g_n$ is a family of equivalences.
\end{defn}

\begin{defn}
Let $\sequence{A}{f}$ and $\sequence{A'}{f'}$ be type sequences.
A \emph{natural transformation} $\sequence{A}{f}\to\sequence{A'}{f'}$
is a pair $\sequence{\tau}{H}$ consisting of a family of maps
\begin{equation*}
\tau : (n:\N) \to A_n \to A'_n
\end{equation*}
and a family $H_n$ of homotopies witnessing that the diagram
\begin{equation*}
\begin{tikzcd}
A_{0} \arrow[r,"f_0"] \arrow[d,"\tau_0"] & A_{1} \arrow[r,"f_1"] \arrow[d,"\tau_1"] & A_{2} \arrow[r,"f_2"] \arrow[d,"\tau_2"] & \cdots \\
A'_0 \arrow[r,"{f'_0}"] & A'_1 \arrow[r,"{f'_1}"] & A'_2 \arrow[r,"{f'_2}"] & \cdots
\end{tikzcd}
\end{equation*}
commutes.
\end{defn}

\begin{defn}
A \emph{natural equivalence} is a natural
transformation $(\tau,H)$ such that each $\tau_n$ is an equivalence. 
The type of natural equivalences from $\sequence{A}{f}$ to $\sequence{A'}{f'}$
is called $\mathsf{NatEq}(\sequence{A}{f},\sequence{A'}{f'})$.
\end{defn}

\begin{lem}
The canonical dependent function $\mathsf{idtonateq}$ 
\begin{equation*}
(\sequence{A}{f}=\sequence{A'}{f'})\to \mathsf{NatEq}(\sequence{A}{f},\sequence{A'}{f'})
\end{equation*}
that sends $\refl_{\sequence{A}{f}}$ to the identity natural transformation, is
an equivalence.
\end{lem}

\begin{proof}
Straightforward application of univalence.
\end{proof}

Every type sequence $\sequence{B}{g}$ over $\sequence{A}{f}$ gives rise to a natural transformation, by the following definition. 

\begin{defn}
Let $\sequence{B}{g}$ be a sequence over $\sequence{A}{f}$. Then we define the
sequence $\msm{\sequence{A}{f}}{\sequence{B}{g}}$ to consist of the diagram
\begin{equation*}
\begin{tikzcd}[column sep=large]
(a:A_0)\times B_0(a) \arrow[r,"\pairr{f_0,g_0}"] & (a:A_1)\times B_1(a) \arrow[r,"\pairr{f_1,g_1}"]
& (a:A_2)\times B_2(a) \arrow[r,"\pairr{f_2,g_2}"] & \cdots
\end{tikzcd}
\end{equation*}
where we take the usual definition
\begin{equation*}
\pairr{f_n,g_n} \defeq \lam{\pairr{a,b}}\pairr{f_n(a),g_n(a,b)}.
\end{equation*}
Furthermore, we define a natural transformation 
\begin{equation*}
\sequence{\pi}{\theta}:\msm{\sequence{A}{f}}{\sequence{B}{g}}\to \sequence{A}{f}
\end{equation*}
by taking 
\begin{align*}
\pi_n & \defeq \proj1 & & : ((a:A_n)\times B_n(a))\to A_n \\
\theta_n(a,b) & \defeq\refl_{f_n(a)} & & : f_n(\proj 1(a,b))= \proj 1(f_n(a),g_n(b)).
\end{align*}
\end{defn}




We will now look at the shift operation on type sequences, in particular to bring up subtleties that come up in the formalization of mathematics in homotopy type theory. The issue we face is that equality in the natural numbers is not always strict. For instance, when addition is defined by induction on the second argument, then $n+0$ is judgmentally equal to $n$, while $0+n$ is not. This implies that sometimes we might have to \emph{transport} along the equalities in the natural numbers (such as $n=0+n$), and this complicates the formalization process.

We define the shift operation.
\begin{defn}
For any type sequence $\sequence{A}{f}$ we define a new type sequence $(S(A),S(f))$ by taking
\begin{align*}
S(A)_n & \defeq A_{n+1} \\
S(f)_n & \defeq f_{n+1}.
\end{align*}
\end{defn}

Of course we can iterated the shift operation, defining a type sequence $(S^k(A),S^k(f))$ for every $k:\N$. However, while the type $S^k(A)_n$ is $A_{n+k}$, the function $S^k(f)_n$ is some function $A_{n+k}\to A_{(n+1)+k}$ that is not judgmentally equal to a function of the form $f_m$ for some $m:\N$. Therefore, we make an alternative definition of the $k$-shift that is different from $S^k$, the type sequence obtained from iterating the shift $S$.

\begin{defn}
Given a type sequence $(A,f)$, we define $S_k(A,f)\jdeq(S_k(A),S_k(f))$ to be the type sequence given by
\begin{align*}
S_k(A)_n & \defeq A_{k+n} \\
S_k(f)_n & \defeq f_{k+n}.
\end{align*}
Given a dependent sequence $(B,g)$ over $(A,f)$, we also define $S_k(B,g)\jdeq (S_k(B),S_k(g))$ by 
\begin{align*}
S_k(B)_n & \defeq B_{k+n} \\
S_k(g)_n & \defeq g_{k+n}.
\end{align*}
\end{defn}

Note that the sequence $(S_{k+1}(A),S_{k+1}(f))$ is not judgmentally equal to the sequence $S(S_k(A),S_k(f))$, since in general we do not have $(k+1)+n\jdeq (k+n)+1$. Therefore we have the following lemma.

\begin{lem}\label{lem:iterate_succ}
For any $k,n:\nat$ and $a : A_k$, one has $q_{k,n}(a):\dpath{A}{p(k,n)}{f_k^{n+1}(a)}{f_{k+1}^n(f_k(a))}$ where $p(k,n):(k+n)+1=(k+1)+n$ is the canonical path in $\nat$.
\end{lem}

\begin{proof}
By induction on $n:\N$.
\end{proof}

\begin{cor}
For any type sequence $\sequence{A}{f}$, the type sequence $(S_{k+1}(A),S_{k+1}(f))$ is naturally equivalent to the type sequence $(S(S_k(A)),S(S_k(f)))$. 
\end{cor}

\subsubsection*{Sequential Colimits}
\begin{rmk}
The induction principle for sequential colimits tells us how to construct a dependent function $f:(a:A_\infty)\to P(a)$ for a type family $P:A_\infty\to\type$. 

Given $s:(a:A_\infty)\to P(a)$, we get
\begin{align*}
\lam{n}{a} s(\iota_n(a)) & : (n:\N)(a:A_n)\to P(\iota_n(a)) \\
\lam{n}{a} \apd{s}{\kappa_n(a)} & : (n:\N)(a:A_n)\to s(\iota_n(a)) =_{\kappa_n(a)}^P s(\iota_{n+1}(f_n(a)))
\end{align*}
In other words, we have a canonical map
\begin{align*}
&\Big((a:A_\infty)\to P(a)\Big)\to\\
&\Big((h:(n:\N)(a:A_n)\to P(\iota_n(a)))\times(n:\N)(a:A_n)\to h_n(a) =_{\kappa_n(a)}^P h_{n+1}(f_n(a))\Big)
\end{align*}
Now we can state the induction principle and computation rule concisely: the canonical map described above comes equipped with a section. We assume that that the computation rule is strict on the point constructors.
\end{rmk}

The universal property of sequential colimits is a straightforward consequence of the induction principle.

\begin{thm}
Let $\sequence{A}{f}$ be a type sequence, and let $X$ be a type. Then the canonical map
\begin{equation*}
(A_\infty\to X)\to (h:(n:\N) \to A_n\to X)\times(n:\N)\to h_n\htpy h_{n+1}\circ f_n
\end{equation*}
is an equivalence.
\end{thm}

The following theorem is a descent theorem for sequential colimits.

\begin{thm}\label{thm:descent}
Consider a sequence $\sequence{A}{f}$. The type $A_\infty\to\UU$ is equivalent to the type of equifibered type sequences over $\sequence{A}{f}$.
\end{thm}

\begin{proof}
By the universal property of $A_\infty$ and by univalence we have
\begin{align*}
(A_\infty\to \UU) & \eqvsym (B:(n:\N)\to A_n\to \UU)\times (n:\N) \to B_n\htpy B_{n+1}\circ f_n \\
& \eqvsym (B:(n:\N) \to A_n\to \UU)\times(n:\N)(x:A_n)\to B_n(x)\eqvsym B_{n+1}(f_n(x))\qedhere
\end{align*}
\end{proof}

\begin{lem}\label{lem:seq_colim_functor}
  Suppose given a natural transformation $(\tau,H):\sequence{A}{f}\to\sequence{A'}{f'}$.
  \begin{enumerate}
    \item\label{part:seq_colim_functor} We get a function $\tfcolim(\tau,H)$ or $\tau_\infty:A_\infty\to A'_\infty$.
    \item\label{part:1_functoriality} The sequential colimit is 1-functorial. This means the following three things. If $(\sigma,K):\sequence{A'}{f'}\to\sequence{A''}{f''}$, then $(\tau\circ\sigma)_\infty\sim\tau_\infty\circ\sigma_\infty$. Moreover, $1_\infty\sim \idfunc$, where $1$ is the identity natural transformation. Lastly, if $(\tau',H'):\sequence{A}{f}\to\sequence{A'}{f'}$ and $q:(n:\N)\to \tau_n\sim\tau'_n$ and we can fill the following square for all $a:A_n$
    \begin{center}\begin{tikzcd}[column sep=25mm]
      \tau_{n+1}(f_na)
      \ar[r,equal,"{q_{n+1}(f_na)}"] 
      \ar[d,equal,"{H_n(a)}"] &
      \tau'_{n+1}(f_na)
      \ar[d,equal,"{H'_n(a)}"]\\
      f'_n(\tau_n(a))
      \ar[r,equal,"{\apfunc{f'_n}(q_n(a))}"] & 
      f'_n(\tau'_n(a))
    \end{tikzcd}\end{center}
    then $\tau_\infty\sim\tau'_\infty$.
    \item\label{part:functor_equivalence} If $\tau$ is a natural equivalence, then $\tau_\infty$ is an equivalence.
  \end{enumerate}
\end{lem}
\begin{proof}\mbox{}
  \begin{enumerate}
    \item 
    We define $\tau_\infty(\iota_n(a))\defeq\iota_n(\tau_n(a))$ and 
    $$\apfunc{\tau_\infty}(\kappa_n(a))\vcentcolon=\apfunc{\iota_{n+1}}(H(a))\cdot\kappa_n(\tau_n(a)):
    \iota_{n+1}(\tau_{n+1}(f_na))=\iota_n(\tau_n(a)).$$
    \item All three parts are by induction on the element of $A_\infty$, and all parts are straightforward.
    \item We define $(\tau_\infty)^{-1}\defeq(\tau^{-1})_\infty$ where $\tau^{-1}$ is the natural transformation by inverting $\tau_n$ for each $n$. Now we can check that this is really the inverse by using all three parts of the 1-functoriality.
    $$\tau_\infty^{-1}\circ\tau_\infty\sim (\tau^{-1}\circ\tau)_\infty\sim 1_\infty\sim\idfunc[A_\infty].$$
    For the second homotopy we need to show that we can fill a certain square, which is straightforward.
    The other composite is homotopic to the identity by a similar argument.\qedhere
  \end{enumerate}
\end{proof}

The following lemma states that $\iota_0$ is an equivalence if all maps in the sequence are an equivalence. We will have a more general result in \autoref{cor:trunc_colim}\ref{part:iota_is_trunc_conn}, but in that proof we will use some special cases of this lemma.
\begin{lem}\label{lem:equiv_equiseq}
  Suppose given a sequence $\sequence{A}{f}$ where $f_n$ is an equivalence for all $n$. 
  Then $\iota_0:A_0\to A_\infty$ is an equivalence.
\end{lem}
\begin{proof}
  First note that the map $f^{0\le n}:A_0\to A_n$ is an equivalence, which is an easy induction on the proof that $0\le n$, because $f^{0\le0}\jdeq\idfunc$ is an equivalence and $f^{0\le n+1}\jdeq f_n\circ f^{0\le n}$ is a composition of two equivalences.

  Also note that we have paths $\kappa^{n\le m}(a):\iota_m(f^{n\le m}(a))=\iota_n(a)$ for $a:A_n$.

  Now we define $\iota_0^{-1}:A_\infty\to A_0$ as 
  $$\iota_0^{-1}(\iota_n(a))\defeq (f^{0\le n})^{-1}(a)$$
  and we define
  $$\apfunc{\iota_0^{-1}}(\kappa_n(a)):(f^{0\le n})^{-1}(f_n^{-1}(f_n(a)))=(f^{0\le n})^{-1}(a)$$
  as $\apfunc{(f^{0\le n})^{-1}}(\ell_n(a))$, where $\ell_n(a):f_n^{-1}(f_n(a))$ is the canonical path.

  Now $\iota_0^{-1}\circ\iota_0\sim\idfunc$ is true by definition. To show that for $x:A_\infty$ we have 
  $p(x):\iota_0(\iota_0^{-1}(x))=x$, we use induction on $x$. If $x\jdeq\iota_n(a)$, we have
  \begin{align*}
    \iota_0(\iota_0^{-1}(\iota_n(a)))
    &\jdeq \iota_0((f^{0\le n})^{-1}(a))\\
    &=\iota_n(f^{0\le n}((f^{0\le n})^{-1}(a)))\\
    &=\iota_n(a).
  \end{align*}
  If we write $r^{0\le n}: f^{0\le n} \circ (f^{0\le n})^{-1}\sim\idfunc$ for the canonical homotopy, then we explicitly define $p(\iota_n(a))$ as
  $$p(\iota_n(a))\defeq(\kappa^{0\le n}((f^{0\le n})^{-1}(a)))^{-1}\cdot \apfunc{\iota_n}(r^{0\le n}(a)).$$

  If $x$ varies over $\kappa_n(a)$, then we need to fill the following square.
  \begin{center}\begin{tikzcd}[column sep=25mm]
    \iota_0(\iota_0^{-1}(\iota_{n+1}(f_na)))
    \ar[r,equal,"{p(\iota_{n+1}(f_na))}"] 
    \ar[d,equal,"{\mapfunc{\iota_0\circ\iota_0^{-1}}(\kappa_n(a))}"] &
    \iota_{n+1}(f_na)
    \ar[d,equal,"{\kappa_n(a)}"]\\
    \iota_0(\iota_0^{-1}(\iota_n(a)))
    \ar[r,equal,"{p(\iota_n(a))}"] & 
    \iota_n(a)
  \end{tikzcd}\end{center}
  If we unfold the definitions of $\iota_0^{-1}$ and $p$, we can fill this as the horizontal concatenation of the following two squares (where we have left out some arguments to the paths)
  \begin{center}\begin{tikzcd}[column sep=15mm]
    \iota_0((f^{0\le n})^{-1}(f_n^{-1}(f_na)))
    \ar[r,equal,"{(\kappa^{0\le n})^{-1}}"] 
    \ar[d,equal,"{\mapfunc{\iota_0\circ(f^{0\le n})^{-1}}(\ell)}"] &
    \iota_n(f^{0\le n}((f^{0\le n})^{-1}(f_n^{-1}(f_na))))
    \ar[d,equal,"{\mapfunc{\iota_n\circ f^{0\le n}\circ (f^{0\le n})^{-1}}(\ell)}"] \\
    \iota_0((f^{0\le n})^{-1}(a))
    \ar[r,equal,"{(\kappa^{0\le n})^{-1}}"] & 
    \iota_n(f^{0\le n}((f^{0\le n})^{-1}(a)))
  \end{tikzcd}\end{center}
  \begin{center}\begin{tikzcd}[column sep=15mm]
    \iota_n(f^{0\le n}((f^{0\le n})^{-1}(f_n^{-1}(f_na))))
    \ar[rr,equal,"{\kappa^{-1}\cdot\apfunc{\iota_{n+1}}(\apfunc{f}(r^{0\le n})\cdot r)}"] 
    \ar[dd,equal,"{\mapfunc{\iota_n\circ f^{0\le n}\circ (f^{0\le n})^{-1}}(\ell)}"] 
    \ar[rd,equal,"{\mapfunc{\iota_n}(r^{0\le n})}"] & &
    \iota_{n+1}(f_na)
    \ar[dd,equal,"{\kappa}"]\\
    & \iota_0(f_n^{-1}(f_na))
    \ar[rd,equal,"{\mapfunc{\iota_n}(\ell)}"] & \\
    \iota_n(f^{0\le n}((f^{0\le n})^{-1}(a)))
    \ar[rr,equal,"{\apfunc{\iota_n}(r^{0\le n})}"] & &
    \iota_n(a)
  \end{tikzcd}\end{center}
  The first square is a naturality square, as is the bottom-left part of the second square. 
  We can use the triangle equalities of $f$ to rewrite the $r$ in the top part to $\apfunc{f}(\ell)$. After doing that, the top-right square becomes the following naturality square.

  \begin{center}\begin{tikzcd}[column sep=20mm]
    \iota_n(f_n(f^{0\le n}((f^{0\le n})^{-1}(f_n^{-1}(f_na)))))
    \ar[r,equal,"{\apfunc{\iota_{n+1}\circ f}(r^{0\le n}\cdot \ell)}"] 
    \ar[d,equal,"{\kappa}"] &
    \iota_{n+1}(f_na)
    \ar[d,equal,"{\kappa}"]\\
    \iota_0(f_n^{-1}(f_na))
    \ar[r,equal,"{\mapfunc{\iota_n}(r^{0\le n}\cdot\ell)}"] &
    \iota_n(a)
  \end{tikzcd}\end{center}
\end{proof}

\begin{lem}\label{lem:colim_shift_one}
For any type sequence $(A,f)$, the colimits of $(A,f)$ and $S(A,f)$ are equivalent.
\end{lem}

\begin{proof}
We construct a map $\varphi:A_\infty \to S(A)_\infty$ by induction on $A_\infty$, by taking
\begin{align*}
(x:A_n) & \mapsto \iota^{S(A),S(f)}_n(f_n(x)) \\
(x:A_n) & \mapsto \kappa^{S(A),S(f)}_n(f_n(x)).
\end{align*}

Next, we construct a map $\psi:S(A)_\infty\to A_\infty$ by induction on $S(A)_\infty$, by taking
\begin{align*}
(x:S(A)_n) & \mapsto \iota^{A,f}_{n+1}(x) \\
(x:S(A)_n) & \mapsto \kappa^{A,f}_{n+1}(x).
\end{align*}

Then we prove that $\psi\circ \varphi\htpy \idfunc$ by induction on $A_\infty$, by taking
\begin{align*}
(x:A_n) & \mapsto \kappa^{A,f}_n(x)
\end{align*}
Now we compute
\begin{align*}
\mathsf{ap}_{\psi\circ\varphi}(\kappa^{A,f}_n(x))) & = \mathsf{ap}_\psi(\mathsf{ap}_\varphi(\kappa^{A,f}_n(x))) \\
& = \mathsf{ap}_\psi(\kappa^{S(A),S(f)}_n(f_n(x))) \\
& = \kappa^{A,f}_{n+1}(f_n(x))
\end{align*}
from the computation rules of $A_\infty$ and $S(A)_\infty$.

We construct the homotopy $\varphi\circ\psi\htpy\idfunc$ by induction on $A_\infty$, by taking
\begin{align*}
(x:S(A)_n) & \mapsto \kappa_{n+1}(S(f)_n(x))
\end{align*}
Now we compute
\begin{align*}
\mathsf{ap}_{\varphi\circ\psi}(\kappa^{S(A),S(f)}_n(x)) & = \mathsf{ap}_\varphi(\mathsf{ap}_\psi(\kappa^{S(A),S(f)}_n(x))) \\
& = \mathsf{ap}_\varphi(\kappa^{A,f}_{n+1}(x)) \\
& = \kappa^{S(A),S(f)}_{n+1}(f_{n+1}(x)).\qedhere
\end{align*}
\end{proof}

\begin{lem}\label{lem:colim_shift_k}
For any type sequence $(A,f)$, we have an equivalence
\begin{equation*}
\kshiftequiv_{k} : \tfcolim(A,f)\eqvsym\tfcolim (S_k(A,f)).
\end{equation*}
\end{lem}

The shift operations and the corresponding equivalences on the sequential colimits can be used to turn an arbitrary sequence $\sequence{B}{g}$ over $\sequence{A}{f}$ into an equifibered sequence over $\sequence{A}{f}$. 

\begin{defn}
Given a dependent sequence $(B,g)$ over $(A,f)$ and $x:A_0$, we define a type sequence $(B[x],g[x])$ by
\begin{align*}
B[x]_n & \defeq B_n(f^n(x)) \\
g[x]_n & \defeq g_n(f^n(x),\blank).
\end{align*}
\end{defn}

\begin{defn}
Given any sequence $\sequence{B}{g}$ over $\sequence{A}{f}$, we define an equifibered sequence
$\sequence{\square B}{\square g}$ over the sequence $\sequence{A}{f}$.
\end{defn}

\begin{constr}
For $x:A_n$ we define
\begin{equation*}
(\square B)_n(x) \defeq S_n(B)[x]_{\infty} \jdeq \tfcolim_m(B_{n+m}(f^m(x))).
\end{equation*}
Now note that
\begin{align*}
  (\square B)_{n+1}(f(x)) & \jdeq \tfcolim_m(B_{(n+1)+m}(f^m(f(x))) \\
  & \simeq \tfcolim_m(B_{n+(m+1)}(f^{m+1}(x))\\
  & \simeq \tfcolim_m(B_{n+m}(f^m(x))\\
  & \jdeq (\square B)_n(x)
\end{align*}
The first equivalence $u_{n,m}$ is given by transporting along the dependent path in \autoref{lem:iterate_succ} in the family $B$. This forms a natural equivalence, because $\mathsf{transport}$ is natural. The second equivalence is given by applying \autoref{lem:colim_shift_one}. We call the composite equivalence $F$, which shows that $\square B$ is an equifibered sequence.
\end{constr}

\begin{defn}
Let $\sequence{B}{g}$ be a sequence over $\sequence{A}{f}$. Then we define
\begin{equation*}
B_\infty : A_\infty\to\UU
\end{equation*}
to be the family over $A_\infty$ associated to the equifibered sequence $(\square B,\square g)$ via the equivalence of \autoref{thm:descent}.
\end{defn}

By construction of $B_\infty$ we get the equality 
$$r(y):\transfib{B_\infty}{\kappa_n(x)}{y}=F(y)$$ 
for $y:B_\infty(\iota_{n+1}(f_n(x)))$ witnessing that $B_\infty$ is defined by the equivalence $F$ on the path constructor. 

We now state our main result, which could be seen as a flattening lemma for sequential colimits,
with the added generality that the sequence $\sequence{B}{g}$ over
$\sequence{A}{f}$ is not required to be equifibered. 

\begin{thm}\label{thm:colim_sm}
  Let $P\defeq\sequence{P}{f}$ be a sequence over $A\defeq\sequence{A}{a}$. Then we have a
  commuting triangle
  \begin{equation*}
  \begin{tikzcd}[column sep=huge]
  \tfcolim(\msm{A}{P}) \arrow[rr,"{\alpha}"] \arrow[dr,swap,"{p\defeq\mathsf{rec}(\iota_n\circ \proj 1,\blank)}"]
  & & (x:A_\infty)\times P_\infty(x) \arrow[dl,"{\proj 1}"] \\
  & A_\infty
  \end{tikzcd}
  \end{equation*}
  in which $\alpha$ is an equivalence.
  \end{thm}
  
  The strategy of the proof is to first show that $(x:A_\infty)\times P_\infty(x)$ has the induction principle of $\colim((x:A_n)\times P_n(x),(a_n,f_n))_n$. This simplifies giving the equivalence, because 
  $(x:A_\infty)\times P_\infty(x)$ is a 2-HIT, being a sigma-type of two 1-HITs, while $\colim((x:A_n)\times P_n(x),(a_n,f_n))_n$ is a 1-HIT. Before we continue, we first define $\alpha$.
  
  The map $\alpha$ is defined by induction on $\tfcolim(\msm{A}{P})$. On the point constructors we define
  \begin{equation*}
    \alpha(\iota_n(x,y))\defeq \pairr{\iota_n(x),\iota_0(y)}.
  \end{equation*}
  For the path constructor we need to define 
  $$\kappa'_n(x,y):(\iota_{n+1}(a_nx),\iota_0(f_n(x,y)))=(\iota_n(x),\iota_0(y))$$
  The first components are equal by $\kappa_n(x)$. By the definition of $P_\infty$, transporting along $\kappa_n(x)$ takes $\iota_0(f_n(x,y))$ to $\iota_1(f_n(x,y))$, which is equal to $\iota_0(y)$ by $\kappa_0(y)$. Explicitly, we define
  $$\apfunc{\alpha}(\kappa_n(x,y))\vcentcolon=\kappa'_n(x,y)\defeq (\kappa_n(x),r(\iota_0(f_n(x,y)))\cdot\kappa_0(y)).$$
  
  \begin{thm}\label{thm:sigma-colim-induction}
    Let $E: (x:A_{\infty})\to P_{\infty}(x) \to \UU$ such that
    \begin{enumerate}
    \item For each $n:\N$, $x:A_n$, $y:P_n(x)$, a term $e_n(x,y):E(\iota_n(x),\iota_0(y))$. 
    \item For each $n:\N$, $x:A_n$, $y:P_n(x)$, a path
    \begin{equation*}
      w_{n}(x,y):e_{n+1}(a_nx,f_n(x,y))=_{\kappa'_n(x,y)}^Ee_n(x,y). 
    \end{equation*}
    \end{enumerate}
    Then there exists a function $s:(x:A_\infty)(y:P_\infty(x))\to E(y)$. 
    \end{thm}
    
  \begin{proof}
    We define the function $s$ by induction on both $x$ and $y$. We need to consider four cases, since both $x$ and $y$ can be a point constructor or vary over a path constructor.
  
  \emph{(point-point)}
  Fix $x:A_n$, we first define $g(n,x):(p:P_{\infty}(\iota_n(x))) \to E(\iota_n(x),p)$. To obtain $g(n,x)$, we do induction on $p:P_\infty(\iota_n(x))$. Fix $y:P_{n+k}(a_n^k(x))$, we need to construct a term of type $g_\ast(k,n,x,y) : E(\iota_n(x),\iota_k(y))$. Proceed by induction on $k$. We can define 
  $$g_\ast(0,n,x,y)\defeq e_n(x,y) : E(\iota_n(x),\iota_0(y)).$$ 
  Assume that $g_\ast(k)$ is defined. 
  We need to define $g_\ast(k+1,n,x,y):E(\iota_n(x),\iota_{k+1}(y))$, where $y:P(n+(k+1),a_n^{k+1}(x))$. However, the type of $y$ is equivalent to the type $P((n+1)+k,a_{n+1}^k(a_n(x)))$ via the equivalence $u_{n,k}$. Therefore, it suffices to define for $z:P_{(n+1)+k}(a_{n+1}^k(a_n(x)))$
  \begin{equation*} 
  g_\ast(k+1,n,x,u_{n,k}(z)) : E(\iota_n(x),\iota_{k+1}(u_{n,k}(z))).
  \end{equation*}
  
  By induction hypothesis we have $g_\ast(k,n+1,a_n(x),z):E(\iota_{n+1}(a_n(x)),\iota_k(z))$, so it suffices to show that
  
  $$\kappa^\ast_{n,k}(x,z):(\iota_{n+1}(a_n(x)),\iota_k(z))=(\iota_n(x),\iota_{k+1}(u_{n,k}(z))).$$ 
  This construction is similar to that of $\kappa'_n(x,y)$.
  The first components are equal by $\kappa_n(x)$, and for the second components we need to show that $\transfib{P_\infty}{\kappa_n(x)}{\iota_k(z)}=\iota_{k+1}(u_{n,k}(z))$. This follows from the computation rule of $P_\infty$ on paths, since the equivalence used to define $P_\infty$ sends $\iota_k(z)$ to $\iota_{k+1}(u_{n,k}(z))$. Specifically,
  $$\kappa^\ast_{n,k}(x,z)\defeq(\kappa_n(x),r(\iota_k(y))).$$
  This finishes the construction of $g_\ast$, hence also of $g$ on points. By construction, we get the following equation: 
  $$\mu_{n,k}(x,z):\dpath{E}{\kappa^\ast_{n,k}(x,z)}{g_\ast(k,n+1,a_n(x),z)}{g_\ast(k+1,n,x,u_{n,k}(z))}$$
  
  \emph{(point-path)} To define $g$ on paths $\kappa_k(y):\iota_{k+1}(f(y))=\iota_k(y)$, we need to give a dependent path
  $$\nu(k,n,x,y):g_\ast(k+1,n,x,f(y))=^{E(\iota_n(x))}_{\kappa_k(y)}g_\ast(k,n,x,y).$$
  We do this by induction on $k$. For $k=0$ note that $u_{n,0}$ is the identity function, and the goal definitionally reduces to
  $$\nu(k,n,x,y):\transfib{E}{\kappa^\ast_{n,0}(x,f_n(x,y)}{e_{n+1}(a_n(x),f_n(x,y)}=^{E(\iota_n(x))}_{\kappa_0(y)}e_n(x,y).$$
  Note that $\kappa'_n(x,y)=\kappa^\ast_{n,0}(f_n(x,y)\cdot(1,\kappa_0(y))$, which means we get this from $w_n(x,y)$. Now suppose that $\nu(k)$ is defined. We need to define for $y:P(n+(k+1),a_n^{k+1}(x))$
  $$\nu(k+1,n,x,y):g_\ast(k+2,n,x,f(y))=^{E(\iota_n(x))}_{\kappa_k(y)}g_\ast(k+1,n,x,y).$$
  Now we again write $y=u_{n,k}(z)$ for $z:P((n+1)+k,a_{n+1}^k(a_n(x)))$ and we equivalently need to give
  $$\nu(k+1,n,x,u_{n,k}(z)):g_\ast(k+2,n,x,f(u_{n,k}(z)))=^{E}_{(1,\kappa_k(y))}g_\ast(k+1,n,x,u_{n,k}z).$$
  We will define this as the composition of a square that we will give later in the proof.
  
  \emph{(path-point)} We have defined $s$ on points constructors of $A_\infty$. To define it on the path $\kappa_n(x):\iota_{n+1}(a_n(x))=\iota_n(x)$ we need a path $g(n+1,a_n(x))=g(n,x)$ over $\kappa_n(x)$. By function extensionality, we can characterize dependent paths in a function type, which means we need to show:
  \begin{equation*}
  (p:P_\infty(\iota_{n+1}(a_n(x))))\to g(n+1,a_n(x),p) =^E_{(\kappa_n(x),1)} g(n,x,\transfib{P_\infty}{\kappa_n(x)}{p}).
  \end{equation*}
  Now for $p:P_\infty(\iota_{n+1}(a_n(x)))$, we can apply the path $r(p)$, which means we need to construct the following path (note that $r(p)$ is added to the path, since $g$ is a dependent function):
  \begin{equation*}
  g(n+1,a_n(x),p) =^E_{(\kappa_n(x),r(p))} g(n,x,F(p)).
  \end{equation*}
  We proceed by induction on $p$. If $p\jdeq\iota_k(y)$ for $k:\N$, $y:P_{(n+1)+k}(a_{n+1}^k(a_n(x)))$, then $F(p)\jdeq \iota_{k+1}(u_{n,k}(y))$ and we need a path
  \begin{equation*}
  g_\ast(k,n+1,a_n(x),y) =^E_{(\kappa_n(x),r(\iota_k(y))))} g_\ast(k+1,n,x,u_{n,k}(y)).
  \end{equation*}
  Now the path $(\kappa_n(x),r(\iota_k(y)))\jdeq\kappa^\ast_{n,k}(x,y)$, hence this dependent path is given by $\mu_{n,k}(x,y)$.
  
  \emph{(path-path)} If $p$ varies over $\kappa_k(y)$, we need to give a dependent path in a family of dependent paths. 
  This is equivalent to filling the following dependent square in the family $E$, 
  which lies over the naturality square form by applying $\lam{p}(\kappa_n(x),r(p))$ to the path $\kappa_k(y)$.\footnote{The 
  left and right sides of the square are not quite correct, 
  the dependent function applied to $\kappa_k(y)$ are pathovers lying over $\kappa_k(y)$, and not $(1,\kappa_k(y))$. 
  However, pathovers lying over $\kappa_k(y)$ in the family $E(\iota_n(x))$ are equivalent to pathovers lying over $(1,\kappa_k(y))$ in the family $E$, 
  and this equivalence commutes with all operations we perform, therefore we omit them in this proof. 
  The following calculations are only type correct when these equivalences are inserted back. 
  Furthermore, we omit some other details. For example, if $p = q$, 
  then $\mapdep{g}{p}$ and $\mapdep{g}{q}$ have different types: 
  the former is a dependent path over $p$ and the latter one over $q$. 
  However, if you modify the path over which they lie, they become equal. 
  These ``modifications'' can be pushed down to the square in $(x:A_\infty)\times P_\infty(x)$, 
  and the proof still goes through. For the full details, consult the formal proof.}
  \begin{xcenter}\begin{tikzcd}[column sep=25mm]
    g_\ast(k+1,n+1,a_n(x),f_{(n+1)+k)}(y)) 
    \ar[r,equal,"{\mu_{n,k+1}(x,f_{(n+1)+k}(y))}","{\kappa^\ast_{n,k+1}(x,f_{n+k}(y))}" swap] 
    \ar[d,equal,"{\mapdep{g(n+1,a_n(x))}{\kappa_k(y))})}","{(1,\kappa_k(y)}" swap] &
    g_\ast(k+2,n,x,u_{n,(k+1)}(f_{(n+1)+k}(y))) 
    \ar[d,equal,"{\mapdep{g(n,x)\circ F}{\kappa_k(y))})}","{(1,\kappa_k(y)}" swap]\\
    g_\ast(k,n+1,a_n(x),y) 
    \ar[r,equal,"{\mu_{n,k}(x,y)}","{\kappa^\ast_{n,k}(x,y)}" swap] & 
    g_\ast(k+1,n,x,u_{n,k}(y))
  \end{tikzcd}\end{xcenter}
  Below and to the left of each equal sign we give the path in $(x:A_\infty)\times P_\infty(x)$ over which the pathover lie. Above and to the right of each equal sign we give the value of the dependent path. 
  
  Now $\mapdep{g(n+1,a_n(x))}{\kappa_k(y)}$ (occurring in the left pathover) is equal to $\nu(k,n+1,a_n(x),y)$ by definition of $g$. On the right, we have ($\delta$ is the naturality of $u_{n,k}$)
  \begin{align*}
    \mapdep{g(n,x)\circ F}{\kappa_k(y)} &= 
    \mapdep{g(n,x)}{\map{F}{\kappa_k(y)}}\\
    &= \mapdep{g(n,x)}{\map{\iota_{k+2}}{\delta(y)}\cdot\kappa_{k+1}(u_{n,k}(y))}\\
    &= \mapdep{g_\ast(k+2,n,x)}{\delta(y)} \cdot \mapdep{g(n,x)}{\kappa_{k+1}(u_{n,k}(y))}\\
    &= \mapdep{g_\ast(k+2,n,x)}{\delta(y)} \cdot \nu(k+1,n,x,u_{n,k}(y))\\
  \end{align*}
  Now we can move the first part of the expression to the top of the square, which means we need to fill the following squareover (where we made some arguments implicit).
  \begin{center}\begin{tikzcd}[column sep=25mm]
    g_\ast(f(y)) 
    \ar[r,equal,"{\mu(f(y))}","{\kappa^\ast(f(y))}" swap] 
    \ar[d,equal,"{\nu(k,n+1,a_n(x),y)}","{(1,\kappa(y))}" swap] &
    g_\ast(u(f(y))) 
    \ar[r,equal,"{\mapdep{g_\ast(k+2,n,x)}{\delta(y)}}","{(1,\delta(y))}" swap] &
    g_\ast(f(u(y))) 
    \ar[d,equal,"{\nu(k+1,n,x,u_{n,k}(y))}","{(1,\kappa(u_{n,k}(y)))}" swap]\\
    g_\ast(y) 
    \ar[rr,equal,"{\mu(y)}","{\kappa^\ast(y)}" swap] & &
    g_\ast(u_{n,k}(y))
  \end{tikzcd}\end{center}
  Note that in this squareover $g$ does not occur, and $\nu$ only occurs on the left side (applied to $k$) and on the right side (applied to $k+1$). Therefore, the top, bottom and left side form a valid open box, and we define $\nu(k+1,n,x,u_{n,k}(y))$ to be the composition of \emph{this} open box. This inductively defines $\nu$, and makes the filler for this square automatic. This finishes the proof.
  \end{proof}
  \begin{proof}[Proof (of Theorem \ref{thm:colim_sm})]
    We first define a map $$\beta: \big((x:A_\infty)\times P_\infty(x)\big) \to \tfcolim(\msm{A}{P}).$$ 
    We do this by induction on $x:A_\infty$ and $p:P_\infty(x)$ individually, so we get four cases again 
    (we do not use our newly defined induction principle, because we have not proven a computation rule for it).
  
    \emph{(point-point)} Suppose $x:A_n$ and $y:P_{n+k}(a_n^k(x))$. We define 
    $$\beta(\iota_n(x),\iota_k(y))\defeq\iota_{n+k}(a_n^k(x),y).$$
    \emph{(point-path)} To show that the second argument respects $\kappa_k(y)$, we define
    $$\mapfunc{\beta(\iota_n(x))}(\kappa_k(y))\vcentcolon=\kappa_{n+k}(a_n^k(x),y):\iota_{n+(k+1)}(a_n^{k+1}(x),f(y))=\iota_{n+k}(a_n^k(x),y).$$
    \emph{(path-point)} To show that the first argument respects $\kappa_n(x)$, we need to give a dependent path
    $$\beta(\iota_{n+1}(a_n(x)))=^{P_\infty({-})\to\tfcolim(\msm{A}{P})}_{\kappa_n(x)}\beta(\iota_n(x)).$$
    By function extensionality, this is equivalent to showing for $p:P_\infty(\iota_{n+1}(a_n(x)))$ that
    $$\beta(\iota_{n+1}(a_n(x)),p)=\beta(\iota_n(x),\transfib{P_\infty}{\kappa_n(x)}{p}).$$
    We apply $\apfunc{\beta(\iota_n(x))}\inv{(r(p))}$ on the right, so that we have to show
    $$\beta(\iota_{n+1}(a_n(x)),p)=\beta(\iota_n(x),F(p)).$$
    Now we apply induction on $p$. If $p\equiv\iota_k(y)$, then $F(p)\jdeq \iota_{k+1}(u_{n,k}(y))$ and we need to show
    $$\mu_{n,k}(x,y):\iota_{(n+1)+k}(a_{n+1}^k(a_n(x)),y)=\iota_{n+(k+1)}(a_n^{k+1},u_{n,k}(y)).$$
    But the triples $((n+1)+k,a_{n+1}^k(a_n(x)),y)$ and $(n+(k+1),a_n^{k+1},u_{n,k}(y))$ are equal: the first two components by \autoref{lem:iterate_succ} and the last component because $u_{n,k}$ was defined by transporting along the equality of the first components. Let us call this equality $s$. So we define $\mu_{n,k}(x,y)$ by applying $\iota$ to $s$.
    
    \emph{(path-path)} Suppose $p$ varies along $\kappa_k(y)$, we need to construct a proof of a pathover in an equality type. This is equivalent to filling the following square.
    \begin{center}\begin{tikzcd}[column sep=25mm]
      \beta(\iota_{n+1}(a_n(x)),\iota_{k+1}(f(y)))
      \ar[r,equal,"{\mu_{n,k+1}(x,f(y))}"] 
      \ar[d,equal,"{\mapfunc{\beta(\iota_{n+1}(a_nx))}(\kappa_k(y)})"] &
      \beta(\iota_n(x),\iota_{k+2}(u_{n,k+1}(f(y))))
      \ar[d,equal,"{\mapfunc{\beta(\iota_n(x))\circ F}(\kappa_k(y))}"]\\
      \beta(\iota_{n+1}(a_n(x)),\iota_k(y))
      \ar[r,equal,"{\mu_{n,k}(x,y)}"] & 
      \beta(\iota_n(x),\iota_{k+1}(u_{n,k}(y)))
    \end{tikzcd}\end{center}
    By simplifying the left and right path, this reduces to
    \begin{center}\begin{tikzcd}[column sep=10mm]
      \beta(\iota(a(x)),\iota(f(y)))
      \ar[r,equal,"{\mu(f(y))}"] 
      \ar[d,equal,"{\kappa(a^k(a(x)),y)}"] &
      \beta(\iota(x),\iota(u(f(y))))
      \ar[r,equal,"{\mapfunc{\beta(\iota(a(x)),\iota({-}))}(\kappa(y))}"] &[15mm]
      \beta(\iota(x),\iota(f(u(y))))
      \ar[d,equal,"{\kappa(a^k(x),u(y))}"]\\
      \beta(\iota(a(x)),\iota_k(y))
      \ar[rr,equal,"{\mu(y)}"] & &
      \beta(\iota(x),\iota(u(y)))
    \end{tikzcd}\end{center}
    Now the concatenation of the two paths on the top reduces to the function $i\defeq\lam{n}{x}{y}\iota_{n+1}(a(x),f(y))$ applied to $s$. Then the square is exactly the naturality square of the homotopy $\kappa:i\sim\iota$ applied to the path $s$. This finishes the definition of $\beta$.
  
    Now we need to show that $\beta\circ\alpha\sim\idfunc$. Take $p:\tfcolim(\msm{A}{P})$, we apply induction to $p$. If $p\jdeq\iota_n(x,y)$, then the equality holds by reflexivity:
    $$\beta(\alpha(p))\jdeq\beta(\iota_n(x),\iota_0(y))\jdeq\iota_n(x,y)\jdeq p.$$
    If $p$ varies over $\kappa_n(x,y)$, we need to fill a square with two degenerate sides, so we need to prove that 
    $\apfunc{\beta\circ\alpha}(\kappa_n(x,y))=\kappa_n(x,y).$
    We can show this as follows.
    \begin{align*}
      &\mathrel{\hphantom{=}}\apfunc{\beta\circ\alpha}(\kappa_n(x,y))\\
      &=\apfunc{\beta}(\kappa_n(x),r(\iota_0(f(y)))\cdot\kappa_0(y))\\
      &=\apfunc{\beta}(\kappa'_n(x,y))\\
      &=\mu_{n,1}(x,f(y))\cdot\apfunc{\beta(\iota_n(x))}\inv{(r(\iota_0(f(y))))}\cdot\apfunc{\beta(\iota_n(x))}(r(\iota_0(f(y)))\cdot\kappa_0(y))\\
      &=\mu_{n,1}(x,f(y))\cdot\apfunc{\beta(\iota_n(x))}(\kappa_0(y))\\
      &=\apfunc{\beta(\iota_n(x))}(\kappa_0(y))\\
      &=\kappa_n(x,y)
    \end{align*}
    In the third step we use that $\apfunc{\beta}(p,q)=\mapdep{\beta}{p}(f(y))\cdot\mapfunc{\beta{\iota_n(x)}(q)}$ and in the fifth step that $\mu_{n,k}(x,y)=1$ for any \emph{numeral} $k$.
  
    Finally we need to show that $\alpha\circ\beta\sim\idfunc$. Take $p:(x:A_\infty)\times P_\infty(x)$. We apply the induction principle proven in \autoref{thm:sigma-colim-induction} to $p$. Suppose that $p\jdeq(\iota_n(x),\iota_0(y))$. Then the equality holds by reflexivity:
    $$\alpha(\beta(p))\jdeq\alpha(\iota_n(x,y))\jdeq(\iota_n(x),\iota_0(y))\jdeq p.$$
    If $p$ varies over $\kappa'_n(x,y)$, then we have to show (similar to the proof $\beta\circ\alpha\sim\idfunc$) that 
    $$\apfunc{\alpha\circ\beta}(\kappa'_n(x,y))\kappa'_n(x,y).$$
    But by the previous computation, $\apfunc{\beta}(\kappa'_n(x,y))=\kappa_n(x,y)$, so we have
    $$\apfunc{\alpha\circ\beta}(\kappa'_n(x,y))=\apfunc{\alpha}(\kappa_n(x,y))=\kappa'_n(x,y).$$
    This finishes the proof. 
  \end{proof}
  
  \begin{cor}\label{cor:eq_colim}
    Consider a sequence $(A_n,f_n)_n$. Then for any $a,a':A_n$ there is an equivalence
    \begin{equation*}
    \eqv{(\iota_n(a)=_{A_\infty}\iota_n(a'))}{\tfcolim(f^k(a)=_{A_{n+k}}f^k(a'))}.
    \end{equation*}
    \end{cor}
    \begin{proof}
    We first prove this for $n \jdeq 0$.   
    Note that for any $a:A_0$, we have the diagram
    \begin{equation*}
    \begin{tikzcd}
    \lam{a':A_0}a=a' \arrow[r] \arrow[d,->>] & \lam{a':A_1} f(a)=a' \arrow[r] \arrow[d,->>] & \lam{a':A_2} f^2(a)=a' \arrow[r] \arrow[d,->>] & \cdots \\
    A_0 \arrow[r,"f_0"] & A_1 \arrow[r,"f_1"] & A_2 \arrow[r,"f_2"] & \cdots
    \end{tikzcd}
    \end{equation*}
    This defines a type family $P:A_\infty\to\type$ with 
    $$P(\iota_n(a'))\defeq\tfcolim_k(f^{0\le n+k}(a)=_{A_{n+k}}f^k(a')).$$
    Now we use \autoref{thm:colim_sm} to see that the total space of $P$ is contractible.
    \begin{align*}
    (a':A_\infty)\times P(a')
    & \eqvsym \tfcolim_n((a':A_n)\times f^{n}(a)=a') \\
    & \eqvsym \tfcolim_n(\unit) \\
    & \eqvsym \unit.
    \end{align*}
    Since $\iota_0(\refl{a}):P(\iota_0(a))$ and noting that $f^{0\le 0+k}(a)\jdeq f^k(a)$ we can now conclude by the total space method to characterize the identity type that 
    $$(\iota_0(a)=_{A_\infty}\iota_0(a'))\simeq P(\iota_0(a'))\jdeq\tfcolim(f^k(a)=_{A_{0+k}}f^k(a')).$$
    
    For general $n$, we use \autoref{lem:colim_shift_k}, which gives us an equivalence 
    $\kshiftequiv_n: A_\infty\simeq \tfcolim(S_n(A,f))$. For $a,a':A_n$ we can now compute:
    \begin{align*}
      (\iota_n(a)=_{A_\infty}\iota_n(a'))
      & \eqvsym (\kshiftequiv_n(\iota_n(a))=_{\tfcolim(S_n(A,f))}\kshiftequiv_n(\iota_n(a'))) \\
      & \eqvsym (\iota_0(a)=_{\tfcolim(S_n(A,f))}\iota_0(a')) \\
      & \eqvsym \tfcolim(S_n(f)^k(a)=_{S_n(A)_{0+k}}S_n(f)^k(a')) \\
      & \eqvsym \tfcolim(f^k(a)=_{A_{n+k}}f^k(a'))..
    \end{align*}
    The last equivalence comes from a natural equivalences of the sequences, because there is a dependent path between
    $S_n(f)^k(a)$ and $f^k(a)$ over the canonical path that $n+(0+k)=n+k$.
    \end{proof}
        
    \begin{cor}\label{cor:fiber_functor}
      Suppose given a natural transformation $\tau : \sequence{A'}{f'}\to\sequence{A}{f}$ and a point $a:A_n$. Then
      $$\hfib{\tau_\infty}{\iota_n(a)}\simeq\colim(\square\mathsf{fib}_\tau[a])\jdeq\colim_k(\hfib{\tau_{n+k}}{f^k(a)}).$$
    \end{cor}
    \begin{proof}
      Consider the following diagram, where the equivalences on the top are given by \autoref{thm:colim_sm} and the fact that the total space of the fiber of a function is the domain of that function.
      \begin{center}\begin{tikzcd}[column sep=10mm]
        (x:A_\infty)\times (\mathsf{fib}_\tau)_\infty(x) \ar[r,"{\sim}"] \ar[dr,"{\pi_1}"] &
        \colim_k((x:A_n))\times \hfib{\tau_n}{x}) \ar[r,"{\sim}"]\ar[d,"p"] &
        A'_\infty \ar[dl,"{\tau_\infty}"] \\
        & A_\infty &
      \end{tikzcd}\end{center}
      This diagram commutes: the left triangle commutes by \autoref{thm:colim_sm} and the right triangle commutes by the 1-functoriality of the colimit, \autoref{lem:seq_colim_functor}. Therefore,
      $$\hfib{\tau_\infty}{\iota_n(a)}\simeq\hfib{\pi_1}{\iota_n(a)}\simeq (\mathsf{fib}_\tau)_\infty(\iota_n(a))\jdeq \colim(\square\mathsf{fib}_\tau[a]).$$
    \end{proof}
    
    \begin{cor}\label{cor:trunc_colim}
      Consider a sequence $\sequence{A}{f}$ and some $k\geq-2$.
      \begin{enumerate}
        \item\label{part:colim_is_trunc} If $A_n$ is $k$-truncated for all $n:\nat$, then $A_\infty$ is $k$-truncated.
        \item\label{part:colim_trunc} We have an equivalence $$\trunc{k}{A_\infty}\simeq\colim(\trunc{k}{A_n},\trunc{k}{f_n})_n.$$
        \item\label{part:colim_is_connected} If $A_n$ is $k$-connected for all $n:\nat$, then $A_\infty$ is $k$-connected.
        \item\label{part:functor_is_trunc_conn} Given a natural transformation $(\tau,H):\sequence Af \to \sequence{A'}{f'}$ such that $\tau_n$ is $k$-truncated ($k$-connected) for all $n$, then $\tau_\infty$ is $k$-truncated ($k$-connected).
        \item\label{part:iota_is_trunc_conn} If $f_n$ is $k$-truncated ($k$-connected) for all $n$, then $\iota_0$ is $k$-truncated ($k$-connected).
      \end{enumerate}
    \end{cor}
      
    \begin{rmk}
      By \autoref{lem:colim_shift_k} we can generalize the quantification ``for all $n:\nat$'' in this Corollary to the weaker ``there exists an $m:\nat$ such that for all $n\ge m$''. In part \ref{part:iota_is_trunc_conn} the conclusion then becomes that $\iota_m$ is $k$-truncated ($k$-connected).
    \end{rmk}
    
    \begin{proof}\mbox{}
      \begin{enumerate}
        \item We prove this by induction on $k$. Suppose $k=-2$, then $f_n$ is an equivalence for all $n$. Therefore $A_\infty\simeq A_0$ by \autoref{lem:equiv_equiseq}, hence $A_\infty$ is contractible.
    
        Now suppose $k\jdeq k'+1$. Take $x,x' : A_\infty$, we need to show that $x=x'$ is $k'$-truncated. Since being truncated is a mere proposition, by induction on $x$ and $x'$ we may assume that $x\jdeq\iota_n(a)$ and $x'\jdeq\iota_m(a')$. Now
        $\iota_n(a)=\iota_{\max(n,m)}(f^{n\le \max(n,m)}(a))$ and $\iota_m(a')=\iota_{\max(n,m)}(f^{m\le \max(n,m)}(a'))$, therefore the type $\iota_n(a)=\iota_m(a')$ is equivalent to
        $$\iota_{\max(n,m)}(f^{n\le \max(n,m)}(a))=\iota_{\max(n,m)}(f^{m\le \max(n,m)}(a')).$$ Therefore it suffices to show that the latter equality type is $k'$-truncated. By \autoref{cor:eq_colim} we need to show that
        $$\tfcolim(f^\ell(f^{n\le \max(n,m)}(a))=f^\ell(f^{m\le \max(n,m)}(a')))_\ell$$
        is $k'$-truncated, which follows from the induction principle and the fact that $A_{\max(n,m)+\ell}$ is $(k'+1)$-truncated.
        \item From the functoriality of the sequential colimit, we get a function
        $$A_\infty\to\colim_n(\trunc{k}{A_n},\trunc{k}{f_n}).$$ 
        Because the right hand side is $k$-truncated, this induces a map
        $$g:\trunc{k}{A_\infty}\to\colim(\trunc{k}{A_n},\trunc{k}{f_n})_n.$$ 
        For the other direction, we define the function
        $$h:\colim(\trunc{k}{A_n},\trunc{k}{f_n})_n\to\trunc{k}{A_\infty}$$ 
        by 
        $$h(\iota_n(\tproj{k}{a}))\defeq\tproj{k}{\iota_n(a)}$$ 
        and 
        $$\mapfunc{h}(\kappa_n(\tproj{k}{a}))\vcentcolon=\mapfunc{\tprojf{k}}(\kappa_n(a)).$$ It is straightforward to show that both $h\circ g$ and $g\circ h$ are homotopic to the identity.
        \item Since $A_n$ is $k$-connected, $\trunc{k}{A_n}$ is contractible, and therefore $\colim_n(\trunc{k}{A_n})\simeq \trunc{k}{A_\infty}$ is contractible.
        \item A function is $k$-truncated ($k$-connected) whenever its fibers are $k$-truncated ($k$-connected). 
        Let $x:A_\infty$. We need to show a proposition, so we may assume that $x\jdeq \iota_n(a)$ for some $a:A_n$. Now 
        $\hfib{\tau_\infty}{\iota_n(a)}\simeq\tfcolim(\square\mathsf{fib}_\tau[a])$ by \autoref{cor:fiber_functor}. Since $\hfib{\tau_n}{x}$ is $k$-truncated ($k$-connected) for all $n$, we know that $\tfcolim(\square\mathsf{fib}_\tau[a])$ is $k$-truncated ($k$-connected) for all $n$, by part \ref{part:colim_is_trunc} or \ref{part:colim_is_connected}.
        \item Consider the natural transformation
        \begin{center}\begin{tikzcd}[column sep=10mm]
          A_0 \ar[r,equal] \ar[d,equal] &
          A_0 \ar[r,equal] \ar[d,equal,"f"] &
          A_0 \ar[r,equal] \ar[d,equal,"{f^{0\le2}}"] & 
          A_0 \ar[r,equal] \ar[d,equal,"{f^{0\le3}}"] & 
          \cdots \ar[r,equal] &
          \colim(A_0)_n \ar[d] \\
          A_0 \ar[r] & 
          A_1 \ar[r] & 
          A_2 \ar[r] & 
          A_3 \ar[r] & 
          \cdots \ar[r] & 
          A_\infty
        \end{tikzcd}\end{center}
        The maps $f^{0\le n}:A_0\to A_n$ are $k$-truncated ($k$-connected) and form a natural transformation. Therefore, by part \ref{part:functor_is_trunc_conn} the map $f^{0\le\infty}:\colim_n(A_0)\to A_\infty$ is $k$-truncated ($k$-connected). The fiber of $\iota_0$ over $x:A_\infty$ is the same as the fiber of $f^{0\le\infty}$ over $x$, and therefore $\iota_0$ is $k$-truncated ($k$-connected).
      \end{enumerate}
    \end{proof}
    
We can use this machinery, in particular \autoref{thm:colim_sm}, to define the localization for maps between $\omega$-compact types. We will omit the construction here, but this will be published in an upcoming preprint.

\chapter{Homotopy Theory}\label{cha:homotopy-theory}

As discussed in the introduction, one very useful application of HoTT is synthetic homotopy
theory. Many results in homotopy theory have been stated and proven in HoTT in a synthetic way. Most
of these results have also been formalized in a proof assistant. This is important,
because one of the advantages of HoTT is to make verification of proofs by a proof assistant
practically possible. Formalizing results that have been proved internally in HoTT provides more evidence for this. 

In this chapter we will look at various topics in homotopy theory and give proofs for them in HoTT
that are fully checked by the Lean proof assistant. In \cref{sec:computing-pi3s2} we will describe
a formalization of the proof that $\pi_3(\S^2)=\Z$. This was already known to be provable in HoTT, but no fully
formalized proof has been given before. We will discuss some new properties proven about
Eilenberg-MacLane spaces in HoTT in \cref{sec:eilenb-macl-spac}, namely that the Eilenberg-MacLane
space functor induces an equivalence of categories. In \cref{sec:smash-product} we prove the
adjunction of the smash product and pointed maps, from which we can conclude that the smash product
is associative. 

None of these results have been formalized before, even including formalization in foundations other than HoTT. 
In fact, not much homotopy theory has been formalized in other foundations. 
The most notable examples of formalizations are the formalization of basic properties of the fundamental group~\cite{zhan2017auto2} 
and the formalization of singular homology theory~\cite[\texttt{Multivariate/homology.ml}]{hollight}.





\section{Computing \texorpdfstring{$\pi_3(\S^2)$}{pi\textunderscore3(S\textasciicircum2)}}%
\label{sec:computing-pi3s2}
Computing that $\pi_3(\S^2)=\Z$ has been done before in Homotopy Type Theory, but it has not been
formalized in a proof assistant before. In this section we will discuss some considerations of
formalizing the proof that $\pi_3(\S^2)=\Z$. The Hopf fibration was formalized in Lean by Ulrik
Buchholtz and was formalized before in Agda by Guillaume Brunerie. The remaining results are formalized by the author.



\subsection{The long exact sequence of homotopy groups}\label{sec:les-homotopy}
We start with an important result in homotopy theory, the long exact sequence of homotopy
groups. 

This has been proven before in HoTT. Two different proofs are given in~\cite[Section 8.4]{hottbook} 
and~\cite[Section 2.5.1]{brunerie2016spheres},
although these proofs have not been formalized. There have been previous
formalizations of parts of this
result~\cite{avigad2015limits,voevodsky2015lecture,unimath}; however none of
these formalizations are complete in the sense that they can be used to deduce
the results in this section. 

The statement is as follows.
\begin{thm}[Long exact sequence of homotopy groups]\label{thm:les-homotopy}
  Suppose $f : X \to Y$ is a pointed map. Then the following is an exact sequence

  \begin{center}\begin{tikzpicture}[node distance=3cm, thick, node
    distance=12mm]
  \node (Y)     at (0,0) {$\pi_0(Y)$};
  \node[left = of Y] (X) {$\pi_0(X)$};
  \node[left = of X] (F) {$\pi_0(F)$};
  \node[above = of Y] (OY) {$\pi_1(Y)$};
  \node[above = of X] (OX) {$\pi_1(X)$};
  \node[above = of F] (OF) {$\pi_1(F)$};
  \node[above = of OY] (O2Y) {$\pi_2(Y)$};
  \node[above = of OX] (O2X) {$\pi_2(X)$};
  \node[above = of OF] (O2F) {$\pi_2(F)$};
  \node[above = 5mm of O2X] (dots) {$\vdots$};
  \path[every node/.style={font=\sffamily\small}]
  (X) edge[->] node [above] (f){$\pi_0(f)$} (Y)
  (F) edge[->] node [below] (f){$\pi_0(p_1)$} (X)
  (OY) edge[->] node [left] (f){$\pi_0(\delta)\quad\mbox{}$} (F)
  (OX) edge[->] node [above] (f){$\pi_1(f)$} (OY)
  (OF) edge[->] node [below] (f){$\pi_1(p_1)$} (OX)
  (O2Y) edge[->] node [left] (f){$\pi_1(\delta)\quad\mbox{}$} (OF)
  (O2X) edge[->] node [above] (f){$\pi_2(f)$} (O2Y)
  (O2F) edge[->] node [below] (f){$\pi_2(p_1)$} (O2X);
\end{tikzpicture}
\end{center}
Here $F\defeq\fib_f$ is the fiber of $f$, $p_1:F\to X$ is the first projection, and
$\delta:\Omega Y \to F$ is defined in the proof.
\end{thm}
First of all, we have to carefully formulate the statement of this theorem in type theory. The naive
thing to do is to say that there is a sequence $A : \N \to \set^*$ and maps $f : (n : \N) \to
A_{n+1} \to A_n$ such that
$$A_0\defeq\pi_0(Y),\quad A_1\defeq\pi_0(X),\quad A_2\defeq\pi_0(F),$$
and so forth. Continuing, this means that
$$A_{3n}=\pi_n(Y),\quad A_{3n+1}=\pi_n(X),\quad A_{3n+2}=\pi_n(F).$$
However, there is no way to make these equalities definitional, the elimination principle for the
natural numbers does not allow for computation rules like that. This means that the map
$f_{3n}:A_{3n+1}\to A_{3n}$ cannot be compared directly to $\pi_n(f)$ since the domain and codomain
are note definitionally equal. Setting things up this way is possible, but makes reasoning about it
unnecessarily complicated. Instead, we change the indexing set, using $\N\times\fin_3$ instead of $\N$. We will work with a general
notion of sequences with a flexible choice of indexing set.
\begin{defn}\label{def:chain-complex}
  A \emph{successor structure} is a type $I$ with endomap $S : I \to I$ called the
  \emph{successor}. We will write $i+n$ for $i:I$ and $n:\N$ to mean iterated application of the
  successor function, $i+n\defeq S^n(i)$.

  A \emph{chain complex} indexed by a successor structure $I$ is a family of pointed sets
  $A : I \to \set^*$ and maps $f : (i : I) \to A_{i+1} \to A_i$ with the property that
  $(i : I) \to (a : A_{i+2}) \to f_i(f_{i+1}(a))=a_0^i$ where $a_0^i$ is the basepoint of $A_i$. We call a chain complex \emph{exact} or a \emph{long exact sequence} if
  $$(i : I) \to (a : A_{i+1}) \to f_i(a) = a_0^i \to \|(a' : A_{i+2}) \times f_{i+1}(a')=a\|.$$
  A \emph{type-valued chain complex} is the same, except that $A_i$ is only required to be a pointed type (not a pointed set). A type-valued chain complex is \emph{exact} or a \emph{type-valued exact sequence} if the above property holds without any propositional truncation, i.e. if
  $$(i : I) \to (a : A_{i+1}) \to f_i(a) = a_0^i \to (a' : A_{i+2}) \times f_{i+1}(a')=a.$$
\end{defn}

\begin{rmk}
Note that a type-valued exact sequence gives part of the structure of a \emph{fiber sequence}. A \emph{fiber sequence} is a sequence where $A_{i+2}$ ``is'' the fiber of $f_i$. This means that $(A_{i+2},f_{i+1})=(\fib_{f_i},p_1)$ for all $i$. Using univalence this can be unpacked in an equivalence and a commuting triangle. 
In a type-valued exact sequence we just require two maps back and forth $A_{i+2}\leftrightarrow \fib_{f_i}$ such that the corresponding triangles commute, but we do not require that these maps are mutual inverses. In the text below we will have sequences that are not fiber sequences, so we require this additional generality.
\end{rmk}

\begin{ex}\label{ex:succ_structure}
  Some useful examples of successor structures are $(\N,\lam{n}n+1)$ and $(\Z,\lam{n}n+1)$. Sequences over these successor structures correspond to one-sided and two-sided infinite sequences. We can also mimic one-sided infinite sequences in the other direction using the successor structure $(\N,\lam{n}n-1)$ (with the convention that $0-1=0$). This has the disadvantage that there is one extra map $A_0\to A_0$. Whenever we use $\N$ as successor structure in this section, we use $\lam{n}n+1$ as its successor.

  Furthermore, if $N$ is a successor structure and $k : \N$, then we define a successor structure on
  $N \times \fin_{k+1}$ by defining
  $$S(n,i)\defeq\begin{cases}(n+1,0) & \text{if $i=k$}\\
    (n,i+1) & \text{otherwise}\end{cases}$$
  Note that $n+1$ is addition in the successor structure $N$.
\end{ex}

We now build the long exact sequence of homotopy groups in five steps. The order of these steps is somewhat arbitrary and can be altered. 
We perform the 0-truncation of the sequence as the last step, so that the intermediate sequences contain as much information as possible.
\begin{enumerate}[(1)]
\item First we define the fiber sequence of $f$.
\item Then we show that this sequence is equivalent to a sequence involving iterated loop spaces.
\item We fix some negation signs in the exact sequence.
\item We index the sequence over $\N\times\fin_3$.
\item We 0-truncate the sequence to obtain the sequence in \cref{thm:les-homotopy}.
\end{enumerate}

We first need some lemmas about fibers.

\begin{lem}\label{lem:fibers}
  Suppose given a pointed map $f : A \to^* B$. Let $p_1 : \fib_f \to^* A$ be the first projection. Then there is a pointed natural equivalence $e_f:\fib_{p_1} \simeq^* \Omega B$. 

  Furthermore, if $q_1:\fib_{p_1}\to\fib_f$ is the first projection, we get a commuting square
  \begin{center}\begin{tikzpicture}[node distance=2cm, thick, node
    distance=12mm]
    \node (tl)       at (0,0) {$\Omega A$};
    \node[right = of tl] (tr) {$\Omega B$};
    \node[below = of tl] (bl) {$\fib_{q_1}$};
    \node (br) at (tr |- bl)   {$\fib_{p_1}$};
    \path[every node/.style={font=\sffamily\small}]
    (tl) edge[->] node {$-\Omega f$} (tr)
         edge[<-] node {$e_{p_1}$} (bl)
    (tr) edge[<-] node {$e_f$} (br)
    (bl) edge[->] node {$r_1$} (br);
  \end{tikzpicture}
  \end{center}
  where $r_1$ is (also) the first projection. We write $-\Omega f$ for the map $\Omega f \o ({-})\sy$.
\end{lem}
\begin{proof}
  The underlying equivalence is the following composite
  \begin{align*}
    \fib_{p_1}&\simeq((a,p):\fib_f)\times a = a_0\\
    &\simeq (a:A)\times a = a_0 \times f(a)=b_0\\
    &\simeq f(a_0)=b_0\\
    &\simeq b_0=b_0 \equiv \Omega B
  \end{align*}
  This equivalence sends $((a, p), q):\fib_{p_1}$ (with $p : fa=b_0$ and $q:a=a_0$) to 
  $f_0\sy \tr \ap fq \tr p$. So there is a path
  $$r(a,p,q): e_f((a, p), q)=f_0\sy \tr f(q\sy) \tr p.$$
  This path satisfies $r(a_0,f_0,1)=1$ (equality is type correct since $e_f((a_0, f_0), q)\equiv f_0\sy \tr q$). We also have $e_f\sy(p)=((a_0,f_0\cdot p),1)$ for $p:\Omega B$.

  Now $e$ respects the basepoint, because $$e(a_0,f_0,1)=f_0\sy\cdot f_0=1.$$
  We will not prove naturality here, since it is not required for the results in this section.
  For the commuting square, we will prove that 
  $$h:e_f \o r_1 \o e_{p_1}\sy\sim^*-\Omega f$$
  For the underlying homotopy, we compute for $p:\Omega A$
  \begin{align*}
    e_f(r_1(e_{p_1}\sy p))&=e_f(r_1(((a_0,f_0),1 \tr p),1))\\
    &=e_f((a_0,f_0),p)\\
    &=f_0\sy \cdot f(p\sy) \cdot f_0 \equiv -\Omega f(p).
  \end{align*}
  To show that $h$ respects the basepoint, suppose that $p\equiv1$. In that case, the first two steps of the above equation becomes definitional equalities. Since we know that $r(a_0,f_0,1)=1$, the last equality is also reflexivity. Since the maps $e_f\o r\o e_{p_1}\sy$ and $-\Omega f$ respect the basepoints using the same path, this shows that $h$ is a pointed homotopy, which finishes the proof.
\end{proof}

\subsubsection{Step 1}
Denote $\arrow^*\defeq(X\ Y : \U^*_i)\times (X \to^* Y)$. We define $F:\arrow^*\to\arrow^*$ by 
$F(X,Y,f)\defeq(\fib_f,X,p_1)$. 
Given a pointed map $f:X\to^* Y$, we define its fiber sequence $A:\N\to\U$ by $A_n\defeq p_2(F^n(X,Y,f))$, and we define $f_n:A_{n+1}\to A_n$ by $p_3(F^n(X,Y,f))$ (which is well-typed, since $A_{n+1}\equiv p_1(F^n(X,Y,f))$ by unfolding the definition of $F$). It is easy to show that $(A_n,f_n)_n$ is a type-valued exact sequence, since $A_{n+2}$ is (definitionally) the fiber of $f_n$. 

Note that by \autoref{lem:fibers} there is a pointed equivalence $e_f:A_3\simeq^* \Omega Y$. We define the diagonal map $\delta\defeq p_1\o e_f\sy : \Omega Y \to \fib_f$.

\subsubsection{Step 2}
Define the sequence $B:\N\to\U$ and $g_n:B_{n+1}\to B_n$ by
\begin{align*}
  B_0&\defeq Y&&\\
  B_1&\defeq X & g_0&\defeq f\\
  B_2&\defeq \fib_f & g_1&\defeq p_1\\
  B_{n+3}&\defeq \Omega B_n & g_2&\defeq \delta\\
  && g_{n+3}&\defeq -\Omega g_n
\end{align*}
Note that $g_2$ has the correct type, since $A_3\equiv B_3$.

Now we can show that $(B,g)$ is a type-valued exact sequence by showing that it is equivalent to $(A,f)$.

\begin{lem}
  There is a natural equivalence $(A_n, f_n)_n\simeq(B_n,g_n)_n$. This means that there are pointed equivalences $\eta_n:A_n\simeq^* B_n$ such that for all $n:\N$ we have $$\eta_n\o f_n\sim^*g_n\o \eta_{n+1}.$$
\end{lem}
\begin{proof}
  We define the equivalence $\eta_n$ by induction on $n$. Note that $A_k\equiv B_k$ for $k=0,1,2$. Now suppose we have an equivalence $\eta_k: A_k\simeq B_k$. Then by \autoref{lem:fibers} we have 
  $$A_{k+3}\equiv\fib_{f_{k+1}}\stackrel{e_{f_k}}\simeq \Omega A_k\stackrel{\Omega\eta_k}\simeq \Omega B_k\equiv B_{k+3}.$$
  We also show the naturality by induction on $n$.\\ 
  For $n\equiv0$ we have $\idfunc[Y]\o f\sim^* f\o\idfunc[X].$\\
  For $n\equiv1$ we have $\idfunc[X]\o p_1\sim^* p_1\o\idfunc[\fib_f].$\\
  For $n\equiv2$ we have 
  $$\idfunc[\fib_f]\o p_1\equiv p_1\sim^* (p_1 \o e_f^{-1}) \o e_f \sim^* \delta\o (\Omega\idfunc[Y] \o e_f).$$
  Now suppose the naturality holds for $k$, then we get the following diagram.
  \begin{center}\begin{tikzpicture}[node distance=2cm, thick, node
    distance=12mm]
    \node (tl)       at (0,0) {$\Omega B_{k+1}$};
    \node[right = of tl] (tr) {$\Omega B_k$};
    \node[below = of tl] (l)  {$\Omega A_{k+1}$};
    \node (r) at (tr |- l)    {$\Omega A_k$};
    \node[below = of l] (bl)  {$A_{k+4}$};
    \node (br) at (r |- bl)   {$A_{k+3}$};
    \path[every node/.style={font=\sffamily\small}]
    (tl) edge[->] node {$-\Omega g_k$} (tr)
         edge[<-] node {$\Omega\eta_{k+1}$} (l)
    (tr) edge[<-] node {$\Omega\eta_k$} (r)
    (l)  edge[->] node {$-\Omega f_k$} (r)
         edge[<-] node {$e_{f_{k+1}}$} (bl)
    (r)  edge[<-] node {$e_{f_k}$} (br)
    (bl) edge[->] node {$f_{k+3}$} (br);
  \end{tikzpicture}\end{center}
  The bottom square can be filled by the second part of \autoref{lem:fibers}. The top square can be filled by applying the functor $\Omega$ to the naturality for $k$ and then noticing that $({-})\sy \o \Omega\eta_k\sim^* \Omega\eta_k \o ({-})\sy,$ which is easily proven for an arbitrary pointed map.
\end{proof}

\subsubsection{Step 3}
We now remove the inverses in our sequence. More precisely, we define a second sequence $h_n:B_{n+1}\to B_n$ by
$$h_0\defeq f\qquad h_1\defeq p_1\qquad h_2\defeq \delta\qquad h_{n+3}\defeq \Omega h_n.$$

To show that $(B,h)$ is a type-valued exact sequence we use the following lemma.
\begin{lem}\label{lem:LESstep3}
  Suppose $N$ is a successor structure and $(B,g)$ is a type-valued exact sequence over $N$. Suppose $h_n:B_{n+1}\to^* B_n$ is another sequence of maps, and suppose that there are pointed maps $e_n, \ell_n, r_n : B_n \to^* B_n$ such that $e_n$ is an equivalence and the following diagrams commute as homotopies (not necessarily pointed):
  \begin{center}\begin{tikzpicture}[node distance=2cm, thick, node distance=12mm]
    \node (tl)       at (0,0) {$B_{n+1}$};
    \node[below = of tl] (bl) {$B_{n+1}$};
    \node[right = of bl] (br) {$B_n$};
    \path[every node/.style={font=\sffamily\small}]
    (tl) edge[->] node {$h_n$} (br)
    (bl) edge[->] node {$e_{n+1}$} (tl)
         edge[->] node {$g_n$} (br);
    \node[right = of tr] (tl) {$B_{n+1}$};
    \node[right = of tl] (tr) {$B_n$};
    \node[below = of tl] (bl) {$B_{n+1}$};
    \node (br) at (tr |- bl)   {$B_n$};
    \path[every node/.style={font=\sffamily\small}]
    (tl) edge[->] node {$h_n$} (tr)
         edge[->] node {$\ell_{n+1}$} (bl)
    (tr) edge[->] node {$e_n$} (br)
    (bl) edge[->] node {$h_n$} (br);
    \node[right = of tr] (tl) {$B_{n+1}$};
    \node[right = of tl] (tr) {$B_n$};
    \node[below = of tl] (bl) {$B_{n+1}$};
    \node (br) at (tr |- bl)   {$B_n$};
    \path[every node/.style={font=\sffamily\small}]
    (tl) edge[->] node {$h_n$} (tr)
         edge[->] node {$e_{n+1}$} (bl)
    (tr) edge[<-] node {$r_n$} (br)
    (bl) edge[->] node {$h_n$} (br);
    \end{tikzpicture}\end{center}
    Then $(B,h)$ is a type-valued exact sequence over $N$.
\end{lem}
\begin{proof}
  First we need to show that for $x:B_{n+2}$ we have $h_n(h_{n+1}(x))=b_0^n$. We compute
  \begin{align*}
    h_n(h_{n+1}(x))&=r_n(h_n(e_{n+1}(h_{n+1}(x))))\\
    &=r_n(g_n(h_{n+1}(x)))\\
    &=r_n(g_n(g_{n+1}(e_{n+2}\sy(x))))\\
    &=r_n(b_0^n)\\
    &=b_0^n.
  \end{align*}
  For exactness, suppose that $y:B_{n+1}$ such that $h_n(y)=b_0^n$. Then 
  $g_n(e_{n+1}\sy(y))=h_n(y)=b_0^n$, therefore, by exactness of $g$ there (purely) exists 
  an $x:B_{n+2}$ such that $g_{n+1}(x)=e_{n+1}\sy(y)$. Now we compute
  \begin{align*}
    h_{n+1}(\ell_{n+2}(e_{n+2}(x)))&=e_{n+1}(h_{n+1}(e_{n+2}(x)))\\
    &=e_{n+1}(g_{n+1}(x))\\
    &=e_{n+1}(e_{n+1}\sy(y))\\
    &=y.
  \end{align*}
  This finishes the proof.
\end{proof}

\begin{lem}
  The sequence $(B,h)$ is a type-valued exact sequence.
\end{lem}
\begin{proof}
  We first define for $k\ge2$ we the pointed equivalence 
  $\epsilon_n^k:B_n\simeq^* B_n$ by induction on $n$. For $n\le k$
  $\epsilon_n^k\defeq\idfunc:B_n\simeq^* B_n$ we define
  $\epsilon_{n+3}^k\defeq -\Omega\epsilon_n^k:B_{n+3}\simeq^* B_{n+3}$ for $n+3>k$.
  Now define $e_n\defeq\epsilon_n^3$ and $\ell_n\defeq\epsilon_n^4$ and $r_n\defeq\epsilon_n^2$.
  We apply \autoref{lem:LESstep3} using these equivalences to obtain the desired result. To do this we need to check three commuting triangles. We will check $h_n\o e_{n+1}\sim g_n$, the other two proofs are similar. Apply induction on $n$. For $n=0,1,2$ it is trivial, reducing to $g_n\o\id \sim g_n$. Suppose the homotopy is true for $n=k$. Then 
  $$h_{k+3}\o e_{k+4}\equiv \Omega h_k \o -\Omega e_{k+1} \sim -\Omega (h_k \o e_{k+1}) \sim -\Omega g_k\equiv g_{k+3}.$$
\end{proof}

\subsubsection{Step 4}
We now define a type-valued chain complex over $\N\times\fin_3$, which has a successor structure by \autoref{ex:succ_structure}. Let $\rho_X$ be the equivalence $\Omega^{n+1}X \simeq^* \Omega^n(\Omega X)$. We now define the sequence $C:\N\times\fin_3$ and $k_n:C_{n+1}\to C_n$ by
\begin{align*}
  C_{(n,0)}&\defeq \Omega^nY& k_{(n,0)}&\defeq \Omega^nf\\
  C_{(n,1)}&\defeq \Omega^nX& k_{(n,1)}&\defeq \Omega^np_1\\
  C_{(n,2)}&\defeq \Omega^n\fib_f & k_{(n,2)}&\defeq \Omega^n\delta \o \rho_X
\end{align*}
In a diagram, $(C,k)$ looks like the following.
\begin{center}\begin{tikzpicture}[node distance=3cm, thick, node
  distance=12mm]
\node (Y)     at (0,0) {$Y$};
\node[left = of Y] (X) {$X$};
\node[left = of X] (F) {$F$};
\node[above = of Y] (OY) {$\Omega Y$};
\node[above = of X] (OX) {$\Omega X$};
\node[above = of F] (OF) {$\Omega F$};
\node[above = of OY] (O2Y) {$\Omega^2Y$};
\node[above = of OX] (O2X) {$\Omega^2X$};
\node[above = of OF] (O2F) {$\Omega^2F$};
\node[above = 5mm of O2X] (dots) {$\vdots$};
\path[every node/.style={font=\sffamily\small}]
(X) edge[->] node [above] (f){$f$} (Y)
(F) edge[->] node [below] (f){$p_1$} (X)
(OY) edge[->] node [left] (f){$\delta\quad\mbox{}$} (F)
(OX) edge[->] node [above] (f){$\Omega f$} (OY)
(OF) edge[->] node [below] (f){$\Omega p_1$} (OX)
(O2Y) edge[->] node [left] (f){$\Omega\delta\quad\mbox{}$} (OF)
(O2X) edge[->] node [above] (f){$\Omega^2f$} (O2Y)
(O2F) edge[->] node [below] (f){$\Omega^2p_1$} (O2X);
\end{tikzpicture}
\end{center}
There is an equivalence $e:\N \simeq \N\times\fin_3$ that sends $n$ to its quotient and remainder when dividing $n$ by 3. The proof of the following lemma is straightforward and omitted.
\begin{lem}
  The sequence $(B,h)$ is naturally equivalent to $(C,k)$ over the equivalence $e$. 
  Therefore, $(C,k)$ is a type-valued exact sequence.
\end{lem}

\subsubsection{Step 5}
If we 0-truncate the sequences at step 4, we get the sequence $(D,\ell)\defeq(\|C\|_0,\|k\|_0)$. This is exactly the sequence in \autoref{thm:les-homotopy}. It is now easy to show that this is a long exact sequence.
\begin{proof}[Proof of \autoref{thm:les-homotopy}]
First note that it is a chain complex by the following computation:
$$\ell_n \o \ell_{n+1}\sim \|k_n \o k_{n+1}\|_0\sim \|0\|_0 \sim 0.$$
To show that it is exact, suppose given $x : D_{n+1}$ and $p:\ell_n(x)=d_0^n$. We need to construct an element in a proposition, so we may assume by induction that $x\equiv |y|_0$. Now the type of $p$ reduces to $|k_n(y)|_0=|c_0^n|_0$, which is equivalent to $\|k_n(y)=c_0^n\|_{-1}$ by the characterization of the identity type in truncations. Therefore, the latter type is inhabited, and by induction, we may assume that we have a path $k_n(y)=c_0^n$. By exactness of $(C,k)$ we get an element $z:C_{n+2}$ such that $q:k_{n+1}(z)=y$. Now we can find $|z|_0:D_{n+2}$ and the path $\mapfunc{|{-}|_0}(q):\ell_{n+1}(|z|_0)=x$, showing exactness.
\end{proof}

\subsection{Computation of homotopy groups}

An important application of the long exact sequence of homotopy groups comes in combination with the Hopf fibration. Combining these tools, we can compute more homotopy groups of spheres. The Hopf fibration was  constructed in~\cite[Theorem 8.5.1]{hottbook} and has been formalized by Ulrik Buchholtz. We will not give the construction here.

\begin{thm}[Hopf Fibration] \label{thm:hopf}
  There is a pointed map $\S^3\to\S^2$ with fiber $\S^1$.
\end{thm}

The quaternionic Hopf fibration has also been constructed in HoTT and formalized in Lean~\cite{buchholtz2016cayleydickson}. This gives a fibration $\S^7\to \S^4$ with fiber $\S^3$.

\begin{cor}\label{cor:homotopy-group-spheres-1}
  $\pi_2(\S^2)=\Z$ and $\pi_n(\S^3)=\pi_n(\S^2)$ for $n\geq 3$.
\end{cor}
\begin{proof}
  We know by the connectedness of spheres that $\pi_1(\S^3)$ and $\pi_2(\S^3)$ are trivial, and by the truncatedness of the circle that $\pi_k(\S^1)$ is trivial for $k>1$ and $\Z$ for $k=1$. We now get the following long exact sequence, from which the result immediately follows.
  \begin{center}\begin{tikzpicture}[node distance=3cm, thick, node
    distance=12mm]
  \node (Y)     at (0,0) {$0$};
  \node[left = of Y] (X) {$0$};
  \node[left = of X] (F) {$\Z$};
  \node[above = 8mm of Y] (OY) {$\pi_2(\S^2)$};
  \node (OX) at (X |- OY) {$0$};
  \node (OF) at (F |- OY) {$0$};
  \node[above = 8mm of OY] (O2Y) {$\pi_3(\S^2)$};
  \node (O2X) at (X |- O2Y) {$\pi_3(\S^3)$};
  \node (O2F) at (F |- O2Y) {$0$};
  \node[above = 8mm of O2Y] (O3Y) {$\pi_4(\S^2)$};
  \node (O3X) at (X |- O3Y) {$\pi_4(\S^3)$};
  \node (O3F) at (F |- O3Y) {$0$};
  \node[above = 5mm of O3X] (dots) {$\vdots$};
  \path[every node/.style={font=\sffamily\small}]
  (X) edge[->] (Y)
  (F) edge[->] (X)
  (OY) edge[->] (F)
  (OX) edge[->] (OY)
  (OF) edge[->] (OX)
  (O2Y) edge[->] (OF)
  (O2X) edge[->] (O2Y)
  (O2F) edge[->] (O2X)
  (O3Y) edge[->] (O2F)
  (O3X) edge[->] (O3Y)
  (O3F) edge[->] (O3X);
\end{tikzpicture}
\end{center}
\end{proof}

The last ingredient we need is the Freudenthal Suspension
Theorem. This has been formalized before by Dan Licata in Agda, and our
formalization is a direct port of that proof to Lean. For the proof we refer to~\cite[Section 8.6]{hottbook}.
\begin{thm}[Freudenthal Suspension Theorem] \label{thm:freudenthal}
  Suppose that $X$ is $n$-connected. Then $\|X\|_{2n}\simeq \|\Omega\Sigma X\|_{2n}$.
\end{thm}

We can combine these results to compute the following homotopy groups.
\begin{cor}\label{cor:homotopy-group-spheres-2}
  $\pi_n(\S^n)=\Z$ and $\pi_3(\S^2)=\Z$
\end{cor}
\begin{proof}
  Note that $\S^n$ is $(n-1)$-connected. Therefore, by the Freudenthal suspension theorem we have
  $$\|\S^n\|_{2(n-1)}\simeq \|\Omega\S^{n+1}\|_{2(n-1)}.$$
  For $n\geq 2$ we have $2(n-1)\geq n$, and therefore we also have
  $$\|\S^n\|_{n}\simeq \|\Omega\S^{n+1}\|_{n}.$$ 
  Taking the $n$-th homotopy group, we get
  $$\pi_n(\S^n)\simeq \pi_{n+1}(\S^{n+1}).$$ 
  Combining this with \autoref{cor:homotopy-group-spheres-1}, we also get $\pi_3(\S^2)\simeq\Z$, as desired.
\end{proof}

\section{Eilenberg-MacLane Spaces}\label{sec:eilenb-macl-spac}

In this section we give an important equivalence between groups and Eilenberg-MacLane spaces~\cite{eilenberg1945spaces}.\footnote{Some of the contents of this section have been published in~\cite{buchholtz2018groups}. 
The work in this section is joint work with Ulrik Buchholtz and Egbert Rijke.} Eilenberg-MacLane space are play an important role in homotopy theory, since they are spaces with simple homotopy groups. Therefore, they can be used to build up more complicated spaces with complicated homotopy groups. Also, they can be used to define homology and cohomology in HoTT, see Sections \ref{sec:spectral-sequence-cohomology} and \ref{sec:spectral-sequence-homology}.

We prove in this section that the category of $n$-connected $(n+1)$-truncated pointed types is equivalent to the
category of groups for $n = 0$ and the category of abelian groups for $n \geq 1$.

If $G$ is a (pre-)groupoid, the groupoid quotient is a higher inductive type with constructors
\begin{inductive}
\texttt{HIT} $\groupoidquotient(G) :=$ \\
$\bullet\ i : G_0 \to \groupoidquotient(G)$; \\
$\bullet\ p : (x\ y : G_0) \to \homm(x,y)\to x=y$; \\
$\bullet\ q : (x\ y\ z : G_0) \to (g : \homm(y,z)) \to (f : \homm(x,y)) \to
p(g \circ f) = p(f) \cdot p(g)$; \\
$\bullet\ \eps : \istrunc{1}(\groupoidquotient(G))$.
\end{inductive}

The groupoid quotient can be constructed purely from homotopy pushouts. 
The untruncated version was constructed in \autoref{sec:non-recursive-2}. 
Then we can apply the 1-truncated afterwards, and we can also construct truncations from homotopy pushouts~\cite{rijke2017join}.

In~\cite{licata2014em} the authors define Eilenberg-MacLane spaces. We use the same approach as in
that paper. We first quickly review the results in that paper.

\subsection{Construction of Eilenberg-MacLane spaces}

If $G$ is any group, the 1-dimensional Eilenberg-MacLane space $K(G,1)$ can be defined by viewing
$G$ as a groupoid, and taking the groupoid quotient of $G$. It is not hard to see that $K(G,1)$ is
0-connected and 1-truncated. Using an encode-decode proof, we can show that $\Omega K(G,1) \simeq G$
and that this equivalence sends concatenation to multiplication. Hence the composite
$\pi_1K(G,1)\simeq \|G\|_0\simeq G$ is a group isomorphism.

If $G$ is abelian, the higher Eilenberg-MacLane spaces can be defined recursively
as $$K(G,n+1):\equiv \|\Sigma K(G,n)\|_{n+1}$$ for $n\geq 1$. This definition is slightly different
than the one given in~\cite{licata2014em}, where $K(G,n+1)$ was defined using the iterated
suspension as $\|\Sigma^n K(G,1)\|_{n+1}$. We chose to modify the definition, since a lot of
properties of Eilenberg-MacLane spaces are proven by induction on $n$, so it is more convenient to
have $K(G,n+1)$ defined directly in terms of $K(G,n)$.

It is easy to show that $K(G,n)$ is $(n-1)$-connected and $n$-truncated. It is trickier to show that
$\Omega K(G,n+1)\simeq K(G,n)$. This is done separately for $n=1$ and for $n\geq 2$.

For $n=1$ we need the result that for every type $X$ with a coherent h-structure, the type $\|\Sigma
X\|_2$ is a \emph{delooping} of $X$, which means that $\Omega \|\Sigma X\|_2 \simeq X$. If $G$ is abelian, then $K(G,1)$ can be equipped with a coherent h-structure,
showing that $\Omega K(G,2)\simeq K(G,1)$.

For $n\geq 2$, this can be done using the Freudenthal suspension theorem,
\cref{thm:freudenthal}. Then the equivalence follows from the following chain of equivalences:
$$\Omega K(G,n+1)\equiv \Omega\|\Sigma K(G,n)\|_{n+1}\simeq \|\Omega\Sigma K(G,n)\|_n\simeq
\|K(G,n)\|_n\simeq K(G,n).$$ The Freudenthal Suspension Theorem is applied in the third step, which
is allowed since $K(G,n)$ is $(n-1)$-connected and $n \leq 2(n-1)$ for $n\geq2$.

This finishes the proof sketch that $\Kloop(G,n):\Omega K(G,n+1)\simeq K(G,n)$. By induction,
$\Omega^n K(G,n+1)\simeq K(G,1)$, hence we get the following group isomorphism
$\pi_{n+1}(K(G,n+1))\simeq \pi_1(K(G,1))\simeq G$.

\subsection{Uniqueness}

In this section we prove that Eilenberg-MacLane spaces are unique, which means that if $X$ and $Y$
are both $(n-1)$-connected, $n$-truncated pointed types such that $\pi_n(X)\simeq \pi_n(Y)$, then
$X\simeq Y$. Note that from these assumptions one can show that $\pi_k(X)\simeq 1\simeq \pi_k(Y)$
for $k < n$ since $X$ and $Y$ are $(n-1)$-connected, but also for $k > n$ since $X$ and $Y$ are
$n$-truncated. Hence from the assumptions we actually have that $\pi_k(X)\simeq\pi_k(Y)$ for all
natural numbers $k$.

This is similar to Whitehead's Theorem, which states that if $f : X \to Y$ is a pointed map
that induces an equivalence on all homotopy groups, then $f$ is an equivalence. Whitehead's Theorem
is not true in general, but it is true under the assumption that both $X$ and $Y$ are $n$-truncated
for some $n$. For the special case that $X$ and $Y$ are both $(n-1)$-connected and $n$-truncated one
does not need to find a map between $X$ and $Y$ to show that they are equivalent, as long as they
have isomorphic homotopy groups.

We first give an elimination principle for $K(G,n)$.
\begin{defn}\label{def:Kelim}
Suppose that $X$ is an $n$-truncated pointed type, and suppose that for some group
$G$ there is an map $\phi : G \to \Omega^n X$ that sends multiplication to
concatenation. Then there is a pointed map $\Kelim(\phi, n) : K(G,n)\to X$.
\end{defn}
\begin{proof}[Construction]
  We construct this by induction on $n$.
  For $n=1$ this follows directly from the induction principle of $K(G,1)$. For $n=k+1>1$ we can define the group homomorphism $\widetilde\phi$ as the composite $G \xrightarrow{\phi} \Omega^{k+1} X \simeq \Omega^k(\Omega X)$, and apply the induction hypothesis to get a map
  $\Kelim(\widetilde\phi, k):K(G,k)\to^* \Omega X$. By the adjunction $\Sigma\dashv\Omega$ we get a pointed map $\Sigma K(G,k)\to^* X$, and by the elimination principle of the truncation we get a map 
  $K(G,k+1)\equiv\|\Sigma K(G,k)\|_{k+1}\to^* X$. 
\end{proof}

\begin{lem}\label{lem:Kelim1}
There is a pointed homotopy making the following diagram commute.
\begin{center}
  \begin{tikzpicture}[node distance=3cm,
      thick,main node/.style={font=\sffamily\bfseries}]
    \node[main node] (Kn) at (0,0) {$K(G,n)$};
    \node[main node] (OK) at (4,0) {$\Omega K(G,n+1)$};
    \node[main node] (OX) at (2,-2) {$\Omega X$};
         \path[every node/.style={font=\sffamily\small}]
         (Kn) edge [->] node {$\sim$} (OK)
              edge [->] node [below left] {$\Kelim(\widetilde\phi,n)$} (OX)
         (OK) edge [->] node [below right] {$\Omega(\Kelim(\phi,n+1))$} (OX);
  \end{tikzpicture}
\end{center}
\end{lem}
\begin{proof}
  This follows by unwinding the definition of the function $\Kelim(\phi,n+1)$ in terms of $\Kelim(\phi,n)$.
\end{proof}

\begin{lem}\label{lem:Kelim2}
The following diagram commutes. 
\begin{center}
  \begin{tikzpicture}[node distance=3cm,
      thick,main node/.style={font=\sffamily\bfseries}]
    \node[main node] (OK) at (0,0) {$\Omega^nK(G,n)$};
    \node[main node] (G)  at (4,0) {$G$};
    \node[main node] (OX) at (2,-2) {$\Omega^n X$};
         \path[every node/.style={font=\sffamily\small}]
         (OK) edge [->] node {$\sim$} (G)
              edge [->] node [below left] {$\Omega^n(\Kelim(\phi,n))$} (OX)
         (OX) edge [<-] node [below right] {$\phi$} (G);
  \end{tikzpicture}
\end{center}
\end{lem}
\begin{proof}
  This follows by repeatedly applying \autoref{lem:Kelim1}.
\end{proof}

\begin{thm}\label{thm:em-unique}
Suppose that $X$ is an $(n-1)$-connected $n$-truncated pointed type, and suppose that for some group
$G$ there is an equivalence $\phi : G \simeq \Omega^n X$ that sends multiplication to
concatenation. Then the map $\Kelim(\phi, n) : K(G,n)\to X$ is an equivalence.
In particular this means that if $X$ is an $(n-1)$-connected $n$-truncated pointed type, and there is a group isomorphism $e : \pi_n(X) \simeq G$, then $X\simeq^* K(G,n)$.
\end{thm}
\begin{proof}
  We apply Whitehead's principle for truncated types. This states that a \emph{weak equivalence} (a map inducing an isomorphism on all homotopy groups) between truncated types is an equivalence. The proof can be found in~\cite[Theorem 8.8.3]{hottbook}. Since both $K(G,n)$ and $X$ are $(n-1)$-connected and $n$-truncated, the map $\Kelim(\phi,n)$ trivially induces an isomorphism on all homotopy groups for all levels other than $n$. It also induces an isomorphism on level $n$ by \autoref{lem:Kelim2}. This finishes the proof.
\end{proof}

\begin{cor}
The type of $(n-1)$-connected, $n$-truncated pointed types is equivalent to the type of groups for
$n=1$ and equivalent to the type of abelian groups for $n \geq 2$.
\end{cor}
\begin{proof}
  The maps back and forth are $K({-},n)$ and $\pi_n$. The composites are homotopic to the identity map, since $\pi_n(K(G,n))\simeq G$ and $K(\pi_n(X),n)\simeq^* X$ (the last equivalence comes from \autoref{thm:em-unique}).
\end{proof}

\subsection{Equivalence of categories}
\begin{defn}
  If $\phi : G \to H$ is a homomorphism between groups, then there is a pointed map $K(\phi, n) :
  K(G, n) \to K(H, n)$. This action is functorial, i.e. it respects composition and identity maps.
\end{defn}
\begin{proof}[Construction]
  The functorial action comes from \autoref{def:Kelim}. We omit the proof of the other properties.
\end{proof}

To show that we get the desired equivalence of categories, we need to fill the following naturality squares. We will omit the proofs here.

\begin{center}
  \begin{tikzpicture}[node distance=3cm,
      thick,main node/.style={font=\sffamily\bfseries}]
    \node[main node] (KG) at (0,0) {$\pi_n(K(G,n))$};
    \node[main node] (KH) at (4,0) {$\pi_n(K(H,n))$};
    \node[main node] (G) at (0,-4) {$G$};
    \node[main node] (H) at (4,-4) {$H$};
         \path[every node/.style={font=\sffamily\small}]
         (KG) edge [->] node {$\pi_n(K(\phi,n))$} (KH)
         (G)  edge [->] node [left] {$\sim$} (KG)
              edge [->] node {$\phi$} (H)
         (H)  edge [->] node [right] {$\sim$} (KH);
  \end{tikzpicture}
\qquad
  \begin{tikzpicture}[node distance=3cm,
      thick,main node/.style={font=\sffamily\bfseries}]
    \node[main node] (X)  at (0,0) {$X$};
    \node[main node] (Y)  at (4,0) {$Y$};
    \node[main node] (KX) at (0,-4) {$K(\pi_n(X),n)$};
    \node[main node] (KY) at (4,-4) {$K(\pi_n(Y),n)$};
         \path[every node/.style={font=\sffamily\small}]
         (X)  edge [->] node {$f$} (Y)
         (KX) edge [->] node [left] {$\sim$} (X)
              edge [->] node {$K(\pi_n(f),n)$} (KY)
         (KY) edge [->] node [right] {$\sim$} (Y);
  \end{tikzpicture}
\end{center}
These diagrams show the following result.
\begin{thm}\label{thm:EM-equiv-categories}
  $K({-},n)$ is an equivalence from the category of $(n-1)$-connected $n$-truncated pointed types to
  the category of groups (for $n=1$) or abelian groups (for $n\geq2$).
\end{thm}
\begin{rmk}
  In particular this shows that the type of pointed maps between two $(n-1)$-connected $n$-truncated types is a set. 
  This is a special case of the more general fact that the type of pointed maps from an $n$-connected type to a $(n+k+1)$-truncated type is $k$-truncated (for $n\geq-1$).
\end{rmk}
\begin{rmk}
  It would be interesting, but a lot more work, to do this one level up. In that case, it should be possible to show that crossed modules or 2-groups correspond to pointed connected 2-types. Furthermore, pointed $(n-2)$ connected $n$-types should correspond to braided 2-groups for $n=3$ and to symmetric 2-groups for $n\geq 4$. A start of this project was given in~\cite{raumer2016doublegroupoids}.
\end{rmk}
\section{The Smash Product}\label{sec:smash-product}

In this section we will discuss the smash product and its
properties.\footnote{The work in this section is joint work with Stefano
Piceghello. Parts of this section are based on ideas from Robin Adams, Marc
Bezem, Ulrik Buchholtz and Egbert Rijke.} The smash product has many uses in
homotopy theory. It can be used to define generalized homology theory (see
\autoref{sec:spectral-sequence-homology}) and it is used to define the cup
product for cohomology~\cite[Section 5.1]{brunerie2016spheres}.

The goal is to prove that the smash product defines a \emph{1-coherent symmetric
monoidal product on pointed types}~\cite[Definition 4.1.1]{brunerie2016spheres},
which we repeat in \autoref{def:symmonprod}. Our proof strategy is to show that
the smash product is left adjoint to pointed maps and then use a Yoneda-style
argument to show that we get a 1-coherent symmetric monoidal product. 

This proof is known in 1-category theory~\cite[Chapter 2, Theorem
5.3]{eilenberg1966closedcategories}. Suppose given a closed category\footnote{A
\emph{closed category} is a category with internal hom-objects. We can view pointed
types as a higher closed category, where the internal hom-object is the type of pointed
maps, pointed by the constant map.} $\mc C$ with internal hom $[{-},{-}]:\mc C\op\times \mc C\to \mc C$.
Moreover suppose that for every $A,B:\mc C$ the functor $[A,[B,{-}]]:\mc C\to \mc C$ is
representable as a $\mc C$-enriched functor. This means that there is an object
$A\otimes B:\mc C$ and a $\mc C$-enriched natural transformation $[A\otimes
B,C]\cong [A,[B,C]]$. Then $\mc C$ is a monoidal closed category. We will spell
out the precise formulation for pointed types in \autoref{def:naturality}, where we will call
$\type^*$-enriched functors \emph{pointed functors} and $\type^*$-enriched
natural transformations \emph{pointed natural transformations}.

In this section we will prove two main claims.
\begin{itemize}
  \item We prove that $A\smsh B$ represents the functor $A\to^* B \to^* ({-})$ on
pointed types. In other words, that we have a natural equivalence 
  $$(A\smsh B \to^* C)\simeq^* (A\to^* B \to^* C).$$
\item We prove that if we have a \emph{pointed} natural equivalence 
$$(A\smsh B \to^* C)\simeq^* (A\to^* B \to^* C),$$
then the smash product forms a 1-coherent symmetric
monoidal product on pointed types.
\end{itemize}
There is still a gap in this argument: we still need to show that the natural
equivalence above is a pointed natural equivalence. We did not manage to do
this, because of the high level of the path algebra involved, but we do not
expect theoretical difficulties.

In this section, all types, maps, homotopies and equivalences are pointed,
unless mentioned otherwise. We will denote pointed homotopies using equalities
in diagrams. We will start with defining some categorical properties of pointed
types. We will use the notation established in \autoref{sec:pointed}.

\subsection{The Category of Pointed Types}\label{sec:pointed-category}

\begin{defn}\label{def:naturality}
  Suppose we are given $F:\type^*\to\type^*$. We say that $F$ is a \emph{1-coherent functor} if
  \begin{itemize}
    \item $F$ acts on pointed maps: given $f : A \to A'$, there is a pointed map $Ff : F(A) \to F(A');$
    \item it respects identities: $F(\id_A)\sim \id_{FA};$
    \item it respects composition: $F(f' \o f) \sim Ff' \o Ff.$
  \end{itemize}
  We will call a 1-coherent functor a \emph{functor} for short.\footnote{While 
  this is an abuse of terminology, it will not cause confusion in practice. Note that internally in the language of HoTT 
  it is an open problem whether we can even formulate the type of fully coherent functors.}
  We say that a functor $F$ is a \emph{pointed functor} if moreover
  $F\unit=\unit$, where $\unit$ is the unit type (which is the zero object in
  pointed types). In this case we can show that
  $F(\const_{A,B})=\const_{FA,FB}$, where $\const_{A,B}$ is the constant map.

	Let $F$, $G$ be functors of pointed types and suppose that $\theta$ is a
	family of pointed maps $(X : \type^*) \to F(X) \to G(X)$. We say that $\theta$
	is a (1-coherent) \emph{natural transformation} or \emph{natural} if for every $f : A \to
	B$ there is a diagram:
	\begin{center}
	\begin{tikzcd}
		F(A)
			\arrow[r, "F(f)"]
			\arrow[d, swap, "\theta_A"]
		& F(B)
			\arrow[d, "\theta_B"]
		\\
		G(A)
			\arrow[r, swap, "G(f)"]
		& G(B)
	\end{tikzcd}
  \end{center}
  That is, a pointed homotopy 
  \[p_\theta(f) : \theta_B \o F(f) \sim G(f) \o \theta_A.\]

	We say that $\theta$ is \emph{pointed natural} if $\theta$ is natural and $p_\theta(\const) = (p_\theta)_0$, where
		\[(p_\theta)_0 : G(\const) \o \theta_A \sim \const \o \theta_A \sim \const \sim \theta_B \o \const \sim \theta_B \o F(\const)\]
    is the canonical proof of the pointed homotopy $G(\const) \o \theta_A \sim \theta_B \o F(\const)$.
  
  For $n$-ary functions $F:\type^*\to\cdots\to\type^*$ we define functoriality
  similarly. We say that transformations between $n$-ary functors are natural if
  they are natural in all arguments.
\end{defn}

\begin{rmk}
  We could define a notion of \emph{weak naturality}, which is like naturality,
  but where the homotopy is not required to be pointed. However, this is
  generally ill-behaved. For example, if $\theta$ is weakly natural, neither
  $X\to\theta$ nor $\theta\to X$ needs to be weakly natural.
\end{rmk}

\begin{defn}\label{def:symmonprod}
  A 1-coherent symmetric monoidal product for pointed types is a binary operation $\otimes:\type^*\to\type^*\to\type^*$ that is functorial. Explicitly, this means that
  \begin{itemize}
  \item Given $f : A \to A'$ and $g : B \to B'$, there is a map $f\otimes g : A \otimes B \to A' \otimes B.$
  \item It respects identities: $\id_A\otimes\id_B\sim \id_{A\otimes B}.$
  \item It respects composition: $(f' \o f) \otimes (g' \o g) \sim (f' \otimes g') \o (f \otimes g).$
  \end{itemize}
  Furthermore, there is a pointed type $I$ and natural equivalences
  \begin{itemize}
    \item $\alpha : (A \otimes B) \otimes C \simeq A \otimes (B \otimes C)$ (associativity of the smash product);
    \item $\lambda : I \otimes B \simeq B$ (left unitor for the smash product);
    \item $\gamma : A \otimes B \simeq B \otimes A$ (braiding for the smash product).
  \end{itemize}
  With pointed homotopies filling the following three diagrams.
  \begin{center}
  \begin{tikzcd}
    &((A \otimes B) \otimes (C \otimes D))
      \arrow[dr, "\alpha"]
    \\
    (((A \otimes B) \otimes C) \otimes D)
      \arrow[ru, "\alpha"]
      \arrow[d, swap, "\alpha \otimes D"]
    && (A \otimes (B \otimes (C \otimes D)))
    \\
    ((A \otimes (B \otimes C)) \otimes D)
      \arrow[rr, swap, "\alpha"]
    && (A \otimes ((B \otimes C) \otimes D))
      \arrow[u, swap, "A \otimes \alpha"]
  \end{tikzcd}
  \end{center}
  \begin{center}
  \begin{tikzcd}
  ((I \otimes A) \otimes B)
    \arrow[rr, "\alpha"]
    \arrow[dr, swap, "\lambda \otimes B"]
  && (I \otimes (A \otimes B))
    \arrow[dl, "\lambda"]
  \\
  & (A \otimes B)
  \end{tikzcd}
  \end{center}


  \begin{center}
  \begin{tikzcd}
    ((A \otimes B) \otimes C)
      \arrow[r, "\alpha"]
      \arrow[d, swap, "\gamma \otimes C"]
    &(A \otimes (B \otimes C))
      \arrow[r, "\gamma"]
    & ((B \otimes C) \otimes A)
      \arrow[d, "\alpha"]
    \\
    ((B \otimes A) \otimes C))
      \arrow[r, swap, "\alpha"]
    & (B \otimes (A \otimes C))
      \arrow[r, swap, "B \otimes \gamma"]
    & (B \otimes (C \otimes A))
  \end{tikzcd}
  \end{center}  
\end{defn}

We have a version of the Yoneda Lemma for pointed types.
\begin{lem}[Yoneda]\label{lem:yoneda}
	Let $A$, $B$ be pointed types, and assume, for all pointed types $X$, a pointed equivalence $\phi_X : (B \to X) \simeq (A \to X)$, natural in $X$, i.e. for all $f : X \to X'$ there is a homotopy \[ p_\phi(f) : (A \to f) \o \phi_X \sim \phi_X' \o (B \to f) \]
	Then there exists a pointed equivalence $\psi_\phi : A \simeq B$.
\end{lem}
\begin{proof}
	We define $\psi_\phi \defeq \phi_B(\idfunc[B]) : A \to B$ and $\psi_\phi\sy \defeq \phi_A\sy(\idfunc[A])$. The given naturality square for $X \defeq B$ and $g \defeq \psi_\phi\sy$ yields $\psi_\phi\sy \o \phi_B (\idfunc[B]) \judgeq \psi_\phi\sy \o \psi_\phi \sim \phi_A (\psi_\phi\sy \o \idfunc[B]) \judgeq \phi_A (\phi_A\sy (\idfunc[A])) \sim \idfunc[A]$, and similarly for the inverse composition.
\end{proof}

\begin{lem}\label{lem:yoneda-pointed}
	Assume $A$, $B$, $\phi_X$ and $p$ as in \autoref{lem:yoneda}, and assume moreover that $\phi_X$ is pointed natural. Then there is a pointed homotopy $(\psi_\phi \to X) \sim \phi_X$.
\end{lem}

\begin{proof}
	Let $f : B \to X$. The underlying homotopy is obtained by:
	\begin{align*}
		(\psi_\phi \to X)(f) &\judgeq f \o \psi_\phi\\
		&\sim \phi_X (f \o \idfunc) &&\text{(by $p_\phi(f)(\idfunc)$)}\\
		&\sim \phi_X (f) &&\text{(by $\mapfunc{\phi_X}(\oneh_f)$)}
	\end{align*}
	To show that this is a pointed homotopy, we need to prove that the following diagram commutes:
	\begin{center}
	\begin{tikzcd}[column sep=4em]
		(\psi_\phi \to X)(\const)
			\arrow[rr, equals, "p_\phi(\const)(\idfunc)\tr\mapfunc{\phi_X}(\oneh_\const)"]
			\arrow[dr, equals, swap, "\zeroh_{\psi_\phi}"]
		&&\phi_X(\const)
			\arrow[dl, equals, "(\phi_X)_0"]
		\\
		&\const
	\end{tikzcd}
	\end{center}
	where the top-left expression is definitionally equal to $\const \o \phi_X(\idfunc)$, the horizontal path comes from the underlying homotopy and $(\phi_X)_0$ is the canonical path from $\phi_X(\const)$ to $\const$. Since $\phi_X$ is pointed natural, we have that
	$p_{\phi_X}(\const)(\idfunc) = (p_{\phi_X})_0(\idfunc)$, which is the concatenation:
	\begin{align*}
	\const\o \phi_X(\idfunc)
	&= \const &&\text{(by $\zeroh_{q_X(\idfunc)}$)}\\
	&= \phi_X(\const) &&\text{(by $(\phi_X)_0\sy$)}\\
	&= \phi_X(\const\o 1) &&\text{(by $(\mapfunc{\phi_X}(\zeroh_{\idfunc}))\sy$)}
	\end{align*}
	The diagram then commutes by cancellation of inverses and using that $\zeroh_{\idfunc} = \oneh_\const$.
\end{proof}
  
\subsection{Basic Properties of the Smash Product}\label{sec:smash-basic}

\begin{defn}
  The smash of $A$ and $B$ is the HIT generated by the point constructor $(a,b)$ for $a:A$ and $b:B$
  and two auxiliary points $\auxl,\auxr:A\smsh B$ and path constructors $\gluel_a:(a,b_0)=\auxl$
  and $\gluer_b:(a_0,b)=\auxr$ (for $a:A$ and $b:B$). $A\smsh B$ is pointed with point $(a_0,b_0)$.
\end{defn}
\begin{rmk}
  This definition of $A\smsh B$ is basically the pushout of
  $\bool\leftarrow A+B\to A \times B$. A more traditional definition of $A\smsh B$ is the pushout
  $\unit\leftarrow A\vee B\to A \times B$; here $\vee$ denotes the wedge product, which can be
  equivalently described as either the pushout $A\leftarrow \unit\to B$ or
  $\unit\leftarrow \bool\to A + B$. These two definitions of $A\smsh B$ are equivalent, because in
  the following diagram the top-left square and the top rectangle are pushout squares, hence the
  top-right square is a pushout square by applying the pushout lemma. Another application of the
  pushout lemma then states that the two definitions of $A\smsh B$ are equivalent.
\begin{center}
\begin{tikzcd}
\bool \arrow[r]\arrow[d] & A+B     \arrow[r]\arrow[d] & \bool \arrow[d] \\
\unit \arrow[r]          & A\vee B \arrow[r]\arrow[d] & \unit \arrow[d] \\
                     & A\times B        \arrow[r] & A\smsh B
\end{tikzcd}
\end{center}

\end{rmk}
\begin{lem}\label{lem:smash-general}
	The smash product is functorial: if $f:A\pmap A'$ and $g:B\pmap B'$, then
    $f\smsh g:A\smsh B\pmap A'\smsh B'$. We write $A\smsh g$ or $f\smsh B$ if one of the
    functions is the identity function. Moreover, if $p:f\sim f'$ and $q:g\sim g'$, then $p\smsh q:f\smsh g\sim f'\smsh g'$; this operation preserves reflexivities, symmetries and transitivies. We will write $p \smsh g$ or $f \smsh q$ if one of the homotopies is reflexivity.
\end{lem}

\begin{lem}\label{lem:interchange}
	The smash product preserves composition, which gives rise to the interchange law:
    \[i:(f_2 \o f_1)\smsh (g_2 \o g_1) \sim f_2 \smsh g_2 \o f_1 \smsh g_1\]
    for maps $A_1\lpmap{f_1}A_2\lpmap{f_2}A_3$ and $B_1\lpmap{g_1}B_2\lpmap{g_2}B_3$.
\end{lem}
\begin{proof}
	Let us denote the basepoints of $A_i$ and $B_i$ with $a_i$ and $b_i$ respectively. We first apply induction on the paths that all the maps in the statement respect the basepoint. We verify the underlying homotopy of $i$ by induction on terms $x$ of the domain $A_1 \smsh B_1$ of the two maps; this can be defined on point constructors $(a,b)$, $\auxl$ and $\auxr$ to be the identity path. If $x$ varies over $\gluel_a$, we need to fill the following square:
	\begin{equation}\label{eq:i-gluel}
	\begin{tikzcd}
		(f_2(f_1(a)), b_3)
			\arrow[r,equals,"1"]
			\arrow[d,swap,equals,"\mapfunc{(f_2 \o f_1)\smsh (g_2 \o g_1)}(\gluel_a)"]
		& (f_2(f_1(a)), b_3)
			\arrow[d,equals,"\mapfunc{f_2 \smsh g_2 \o f_1 \smsh g_1}(\gluel_a)"]
		\\
		\auxl
			\arrow[r,swap,equals,"1"]
		&\auxl
	\end{tikzcd}
	\end{equation}
	This reduces to proving that
	\[\mapfunc{(f_2(f_1(a)),-)}(g_2\o g_1)_0 \tr \gluel_{f_2(f_1(a))} = \mapfunc{(f_2(f_1(a)),-)}(\mapfunc{g_2}{(g_1)}_0 \tr {(g_2)}_0) \tr \gluel_{f_2(f_1(a))}\]
	Since we assumed that ${(g_1)}_0$ and ${(g_2)}_0$ are the identity path, the claim is easily verified. The case for $x$ varying over $\gluer_b$ is entirely analogous, giving the square:
	\begin{equation}\label{eq:i-gluer}
	\begin{tikzcd}
		(a_3, g_2(g_1(b))
			\arrow[r,equals,"1"]
			\arrow[d,swap,equals,"\mapfunc{(f_2 \o f_1)\smsh (g_2 \o g_1)}(\gluer_b)"]
		& (a_3, g_2(g_1(b))
			\arrow[d,equals,"\mapfunc{f_2 \smsh g_2 \o f_1 \smsh g_1}(\gluer_b)"]
		\\
		\auxr
			\arrow[r,swap,equals,"1"]
		&\auxr
	\end{tikzcd}
	\end{equation}
	The resulting homotopy is pointed, as $i(a_1,b_1) \judgeq 1$ and the proofs that the two maps respect the basepoint are assumed to be the identity path.
\end{proof}

\begin{lem}\label{lem:smash-zero}
	There are homotopies 
	\begin{align*}
	t_g : \const\smsh g\sim \const && t'_f : f\smsh\const\sim\const
	\end{align*}
	such that the following diagrams
    commute for given homotopies $p : g\sim g'$ and $q : f\sim f'$.
    \begin{equation}\label{eq:t-triangles}
	\begin{tikzcd}
	\const\smsh g
		\arrow[rr, equals,"1\smsh p"]
		\arrow[dr,equals,swap,"t_g"]
	&& \const\smsh g'\arrow[dl,equals,"t_{g'}"]
	&f\smsh \const
		\arrow[rr, equals,"q\smsh 1"]
		\arrow[dr,equals,swap, "t'_f"]
	&& f'\smsh \const\arrow[dl,equals,"t'_{f'}"]
	\\
	& \const
	&&& \const
	\end{tikzcd}
	\end{equation}
\end{lem}
\begin{proof}
	We will define the homotopy $t_g : \const \smsh g$, with $\const : A_1 \to A_2$ and $g : B_1 \to B_2$ (with the notational convention for the basepoints as in \autoref{lem:interchange}); the definition for $t'_f$ is analogous. First, we apply induction on the path that $g$ respects the basepoint. The underlying homotopy of $t_g$ is given by induction on terms $x : A_1 \smsh B_1$. On point constructors, we define:
	\begin{align*}
	t_g (a,b) &\defeq \gluer_{g(b)} \tr \gluer_{b_2}\sy && : (a_2, g(b)) = (a_2, b_2)\\
	t_g (\auxl) &\defeq \gluel_{a_2}\sy && : \auxl = (a_2, b_2)\\
	t_g (\auxr) &\defeq \gluer_{b_2}\sy && : \auxr = (a_2, b_2)
	\end{align*}
	If $x$ varies over $\gluel_a$, after some reductions, we need to fill the following square:
	\begin{equation}\label{eq:t-gluel}
	\begin{tikzcd}[column sep=7em]
		(a_2, g(b_1))
			\arrow[r,equals,"\gluer_{b_2} \tr \gluer_{b_2}\sy"]
			\arrow[d,swap,equals, "\gluel_{a_2}"]
		& (a_2, b_2)
			\arrow[d,equals,"1"]
		\\
		\auxl
			\arrow[r,swap,equals, "\gluel_{a_2}\sy"]
		& (a_2, b_2)
	\end{tikzcd}
	\end{equation}
	Similarly, if $x$ varies over $\gluer_b$, we need to fill the following square:
	\begin{equation}\label{eq:t-gluer}
	\begin{tikzcd}[column sep=7em]
		(a_2, g(b))
			\arrow[r,equals,"\gluer_{g(b)} \tr \gluer_{b_2}\sy"]
			\arrow[d,swap,equals, "\gluer_{g(b)}"]
		& (a_2, b_2)
			\arrow[d,equals,"1"]
		\\
		\auxr
			\arrow[r,swap,equals, "\gluer_{b_2}\sy"]
		& (a_2, b_2)
	\end{tikzcd}
	\end{equation}
	The squares in (\ref{eq:t-gluel}) and (\ref{eq:t-gluer}) can both be filled by simple path algebra. The resulting homotopy is pointed, as $t_g(a_1,b_1)$ is equal to the identity path and the proof that $g$ respects the basepoint is also assumed to be the identity path. Finally, for $p : g \sim g'$, the diagram on the left in (\ref{eq:t-triangles}) commutes by induction on $p$.
\end{proof}

\begin{lem}\label{lem:smash-coh}
 	Suppose that we have maps $A_1\lpmap{f_1}A_2\lpmap{f_2}A_3$ and $B_1\lpmap{g_1}B_2\lpmap{g_2}B_3$
 	and suppose that either $f_1$ or $f_2$ is constant. Then there are two homotopies
  $(f_2 \o f_1)\smsh (g_2 \o g_1)\sim \const$, one of which uses the interchange law and one that does not. These two homotopies are equal. Specifically, the following two diagrams commute:
	\begin{center}
	\begin{tikzcd}
		(f_2 \o \const)\smsh (g_2 \o g_1)
			\arrow[r, equals, "i"]
			\arrow[dd, swap, equals, "\zeroh' \smsh (g_2 \o g_1)"]
		&(f_2 \smsh g_2)\o (\const \smsh g_1)
			\arrow[d, equals, "(f_2 \smsh g_2) \o t_{g_1}"]
		\\
		& (f_2 \smsh g_2)\o \const
			\arrow[d,equals, "\zeroh'"]
		\\
		\const\smsh (g_2 \o g_1)
			\arrow[r,equals, swap, "t_{g_2 \o g_1}"]
		& \const
	\end{tikzcd}
	\qquad
	\begin{tikzcd}
		(\const \o f_1)\smsh (g_2 \o g_1)
			\arrow[r, equals, "i"]
			\arrow[dd, swap, equals, "\zeroh \smsh (g_2 \o g_1)"]
		& (\const \smsh g_2)\o (f_1 \smsh g_1)
			\arrow[d,equals, "t_{g_2} \o (f_1 \smsh g_1)"]
		\\
		& \const\o (f_1 \smsh g_1)
			\arrow[d,equals, "\zeroh"]
		\\
		\const\smsh (g_2 \o g_1)
			\arrow[r,swap, equals, "t_{g_2 \o g_1}"]
		& \const
	\end{tikzcd}
	\end{center}

\end{lem}
\begin{proof}
	
	We start by filling the diagram on the left. First apply induction on the paths that $f_2$, $g_1$ and $g_2$
  respect the basepoint. In this case $f_2\o\const$ is definitionally equal to $\const$, and the canonical
  proof that $f_2\o \const\sim\const$ is (definitionally) equal to reflexivity. This means that the homotopy
  $(f_2 \o \const)\smsh (g_2 \o g_1)\sim\const\smsh (g_2 \o g_1)$ is also equal to reflexivity, and also the
  path that $f_2 \smsh g_2$ respects the basepoint is reflexivity, hence the homotopy
  $(f_2 \smsh g_2)\o \const\sim\const$ is also reflexivity. This means we need to fill the following square:
	\begin{center}
	\begin{tikzcd}
		(f_2 \o \const)\smsh (g_2 \o g_1)
			\arrow[r, equals,"i"]
			\arrow[d, swap, equals,"1"]
		& (f_2 \smsh g_2)\o (\const \smsh g_1)
			\arrow[d,equals,"(f_2\smsh g_2)\o t_{g_1}"]
		\\
		\const \smsh (g_2 \o g_1)
			\arrow[r, swap, equals,"t_{g_1 \o g_2}"]
		& \const
	\end{tikzcd}
	\end{center}
  For the underlying homotopy, take $x : A_1\smsh B_1$ and apply induction on $x$. Suppose
  $x\equiv(a,b)$ for $a:A_1$ and $b:B_1$. With the notational convention for basepoints as in \autoref{lem:interchange}, we have to fill the square (we use that the paths that the maps respect the basepoints are reflexivity):
  \begin{equation}\label{eq:pent-left-ab}
    \begin{tikzcd}[column sep=5em]
	(a_3,g_2(g_1(b)))
      	\arrow[r, equals,"1"]
      	\arrow[d,swap,equals,"1"]
	& (a_3,g_2(g_1(b)))
		\arrow[d,equals,"\mapfunc{f_2\smsh g_2}(\gluer_{g_1(b)}\tr\gluer_{b_2}\sy)"]
	\\
	(a_3,g_2(g_1(b)))
		\arrow[r,swap,equals,"\gluer_{g_2(g_1(b))}\tr\gluer_{b_3}\sy"]
	& (a_3,b_3)
    \end{tikzcd}
    \end{equation}  
   Now $\mapfunc{h\smsh k}(\gluer_z)=\gluer_{k(z)}$, so by general groupoid laws we see that the path on the bottom is equal to the path on the right, which means we can fill the square. For the other point constructors, the squares to fill are similar. If $x \judgeq \auxl$, we have:
      \begin{equation}\label{eq:pent-left-auxl}
    \begin{tikzcd}[column sep=5em]
      \auxl \arrow[r, equals,"1"]
      \arrow[d,swap,equals,"1"] &
      \auxl \arrow[d,equals,"\mapfunc{f_2\smsh g_2}(\gluel_{a_2}\sy)"] \\
      \auxl \arrow[r,swap, equals,"\gluel_{a_3}\sy"] &
      (a_3,b_3)
    \end{tikzcd}
 \end{equation}
	We can fill this square, as the path on the bottom is definitionally equal to $\gluel_{a_3}\sy$ (as we applied path induction on the path that $f_2$ respects the basepoint) and the path on the right  also reduces to $\gluel_{a_3}\sy$ using that $\mapfunc{h\smsh k}(\gluel_z)=\gluel_{h(z)}$. Similarly, we can fill the square for $x \judgeq \auxr$, which is:
  \begin{equation}\label{eq:pent-left-auxr}
    \begin{tikzcd}[column sep=5em]
      \auxr \arrow[r, equals,"1"]
      \arrow[d,swap,equals,"1"] &
      \auxr \arrow[d,equals,"\mapfunc{f_2\smsh g_2}(\gluer_{b_2}\sy)"] \\
      \auxr \arrow[r,swap, equals,"\gluer_{b_3}\sy"] &
      (a_3,b_3)
    \end{tikzcd}
  \end{equation}
	If $x$ varies over $\gluel_a$, after some reductions, we need to fill the following cube, where the front and the back are the squares in (\ref{eq:pent-left-ab}) for $(a,b_1)$ and (\ref{eq:pent-left-auxl}) respectively; the left square is degenerate; the other three sides are the squares in the definition of $i$ and $t$ to show that they respect $\gluel_a$ (given in (\ref{eq:i-gluel}) and (\ref{eq:t-gluel}) respectively), where we also apply $f_2 \smsh g_2$ to the square on the right. We suppress in the diagram the arguments of $\gluer$ in $\gluer\tr\gluer\sy$ (which match, so the concatenation results equal to the identity path).
	\begin{equation}\label{eq:pent-left-gluel}
	\begin{tikzcd}[column sep=5em]
	& \auxl
		\arrow[rr, equals, "1"]
		\arrow[dd, swap, equals, near end, "1"]
	&& \auxl
		\arrow[dd, equals, "\mapfunc{f_2\smsh g_2} (\gluel_{a_2}\sy)"]
	\\
	(a_3,b_3)
		\arrow[ur, equals, "\gluel_{a_3}"]
		\arrow[dd, swap, equals, "1"]
		\arrow[rr, equals, crossing over, near end, "1"]
	&& (a_3,b_3)
		\arrow[ur, equals, near start, "\mapfunc{f_2\smsh g_2}(\gluel_{a_2})"]
	\\
	& \auxl
		\arrow[rr, swap, equals, near start, "\gluel_{a_3}\sy"]
	&& (a_3,b_3)
	\\
	(a_3,b_3)
		\arrow[ur, equals, "\gluel_{a_3}"]
		\arrow[rr, swap, equals, "\gluer\tr\gluer\sy"]
	&& (a_3,b_3)
		\arrow[ur, swap, equals, "1"] 
		\arrow[from=uu, equals, crossing over, very near start, "\mapfunc{f_2 \smsh g_2}(\gluer\tr\gluer\sy)"]
	\end{tikzcd}
	\end{equation}
	Similarly, if $x$ varies over $\gluer_b$, we need to fill the cube below: the front and the back are the squares in (\ref{eq:pent-left-ab}) for $(a_1,b)$ and (\ref{eq:pent-left-auxr}) respectively; the left square is again degenerate; the other three sides come from the fact that $i$ and $t$ respect $\gluer_b$ (given in (\ref{eq:i-gluer}) and (\ref{eq:t-gluer}) respectively). Again, we omit the arguments of $\gluer$ in $\gluer\tr\gluer\sy$ (in this case, not a priori judgmentally equal).
	\begin{equation}\label{eq:pent-left-gluer}
	\begin{tikzcd}[column sep=4em]
	& \auxr
		\arrow[rr, equals,"1"]
		\arrow[dd, swap, equals, near end,"1"]
	&& \auxr
		\arrow[dd,equals,"\mapfunc{f_2\smsh g_2}(\gluer_{b_2}\sy)"]
	\\
	(a_3,g_2(g_1(b)))
		\arrow[rr, equals, near end, crossing over, "1"]
		\arrow[dd, swap, equals, "1"]
		\arrow[ur, equals, "\gluer_{g_2(g_1(b))}"]
	&& (a_3,g_2(g_1(b)))
		\arrow[ur, equals, near start, "\mapfunc{f_2\smsh g_2}(\gluer_{g_1(b)})"]
	\\
	& \auxr
		\arrow[rr, swap, equals, near start, "\gluer_{b_3}\sy"]
	&& (a_3,b_3)
	\\
	(a_3,g_2(g_1(b)))
		\arrow[rr, swap, equals,"\gluer\tr\gluer\sy"]
		\arrow[ur, equals, near end, "\gluer_{g_2(g_1(b))}"]
	&& (a_3,b_3)
		\arrow[from=uu, equals, crossing over, very near start, "\mapfunc{f_2\smsh g_2}(\gluer\tr\gluer\sy)"]
		\arrow[ur, swap, equals, "1"]
	\end{tikzcd}
	\end{equation}
  In order to fill the cubes in (\ref{eq:pent-left-gluel}) and (\ref{eq:pent-left-gluer}), we generalize the paths and fill the cubes by path induction. The cube in (\ref{eq:pent-left-gluel}) can be generalized to a cube:
  \begin{center}
  	\begin{tikzcd}[column sep=3em]
	& h(y)
		\arrow[rr, equals,"1"]
		\arrow[dd, swap, equals, near end,"1"]
	&& h(y)
		\arrow[dd,equals,"\mapfunc{h}(p_l\sy)"]
	\\
	h(x)
		\arrow[rr, equals, near end, crossing over, "1"]
		\arrow[dd, swap, equals, "1"]
		\arrow[ur, equals, "q_l"]
	&& h(x)
		\arrow[ur, equals, near start, "\mapfunc{h}(p_l)"]
	\\
	& h(y)
		\arrow[rr, swap, equals, near start, "q_l\sy"]
	&& h(x)
	\\
	h(x)
		\arrow[rr, swap, equals,"q_r\tr q_r\sy"]
		\arrow[ur, equals, "q_l"]
	&& h(x)
		\arrow[from=uu, equals, crossing over, near start, "\mapfunc{h}(p_r\tr p_r\sy)"]
		\arrow[ur, swap, equals, "1"]
	\end{tikzcd}
	\end{center}
	for $X$ and $X'$ pointed types; a map $h : X \to X'$; terms $x$, $y$ $z : X$; paths $p_l : x = y$, $p_r : x = z$, $q_l : h(x) = h(y)$, $q_r : h(x) = h(z)$; and 2-paths $s_l : \mapfunc{h}(p_l) = q_l$ (for the back and the top) and $s_r : \mapfunc{h}(p_r) = q_r$ (for the right side). This cube is filled by path induction on $s_l$, $s_r$, $p_l$ and $p_r$. The cube in (\ref{eq:pent-left-gluer}) can be generalized to a similar cube:
	 \begin{center}
  	\begin{tikzcd}[column sep=3em]
	& h(y)
		\arrow[rr, equals,"1"]
		\arrow[dd, swap, equals, near end,"1"]
	&& h(y)
		\arrow[dd,equals,"\mapfunc{h}(p_b)"]
	\\
	h(x)
		\arrow[rr, equals, near end, crossing over, "1"]
		\arrow[dd, swap, equals, "1"]
		\arrow[ur, equals, "q_l"]
	&& h(x)
		\arrow[ur, equals, near start, "\mapfunc{h}(p_l)"]
	\\
	& h(y)
		\arrow[rr, swap, equals, near start, "q_b"]
	&& h(z)
	\\
	h(x)
		\arrow[rr, swap, equals,"q_l\tr q_b"]
		\arrow[ur, equals, "q_l"]
	&& h(z)
		\arrow[from=uu, equals, crossing over, near start, "\mapfunc{h}(p_l\tr p_b)"]
		\arrow[ur, swap, equals, "1"]
	\end{tikzcd}
	\end{center}
	for paths $p_l : x = y$, $p_b : y = z$, $q_l : h(x) = h(y)$, $q_b : h(y) = h(z)$ and for 2-paths $s_l : \mapfunc{h}(p_l) = q_l$ (for the top) and $s_b : \mapfunc{h}(p_b) = q_b$ (for the back).
	
	The diagram on the right is similar to the previous case. It is not hard to show that these homotopies are pointed.

\end{proof}

\begin{thm}\label{thm:smash-functor-right}
Given pointed types $A$, $B$ and $C$, the functorial action of the smash product induces a map
$$({-})\smsh C:(A\pmap B)\pmap(A\smsh C\pmap B\smsh C)$$
that is natural in $A$ and $B$ and dinatural in $C$.
\end{thm}
The naturality and dinaturality means that the following squares commute for $f : A' \to A$ $g:B\to B'$ and $h:C\to C'$.
\begin{center}
\begin{tikzcd}[column sep=5em]
(A\pmap B) \arrow[r,"({-})\smsh C"]\arrow[d,"f\pmap B"] &
(A\smsh C\pmap B\smsh C)\arrow[d,"f\smsh C\pmap B\smsh C"] \\
(A'\pmap B) \arrow[r,"({-})\smsh C"] &
(A'\smsh C\pmap B\smsh C)
\end{tikzcd}
\qquad
\begin{tikzcd}[column sep=5em]
(A\pmap B) \arrow[r,"({-})\smsh C"]\arrow[d,"A\pmap g"] &
(A\smsh C\pmap B\smsh C)\arrow[d,"A\smsh C\pmap g\smsh C"] \\
(A\pmap B') \arrow[r,"({-})\smsh C"] &
(A\smsh C\pmap B'\smsh C)
\end{tikzcd}
\begin{tikzcd}[column sep=5em]
(A\pmap B) \arrow[r,"({-})\smsh C"]\arrow[d,"({-})\smsh C'"] &
(A\smsh C\pmap B\smsh C)\arrow[d,"A\smsh C\pmap B\smsh h"] \\
(A\smsh C'\pmap B\smsh C') \arrow[r,"A\smsh h\pmap B\smsh C'"] &
(A\smsh C\pmap B\smsh C')
\end{tikzcd}
\end{center}
\begin{proof}
First note that $\lam{f}f\smsh C$ preserves the basepoint so that the map is indeed pointed.

Let $k:A\pmap B$. Then as homotopy the naturality in $A$ becomes
$(k\o f)\smsh C=k\smsh C\o f\smsh C$. To prove an equality between pointed maps, we need to give
a pointed homotopy, which is given by interchange. To show that this homotopy is pointed, we need to
fill the following square (after reducing out the applications of function extensionality), which follows from \autoref{lem:smash-coh}.
\begin{center}
\begin{tikzcd}
(\const \o f)\smsh C \arrow[r, equals]\arrow[dd,equals] &
(\const \smsh C)\o (f \smsh C)\arrow[d,equals] \\
& \const \o (f \smsh C)\arrow[d,equals] \\
\const\smsh C \arrow[r,equals] &
\const
\end{tikzcd}
\end{center}
The naturality in $B$ is almost the same: for the underlying homotopy we need to show
$i:(g \o k)\smsh C = g\smsh C \o k\smsh C$. For the pointedness we need to fill the following
square, which follows from the left pentagon in \autoref{lem:smash-coh}.
\begin{center}
\begin{tikzcd}
(g \o \const)\smsh C \arrow[r, equals]\arrow[dd,equals] &
(g \smsh C)\o (\const \smsh C)\arrow[d,equals] \\
& (g\smsh C) \o \const\arrow[d,equals] \\
\const\smsh C \arrow[r,equals] &
\const
\end{tikzcd}
\end{center}

The dinaturality in $C$ is a bit harder. For the underlying homotopy we need to show
$B\smsh h\o k\smsh C=k\smsh C'\o A\smsh h$. This follows from applying interchange twice:
$$B\smsh h\o k\smsh C\sim(\idfunc[B]\o k)\smsh(h\o\idfunc[C])\sim(k\o\idfunc[A])\smsh(\idfunc[C']\o h)\sim k\smsh C'\o A\smsh h.$$
To show that this homotopy is pointed, we need to fill the following square:
\begin{xcenter}
  \begin{tikzcd}
    B\smsh h\o \const\smsh C \arrow[r, equals]\arrow[d,equals] &
    (\idfunc[B]\o \const)\smsh(h\o\idfunc[C]) \arrow[r, equals]\arrow[d,equals] &
    (\const\o\idfunc[A])\smsh(\idfunc[C']\o h)\arrow[r, equals]\arrow[d,equals] &
    \const\smsh C'\o A\smsh h\arrow[d,equals] \\
    B\smsh h\o \const \arrow[d,equals] &
    \const\smsh(h\o\idfunc[C]) \arrow[r, equals]\arrow[d,equals] &
    \const\smsh(\idfunc[C']\o h) \arrow[d,equals] &
    \const\o A\smsh h\arrow[d,equals] \\
    B\smsh h\o \const \arrow[r, equals] &
    \const \arrow[r, equals] &
    \const \arrow[r, equals] &
    \const
  \end{tikzcd}
\end{xcenter}
The left and the right squares are filled by \autoref{lem:smash-coh}. The squares in the middle
are filled by (corollaries of) \autoref{lem:smash-general}.
\end{proof}

\subsection{Adjunction}\label{sec:smash-adjunction}

\begin{lem}\label{lem:unit-counit}
  There is a unit $\eta_{A,B}\equiv\eta:A\pmap B\pmap A\smsh B$ natural in $A$ and counit
  $\epsilon_{B,C}\equiv\epsilon : (B\pmap C)\smsh B \pmap C$ dinatural in $B$ and natural in $C$.
  These maps satisfy the unit-counit laws:
  $$(A\to\epsilon_{A,B})\o \eta_{A\to B,A}\sim \idfunc[A\to B]\qquad
  \epsilon_{B,B\smsh C}\o \eta_{A,B}\smsh B\sim\idfunc[A\smsh B].$$
\end{lem}
Note: $\eta$ is also dinatural in $B$, but we do not need this.
\begin{proof}
  We define $\eta ab=(a,b)$. We define the path that $\eta a$ respects the basepoint as
  $$(\eta a)_0\defeq\gluel_a\tr\gluel_{a_0}\sy:(a,b_0)=(a_0,b_0).$$ Also, $\eta$ itself respects the basepoint. To show this, we need to give $\eta_0:\eta (a_0)\sim \const$. The underlying maps are homotopic, by $$\eta_0b\defeq\gluer_b\cdot\gluer_{b_0}\sy:(a_0,b)=(a_0,b_0).$$ To show that
  this homotopy is pointed, we need to show that the two given proofs of $(a_0,b_0)=(a_0,b_0)$ are
  equal, but they are both equal to reflexivity:
  $$\eta_{00}:\gluel_{a_0}\tr\gluel_{a_0}\sy=1=\gluer_{b_0}\tr\gluer_{b_0}\sy.$$
  This defines the unit. To show that it is natural in $A$, we need to give the following pointed homotopy $p_\eta(f)$ for $f:A\to A'$.
  \begin{center}
	\begin{tikzcd}
	A \arrow[r,"\eta"]\arrow[d,"f"] &
	(B\pmap A \smsh B)\arrow[d,"B\to f\smsh B"] \\
	A' \arrow[r,"\eta"] &
	(B\pmap A'\smsh B)
	\end{tikzcd}
  \end{center}
  We may assume that $f_0$ is reflexivity. For the underlying homotopy we need to define for $a:A$ that $p_\eta(f,a):\eta(fa)\sim f\smsh B \circ \eta a$, which is another pointed homotopy. For $b:B$ we have $\eta(fa,b)\equiv(fa,b)\equiv(f\smsh B)(\eta ab).$
  The homotopy $p_\eta(f,a)$ is pointed, since $$(f\smsh B \circ \eta a)_0=\apfunc{f\smsh B}(\gluel_a\cdot\gluel_{a_0}\sy)=\gluel_{fa}\cdot\gluel_{a_0'}\sy=(\eta(fa))_0.$$
  Now we need to show that $p_\eta(f)$ is pointed, for which we need to fill the following diagram.
  \begin{center}
	\begin{tikzcd}
	\eta(fa_0) \arrow[equals,rr,"{p_\eta(f,a_0)}"]\arrow[dr,equals,"{\eta_0}"] & &
	f\smsh B \circ \eta a_0\arrow[dl,equals,"{f\smsh B\circ\eta_0}"] \\
	& \const_{B,A'\smsh B} &
	\end{tikzcd}
  \end{center}
  These pointed homotopies have equal underlying homotopies, since for $b:B$ we have 
  $$p_\eta(f,a_0,b)\cdot\apfunc{f\smsh B}(\eta_0 b)=1\cdot\apfunc{f\smsh B}(\gluer_b\cdot\gluer_{b_0}\sy)=\gluer_{b}\cdot\gluer_{b_0}\sy=\eta_0b.$$
  We will skip the proof that these homotopies respect the point in the same way.

  To define the counit, given $x:(B\pmap C)\smsh B$, we construct
  $\epsilon (x):C$ by induction on $x$. If $x\jdeq(f,b)$, we set $\epsilon(f,b)\defeq f(b)$. If $x$
  is either $\auxl$ or $\auxr$, then we set $\epsilon (x)\defeq c_0:C$. If $x$ varies over $\gluel_f$,
  then we need to show that $f(b_0)=c_0$, which is true by $f_0$. If $x$ varies over $\gluer_b$, we
  need to show that $\const(b)=c_0$ which is true by reflexivity. Now $\epsilon_0\defeq 1:\epsilon(\const_{B,C},b_0)=c_0$ shows that $\epsilon$ is pointed.

  We will skip the proof that the counit is dinatural in $B$ and natural in $C$.

  Finally, we need to show the unit-counit laws. For the underlying homotopy of the first one, let
  $f:A\to B$. We need to show that $p_f:\epsilon\o\eta f\sim f$. We define $p_f(a)=1:\epsilon(f,a)=f(a)$. To show that $p_f$ is a pointed homotopy, we need to show that
  $p_f(a_0)\tr f_0=\mapfunc{\epsilon}(\eta f)_0\tr \epsilon_0$, which reduces to
  $f_0=\mapfunc{\epsilon}(\gluel_f\tr\gluel_0\sy)$, but we can reduce the right hand side: (note:
  $\const_0$ denotes the proof that $\const(a_0)=b_0$, which is reflexivity)
  $$\mapfunc{\epsilon}(\gluel_f\tr\gluel_0\sy)=\mapfunc{\epsilon}(\gluel_f)\tr(\mapfunc{\epsilon}(\gluel_0))\sy=f_0\tr \const_0\sy=f_0.$$
  Now we need to show that $p$ itself respects the basepoint of $A\to B$, i.e. that the composite
  $\epsilon\o\eta(\const)\sim\epsilon\o\const\sim\const$ is equal to $p_{\const_{A,B}}$. The underlying
  homotopies are the same for $a : A$; on the one side we have
  $\mapfunc{\epsilon}(\gluer_{a}\tr\gluer_{a_0}\sy)$ and on the other side we have reflexivity
  (note: this type checks since $\const_{A,B}a\equiv\const_{A,B}a_0$). These paths are equal, since
  $$\mapfunc{\epsilon}(\gluer_{a}\tr\gluer_{a_0}\sy)=\mapfunc{\epsilon}(\gluer_{a})\tr(\mapfunc\epsilon(\gluer_{a_0}))\sy=1\cdot1\sy\equiv1.$$
  Both pointed homotopies are pointed in the same way, which requires some path-algebra, and we skip the proof here.

  For the underlying homotopy of the second unit-counit law, we need to show for $x:A\smsh B$ that
  $q(x):\epsilon((\eta\smsh B)x)=x$, which we prove by induction to $x$. If $x\equiv(a,b)$, then we can define $q(a,b)\defeq1_{(a,b)}$. 
  If $x$ is $\auxl$ or $\auxr$, then the left-hand side reduces to $(a_0,b_0)$, 
  so we can define $q(\auxl)\defeq\gluel_{a_0}$ and $q(\auxr)\defeq\gluer_{b_0}$. The following computation shows that $q$ respects $\gluel_a$:
  \begin{align*}
    \apfunc{\epsilon\circ\eta\smsh B}(\gluel_a)\cdot\gluel_{a_0}&= \apfunc{\epsilon}(\gluel_{\eta a})\cdot\gluel_{a_0}=(\eta a)_0\cdot\gluel_{a_0}=\gluel_a\cdot\gluel_{a_0}\sy\cdot\gluel_{a_0}\\
    &=\gluel_a.
  \end{align*}
  To show that it respects $\gluer_b$ we compute
  \begin{align*}
    \apfunc{\epsilon\circ\eta\smsh B}(\gluer_b)\cdot\gluer_{b_0}&=
  \apfunc{\epsilon({-},b)}(\eta_0)\cdot\apfunc{\epsilon}(\gluer_b)\cdot\gluer_{b_0}=
  \apfunc{\lam{f}fb}(\eta_0)\cdot\gluer_{b_0}\\
  &=\eta_0b\cdot\gluer_{b_0}=
  \gluer_b.
  \end{align*}
  To show that $q$ is a pointed homotopy, we need to show that $(\epsilon\circ\eta\smsh B)_0=1$, For this we compute $$(\epsilon\circ\eta\smsh B)_0=\apfunc{\epsilon({-},b_0)}(\eta_0)=\eta_0b_0=\gluer_{b_0}\cdot\gluer_{b_0}\sy=1.$$
\end{proof}

\begin{defn}
The function $e\jdeq e_{A,B,C}:(A\pmap B\pmap C)\pmap(A\smsh B\pmap C)$ is defined as the composite
$$(A\pmap B\pmap C)\lpmap{({-})\smsh B}(A\smsh B\pmap (B\pmap C)\smsh B)\lpmap{A\smsh B \pmap\epsilon}(A\smsh B\pmap C).$$
\end{defn}

\begin{lem}
  The function $e$ is invertible, hence gives a pointed equivalence $$(A\pmap B\pmap C)\simeq(A\smsh B\pmap C).$$
\end{lem}
\begin{proof}
  Define
  $$e\sy_{A,B,C}:(A\smsh B\pmap C)\lpmap{B\pmap({-})}((B\pmap A\smsh B)\pmap (B\pmap
  C))\lpmap{\eta\pmap(B\pmap C)}(A\pmap B\pmap C).$$ It is easy to show that $e$ and $e\sy$ are
  inverses as unpointed maps from the unit-counit laws (\autoref{lem:unit-counit}) and naturality of $\eta$ and $\epsilon$.
\end{proof}
\begin{lem}\label{lem:e-natural}
	The function $e$ is natural in $A$, $B$ and $C$.
\end{lem}
\begin{proof}
	\textbf{Naturality of $e$ in $A$}. Suppose that $f:A'\pmap A$. Then the following diagram commutes. The left square commutes by naturality of $({-})\smsh B$ in the first argument and the right square commutes because composition on the left commutes with composition on the right.
	\begin{center}
	\begin{tikzcd}
		(A\pmap B\pmap C) \arrow[r,"({-})\smsh B"]\arrow[d,"f\pmap B\pmap C"] &
		(A\smsh B\pmap (B\pmap C)\smsh B) \arrow[r,"A\smsh B\pmap\epsilon"]\arrow[d,"f\smsh B\pmap\cdots"]  &
		(A\smsh B\pmap C)\arrow[d,"f\smsh B\pmap C"] \\
		(A'\pmap B\pmap C) \arrow[r,"({-})\smsh B"] &
		(A'\smsh B\pmap (B\pmap C)\smsh B) \arrow[r,"A\smsh B\pmap\epsilon"] &
		(A'\smsh B\pmap C)
	\end{tikzcd}
	\end{center}

	\textbf{Naturality of $e$ in $C$}. Suppose that $f:C\pmap C'$. Then in the following diagram the left square commutes by naturality of $({-})\smsh B$ in the second argument (applied to $B\pmap f$) and the right square commutes by applying the functor $A\smsh B \pmap({-})$ to the naturality of $\epsilon$ in the second argument.
	\begin{center}
	\begin{tikzcd}
		(A\pmap B\pmap C) \arrow[r]\arrow[d] &
		(A\smsh B\pmap (B\pmap C)\smsh B) \arrow[r]\arrow[d] &
		(A\smsh B\pmap C)\arrow[d] \\
		(A\pmap B\pmap C') \arrow[r] &
		(A\smsh B\pmap (B\pmap C')\smsh B) \arrow[r] &
		(A\smsh B\pmap C')
	\end{tikzcd}
	\end{center}

	\textbf{Naturality of $e$ in $B$}. Suppose that $f:B'\pmap B$. Here the diagram is a bit more
complicated, since $({-})\smsh B$ is dinatural (instead of natural) in $B$. Then we get the
following diagram. The front square commutes by naturality of $({-})\smsh B$ in the second argument
	(applied to $f\pmap C$). The top square commutes by naturality of $({-})\smsh B$ in the third
argument, the back square commutes because composition on the left commutes with composition on the
	right, and finally the right square commutes by applying the functor $A\smsh B' \pmap({-})$ to the
	naturality of $\epsilon$ in the first argument.
	\begin{center}
	\begin{tikzcd}[row sep=scriptsize, column sep=-4em]
		& (A\smsh B\pmap (B\pmap C)\smsh B) \arrow[rr] \arrow[dd] & & (A\smsh B'\pmap (B\pmap C)\smsh B)\arrow[dd] \\
		(A\pmap B\pmap C) \arrow[ur] \arrow[rr, crossing over] \arrow[dd] & & (A\smsh B'\pmap (B\pmap C)\smsh B') \arrow[ur] \\
		& (A\smsh B\pmap C)\arrow[rr] &  & (A\smsh B'\pmap C) \\
		(A\pmap B'\pmap C) \arrow[rr] & & (A\smsh B'\pmap (B'\pmap C)\smsh B') \arrow[ur] \arrow[from=uu, crossing over]
	\end{tikzcd}
	\end{center}

\end{proof}
\begin{rmk}
  Instead of showing that $e$ is natural, we could show that $e^{-1}$ is natural. In
  that case we need to show that the map $A\to({-}):(B\to C)\to(A\to B)\to(A\to C)$ is natural in
  $A$, $B$ and $C$. This might actually be easier, since we do not need to work with any higher
  inductive type to prove that.
\end{rmk}

We have now obtained the following theorem
\begin{thm}\label{thm:smash-adjoint}
  There is an equivalence 
  $$(A\pmap B\pmap C)\simeq(A\smsh B\pmap C)$$
  natural in $A$, $B$ and $C$.
\end{thm}

\begin{rmk}
  We can state \autoref{thm:smash-adjoint} as an adjunction $({-})\smsh B\dashv B\to({-})$ or by saying that $A\smsh B$
  represents the functor $A\pmap B\pmap ({-})$.
  
  In \autoref{sec:smash-monoidal} we show that the smash product forms a
  1-coherent symmetric monoidal product from the assumption that this adjunction
  is pointed in $C$. Explicitly, this means that the naturality of $e$ in $C$
  applied to the map $\const_{C,C'}:C\to C'$ is equal to the composite
  $$(A\smsh B \to \const_{C,C'})\o e_{A,B,C}\sim \const\o e_{A,B,C}\sim \const
  \sim e_{A,B,C'}\o \const \sim e_{A,B,C'}\o (A\to B\to \const_{C,C'}).$$ 
  
  To prove this, we need that the counit $\eps$ is pointed natural in $C$. To
  prove that, we need to show that the map $({-})\smsh C$, defined in
  \autoref{thm:smash-functor-right}, is pointed natural in $B$. In order to
  prove that, we need to show that in the situation of \autoref{lem:smash-coh},
  if both $f_1$ and $f_2$ are (judgmentally) the constant map, then the two
  pentagons stated in that lemma are equal (transported appropriately in order
  to make this equality type check). This can be formulated as a 3-path in a
  type of pointed maps, which is hard to fill.
\end{rmk}

\subsection{Symmetric monoidal product}\label{sec:smash-monoidal}
In this section we will prove that the smash product is a 1-coherent symmetric
monoidal product~\autoref{def:symmonprod}, from the assumption that the
adjunction from \autoref{sec:smash-adjunction} is pointed natural in $C$. We
will need to following pointed equivalences. Without the proof that $e$ is
pointed natural, parts of this section are still true. In particular, the
natural equivalences defined in \autoref{def:smash-alrg} do not require pointed
naturality of $e$.

\begin{defn}\label{def:b-and-tw}
  We define the pointed equivalences:
    \[\two : (\pbool \to X) \simeq X\] where $\pbool$ is the type of booleans (pointed in $\bfalse$) with underlying map defined with $\two(f) \defeq f(\btrue)$, and
    \[\twist : (A \to B \to X) \simeq (B \to A \to X)\]
    with underlying map defined with $\twist(f) \defeq \lam{b}\lam{a}f(a)(b)$.
  \end{defn}

Using \autoref{lem:yoneda} (Yoneda) we can prove associativity, left and right
unitality and braiding equivalences for the smash product, in the following way.

\begin{defn}\label{def:equiv-precursors}
	The following pointed equivalences are defined for $A$, $B$, $C$ and $X$ pointed types:
	\begin{itemize}
		\item $\alphabar_X : (A \smsh (B \smsh C) \to X) \simeq ((A \smsh B) \smsh C \to X)$ as the composition of the equivalences:
			\begin{align*}
			    A \smsh (B \smsh C)\to X&\simeq A \to B\smsh C\to X && (e\sy)\\
			    &\simeq A \to B\to C\to X && (A \to e\sy)\\
			    &\simeq A \smsh B\to C\to X && (e)\\
		    	&\simeq (A \smsh B)\smsh C\to X. && (e)
			\end{align*}
		\item $\lambdabar_X : (B \to X) \simeq (\pbool \smsh B \to X)$ as the composition of the equivalences:
			\begin{align*}
				B \to X &\simeq \pbool \to B \to X && (\two\sy)\\
				&\simeq \pbool \smsh B \to X && (e)
			\end{align*}
		\item $\rhobar_X : (A \to X) \simeq (A \smsh \pbool \to X)$ as the composition of the equivalences:
			\begin{align*}
				A \to X &\simeq A \to \pbool \to X && (A \to \two\sy)\\
				&\simeq A \smsh \pbool \to X && (e)
			\end{align*}
		\item $\gammabar_X : (B \smsh A \to X) \simeq (A \smsh B \to X)$ as the composition of the equivalences:
			\begin{align*}
				B \smsh A \to X &\simeq B \to A \to X && (e\sy)\\
				&\simeq A \to B \to X && (\twist)\\
				&\simeq A \smsh B \to X && (e)
			\end{align*}
	\end{itemize}
\end{defn}

\begin{rmk}\label{rmk:alrg-pointed-natural}
	The equivalences in \autoref{def:equiv-precursors} are natural in all their
	arguments and from the assumption that $e$ is pointed natural in $C$ we can
  show that these maps are all pointed natural in $X$. 
\end{rmk}

\begin{defn}\label{def:smash-alrg}
	We define the following equivalences, natural in all their arguments, with inverses provided as in \autoref{lem:yoneda}:
	\begin{itemize}
		\item $\alpha\defeq\alphabar_{A \smsh (B \smsh C)}(\idfunc) : (A \smsh B) \smsh C \simeq A \smsh (B \smsh C)$ (associativity of the smash product), with inverse $\alpha\sy\defeq\alphabar\sy_{(A \smsh B) \smsh C}(\idfunc)$;
		\item $\lambda \defeq \lambdabar_B(\idfunc) : \pbool \smsh B \simeq B$ and $\rho \defeq \rhobar_A(\idfunc) : A \smsh \pbool \simeq A$ (left- and right unitors for the smash product), with inverses $\lambda\sy\defeq \lambdabar_{\pbool\smsh B}\sy(\idfunc)$ and $\rho\sy\defeq \rhobar_{A\smsh \pbool}\sy(\idfunc)$, respectively;
		\item $\gamma \defeq \gammabar_{B\smsh A} (\idfunc) : A \smsh B \simeq B \smsh A$ (braiding for the smash product), with inverse $\gamma\sy \defeq \gammabar_{A \smsh B}\sy (\idfunc)$.
	\end{itemize}
	$\alpha$, $\lambda$, $\rho$ and $\gamma$ are natural in all their arguments,
	as $\alphabar$, $\lambdabar$, $\rhobar$ and $\gammabar$ are. Note that these
	definitions do \emph{not} require pointed naturality of $e$.
\end{defn}

\begin{lem}\label{lem:bar-homotopy}
	There are pointed homotopies
	\begin{align*}
	\alphabar_X &\sim \alpha \to X
		& \lambdabar_X &\sim \lambda \to X
	\\
	\rhobar_X &\sim \rho \to X
		& \gammabar_X &\sim \gamma \to X
	\end{align*}
\end{lem}

\begin{proof}
	This follows directly from \autoref{lem:yoneda-pointed} and \autoref{rmk:alrg-pointed-natural} (this does require pointed naturality of $e$).
\end{proof}

\begin{thm}[Associativity pentagon]\label{thm:smash-associativity-pentagon}
	For $A$, $B$, $C$ and $D$ pointed types, there is a homotopy
	\[\alpha \o \alpha \sim (A \smsh \alpha) \o \alpha \o (\alpha \smsh D)\]
	corresponding to the commutativity of the following diagram:
	\begin{center}
	\begin{tikzcd}
		&((A \smsh B) \smsh (C \smsh D))
			\arrow[dr, "\alpha"]
		\\
		(((A \smsh B) \smsh C) \smsh D)
			\arrow[ru, "\alpha"]
			\arrow[d, swap, "\alpha \smsh D"]
		&& (A \smsh (B \smsh (C \smsh D)))
		\\
		((A \smsh (B \smsh C)) \smsh D)
			\arrow[rr, swap, "\alpha"]
		&& (A \smsh ((B \smsh C) \smsh D))
			\arrow[u, swap, "A \smsh \alpha"]
	\end{tikzcd}
	\end{center}
\end{thm}
\begin{proof}
	We articulate the proof in several steps. A map homotopic to both sides of the sought homotopy will be constructed via the equivalence
	\begin{align*}	
		\alphabar^4 : (A \smsh (B \smsh (C \smsh D)) \to X) &\simeq (((A \smsh B) \smsh C) \smsh D \to X)
		\intertext{(natural in all its arguments), defined as the composite:}
		A \smsh (B \smsh (C \smsh D)) \to X
		&\simeq A \to B \smsh (C \smsh D) \to X && \text{($e\sy$)}\\
		&\simeq A \to B \to C \smsh D \to X &&\text{($A \to e\sy$)}\\
		&\simeq A \to B \to C \to D \to X &&\text{($A \to B \to e\sy$)}\\
		&\simeq A \smsh B \to C \to D \to X &&\text{($e$)}\\
		&\simeq (A \smsh B) \smsh C \to D \to X &&\text{($e$)}\\
		&\simeq ((A \smsh B) \smsh C) \smsh D \to X && \text{($e$)}
		\intertext{giving $\alphabar^4(\idfunc) : ((A \smsh B) \smsh C) \smsh D) \simeq A \smsh (B \smsh (C \smsh D))$. Moreover, in order to simplify the expressions of $\alpha \smsh D$ and $A \smsh \alpha$, we also define:}
		\alphabar^R : ((A \smsh (B \smsh C)) \smsh D \to X) &\simeq (((A \smsh B) \smsh C) \smsh D \to X)
		\intertext{as the composite:}
		(A \smsh (B \smsh C)) \smsh D \to X
		&\simeq A \smsh (B \smsh C) \to D \to X &&\text{($e\sy$)}\\
		&\simeq (A \smsh B) \smsh C \to D \to X &&\text{($\alphabar$)}\\
		&\simeq ((A \smsh B) \smsh C) \smsh D \to X &&\text{($e$)}
		\intertext{and}
		\alphabar^L : (A \smsh (B \smsh (C \smsh D)) \to X) &\simeq (A \smsh ((B \smsh C) \smsh D) \to X)
		\intertext{as the composite:}
		A \smsh (B \smsh (C \smsh D)) \to X
		&\simeq A \to B \smsh (C \smsh D) \to X &&\text{($e\sy$)}\\
		&\simeq A \to (B \smsh C) \smsh D \to X &&\text{($A \to \alphabar$)}\\
		&\simeq A \smsh ((B \smsh C) \smsh D) \to X &&\text{($e$)}
	\end{align*}
	also natural in their arguments. Evaluating these equivalences to the identity function, we get new arrows that fit in the original diagram:
	\begin{center}
	\begin{tikzcd}
		&((A \smsh B) \smsh (C \smsh D))
			\arrow[dr, "\alpha"]
		\\
		(((A \smsh B) \smsh C) \smsh D)
			\arrow[ru, "\alpha"]
			\arrow[d, swap, "\alpha \smsh D"]
			\arrow[d, bend left=40, "\alphabar^R(\idfunc)"]
			\arrow[rr, "\alphabar^4(\idfunc)"]
		&& (A \smsh (B \smsh (C \smsh D)))
		\\
		((A \smsh (B \smsh C)) \smsh D)
			\arrow[rr, swap, "\alpha"]
		&& (A \smsh ((B \smsh C) \smsh D))
			\arrow[u, swap, "A \smsh \alpha"]
			\arrow[u, bend left=40, "\alphabar^L(\idfunc)"]
	\end{tikzcd}
	\end{center}
	
	The theorem is then proved once we show the chain of homotopies:
	\begin{equation}\label{eq:alphafour}
	\alpha \o \alpha
	\sim \alphabar^4(\idfunc)
	\sim \alphabar^L(\idfunc) \o \alpha \o \alphabar^R(\idfunc)
	\sim (A \smsh \alpha) \o \alpha \o (\alpha \smsh D)
	\end{equation}
	
	To verify the first homotopy in (\ref{eq:alphafour}), we see that:
	\begin{align*}
		\alpha \o \alpha
		&\judgeq \alphabar(\idfunc) \o \alphabar(\idfunc)\\
		&\sim (\alphabar \o \alphabar) (\idfunc) &&\text{(naturality of $\alphabar$)}\\
		&\judgeq (e \o e \o (A \to e\sy) \o e\sy \o e \o e \o (A \to e\sy) \o e\sy)(\idfunc)\\
		&\sim (e \o e \o (A \to e\sy) \o e \o (A \to e\sy) \o e\sy)(\idfunc) &&\text{(cancelling)}\\
		&\sim (e \o e \o e \o (B \to A \to e\sy) \o (A \to e\sy) \o e\sy)(\idfunc) &&\text{(naturality of $e$)}\\
		&\judgeq \alphabar^4(\idfunc)
	\end{align*}		

	The second homotopy in (\ref{eq:alphafour}) is verified by (right-to-left):
	\begin{align*}
		\alphabar^L(\idfunc) \o \alpha \o \alphabar^R(\idfunc)
		&\judgeq \alphabar^L(\idfunc) \o \alphabar(\idfunc) \o \alphabar^R(\idfunc)\\
    &\sim (\alphabar^R \o \alphabar \o \alphabar^L)(\idfunc) \\ &\mbox{}\qquad\text{(naturality of $\alphabar$ and $\alphabar^R$)}\\
		&\judgeq (e \o \alphabar \o e\sy \o e \o e \o (A \to e\sy) \o e\sy \o e \o (A \to \alphabar) \o e\sy)(\idfunc)\\
		&\sim (e \o \alphabar \o e \o (A \to e\sy) \o (A \to \alphabar) \o e\sy)(\idfunc) \\ &\mbox{}\qquad\text{(cancelling)}\\
		&\sim (e \o \alphabar \o e \o (A \to (e\sy \o \alphabar)) \o e\sy)(\idfunc) \\ &\mbox{}\qquad\text{(functoriality of $A\to -$)}\\
		&\judgeq (e \o e \o e \o (A \to e\sy) \o e\sy \o e \\
		&\hspace{3em}\o (A \to (e\sy \o e \o e \o (B \to e\sy) \o e\sy)) \o e\sy)(\idfunc)\\
		&\sim (e \o e \o e \o (A \to ((B \to e\sy) \o e\sy)) \o e\sy)(\idfunc) \\ &\mbox{}\qquad\text{(cancelling)}\\
		&\sim (e \o e \o e \o (B \to A \to e\sy) \o (A \to e\sy) \o e\sy)(\idfunc) \\ &\mbox{}\qquad\text{(funct. of $A \to -$)}\\
		&\judgeq \alphabar^4(\idfunc)
	\end{align*}
	
	In order to prove the last homotopy in (\ref{eq:alphafour}), it is sufficient to show that $\alphabar^R(\idfunc) \sim \alpha \smsh D$ and that $\alphabar^L(\idfunc) \sim A \smsh \alpha$. We have:
	\begin{align*}
		\alphabar^R(\idfunc)
		&\judgeq e(\alphabar (e\sy(\idfunc)))\\
		&\sim e (\alphabar (\eta))\\
		&\sim e (\eta \o \alphabar(\idfunc)) &&\text{(naturality of $\alphabar$)}\\
		&\judgeq \epsilon \o (\eta \o \alpha) \smsh D\\
		&\sim \epsilon \o (\eta \smsh D) \o (\alpha \smsh D) &&\text{(distrib. of $\smsh$)}\\
		&\sim \alpha \smsh D &&\text{(\autoref{lem:unit-counit})}
	\end{align*}
	and, lastly,
	\begin{align*}
		\alphabar^L(\idfunc)
		&\judgeq e(\alphabar \o e\sy(\idfunc))\\
		&\sim e(\alphabar \o \eta)\\
		&\sim e((\alpha \to A \smsh (B \smsh (C \smsh D))) \o \eta) &&\text{(\autoref{lem:bar-homotopy})}\\
		&\sim e((B \smsh (C \smsh D) \to A \smsh \alpha) \o \eta) &&\text{(dinaturality of $\eta$)}\\
		&\sim (A \smsh \alpha) \o e(\eta) &&\text{(naturality of $e$)}\\
		&\sim A \smsh \alpha &&\text{(\autoref{lem:unit-counit})}
	\end{align*}
	thus proving the desired homotopy.
\end{proof}

\begin{thm}[Unitors triangle]\label{thm:smash-unitors-triangle}
	For $A$ and $B$ pointed types, there is a homotopy
	\[(A \smsh \lambda) \o \alpha \sim (\rho \smsh B)\]
	corresponding to the commutativity of the following diagram:
	\begin{center}
	\begin{tikzcd}
	((A \smsh \pbool) \smsh B)
		\arrow[rr, "\alpha"]
		\arrow[dr, swap, "\rho \smsh B"]
	&& (A \smsh (\pbool \smsh B))
		\arrow[dl, "A \smsh \lambda"]
	\\
	& (A \smsh B)
	\end{tikzcd}
	\end{center}
\end{thm}
\begin{proof}
	By an argument similar to the one for $\alphabar^L$ and $\alphabar^R$ in \autoref{thm:smash-associativity-pentagon}, one can verify the homotopies $A \smsh \lambda \sim (e \o (A \to \lambdabar) \o e\sy)(\idfunc)$ and $\rho \smsh B \sim (e \o \rhobar \o e)(\idfunc)$, simplifying the expressions in the sought homotopy. Then:
	\begin{align*}
		(A \smsh \lambda) \o \alpha
		&\sim e(\lambdabar \o e\sy(\idfunc)) \o \alphabar(\idfunc) &&\text{(simplification)}\\
		&\sim \alphabar(e(\lambdabar \o e\sy(\idfunc)) &&\text{(naturality of $\alphabar$)}\\
		&\judgeq e(e(e\sy \o e\sy (e(\lambdabar \o e\sy(\idfunc)))))\\
		&\sim e(e(e\sy \o \lambdabar \o e\sy(\idfunc))) &&\text{(cancelling)}\\
		&\judgeq e(e(e\sy \o e \o \two\sy \o e\sy(\idfunc)))\\
		&\sim e(e(\two\sy \o e\sy(\idfunc))) &&\text{(cancelling)}\\
		&\judgeq (e \o \rhobar \o e\sy)(\idfunc)\\
		&\sim \rho \smsh B &&\text{(simplification)}
	\end{align*}
	gives the desired homotopy.
\end{proof}

\begin{thm}[Braiding-unitors triangle]\label{thm:smash-braiding-unitors}
	For a pointed type $A$, there is a homotopy
	\[\lambda \o \gamma \sim \rho\]
	corresponding to the commutativity of the following diagram:
	\begin{center}
	\begin{tikzcd}
	(A \smsh \pbool)
		\arrow[rr, "\gamma"]
		\arrow[dr, swap, "\rho"]
	&& (\pbool \smsh A)
		\arrow[dl, "\lambda"]
	\\
	& A
	\end{tikzcd}
	\end{center}
\end{thm}
\begin{proof}
	We have:
	\begin{align*}
		\lambda \o \gamma
		&\judgeq \lambdabar(\idfunc) \o \gammabar(\idfunc)\\
		&\sim (\gammabar \o \lambdabar)(\idfunc) &&\text{(naturality of $\gammabar$)}\\
		&\judgeq (e \o \twist \o e\sy \o e \o \two\sy)(\idfunc)\\
		&\sim (e \o \twist \o \two\sy)(\idfunc) &&\text{(cancelling)}\\
		&\sim (e \o (A \to \two\sy))(\idfunc)\\
		&\judgeq \rhobar(\idfunc) \judgeq \rho
	\end{align*}		
	where the last homotopy is given by $(A \to c) \o \two \sim \twist : (\pbool \to A \to X) \to (A \to X)$.
\end{proof}

\begin{lem}\label{lem:pentagon-c}
	The following diagram commutes, for $A$, $B$, $C$ and $X$ pointed types:
	\begin{center}
	\begin{tikzcd}[column sep=7em]
		(B \smsh C \to A \to X)
			\arrow[r, "e\sy"]
			\arrow[dd, swap, "\twist"]
		& (B \to C \to A \to X)
			\arrow[d, "B \to \twist"]
		\\
		& (B \to A \to C \to X)
			\arrow[d, "\twist"]
		\\
		(A \to B \smsh C \to X)
			\arrow[r, swap, "A \to e\sy"]
		& (A \to B \to C \to X)
	\end{tikzcd}
	\end{center}
\end{lem}
\begin{proof}
	Unfolding the definition of $e\sy$, we get the diagram:
	\begin{xcenter}
	\begin{tikzcd}[column sep=6em,every node/.style={font=\sffamily\small}]
		(B \smsh C \to A \to X)
			\arrow[rr, bend left=10, "e\sy"]
			\arrow[r, swap, "C\to -"]
			\arrow[dd, swap, "\twist"]
		& ((C \to B \smsh C) \to C \to A \to X)
			\arrow[r, swap, "\eta \to C \to A \to X"]
			\arrow[d, swap, "(C \to B \smsh C) \to \twist"]
		& (B \to C \to A \to X)
			\arrow[d, "B \to \twist"]
		\\
		& ((C \to B \smsh C) \to A \to C \to X)
			\arrow[r, swap, "\eta \to A \to C \to X"]
			\arrow[d, swap, "\twist"]
		& (B \to A \to C \to X)
			\arrow[d, "\twist"]
		\\
		(A \to B \smsh C \to X)
			\arrow[rr, swap, bend right=10, "A \to e\sy"]
			\arrow[r, "A \to (C \to -)"]
		& (A \to (C \to B \smsh C) \to C \to X)
			\arrow[r, "A \to (\eta \to C \to X)"]
		& (A \to B \to C \to X)
	\end{tikzcd}
	\end{xcenter}
	where the squares on the right are instances of naturality of $\twist$, while the commutativity of the pentagon on the left follows easily from the definition of $\twist$.
\end{proof}

\begin{thm}[Associativity-braiding hexagon]\label{thm:smash-associativity-braiding}
	For pointed types $A$, $B$ and $C$, there is a homotopy
	\[\alpha \o \gamma \o \alpha \sim (B \smsh \gamma) \o \alpha \o (\gamma \smsh C)\]
	corresponding to the commutativity of the following diagram:
	\begin{center}
	\begin{tikzcd}
		((A \smsh B) \smsh C)
			\arrow[r, "\alpha"]
			\arrow[d, swap, "\gamma \smsh C"]
		&(A \smsh (B \smsh C))
			\arrow[r, "\gamma"]
		& ((B \smsh C) \smsh A)
			\arrow[d, "\alpha"]
		\\
		((B \smsh A) \smsh C))
			\arrow[r, swap, "\alpha"]
		& (B \smsh (A \smsh C))
			\arrow[r, swap, "B \smsh \gamma"]
		& (B \smsh (C \smsh A))
	\end{tikzcd}
	\end{center}
\end{thm}
\begin{proof}
	The proof is structured similarly to the one for \autoref{thm:smash-associativity-pentagon}: the homotopies
	\begin{align*}
	B \smsh \gamma &\sim \gammabar^L(\idfunc) &\text{with\ \ } \gammabar^L &\defeq e \o (B \to \gammabar) \o e\sy\\
	\gamma \smsh C &\sim \gammabar^R(\idfunc) &\text{with\ \ } \gammabar^R &\defeq e \o \gammabar \o e\sy
	\end{align*}		
	can be proven in exactly the same way and, using these simplifications, we will show that both sides of the sought homotopy are homotopic to the same equivalence. Indeed we have:
	\begin{align*}
		\alpha \o \gamma \o \alpha
		&\judgeq \alphabar(\idfunc) \o \gammabar(\idfunc) \o \alphabar(\idfunc)\\
		&\sim (\alphabar \o \gammabar \o \alphabar)(\idfunc) \\ &\mbox{}\qquad\text{(naturality of $\gammabar$ and $\alphabar$)}\\
		&\judgeq (e \o e \o (A \to e\sy) \o e\sy \o e \o \twist \o e\sy \o e \o e \o (B \to e\sy) \o e\sy)(\idfunc)\\
		&\sim (e \o e \o (A \to e\sy) \o \twist \o e \o (B \to e\sy) \o e\sy)(\idfunc) \\ &\mbox{}\qquad\text{(cancelling)}\\
		&\sim (e \o e \o \twist \o (B \to \twist) \o e\sy \o e \o (B \to e\sy) \o e\sy)(\idfunc) \\ &\mbox{}\qquad\text{(\autoref{lem:pentagon-c})}\\
		&\sim (e \o e \o \twist \o (B \to \twist) \o (B \to e\sy) \o e\sy)(\idfunc) \\ &\mbox{}\qquad\text{(cancelling)}
	\end{align*}
	and
	\begin{align*}
		(B \smsh \gamma) \o \alpha \o (\gamma \smsh C)
		&\sim \gammabar^L(\idfunc) \o \alphabar \o \gammabar^R(\idfunc) \\ &\mbox{}\qquad\text{(simplification)}\\
		&\sim (\gammabar^R \o \alphabar \o \gammabar^L)(\idfunc) \\ &\mbox{}\qquad\text{(naturality of $\alphabar$ and $\gammabar^R$)}\\
		&\judgeq (e \o \gammabar \o e\sy \o e \o e \o (B \to e\sy) \o e\sy \o e \o (B \to \gammabar) \o e\sy)(\idfunc)\\
		&\sim (e \o \gammabar \o e \o (B \to e\sy) \o (B \to \gammabar) \o e\sy)(\idfunc) \\ &\mbox{}\qquad\text{(cancelling)}\\
		&\sim (e \o \gammabar \o e \o (B \to (e\sy \o \gammabar)) \o e\sy)(\idfunc) \\ &\mbox{}\qquad\text{(functoriality of $B \to -$)}\\
		&\judgeq (e \o e \o \twist \o e\sy \o e \o (B \to (e\sy \o e \o \twist \o e\sy)) \o e\sy)(\idfunc)\\
		&\sim (e \o e \o \twist \o (B \to \twist) \o (B \to e\sy) \o e\sy)(\idfunc) \\ &\mbox{}\qquad\text{(cancelling)}
	\end{align*}
	proving the commutativity of the diagram.
\end{proof}

\begin{thm}[Double braiding]\label{thm:smash-double-braiding}
	For $A$ and $B$ pointed types, there is a homotopy
	\[\gamma \o \gamma \sim \idfunc\]
	corresponding to the commutativity of the following diagram:
	\begin{center}
	\begin{tikzcd}
	(A \smsh B)
		\arrow[r, "\gamma"]
		\arrow[dr, equals]
	& (B \smsh A)
		\arrow[d, "\gamma"]
	\\
	& (A \smsh B)
	\end{tikzcd}
	\end{center}
\end{thm}
\begin{proof}
	Using that $\twist \o \twist \sim \idfunc$, we get:
	\begin{align*}
		\gamma \o \gamma
		&\judgeq \gammabar(\idfunc) \o \gammabar(\idfunc)\\
		&\sim (\gammabar \o \gammabar)(\idfunc) &&\text{(naturality of $\gammabar$)}\\
		&\judgeq (e \o \twist \o e\sy \o e \o \twist \o e\sy)(\idfunc)\\
		&\sim \idfunc &&\text{(cancelling)}
	\end{align*}
	as desired.
\end{proof}

Finally we get the result of this section.
\begin{thm}
	$\smsh$ is a 1-coherent symmetric monoidal product, assuming that $e$ is pointed natural in $C$.
\end{thm}
\begin{proof}
	This follows immediately from the theorems in this section.
\end{proof}

\chapter{The Serre Spectral Sequence}\label{cha:serre-spectr-sequ}

Spectral sequences are important tools in algebraic topology.\footnote{The work in this chapter is joint work with Jeremy Avigad, Steve Awodey, Ulrik Buchholtz, Egbert Rijke and Mike Shulman.} They give a
relationship between certain homotopy, homology and cohomology groups, in a way
that generalizes long exact sequences. This generalization comes at a cost of
being a lot more complicated than a long exact sequence.

In this chapter we will start the study of spectral sequences in homotopy type
theory. We will introduce the notion of spectral sequences, and then construct
the Atiyah-Hirzebruch and Serre spectral sequences for cohomology. We follow the
construction due to Michael Shulman given in~\cite{shulman2013spectral}. We will also give a sketch on
how to construct the analogues for homology, and look at some of the
applications of these spectral sequences. 

There are a couple of notable differences between spectral sequences in homotopy
type theory compared to classical homotopy theory. 
\begin{itemize}
  \item As always, in HoTT all constructions have to be homotopy invariant, so
  we cannot use classical constructions that are not homotopy invariant. For
  example, the construction of the Serre spectral sequence for homology in~\cite{hatcher2004spectral} 
  uses CW-approximation of a space and the skeleton
  of the obtained CW-complex to construct the spectral sequence. These
  operations are not homotopy invariant, and therefore cannot be performed in
  HoTT. 
  \item Another difference is that homology and cohomology are defined
  differently in HoTT than in classical homotopy theory. In classical homotopy
  theory (co)homology is defined using singular (co)homology. Since the intermediate 
  steps in the construction of singular (co)homology is not homotopy invariant, we use a different definition of
  (co)homology (see \autoref{def:cohomology}), which impacts the definition of spectral sequences involving (co)homology.
  \item 
  The first page of a spectral sequence is often not homotopy invariant, and
  therefore cannot be constructed in HoTT. For this reason, we start counting
  the pages of spectral sequences at 2.
  \item HoTT offers a convenient language for formalizing proofs. Therefore, we
  have formalized all constructed spectral sequences in this chapter.
\end{itemize}

The spectral sequences we construct are not the most general version of these
spectral sequences. The spectral sequences we construct are still more general
than the formulation of the Serre spectral sequence in many textbooks (we give a
version of generalized and parametrized cohomology), but there exist more
general versions. There are two places where we compromised on generality for
the sake of making the formalization easier. The first compromise is that we
only formalized exact couples for graded $R$-modules for a ring $R$ (which is
not graded). More generally we could do this for any abelian category, which
would require building up the theory of abelian categories (this is done in
UniMath~\cite{unimath}). Furthermore, we did not look at convergence of spectral
sequences in the most general sense, since that can get quite complicated and
subtle. Instead, we only look at spectral sequences that are eventually constant
pointwise, so the $\infty$-page is just the eventual value. This restriction
adds the condition to the spectral sequences we construct that the coefficients
are only in truncated spectra.

\section{Spectral Sequences}\label{sec:spectral-sequences}

A spectral sequence consists of a sequence of pages, each of them containing a two-dimensional grid of abelian groups. There are maps between these groups, called \emph{differentials}. These differentials form (co)chain complexes, and the (co)homology of these complexes determine the groups on the next page. In \ref{fig:spectral-sequence-pages} we show an example of two pages of a spectral sequence, where each dot represents an abelian group. In this figure only the two first quadrants are shown, because in simple applications all other groups are trivial, though that need not be the case in general.

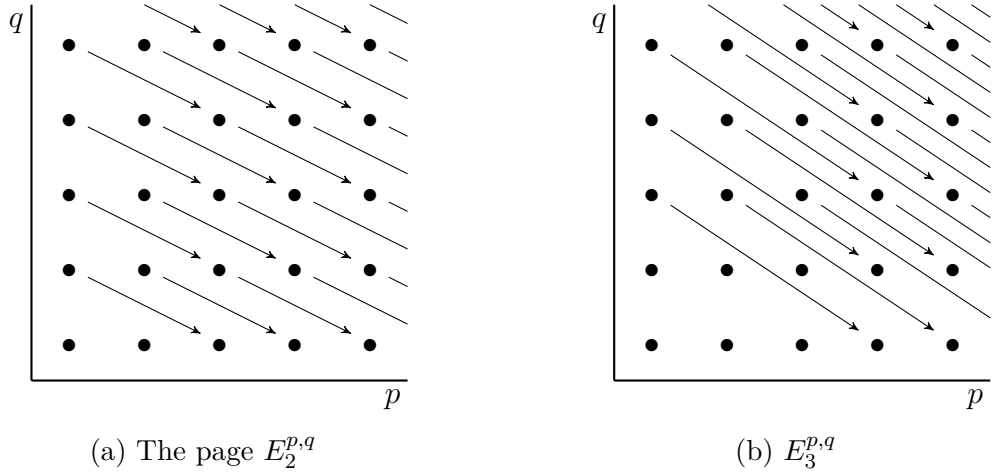
\begin{figure}[ht]
  \centering
  \begin{subfigure}{.5\textwidth}
    \centering
    \begin{tikzpicture}
      \draw [thick] (0.5,0.5) -- (0.5,5.5);
      \draw [thick] (0.5,0.5) -- (5.5,0.5);
      \node at (5.3,0.3) {$p$};
      \node at (0.3,5.3) {$q$};
      \node (x11) at (1,1) {$\bullet$};
      \node (x12) at (1,2) {$\bullet$};
      \node (x13) at (1,3) {$\bullet$};
      \node (x14) at (1,4) {$\bullet$};
      \node (x15) at (1,5) {$\bullet$};
      \node (x21) at (2,1) {$\bullet$};
      \node (x22) at (2,2) {$\bullet$};
      \node (x23) at (2,3) {$\bullet$};
      \node (x24) at (2,4) {$\bullet$};
      \node (x25) at (2,5) {$\bullet$};
      \node (x31) at (3,1) {$\bullet$};
      \node (x32) at (3,2) {$\bullet$};
      \node (x33) at (3,3) {$\bullet$};
      \node (x34) at (3,4) {$\bullet$};
      \node (x35) at (3,5) {$\bullet$};
      \node (x41) at (4,1) {$\bullet$};
      \node (x42) at (4,2) {$\bullet$};
      \node (x43) at (4,3) {$\bullet$};
      \node (x44) at (4,4) {$\bullet$};
      \node (x45) at (4,5) {$\bullet$};
      \node (x51) at (5,1) {$\bullet$};
      \node (x52) at (5,2) {$\bullet$};
      \node (x53) at (5,3) {$\bullet$};
      \node (x54) at (5,4) {$\bullet$};
      \node (x55) at (5,5) {$\bullet$}; 
      \path
      (x12) edge[->] (x31)
      (x13) edge[->] (x32)
      (x14) edge[->] (x33)
      (x15) edge[->] (x34)
      (x22) edge[->] (x41)
      (x23) edge[->] (x42)
      (x24) edge[->] (x43)
      (x25) edge[->] (x44)
      (x32) edge[->] (x51)
      (x33) edge[->] (x52)
      (x34) edge[->] (x53)
      (x35) edge[->] (x54)
      (x42) edge (5.5,1.25)
      (x43) edge (5.5,2.25)
      (x44) edge (5.5,3.25)
      (x45) edge (5.5,4.25)
      (x52) edge (5.5,1.75)
      (x53) edge (5.5,2.75)
      (x54) edge (5.5,3.75)
      (x55) edge (5.5,4.75)
      (2,5.5) edge[->] (x35)
      (3,5.5) edge[->] (x45)
      (4,5.5) edge[->] (x55)
      (5,5.5) edge (5.5,5.25)
      ;
      \end{tikzpicture}
    \caption{The page $E_2^{p,q}$}
    \label{fig:E2-page}
  \end{subfigure}%
  \begin{subfigure}{.5\textwidth}
    \centering
    \begin{tikzpicture}
      \draw [thick] (0.5,0.5) -- (0.5,5.5);
      \draw [thick] (0.5,0.5) -- (5.5,0.5);
      \node at (5.3,0.3) {$p$};
      \node at (0.3,5.3) {$q$};
      \node (x11) at (1,1) {$\bullet$};
      \node (x12) at (1,2) {$\bullet$};
      \node (x13) at (1,3) {$\bullet$};
      \node (x14) at (1,4) {$\bullet$};
      \node (x15) at (1,5) {$\bullet$};
      \node (x21) at (2,1) {$\bullet$};
      \node (x22) at (2,2) {$\bullet$};
      \node (x23) at (2,3) {$\bullet$};
      \node (x24) at (2,4) {$\bullet$};
      \node (x25) at (2,5) {$\bullet$};
      \node (x31) at (3,1) {$\bullet$};
      \node (x32) at (3,2) {$\bullet$};
      \node (x33) at (3,3) {$\bullet$};
      \node (x34) at (3,4) {$\bullet$};
      \node (x35) at (3,5) {$\bullet$};
      \node (x41) at (4,1) {$\bullet$};
      \node (x42) at (4,2) {$\bullet$};
      \node (x43) at (4,3) {$\bullet$};
      \node (x44) at (4,4) {$\bullet$};
      \node (x45) at (4,5) {$\bullet$};
      \node (x51) at (5,1) {$\bullet$};
      \node (x52) at (5,2) {$\bullet$};
      \node (x53) at (5,3) {$\bullet$};
      \node (x54) at (5,4) {$\bullet$};
      \node (x55) at (5,5) {$\bullet$}; 
      \path
      (x13) edge[->] (x41)
      (x14) edge[->] (x42)
      (x15) edge[->] (x43)
      (x23) edge[->] (x51)
      (x24) edge[->] (x52)
      (x25) edge[->] (x53)
      (x33) edge (5.5,1.3333)
      (x34) edge (5.5,2.3333)
      (x35) edge (5.5,3.3333)
      (x43) edge (5.5,2)
      (x44) edge (5.5,3)
      (x45) edge (5.5,4)
      (x53) edge (5.5,2.6667)
      (x54) edge (5.5,3.6667)
      (x55) edge (5.5,4.6667)
      (1.75,5.5) edge[->] (x44)
      (2.75,5.5) edge[->] (x54)
      (3.25,5.5) edge[->] (x45)
      (4.25,5.5) edge[->] (x55)
      (3.75,5.5) edge (5.5,4.3333)
      (4.75,5.5) edge (5.5,5)
      (5.25,5.5) edge (5.5,5.3333)
      ;
      \end{tikzpicture}
    \caption{$E_3^{p,q}$}
    \label{fig:E3-page}
  \end{subfigure}
  \caption{Two pages of a spectral sequence.}
  \label{fig:spectral-sequence-pages}
  \end{figure}

  Before we start, we define the notion of a graded abelian group. We will give a
  nonstandard definition that is equivalent to the standard one.
  \begin{defn}\label{def:graded}
    For an abelian group $G$, an \emph{$G$-graded abelian group} is
    a family of abelian groups indexed over $G$. If $M$ and $M'$ are $G$-graded
    abelian groups, the type of \emph{graded abelian group homomorphism from $M$ to
    $M'$} is a triple consisting of a \emph{degree} $e : G \simeq G$ (this is an equivalence of types, not a group isomorphism), a proof of
    $(g : G) \to e(g)=g+e(0)$ and a term of type 
    $$\{x\ y : I\} \to (p : e(x)=y) \to M_x \to M'_y.$$ 
    We will denote the type of homomorphisms as $M\to M'$.
    For $\phi:M\to M'$ we write $\deg_\phi$ for the first projection. We will
    often call $\deg_\phi(0)$ the degree of $\phi$. For $x : I$ we will write
    $$\phi_x\defeq\phi_{\refl_x}:M_x\to M'_{\deg_\phi(x)}$$ 
    and 
    $$\phi_{[x]}\defeq\phi_{p_x}:M_{\deg_\phi^{-1}(x)}\to M'_x$$
    where $p_x:\deg_\phi(\deg_\phi^{-1}(x))=x$ is the proof obtained from the equivalence $\deg_\phi$.
  \end{defn}
  \begin{rmk}
    This definition looks a bit cumbersome, since the condition on $e$ forces
    $e$ to be homotopic to the function $\lam{g}g+h$ for some group element $h$.
    Furthermore, the type of $\phi$ is equivalently $(x : I) \to M_x \to
    M'_{x+h}$. We will now discuss why we made these choices.
  
    To see why this is more convenient, we consider the composition of two graded homomorphisms. Suppose we
    have two graded homomorphisms $\phi:M\to M'$ and $\psi:M'\to M''$ of degrees $h:G$ and $k:G$,
    respectively. Then the pointwise composition $\lam{g:G}{m:M_g}\psi_{g+h}(\phi_g(m))$ has type
    $(g : G) \to M_g \to M''_{(g+h)+k}$. So to get a graded homomorphism of degree $h+k$, with the more straightforward representation, we would need to
    transport along the equality $(g+h)+k=g+(h+k)$. Since compositions are ubiquitous, this would happen all over the place. 
    However, in our setting, the composite of two graded homomorphisms of degree $e$ and
    $e'$ will have degree $e' \o e$, without using any transports.
  
    We eliminated a transport to define composition, but there are other places where we cannot get rid
    of them so easily. For example, given morphisms $\phi:M\to M'$ and $\psi:M'\to M''$ with
    $\psi\o\phi=0$ (the graded map that is constantly 0), we are interested in the
    homology of $\phi$ and $\psi$. This is the kernel of $\psi$ quotiented by the image of $\phi$ in
    $M'_x$. However, if $\phi_x$ has type $M_x \to M'_{\deg_\phi(x)}$, there is no map that (without
    transports) lands in $M'_x$. We would need to transport along the path
    $p_x:\deg_\phi\sy(\deg_\phi(x))=x$ and take the image of this composite:
    $$M_{\deg_\phi\sy(x)} \xrightarrow{\phi_{\deg_\phi\sy(x)}} M'_{\deg_\phi(\deg_\phi\sy(x))} \xrightarrow{\sim} M'_x.$$
    For this reason, we allow graded homomorphisms to be applied to paths, so that we have a ``built-in''
    transport. Then we can define the homology as $H_x\defeq \ker(\psi_x)/\im(\phi_{[x]})$, or diagramatically
    $$M_{\deg_\phi\sy(x)}\xrightarrow{\phi_{[x]}}M'_x\xrightarrow{\psi_x}M''_{\deg_\psi(x)}.$$

    For the construction of spectral sequences, we do not actually need the
    second component of a graded homomorphism: all constructions also work if
    the degrees are arbitrary equivalences of type $I\simeq I$, where $I$ is an
    arbitrary set. This is the definition used in the formalization. In this
    document we add this condition, so that our definition is equivalent to the
    usual definition of graded morphism.
  \end{rmk}
  
\begin{defn}
  A \emph{spectral sequence} consists of the following data.
  \begin{itemize}
  \item A sequence $E_r$ of abelian groups graded over $\Z\times\Z$ for $r\geq2$. $E_r$ is called the \emph{$r$-page} of the spectral sequence;
  \item \emph{differentials}, which are graded morphisms $d_r:E_r\to E_r$ such that $d_r\circ d_r=0$;
  \item isomorphisms $\alpha_r^{p,q}:H^{p,q}(E_r)\simeq E_{r+1}^{p,q}$ where $H^{p,q}(E_r)=\ker(d_r^{p,q})/\im(d_r^{[p,q]})$ is the cohomology of the cochain complex determined by $d_r$.
  \end{itemize}
\end{defn}
  We use the notation for cohomologically indexed spectral sequences, since we
  will construct spectral sequences in cohomology in this chapter. For the
  spectral sequences in this chapter, the degree of $d_r$ will be $(r,1-r)$, which
  signifies a cohomologically indexed spectral sequence. 

  As mentioned before, we start counting the pages at 2, since the first page of
  the spectral sequences we construct will not be homotopy invariant. In the
  formalization we start counting at 0 for convenience. Also, in the
  formalization, we assume that the grading of $E_r$ is over some set $I$
  instead of fixing it to $\Z\times \Z$. It is not clear whether this extra
  generality is useful. Instead of abelian groups, we could take objects of an
  arbitrary abelian category, but for concreteness and to simplify things, we
  choose to develop the theory only for abelian groups. In the formalization we
  developed the theory for graded $R$-module for a (non-graded) ring $R$, but we
  have only applied it to abelian groups so far.

  Note that $(E_r,d_r)$ determines $E_{r+1}$ but \emph{not} $d_{r+1}$.
  Furthermore, $E_r$ is a subquotient (subgroup of a quotient) of $E_2$, so if
  $E_2^{p,q}$ is trivial, then $E_r^{p,q}$ is trivial for all $r$. 
  
  In many cases, the spectral sequence will \emph{converge}. That means that for
  a fixed $(p,q):\Z\times \Z$ the sequence $E_r^{p,q}$ will be constant for $r$
  large enough. For example, suppose that the degree of $d_r$ is $(r,r-1)$, and
  $E_2$ is limited to the first quadrant. Now for any $(p,q)$ all differentials
  in or out of $E_r^{p,q}$ will go out the first quadrant for sufficiently large
  $r$. This means that the image of $d_r^{[p,q]}$ is trivial, and the kernel of
  $d_r^{p,q}$ is the full group. This implies that $E_{r+1}^{p,q}\simeq
  E_r^{p,q}$, so the spectral sequence converges.

  Whenever a spectral sequence converges, we write $E_\infty^{p,q}$ for the
  eventual value of $E_r^{p,q}$ for $r$ large enough. Now the power of spectral
  sequences is that there is often a relation between $E_2^{p,q}$ and
  $E_\infty^{p,q}$. This relation does not specify $E_\infty^{p,q}$ exactly, but
  specifies that $E_\infty^{p,q}$ build up some group $D^n$ for the diagonals
  where $p+q=n$.

  \begin{defn}
    Suppose given an abelian group $D$ and a finite sequence of abelian groups
    $(E^n)_n$. We say that \emph{$D$ is built from $(E^n)_n$} if there is a
    sequence of abelian groups $(D^n)_n$ and short exact sequences 
    \begin{align*}
      E^0\to &D \to D^1\\
      &\vdots\\
      E^k\to &D^k \to D^{k+1}\\
      E^{k+1}\to &D^{k+1} \to D^{k+2}\\
      &\vdots\\
      E^m\to &D^m \to 0\\
    \end{align*}
    The sequence $(D^n)_n$ is called a \emph{cofiltration} of $D$, they are successive quotients of $D$.
  \end{defn}
  \begin{defn}
    Given a graded abelian group $D^n$ and a bigraded abelian group $C^{p,q}$, we write 
    $$E_2^{p,q}=C^{p,q}\Rightarrow D^{p+q}$$
    if there is a spectral sequence $E$ such that
    \begin{itemize}
      \item $E_2^{p,q}=C^{p,q}$;
      \item $E$ converges to $E_\infty$;
      \item $D^n$ is built from $E_\infty^{p,q}$ where $p+q=n$.
    \end{itemize}
  \end{defn}
  \begin{rmk}
    This definition implicitly requires that for $p+q=n$ only finitely many $E_\infty^{p,q}$ are nontrivial. 
    This is sufficient for the spectral sequences we consider in this chapter, but this condition can be relaxed in more general constructions of spectral sequences.
  \end{rmk}


\section{Exact Couples}\label{sec:exact-couples}

As we said before, the pair $(E_r,d_r)$ in a spectral sequence specifies
$E_{r+1}$, but not $d_{r+1}$. If we have some more information about page $r$,
then we can construct page $r+1$ and the extra information for page $r+1$. Now
we can iterate this construction and obtain a spectral sequence by forgetting
about the extra information.

An \emph{exact couple} exactly gives this extra information~\cite{massey1952exactcouple}. From it, we can compute the \emph{derived exact
couple}, which gives us the information next page of the spectral sequence.
\begin{defn}
An \emph{exact couple} is a pair $(D,E)$ of $\Z\times\Z$-graded abelian groups with graded homomorphisms
\begin{center}
  \begin{tikzpicture}[thick,node distance=1cm]
  \node (tl) at (0,0) {$D$};
  \node[below right = of tl] (b) {$E$};
  \node[above right = of b] (tr) {$D$};
  \path[every node/.style={font=\sffamily\small}]
  (tl) edge[->] node [above] {$i$} (tr)
  (tr) edge[->] node [below right] {$j$} (b)
  (b)  edge[->] node [below left] {$k$} (tl);
  \end{tikzpicture}
\end{center}
that is exact in all three vertices. This means that for all $p:\deg_j(x)=_Iy$ and
$q:\deg_k(y)=z$ that $\ker(k_q)=\im(j_p)$, and similarly for the other two pairs of maps.

For an exact couple we will write $\iota\defeq\deg_i$ and $\eta\defeq\deg_j$ and $\kappa\defeq\deg_k$ for the degrees.
\end{defn}
\begin{lem}
Given an exact couple $(D,E,i,j,k)$, we can define a \emph{derived exact couple} $(D',E',i',j',k')$
where $E'$ is the homology of $d\defeq j\o k:E\to E$. The degrees of the derived maps are
$\deg_{i'}\equiv\iota$, $\deg_{k'}\equiv\kappa$ and $\deg_{j'}\equiv\eta\o\iota\sy$.
\begin{center}
  \begin{tikzpicture}[thick,node distance=1cm]
  \node (tl) at (0,0) {$D'$};
  \node[below right = of tl] (b) {$E'$};
  \node[above right = of b] (tr) {$D'$};
  \path[every node/.style={font=\sffamily\small}]
  (tl) edge[->] node [above] {$i'$} (tr)
  (tr) edge[->] node [below right] {$j'$} (b)
  (b)  edge[->] node [below left] {$k'$} (tl);
  \end{tikzpicture}
\end{center}
\end{lem}
\begin{proof}
  In this proof we will be explicit about the grading of $D$ and $E$, which is a lot trickier (at least in intensional type theory) than a proof without the grading. For a proof that does not take the grading into account, see for example~\cite[Lemma 1.1]{hatcher2004spectral}.
  We define for $x:\Z\times \Z$ the graded abelian groups $D'$ and $E'$ by $D_x'=\im i_{[x]}$ and $E_x'=\ker d_x/\im d_{[x]}$. Now $i_x':D_x'\to D_{\iota x}'$ is defined as the composite
  $$D_x' \hookrightarrow D_x \xrightarrow{i} D_{\iota x}'.$$
  This is sufficient to define $i'$ on all paths $\iota x=y$ as a function $D_x \to D_y$.\\
  We first define $j_{px}':D_{\iota x}'\to E_{\eta x}'$ for the canonical path $px:\eta(\iota\sy(\iota x))=\eta x$, which is sufficient to define $j'$ in general. 
  Note that $D_{\iota x}'\equiv \im i_{[\iota x]}\simeq\im i_x$, so to define $j_{px}'$ it is sufficient to define $\tilde\jmath : D_x \to E_{\eta x}'$ such that $(a : D_x) \to i_x(a)=0 \to \tilde\jmath(a)=0$. We define $\tilde\jmath(a)\defeq [j_xa]$. This is well-defined, since $j_xa\in\ker d_{\eta x}$,\footnote{We use the set-theoretical notation $g\in H$ to say that a group element $g:G$ is in subgroup $H$. Formally, a subgroup $H$ is an element of $G\to\prop$ (containing 0, and closed under addition and negation) and $g\in H$ is defined as $H(g)$. Note that $(g:G)\times H(g)$ can be endowed with a group structure, which is $H$ viewed as a group.} because
  $$d_{\eta x}(j_xa)=j_{\kappa(\eta x)}(k_{\eta x}(j_xa))=j_{\kappa(\eta x)}\const=\const.$$ 
  Now suppose that $i_x(a)=0$. Without loss of generality we may assume that $x\equiv \kappa y$. By exactness, this means that $a\in\im k_y$, so there is $b:E_y$ such that $k_y(b)=a$. 
  Now $j_xa=j_x(k_yb)\equiv d_yb$, so 
  $$j_xa\in\im d_y\simeq \im d_{[\deg_d y]}\equiv \im d_{[\eta x]}.$$
  This shows that $\tilde\jmath(a)=0$, completing the definition of $j'$. Note that 
  $j_{px}'(i_xa)=[j_xa]$.
  To define $k_x':E_x'\to D_{\kappa x}'$, first note that if $a\in\ker d_x$, then $k_xa\in\ker j_{\kappa x}=\im i_{[\kappa x]}$ by exactness. Now we need to show that if $a\in\im d_{[x]}$, then $k_xa=0$. By assumption, we have $b:E_{\eta\sy\kappa\sy x}$ such that $d_{[x]}(b)=a$. Now we compute (using $k_xj_{[x]}=\const$)
  $$k_xa=k_x(d_{[x]}b)=k_x(j_{[x]}(k_{[\eta\sy x]}b))=0.$$
  This defines $k$. 
  
  Showing exactness of the derived couple involves some diagram chasing. 
  To show that $j'i'=\const$ it is sufficient to show that for all $a:D_{\iota x}'$ we have
  $j'_{p(\iota x)}(i'_{\iota x}a)=0$. Since $a\in\im i_{[\iota x]}\simeq i_x$ we know that $a=i_x b$ for some $b: D_x$. We compute
  $$j'_{p(\iota x)}(i'_{\iota x}a)=j'_{p(\iota x)}(i_{\iota x}a)=[j_{\iota x}a]=[j_{\iota x}(i_x b)]=[0]=0.$$
  To show that $\ker j'\subseteq \im i'$, it is sufficient to show that 
  $\ker j'_{p(\kappa x)}\subseteq \im i'_{\kappa x}$. Suppose $a : D'_{\iota(\kappa x)}$ such that $j'_{p(\kappa x)}(a)=0$, we know that $a=i_{\kappa x}(b)$ for some $b$. Now
  $$0=j'_{p(\kappa x)}(a)=j'_{p(\kappa x)}(i_{\kappa x}(b))=[j_{\kappa x}(b)],$$
  which means that $j_{\kappa x}(b)\in\im d_{[{\kappa (\eta x)}]}\simeq \im d_x$. This means that for some $c: E_x$ we have $j_{\kappa x}(b)=d_x(c)=j_{\kappa x}(k_x c)$. This means that 
  $j_{\kappa x}(b-k_x c)=0$, hence $b-k_xc\in \ker j_{\kappa x}=\im i_{[\kappa x]}$. This means that we can define $b-k_xc: D'_{\kappa x}$. Now we compute
  $$i'_{\kappa x}(b-k_xc)=i_{\kappa x}b-i_{\kappa x}(k_xc)=a-0=a,$$
  which means $a\in\im i'_{\kappa x}$, as desired.

  We will omit the other cases, which are similar but easier.
\end{proof}
Repeating the process of deriving exact couples, we get a sequence of exact couples $(D_r,E_r,i_r,j_r,k_r)$.\footnote{We will now put the grading of $D$, $E$ and the maps as superscript, so that we can put the page as subscript.} We get a spectral sequence $(E_r,d_r)$ where $d_r\defeq j_r\o k_r$. Note that
$$\deg_{d_r}=\deg_{j_r}\o\deg_{k_r}=\eta\o\iota^r\o\kappa
$$

Given some extra conditions on the exact couple, we can show that this spectral sequence converges. 
\begin{defn}\label{def:bounded}
We call an exact couple \emph{bounded} if for every $x:\Z\times\Z$ there is are bounds $B_x : \N$ such that for all $s \geq B_x$ we have
$$E^{\iota^{-s}(x)}=0 \qquad\text{and}\qquad D^{\iota^s(x)}=0$$
\end{defn}
\begin{rmk}\label{rmk:bounded}
  The condition on $D$ also shows that if you go sufficiently far in the $\iota$-direction, then $E$ is trivial, since $D \xrightarrow{j} E \xrightarrow{k} D$ is exact and the occurrences of $D$ will be trivial. Converse, the condition on $E$ shows that if you go sufficiently far in the
  $\iota^{-1}$ direction, $i$ will be an equivalence, by the following exact sequence.
  $$E \xrightarrow{k} D \xrightarrow{i} D \xrightarrow{j} E$$
  We call $x:\Z\times\Z$ a \emph{stable index} whenever $i^{[\iota^{-s}x]}$ is surjective for \emph{all} $s\geq0$.
\end{rmk}
Given a bounded exact couple, the pages stabilize pointwise, which is the content of the next lemma.

\begin{lem}\label{lem:exact-couple-stabilize}
  For a bounded exact couple $(D,E,i,j,k)$ we have for all sufficiently large $r$ that $D_{r+1}^x=D_r^x$ and $E_{r+1}^x=E_r^x$.
\end{lem}
\begin{proof}
  Note that $E_{r+1}^x=\ker d_r^{x}/\im d_r^{[x]}$. Since $d_r$ has degree
  $\eta\o\iota^r\o\kappa$, and because $\Z\times \Z$ is an abelian group, the
  degrees commute.\footnote{In the formalization, we do not assume that the
  degrees are shifts by a group element, and we explicitly assume that
  $\kappa\iota=\iota\kappa$ and $\iota\eta=\eta\iota$.} The codomain of $d_r^x$
  is $E_r^{\eta(\iota^r(\kappa x))}=E_r^{\iota^r(\eta(\kappa x))}$, which is
  trivial for sufficiently large $r$ by \autoref{rmk:bounded}. Also, the domain
  of $d_r^{[x]}$ is $E_r^{\kappa\sy(\iota^{-r}(\eta\sy
  x))}=E_r^{\iota^{-r}(\kappa\sy(\eta\sy x))}$, which is trivial for sufficiently large $r$ by the definition of boundedness.

  To show that $D$ stabilizes, first note that if $i_r^{[\iota\sy x]}$ is surjective, then $i_{r+1}^{[x]}$ is surjective. The reason is that 
  $$i_{r+1}: D_{r+1}^{\iota\sy x}\xrightarrow{\sim} D_r^{\iota\sy x} \to D_{r+1}^{x}$$ 
  is now a composite of two surjective maps. This means that if the maps $i_{r_0}^{[\iota^{-s}x]}$ are surjections for all $s\geq B+1$, then the maps $i_{r_0+1}^{[\iota^{-s}x]}$ will be surjections for all $s\geq B$. In this case, for $r\geq r_0+B$ we have that $i_r^{[x]}$ is a surjection, hence that $D_{r+1}^x=D_r^x$. Since $i_0^{[\iota^{-s}x]}$ are surjections for sufficiently large $s$ by \autoref{rmk:bounded}, we finish the proof.
\end{proof}
 
By the proof of \autoref{lem:exact-couple-stabilize} we get explicit bounds $B_x^D$ and $B_x^E$ such that $D_r^x=D_{B_x^D}^x$ and $E_{r'}^x=E_{B_x^E}^x$ for all $r\geq B_x^D$ and $r'\geq B_x^E$. We define $D_\infty^x\defeq D_{B_x^D}^x$ and $E_\infty^x\defeq E_{B_x^E}^x$. Both $B_x^D$ and $B_x^E$ will be the maximum of $B_y$ for some sequence of indices $y$.

\begin{thm}[Convergence Theorem]\label{thm:exact-couple-convergence}
Let $(D,E,i,j,k)$ be a bounded exact couple and let $x$ be a stable index. 
Then $D^{\kappa x}$ is built from $(E_\infty^{\iota^n(x)})_{0\le n < B_{\kappa x}}$. 
\end{thm}
\begin{proof}
  Define 
  $C^n = D_\infty^{\kappa(\iota^nx)}$. Let $n:\N$ be arbitrary, then for sufficiently large $r$ the following is a short exact sequence
  $$0 \xrightarrow{j_r} E_r^{\iota^nx} \xrightarrow{k_r} D_r^{\kappa(\iota^nx)} \xrightarrow{i_r} 
  D_r^{\iota(\kappa(\iota^nx))} \xrightarrow{j_r} 0.$$
  This is the case, because for sufficiently large $r$ the domain of $j_r^{[\iota^nx]}$ and the codomain of $j_r{\iota(\kappa(\iota^nx))}$ are contractible. Now (possibly by increasing $r$) these groups are in the stable range, so we get a short exact sequence
  $$0 \to E_\infty^{\iota^nx} \to C^n \to C^{n+1} \to 0.$$ Moreover, we have
  $C^0\equiv D_\infty^{\kappa x}=D^{\kappa x}$ because $x$ is a stable index.
  Lastly, for $s\geq B_{\kappa x}$ we know that $C^s$ is trivial, because
  $D^{\kappa(\iota^n)}$ is trivial by the condition of being bounded. This shows
  that $D^{\kappa x}$ is built from $(E_\infty^{\iota^n(x)})_{0\le n < B_{\kappa
  x}}$.
\end{proof}

\section{Spectra}\label{sec:spectra}

We have not yet discussed how to get an exact couple in the first place. Recall that from a pointed
map we get a long exact sequence of homotopy groups. For a \emph{sequence} of pointed maps we get a sequence of long exact sequences. However, we do not want to do this for pointed maps, but for maps between \emph{spectra}. 

You can think of a spectrum as a generalized space with negative dimensions. Suppose we are given a pointed type $X$ and a chosen delooping $Y$ of $X$. That is, $Y$ is a pointed type such that $\Omega Y\simeq^* X$. Now the $(n+1)$-th homotopy group of $Y$ is equal to the $n$-th homotopy group of $X$. The $0$-th homotopy group of $Y$ is new information, and we can think of it as the $(-1)$-th homotopy group of $X$. Spectra go further on this idea: it is a pointed type with infinitely many deloopings.

\begin{defn}
  A \emph{prespectrum} is a pair consisting of a sequence of pointed types $Y:\Z\to\type^*$ and a sequence of pointed maps $e:(n:\Z) \to Y_n \to^* \Omega Y_{n+1}$. An \emph{$\Omega$-spectrum} or \emph{spectrum} is a prespectrum $(Y,e)$ where $e_n$ is a pointed equivalence for all $n$. We will often just write $Y$ for the pair $(Y,e)$, and we denote the type of (pre)spectra by $\prespectrum$ and $\spectrum$.\\
  A map between (pre)spectra $(Y,e)\to (Y',e')$ is a pair consisting of $f:(n : \Z) \to Y_n \to Y_n'$ and $p : (n : \N) \to e_n' \o f_n \sim^* \Omega f_{n+1} \o e_n$.
\end{defn}

\begin{rmk}
  Usually a (pre)spectrum is indexed over $\N$ and not over $\Z$. We index it over $\Z$ so that we do not have to do a case split in --- for example --- the definition of homotopy group of a spectrum, see \autoref{def:spectrum-homotopy-group}.
\end{rmk}

\begin{ex}\mbox{}
  \begin{itemize}
  \item If $A$ is an abelian group, we have $HA:\spectrum$ where $(HA)_n=K(A,n)$ for $n\geq0$ and $(HA)_n=\unit$ for $n<0$. 
  \item Given $Y:\spectrum$ and $k:\Z$, we can define two new spectra $\Omega^kY$ and $\susp^k Y:\spectrum$ with 
  \begin{align*}
    (\Omega^kY)_n&\defeq Y_{n-k}&(\susp^kY)_n&\defeq Y_{n+k}
  \end{align*}
  \item Given a spectrum map $f: X \to Y$, we have a spectrum $\fib_f:\spectrum$ with $(\fib_f)_n\defeq \fib_{f_n}$. Furthermore we have a spectrum map $p_1:\fib_f\to X$. This follows from the following two facts about fibers (which we will not prove here).
  \begin{enumerate}
    \item Given a pointed map $g : A \to B$, there is a pointed equivalence $e_1:\Omega\fib_g\simeq^*$ with a pointed homotopy
    \begin{center}
      \begin{tikzpicture}[thick,node distance=1cm]
      \node (tl) at (0,0) {$\Omega\fib_g$};
      \node[below right = of tl] (r) {$A$};
      \node[below left = of r] (bl) {$\fib_{\Omega g}$};
      \path[every node/.style={font=\sffamily\small}]
      (tl) edge[->] node [above right] {$\Omega p_1$} (r)
      (bl) edge[->] node [below right] {$p_1$} (r)
      (tl)  edge[->] node [left] {$e_1$} (bl);
      \end{tikzpicture}
    \end{center}
    \item $\fib$ is a functor from pointed maps to pointed types and $p_1$ is a natural transformation. This means the following. Suppose we are given a square of pointed maps and a homotopy filling the following square.
    \begin{center}
      \begin{tikzpicture}[thick,node distance=1cm]
      \node (tl) at (0,0) {$A$};
      \node[right = of tl] (tr) {$A'$};
      \node[below = of tl] (bl) {$B$};
      \node (br) at (tr |- bl)  {$B'$};
      \path[every node/.style={font=\sffamily\small}]
      (tl) edge[->] node [above] {$f$} (tr)
           edge[->] node [right] {$h$} (bl)
      (bl) edge[->] node [above] {$g$} (br)
      (tr) edge[->] node [right] {$h'$} (br);
      \end{tikzpicture}
    \end{center}
    Then there is a pointed map $e_2:\fib_f\to\fib_g$, functorial in $(h,h')$. In particular this means that if $h$ and $h'$ are equivalences, then $e_2$ is. The naturality of $p_1$ means that we have the following pointed homotopy.
    \begin{center}
      \begin{tikzpicture}[thick,node distance=1cm]
      \node (tl) at (0,0) {$\fib_f$};
      \node[right = of tl] (tr) {$A$};
      \node[below = of tl] (bl) {$\fib_g$};
      \node (br) at (tr |- bl)  {$B$};
      \path[every node/.style={font=\sffamily\small}]
      (tl) edge[->] node [above] {$p_1$} (tr)
           edge[->] node [right] {$e_2$} (bl)
      (bl) edge[->] node [above] {$p_1$} (br)
      (tr) edge[->] node [right] {$h$} (br);
      \end{tikzpicture}
    \end{center}
  \end{enumerate}
  \end{itemize}
\end{ex}

Given an $\Omega$-spectrum $Y$ and $n : \Z$, we define can the $n$-th homotopy group of $Y$ to be 
$$\pi_n(Y)\defeqp\pi_{n+k}(Y_k):\abgroup$$ 
for any $k$ such that $n+k\geq 0$. This is independent of $k$, because 
$$\pi_{n+(k+1)}(Y_{k+1})\simeq \pi_{n+k}(\Omega Y_{k+1})\simeq \pi_{n+k}(Y_k).$$
For concreteness, in the following definition we pick $k=2-n$. We make this choice so that $\pi_n(Y)$ directly carries the structure of an abelian group.

The homotopy group of a prespectrum $Y$ is a bit different, since $\pi_{n+k}(Y_k)$ is not independent of $k$. In this case, it is the colimit as $k\to\infty$. We make the substitution 
$\ell=n+k-2$ to make the index of the homotopy group always positive.

\begin{defn}\label{def:spectrum-homotopy-group}
  Given an $\Omega$-spectrum $Y$ and $n : \Z$, we define the \emph{$n$-th homotopy group of $Y$} as
  $$\pi_n(Y)\defeq \pi_2(Y_{2-n}).$$
  For a prespectrum $Y$ we define
  $$\pi_n(Y)\defeq\colim_{\ell\to\infty}(\pi_{\ell+2}(Y_{\ell+2-n})).$$
\end{defn}

Note that the homotopy group of a prespectrum is a set by \autoref{cor:trunc_colim}\ref{part:colim_is_trunc}, 
and the colimit can be equipped with a group structure, making $\pi_n(Y)$ an abelian group for a prespectrum $Y$.

The long exact sequence of homotopy groups for pointed types, constructed in \autoref{sec:les-homotopy}, induces one on spectra.

\begin{thm}\label{thm:spectrum-LES}
  Given a spectrum map $f:X\to Y$ with fiber $F$, we get the following long exact sequence of homotopy groups indexed over $\Z\times \fin_3$. 
  \begin{center}\begin{tikzpicture}[thick, node distance=18mm]
  \node (Y)     at (0,0) {$\pi_k(Y)$};
  \node[left = of Y] (X) {$\pi_k(X)$};
  \node[left = of X] (F) {$\pi_k(F)$};
  \node[above = of Y] (OY) {$\pi_{k+1}(Y)$};
  \node[above = of X] (OX) {$\pi_{k+1}(X)$};
  \node[above = of F] (OF) {$\pi_{k+1}(F)$};
  \node[above = of OY] (O2Y) {$\pi_{k+2}(Y)$};
  \node[above = of OX] (O2X) {$\pi_{k+2}(X)$};
  \node[above = of OF] (O2F) {$\pi_{k+2}(F)$};
  \node[above = 5mm of O2X] {$\vdots$};
  \node[below = 5mm of X] {$\vdots$};
  \path[every node/.style={font=\sffamily\small}]
  (X) edge[->] node [above] (f){$\pi_k(f)$} (Y)
  (F) edge[->] node [below] (f){$\pi_k(p_1)$} (X)
  (OY) edge[->] (F)
  (OX) edge[->] node [above] (f){$\pi_{k+1}(f)$} (OY)
  (OF) edge[->] node [below] (f){$\pi_{k+1}(p_1)$} (OX)
  (O2Y) edge[->] (OF)
  (O2X) edge[->] node [above] (f){$\pi_{k+2}(f)$} (O2Y)
  (O2F) edge[->] node [below] (f){$\pi_{k+2}(p_1)$} (O2X);
\end{tikzpicture}
\end{center}
\end{thm}
We will use the following lemma. Recall the definition of successor structure from \autoref{def:chain-complex}.
\begin{lem}\label{lem:splice}
  Suppose given two successor structures $N$ and $M$, and for each $n:N$ let $G^n$ be a long exact sequence index by $M$. Let $m:M$ and $k\geq 2$. Suppose that 
  \begin{itemize}
    \item for all $n:N$, $G_m^{n+1} \simeq G_{m+k}^n$ and $G_{m+1}^{n+1} \simeq G_{m+k+1}^n$
    \item for all $n:N$ the following diagram commutes.
    \begin{center}\begin{tikzpicture}[thick, node distance=10mm]
      \node (Y)     at (0,0) {$G_{m+k}^n$};
      \node[left = of Y] (X) {$G_{m+k+1}^n$};
      \node[above = 12mm of Y] (OY) {$G_m^{n+1}$};
      \node[above = 12mm of X] (OX) {$G_{m+1}^{n+1}$};
      \path[every node/.style={font=\sffamily\small}]
      (X) edge[->] (Y)
      (OY) edge[->] node [sloped, above] {$\sim$} (Y)
      (OX) edge[->] node [sloped, above] {$\sim$} (X)
      (OX) edge[->] (OY);
    \end{tikzpicture}\end{center}
  \end{itemize}
  Then there is a long exact sequence $H:N\times\fin_{k-1}\to\set^*$ with $H_{(n,\ell)}\defeq G_{m+\ell}^n$ 
\end{lem}
For $k=3$ the hypotheses can be represented in the diagram below.
\begin{center}\begin{tikzpicture}[thick, node distance=10mm]
\node (Y)     at (0,0) {$G_m^n$};
\node[left = of Y] (X) {$G_{m+1}^n$};
\node[left = of X] (F) {$G_{m+2}^n$};
\node[left = of F] (G) {$G_{m+3}^n$};
\node[left = of G] (H) {$G_{m+4}^n$};
\node[left = of H] (I)  {$\cdots$};
\node[above = 12mm of Y] (OY) {$G_m^{n+1}$};
\node[above = 12mm of X] (OX) {$G_{m+1}^{n+1}$};
\node[above = 12mm of F] (OF) {$G_{m+2}^{n+1}$};
\node[above = 12mm of G] (OG) {$G_{m+3}^{n+1}$};
\node[above = 12mm of H] (OH) {$G_{m+4}^{n+1}$};
\node (OI) at (I |- OH) {$\cdots$};
\node[above = 12mm of OY] (O2Y) {$G_{m}^{n+2}$};
\node[above = 12mm of OX] (O2X) {$G_{m+1}^{n+2}$};
\node[above = 12mm of OF] (O2F) {$G_{m+2}^{n+2}$};
\node[above = 12mm of OG] (O2G) {$G_{m+3}^{n+2}$};
\node[above = 12mm of OH] (O2H) {$G_{m+4}^{n+2}$};
\node (O2I) at (OI |- O2H) {$\cdots$};
\node[above = 5mm of O2F] {$\vdots$};
\node[below = 5mm of F] {$\vdots$};
\path[every node/.style={font=\sffamily\small}]
(X) edge[->] (Y)
(F) edge[->] (X)
(G) edge[->] (F)
(H) edge[->] (G)
(I) edge[->] (H)
(OY) edge[->] node [sloped, below, pos = 0.4] {$\sim$} (G)
(OX) edge[->] node [sloped, above, pos = 0.6] {$\sim$} (H)
(OX) edge[->] (OY)
(OF) edge[->] (OX)
(OG) edge[->] (OF)
(OH) edge[->] (OG)
(OI) edge[->] (OH)
(O2Y) edge[->] node [sloped, below, pos = 0.4] {$\sim$} (OG)
(O2X) edge[->] node [sloped, above, pos = 0.6] {$\sim$} (OH)
(O2X) edge[->] (O2Y)
(O2F) edge[->] (O2X)
(O2G) edge[->] (O2F)
(O2H) edge[->] (O2G)
(O2I) edge[->] (O2H);
\end{tikzpicture}\end{center}
\begin{proof}[Proof (\autoref{lem:splice})]
  The map $H_{(n,\ell+1)}\to H_{(n,\ell)}$ is defined to be the given map $G_{m+\ell+1}^n\to G_{m+\ell}^n$. The map $H_{(n+1,0)}\to H_{(n,k-1)}$ is defined to be the composite 
  $$G_m^{n+1}\xrightarrow{\sim}G_{m+k}^n\to G_{m+k-1}^n.$$
  It is easy to check that this is a long exact sequence from the conditions.
\end{proof}
\begin{proof}[Proof (\autoref{thm:spectrum-LES})]
For each $n:\Z$ we get a long exact sequence of homotopy groups for $f_{2-n}$ by
\autoref{thm:les-homotopy}. We splice them together using \autoref{lem:splice} with $N=(\Z,\lam{n}n+1)$ and $M=(\N,\lam{n}n+1)$ and with $k=3$ and $m=(2,0)$. This means that the resulting sequence is 
\begin{center}\begin{tikzpicture}[thick, node distance=18mm]
  \node (Y)     at (0,0) {$\pi_2(Y_{2-n})$};
  \node[left = of Y] (X) {$\pi_2(X_{2-n})$};
  \node[left = of X] (F) {$\pi_2(F_{2-n})$};
  \node[above = of Y] (OY) {$\pi_2(Y_{2-(n+1)})$};
  \node[above = of X] (OX) {$\pi_2(X_{2-(n+1)})$};
  \node[above = of F] (OF) {$\pi_2(F_{2-(n+1)})$};
  \node[above = of OY] (O2Y) {$\pi_2(Y_{2-(n+2)})$};
  \node[above = of OX] (O2X) {$\pi_2(X_{2-(n+2)})$};
  \node[above = of OF] (O2F) {$\pi_2(F_{2-(n+2)})$};
  \node[above = 5mm of O2X] {$\vdots$};
  \node[below = 5mm of X] {$\vdots$};
  \path[every node/.style={font=\sffamily\small}]
  (X) edge[->] node [above] (f){$\pi_2(f)$} (Y)
  (F) edge[->] node [below] (f){$\pi_2(p_1)$} (X)
  (OY) edge[->] (F)
  (OX) edge[->] node [above] (f){$\pi_2(f)$} (OY)
  (OF) edge[->] node [below] (f){$\pi_2(p_1)$} (OX)
  (O2Y) edge[->] node [left] (f){$\pi_2(\delta)\quad\mbox{}$} (OF)
  (O2X) edge[->] node [above] (f){$\pi_2(f)$} (O2Y)
  (O2F) edge[->] node [below] (f){$\pi_2(p_1)$} (O2X);
\end{tikzpicture}
\end{center}
We still need to check the conditions for the Lemma. The first isomorphism is given by the following composition
$$\pi_2(Y_{2-(n+1)})\simeq\pi_2(\Omega Y_{2-(n+1)+1})\simeq\pi_2(\Omega Y_{2-n})\equiv \pi_3(Y_{2-n}),$$
The second isomorphism is the same, replacing $Y$ by $X$. The square commutes because the two isomorphisms are both natural in $Y$.
\end{proof}

Suppose given a sequence of spectra $A$ and a sequence of spectrum maps
$$\cdots \to A_{s} \xrightarrow{f_{s}} A_{s-1} \xrightarrow{f_{s-1}} A_{s-2} \to \cdots $$
Let $B_s\defeq\fib_{f_s}$. Then $D^{n,s}\defeq\pi_n(A_s)$ and $E^{n,s}\defeq\pi_n(B_s)$ are
graded abelian groups and the maps of the long exact sequences become graded
homomorphisms. This gives exactly the data of an exact couple. 

For cohomology, it is customary to reindex the pages of the spectral sequence
with the base change $(p,q)=(s-n,-s)$, or equivalently $(n,s)=(-(p+q),-q)$.

\begin{thm}\label{thm:spectral-sequence-spectrum}
  Given a sequence of spectra
  $$\cdots \to A_{s} \xrightarrow{f_{s}} A_{s-1} \xrightarrow{f_{s-1}} A_{s-2}
  \to \cdots $$ with fibers $B_s\defeq\fib_{f_s}$, suppose for all $n$ there is
  a $\beta_n$ such that for all $s\le \beta_n$ we have $\pi_n(A_s)=0$ and suppose that
  for all $n$ there is a $\gamma_n$ such that for all $s>\gamma_n$ the map $\pi_n(f_s)$ is
  an isomorphism. Then the exact couple constructed from this sequence is
  bounded. This spectral sequence gives
  $$E_2^{p,q}=\pi_{-(p+q)}(B_{-q})\Rightarrow \pi_{-(p+q)}(A_{\gamma_{-(p+q)}}).$$
\end{thm}
\begin{proof}
  Note that for this spectral sequence we have
  $\iota(n,s)\equiv\deg_i(n,s)\equiv(n,s-1)$ and $\kappa=\id$. This means that
  we need to show that for all $(n,s):\Z\times\Z$ there is a bound $\beta'_{n,s}$
  such that for all $t\geq \beta'_{n,s}$ we have 
  $$E^{n,s+t}\equiv \pi_n(B_{s+t})=0\quad\text{ and }\quad D^{n,s-t}\equiv\pi_n(A_{s-t})=0.$$ 
  Note that the right equation holds if $s-t\le \beta_n$, i.e. if $t\ge s-\beta_n$. By
  the long exact sequence of homotopy groups we know that if $f_s:A_s\to
  A_{s-1}$ induces an equivalence on both $\pi_n$ and $\pi_{n+1}$, then
  $\pi_n(B_{s+t})=0$. So if we define 
  $$\beta'_{n,s}\defeq\max(s-\beta_n,\gamma_n-s,\gamma_{n+1}-s),$$
  we know that the exact couple is bounded with bound $\beta'$. 
  
  Now note that $x=(n,\gamma_n)$ is a stable index, because $\pi_n(f_{\gamma_n+t})$ is
  surjective for all $t\geq0$. Therefore, by
  \autoref{thm:exact-couple-convergence} we know that $D^{n,\gamma_n}$ is built from
  $(E_\infty^{n,\gamma_n-s})_{0\le s\le \beta'_{n,\gamma_n}}$. If we apply the reindexing
  $(p,q)=(s-n,-s)$, we get the desired relation
  $$E_2^{p,q}=\pi_{-(p+q)}(B_{-q})\Rightarrow \pi_{-(p+q)}(A_{\gamma_{-(p+q)}}).$$
\end{proof}

\section{Spectral Sequences for Cohomology}\label{sec:spectral-sequence-cohomology}

Cohomology groups are algebraic invariants of types. They are often easier to
compute than homotopy groups, but they can also be used to compute certain
homotopy groups, often via the universal coefficient theorem and the Hurewicz
theorem (neither of which have been proven in HoTT yet). 

The intermediate steps of most classical constructions of the singular cohomology are not homotopy invariant.
\emph{Cellular cohomology} is only defined for cell complexes and not for arbitrary
spaces, but it can be defined in HoTT~\cite{buchholtz2018cellular}. \emph{Singular
cohomology} is defined as a quotient of a large abelian group that is not
homotopy invariant, which makes this definition impossible in HoTT. However,
classically, Eilenberg-MacLane spaces represent cohomology, and we can use this
fact as the \emph{definition} of cohomology in HoTT~\cite{cavallo2015cohomology}.

Normally cohomology groups have coefficients in an abelian group, but more
generally they can have coefficients in a spectrum, or even a family of spectra.
In this section we will define cohomology groups and construct the
Atiyah-Hirzebruch spectral sequence for cohomology. This is a generalization of
the spectral sequence defined in~\cite{atiyah1961spectral} in the special case
of topological K-theory. From the Atiyah-Hirzebruch spectral sequence we can 
construct the Serre spectral sequence, sometimes also called the Leray-Serre spectral sequence.
\begin{defn}
  Suppose given $X:\type^*$ and $Y:X\to\spectrum$. We define $(x:X)\to^* Yx:\spectrum$ such that 
  $((x:X)\to^* Yx)_n\defeq (x:X)\to^*(Yx)_n$. If $Y$ does not depend on $X$, we write $X\to^* Y$.

  For an unpointed type $X:\type$ and $Y:X\to\spectrum$, we similarly define $(x:X)\to Yx:\spectrum$ such that 
  $((x:X)\to Yx)_n\defeq (x:X)\to(Yx)_n$ (this has as basepoint the constant map into the basepoint of $(Yx)_n$), and abbreviate this to $X\to Y$ if $Y$ does not depend on $X$.
\end{defn}
These spectra are well-defined, since we have 
$$\Omega((a:A)\to^*(Ba))\simeq (a:A)\to^*\Omega(Ba)$$ 
and 
$$\Omega((a:A)\to(Ba))\simeq (a:A)\to\Omega(Ba).$$ 
Moreover, they satisfy the expected properties of dependent product. In particular, if $X:\type^*$ and $Y,Z:X\to\spectrum$ and moreover if we have a fiberwise spectrum map $f:(x:X)\to Yx \to Zx$, this induces a map on the dependent products
$$\Pi_f:((x:X)\to Yx)\to ((x:X)\to Zx).$$
\begin{defn}\label{def:cohomology}
  Suppose given $X:\type^*$, $Y:X\to\spectrum$ and $n:\Z$. We define the \emph{geneneralized, parametrized, reduced cohomology} of $X$ with coefficients in $Y$ as\footnote{We will write $\lam{x}Yx$ in $\eta$-expanded form to remember that this is parametrized cohomology.}
  $$\tilde H^n(X;\lam{x}Yx)\defeq \pi_{-n}((x:X)\to^* Yx)\simeq \|(x:X)\to^* (Yx)_n\|_0.$$
  If $Y$ does not depend on $X$, we have the \emph{unparametrized cohomology} as
  $$\tilde H^n(X;Y)\defeq \pi_{-n}(X\to^* Yx)\simeq \|X\to^* Y_n\|_0.$$
  If $X:\type$ is an arbitrary type, we define the \emph{unreduced cohomology} as
  $$H^n(X;\lam{x}Yx)\defeq \pi_{-n}((x:X)\to Yx)\simeq \|(x:X)\to (Yx)_n\|_0\simeq \tilde H^n(X_+;\lam{x}Y_+x).$$ 
  Here $X_+\defeq X+1:\type^*$ and $Y_+:X_+\to\spectrum$ is defined as $Y_+(\inl(x))\defeq Yx$ and $Y_+(\inr(\star))\defeq 1$.
  If $X:\type^*$ and $A:X\to\abgroup$, we define the \emph{ordinary cohomology} as
  $$\tilde H^n(X;\lam{x}Ax)\defeq \tilde H^n(X;\lam{x}H(Ax).$$
  We can combine the attributes ordinary/generalized, parametrized/unparametrized and reduced/unreduced for cohomology however we want, leading to eight different notions. 
  
  We define
  $$\tilde H^n(X)\defeq \tilde H^n(X;\Z)$$
  and similarly for unreduced cohomology.
\end{defn}
Unparametrized cohomology satisfies the Eilenberg-Steenrod axioms for cohomology. Although we will not use this fact in this chapter, for completeness we will state it here. 

To give the definition we need to introduce one more concept. 
\begin{defn}\label{def:choice}
A type $X$ has $n$-choice for $n\geq-2$ if for all $P:X\to\type$ the canonical map
$$\|(x:X)\to Px\|_n\to ((x:X) \to \|Px\|_n)$$
is an equivalence.
\end{defn} 
Note that in particular $\fin_k$ has $n$-choice for all $k,n$.
\begin{defn}
  A \emph{unparametrized reduced cohomology theory} is a contravariant functor
  $\tilde E^n:\type^*\to\abgroup$ for every $n:\Z$ satisfying the
  \emph{Eilenberg-Steenrod axioms}. Functoriality means that for a pointed map
  $f:X\to^* Y$ there is a map $\tilde E^n(f):\tilde E^n(Y)\to \tilde E^n(X)$ such that $\tilde E(\id)\sim^*\id$ and $\tilde E(g\o
  f)\sim^*\tilde E(f)\o\tilde E g$. The Eilenberg-Steenrod axioms are
  \begin{itemize}
  \item \emph{(Suspension axiom)} There is a natural transformation $\tilde E^{n+1}(\susp X)\simeq \tilde E^n(X)$.
  \item \emph{(Exactness)} Given a cofiber sequence $X\xrightarrow{f}Y\xrightarrow{g}Z$, the sequence
  $$\tilde E^n(Z)\xrightarrow{\tilde E^n(g)}\tilde E^nY\xrightarrow{\tilde E^n(f)}\tilde E^n(X)$$
  is exact at $\tilde E^n(Y)$.
  \item \emph{(Additivity)} Suppose given a type $I$ satisfying 0-choice and $X:I\to\type^*$. Then the canonical homomorphism 
  $$\tilde E^n\big(\bigvee_i Xi\big)\to ((i:I)\to \tilde E^n(Xi))$$
  is an isomorphism.
  \end{itemize}
  A cohomology theory is called \emph{ordinary} if it also satisfies the following axiom.
  \begin{itemize}
    \item \emph{(Dimension)} If $n\neq 0$, then $\tilde E^n(\pbool)$ is trivial.
    \end{itemize}
\end{defn}
The following theorem has been proven in~\cite{cavallo2015cohomology}. We will not repeat the proof here.
\begin{thm}\label{thm:cohomology-theory}
  Unparametrized generalized reduced cohomology is a cohomology theory. Ordinary cohomology also satisfies the dimension axiom.
\end{thm}
We will not use \autoref{thm:cohomology-theory} in the remainder of this chapter.

To construct the Atiyah-Hirzebruch spectral sequence, we need the Postnikov tower of a spectrum.

\begin{defn}
  We say that for $k:\Z$ a spectrum $Y$ is \emph{$k$-truncated} if $Y_n$ is $(k+n)$-truncated for all $n:\Z$ (using the convention that any type is $\ell$-truncated for $\ell\leq-2$).

  The \emph{$k$-truncation} of a spectrum $Y$, written $\|Y\|_k$, is defined as $(\|Y\|_k)_n\defeq \|Y_n\|_{k+n}$ where we define $\|A\|_\ell=\unit$ for $\ell\leq-2$.
\end{defn}
\begin{lem}\label{lem:spectrum-trunc-properties}
  The usual properties of truncations also hold for spectra. In particular we will use that there is a spectrum map $|{-}|_k:Y\to\|Y\|_k$ and that if $Z$ is $k$-truncated, then a spectrum map $f:Y\to Z$ induces a spectrum  map $\|Y\|_k\to Z$.
\end{lem}
\begin{proof}
  The underlying maps are the corresponding facts for pointed maps. The fact that these maps are spectrum maps comes from the fact that these operations commute with taking loop spaces. We omit the details here.
\end{proof}
\begin{lem}[Postnikov Tower for spectra]\label{lem:postnikov-tower-spectra}
  For $s:\Z$ and $Y:\spectrum$ there is a spectrum map $f^s:\|Y\|_s\to\|Y\|_{s-1}$ that levelwise has fiber $\susp^n(H\pi_s(Y))$. That is,
  $$(\fib_{f^s})_k\simeq^*(\susp^s(H\pi_s(Y)))_k.$$
\end{lem}
We should be able to extend this equivalence to a spectrum equivalence, but we do not need this strengthening for the remainder of the proof.
\begin{proof}
  Note that $\|Y\|_{s-1}$ is $(s-1)$-truncated, and therefore $s$-truncated. By
  the elimination of spectrum truncation in
  \autoref{lem:spectrum-trunc-properties} we get a spectrum map
  $f^s:\|Y\|_s\to\|Y\|_{s-1}$. For the levelwise pointed equivalence, we need to
  show that
  $$\fib_{f^s_k}\simeq^*K(\pi_s(Y),s+k)$$
  To show this, by \autoref{thm:em-unique} we need to show that $\fib_{f^s_k}$
  is $(s+k)$-truncated, $(s+k-1)$-connected and $\pi_{s+k}(\fib_{f^s_k})\simeq
  \pi_s(Y)$.

  Note that $f^s_k:\|Y_k\|_{s+k}\to\|Y_k\|_{s+k-1}$, so the truncatedness
  follows because the domain and codomain of $f^s_k$ are both $(s+k)$-truncated.
  For the connectedness, we know that $|{-}|_{s+k-1}:Y_k\to \|Y_k\|_{s+k-1}$ is
  $(s+k-1)$-connected, and the elimination principle for truncations preserve
  connectedness, therefore $f^s_k$ is $(s+k-1)$-connected. To compute the
  homotopy group, we look at a piece of the long exact sequence for homotopy
  groups for $f^s_k$ at level $s+k$ and $s+k+1$.
  \begin{center}\begin{tikzpicture}[node distance=10mm]
    \node (Y)     at (0,0) {$0$};
    \node[left = of Y] (X) {$\pi_{s+k}(Y)$};
    \node[left = of X] (F) {$\pi_{s+k}(\fib_{f^s_k})$};
    \node[above = of Y] (OY) {$0$};
    \node[above = of X] (OX) {$0$};
    \node[above = of F] (OF) {$\bullet$};
    \path[every node/.style={font=\sffamily\small}]
    (X) edge[->] (Y)
    (F) edge[->] (X)
    (OY) edge[->] (F)
    (OX) edge[->] (OY)
    (OF) edge[->] (OX);
  \end{tikzpicture}
  \end{center}
  Since we have the exact sequence
  $0\to\pi_{s+k}(\fib_{f^s_k})\to\pi_{s+k}(Y)\to0$, the middle map must be an
  equivalence, which finishes the proof.
\end{proof}
For a spectrum $Y$, we get the Postnikov tower
$$\cdot\to\|Y\|_s\to\|Y\|_{s-1}\to\|Y\|_{s-2}\to\cdots$$ 
This satisfies the conditions of \autoref{thm:spectral-sequence-spectrum}, but
unfortunately the spectral sequence constructed from this is trivial. We need another ingredient to get an interesting spectral sequence.

\begin{lem}\label{lem:spi-functor}
  Suppose given $X:\type^*$ and two family of spectra $Y,Z:X\to\spectrum$. A family of spectrum maps 
  $$f:(x:X)\to Yx \to Zx$$
  induces a spectrum map between the spectra of sections for $Y$ and $Z$:
  $$f\o({-}):((x:X)\to Yx)\to((x:X)\to Zx).$$
  Moreover, the fiber of this spectrum map is levelwise $(x:X)\to \fib_{fx}$, that is
  $$(\fib_{f\o({-})})_n \simeq^* ((x:X)\to \fib_{fx})_n.$$
\end{lem}
The levelwise equivalence should be extendable to a spectrum equivalence, but we do not need that in this chapter.
\begin{proof}
  We define (see \autoref{lem:pointed-types-basic}.\ref{item:fiber-composition})
  $$(f\o({-}))_n\defeq f_n \o ({-}):((x:X)\to^* (Yx)_n)\to^*((x:X)\to^* (Zx)_n).$$
  This is a spectrum map because of the pointed function extensionality mentioned in \autoref{lem:pointed-types-basic}.\ref{item:pointed-function-extensionality}.
  
  By \autoref{lem:pointed-types-basic}.\ref{item:fiber-composition} the fiber of this map is levelwise $(x:X)\to \fib_{fx}$. 
\end{proof}

We now have all the ingredients of the Atiyah-Hirzebruch spectral sequence.

\begin{thm}[Atiyah-Hirzebruch spectral sequence for reduced cohomology]\label{thm:atiyah-hirzebruch-reduced}
  If $X:\type^*$ is a pointed type and $Y:X\to k\operatorname{-\spectrum}$ is a family of
  $k$-truncated spectra over $X$, then we get a spectral sequence with
  $$E_2^{p,q}=\tilde H^p(X;\lam{x}\pi_{-q}(Yx))\Rightarrow \tilde H^{p+q}(X;\lam{x}Yx).$$
\end{thm}
\begin{proof}
  Define $A_s\defeq((x:X)\to^* \|Yx\|_s)$ and consider the sequence of spectra
  $$\cdots \to A_s\xrightarrow{f_s} A_{s-1} \xrightarrow{f_{s-1}} A_{s-2} \to \cdots$$ 
  where $f_s$ is the map induced by the Postnikov tower. By
  \autoref{lem:spi-functor} and \autoref{lem:postnikov-tower-spectra} $f_s$
  levelwise has fiber $B_s\defeq(x:X)\to^*\susp^sH\pi_s(Yx)$. We want to apply
  \autoref{thm:spectral-sequence-spectrum}, so we need to check the conditions
  of that theorem. For $n:\Z$ we define $\beta_n\defeq n-1$. Notice that $A_s$
  is $s$-truncated, and thus for $s\le B_n$ we have
  $$\pi_n(A_s)\defeq \pi_n((x:X)\to^* \|Yx\|_s)=0.$$
  For $n:\Z$ define $\gamma_n\defeq k$. Then for $s\geq \gamma_n$ the spectrum
  $A_s$ is levelwise equivalent to $(x:X)\to^* Yx$, so for $s>\gamma_n$ the map
  $A_s\to A_{s-1}$ becomes levelwise the identity map under that equivalence.
  This means that $f_s$ is an equivalence, so in particular $\pi_n(f_s)$ is an
  isomorphism. By \autoref{thm:spectral-sequence-spectrum} we now get the
  spectral sequence
  $$E_2^{p,q}=\pi_{-(p+q)}(B_{-q})\Rightarrow \pi_{-(p+q)}(A_k).$$
  We now compute
  \begin{align*}
    \pi_{-(p+q)}(B_{-q})&\simeq\pi_{-(p+q)}((x:X) \to^* \susp^{-q}H\pi_{-q}(Yx))\\
    &\simeq\tilde H^{p+q}(X;\lam{x}\susp^{-q}H\pi_{-q}(Yx))\\
    &\simeq\tilde H^{p}(X;\lam{x}\pi_{-q}(Yx))
    \intertext{and}
    \pi_{-(p+q)}(A_k)&\simeq \pi_{-(p+q)}((x:X)\to^*\|Yx\|_k)\\
    &\simeq\pi_{-(p+q)}((x:X)\to^* Yx)\\
    &\simeq\tilde H^{p+q}(X;\lam{x}Yx).
  \end{align*}
\end{proof}

We also have the corresponding spectral sequence for unreduced cohomology.

\begin{cor}[Atiyah-Hirzebruch spectral sequence for unreduced cohomology]\label{cor:atiyah-hirzebruch-unreduced}
  If $X:\type$ is any type and $Y:X\to k\operatorname{-\spectrum}$ is a family
  of $k$-truncated spectra over $X$, then
  $$E_2^{p,q}=H^p(X;\lam{x}\pi_{-q}(Yx))\Rightarrow H^{p+q}(X;\lam{x}Yx).$$
\end{cor}
\begin{proof}
Apply \autoref{thm:atiyah-hirzebruch-reduced} to $X_+$ and $Y_+$ (defined in \autoref{def:cohomology}).
\end{proof}
From the Atiyah-Hirzebruch spectral sequence we can construct the Serre spectral sequence.
\begin{thm}[Serre spectral sequence for cohomology]\label{thm:serre-spectral-sequence}
  Suppose given $B:\type$, a family of types $F:B\to\type$ and a spectrum
  $Y:\spectrum$ that is $k$-truncated. Then 
  $$E_2^{p,q}=H^p(B;\lam{b}H^q(Fb;Y))\Rightarrow H^{p+q}((b:B)\times Fb;Y).$$
\end{thm}
\begin{proof}
Apply \autoref{cor:atiyah-hirzebruch-unreduced} to the type $B$ and and the family of spectra $\lam{b}Fb\to Y$, which is $k$-truncated. Then we get
$$E_2^{p,q}=H^p(B;\lam{b}\pi_{-q}(Fb\to Y))\Rightarrow H^{p+q}(B;\lam{b}Fb\to Y).$$
Note that $\pi_{-q}(Fb\to Y)\simeq H^q(Fb;Y)$, so the second page is the desired group, and for the $\infty$-page we compute
\begin{align*}
  H^{p+q}(B;\lam{b}Fb\to Y) &= \pi_{-(p+q)}((b:B) \to Fb \to Y)\\
  &= \pi_{-(p+q)}(((b:B) \times Fb)\to Y)\\
  &= H^{p+q}((b:B)\times Fb;Y).
\end{align*}
\end{proof}
Equivalent to the data given in \autoref{thm:serre-spectral-sequence} is a map
$X\to B$ and a $k$-truncated spectrum $Y$. In that case we get the spectral
sequence
$$E_2^{p,q}=H^p(B;\lam{b}H^q(\fib_f(b);Y))\Rightarrow H^{p+q}(X;Y).$$ Analogous
to the proof of \autoref{thm:serre-spectral-sequence} we also have a version
when $Y$ is parametrized over $B$. In that case we get
$$E_2^{p,q}=H^p(B;\lam{b}H^q(\fib_f(b);Y))\Rightarrow H^{p+q}(X;\lam{x}Y(fx)).$$

We get a useful special case of the Serre spectral sequence when the family
$\lam{b}H^q(Fb;Y)$ is constant. This happens in particular when $B$ is simply
connected.
\begin{cor}\label{cor:serre-spectral-sequence-conn}
  Suppose given a simply connected pointed type $B:\type^*$, a family of types $F:B\to\type$ and a spectrum
  $Y:\spectrum$ that is $k$-truncated. Then 
  $$E_2^{p,q}=H^p(B;H^q(Fb_0;Y))\Rightarrow H^{p+q}((b:B)\times Fb;Y).$$
\end{cor}
\begin{proof}
  Apply \autoref{thm:serre-spectral-sequence}. The family
  $\lam{b}H^q(Fb;Y):B\to\abgroup$ is a family of sets. Since $B$ is simply
  connected, every such family is constant, so all fibers are equal to $H^q(Fb_0;Y)$.
\end{proof}

In the spectral sequences constructed we assumed that the spectra were truncated. 
The reason we need this assumption is that the notion of convergence we used for spectral sequence is the eventual value of the sequence. 
If we had a stronger notion of convergence, we might be able to relax the truncatedness condition. 
However, there is another reason why the spectral sequence can become (pointwise) eventually constant. 
Instead of assuming that the spectra are truncated, we can pose a restriction on the base space.
\begin{defn}\label{def:weak-pointed-choice}
  We say that $X:\type^*$ satisfies \emph{weak pointed choice} if there is a natural number $n$ such that for all families $Y:X\to\type^*$ of $n$-connected types the type of dependent pointed maps $(x:X)\to^* Y(x)$ is 0-connected.
\end{defn}
\begin{ex}\mbox{}
  \begin{itemize}
    \item The spheres $\S^n$ satisfy weak pointed choice. The proof is easy for $n=0$, which we will skip.
  For $\S^{n+1}$, note that 
  $$((x:S^{n+1})\to^* Y(x))\simeq \star =_{\surf}^{\Omega^nY({-})} \star$$ 
  where $\surf:\Omega^{n+1}\S^{n+1}$ is the surface of $\S^{n+1}$ and
  $\star:\Omega^nY(\base)$ is the basepoint. Now if $Y$ is a family of $(n+1)$-connected types, then $\Omega^nY({-})$ is a family of 1-connected types, and a pathover in that family is 0-connected, as desired.
  \item Suppose $I$ is a type that satisfies 0-choice (see \autoref{def:choice}). Then the collection of types that satisfy weak pointed choice are closed under $I$-indexed wedges. This follows from the dependent universal property of the wedge.
  $$\big(\bigvee_{i:I}X_i\to^* P(x)\big)\simeq (i:I) \to (x : X_i) \to^* P(\inm_i(x)).$$
  \end{itemize}
\end{ex}

\begin{thm}
  If $X:\type^*$ satisfies weak pointed choice and $Y:X\to\spectrum$ is any family of spectra, we get the spectral sequence in \autoref{thm:atiyah-hirzebruch-reduced}:
  $$E_2^{p,q}=\tilde H^p(X;\lam{x}\pi_{-q}(Yx))\Rightarrow \tilde H^{p+q}(X;\lam{x}Yx).$$
\end{thm}
\begin{proof}
  The proof is the mostly the same as for \autoref{thm:atiyah-hirzebruch-reduced}. The only difference is in showing that the sequence stabilizes on homotopy groups when $s$ is large. Suppose $X$ satisfies choice with respect to $k$-connected families.
  For $n:\Z$ define $\gamma_n\defeq n+k$. Then for $s>\gamma_n$ we know that the fiber of $f_s$ has as $k$-th homotopy group 
  $$\pi_\ell(B_s)=\pi_0(\Omega^\ell B_s)=\|(x:X)\to^* K(\pi_s(Yx),s-\ell)\|_0.$$ 
  This is a product into a family of $(s-\ell-1)$-connected types, which for $\ell=n$ and $\ell=n-1$ is a family of at least $k$-connected types. By the weak choice principle on $X$ this type is 0-connected, so these homotopy groups are trivial. Now by the long exact sequence of homotopy groups for $f_s$ the map $\pi_n(f_s)$ is an isomorphism, as required.
\end{proof}

\section{Spectral Sequences for Homology}\label{sec:spectral-sequence-homology}

Homology theory has not been developed as much as cohomology theory in HoTT. It is known that the homology given by a prespectrum forms a homology theory~\cite{graham2017homology}. Lemma 18 in that paper was not proven carefully, but it follows from the results in \autoref{sec:smash-product}.

In this section, we sketch the construction of the Atiyah-Hirzebruch and Serre spectral sequences for homology~\cite{serre1951homology}. The results in the section are not proven in HoTT, and are therefore stated as remarks without proof.

If $X:\type^*$ and $Y:\prespectrum$, we can define $X\wedge Y:\prespectrum$ with 
$$(X\wedge Y)_n\defeq X\wedge Y_n.$$ 
To show that it is a prespectrum, recall the adjunction between the suspension and the loop. For pointed types $X$ and $Y$ we have a natural equivalence
$$(\susp X \to^* Y)\simeq (X \to^* \susp Y).$$
Therefore, to characterize a prespectrum, it is sufficient to give a map $f_n:\susp Y_n \to Y_{n+1}$. This is given for the smash prespectrum as the composite
$$\susp(X\wedge Y_n)\xrightarrow{\sim} X \wedge \susp Y_n \xrightarrow{X\wedge f_n} X \wedge Y_{n+1}.$$
We can define reduced homology 
$$\tilde H_n(X;Y)\defeq \pi_n(X\wedge Y).$$
For the construction of parametrized homology we need to generalize the smash product. 

\begin{defn}
  Given $A : \type^*$ and $B : A \to \type^*$, we define the parametrized smash
  $$(x : A) \wedge B(x)$$
  to be the pushout 
  \begin{center}
  \begin{tikzpicture}[node distance=10mm,baseline={([yshift={-\ht\strutbox}]current bounding box.north)}]
    \node (tl)       at (0,0) {$A+B$};
    \node[right = 25mm of tl] (tr) {$2$};
    \node[below = of tl] (bl) {$(x:A)\times B(x)$};
    \node (br) at (tr |- bl)   {$(x : A) \wedge B(x)$};
    \path[every node/.style={font=\sffamily\small}]
    (tl) edge[->] node {} (tr)
          edge[->] node {} (bl)
    (tr) edge[->] node {} (br)
    (bl) edge[->] node {} (br);
  \end{tikzpicture}\quad
  \begin{tikzpicture}[baseline={([yshift={-\ht\strutbox}]current bounding box.north)}]
    \draw[fill=black!5, thin] (0,0) ellipse (20mm and 8mm);
    \node[label=below:{$a_0$}] at (0,-1.98) {$\bullet$};
    \node at (2.3,-2) {$A$};
    \node at (2.3,0) {$B$};
    \path[every node/.style={font=\sffamily\scriptsize}]
    (-2,-2) edge[thick] (2,-2)
    (0,-1.1)  edge[->] (0,-1.8)
    (0,0)     edge[bend left = 10, thick] (-2,0)
              edge[bend left = 10, thick] node[below] {$b_0$} (2,0)
              edge[bend left = 5, thick] node[left] {$B(a_0)$} (0,0.8)
              edge[bend left = 5, thick] (0,-0.8);
  \end{tikzpicture}
  \end{center}
\end{defn}
\begin{rmk}The strategy for constructing the spectral sequences for homology is as follows.
\begin{itemize}
\item The parametrized smash is (should be) left adjoint to pointed dependent maps. 
That means that there is a natural equivalence
$$(((x:A) \wedge Bx) \to^* C) \simeq (x:A) \to^* Bx \to^* C.$$
\item From this we get (natural) equivalences
$$\susp((x:A)\wedge Bx) \simeq ((x:A) \wedge \susp(Bx));$$
$$((x:A)\wedge Bx) \wedge C \simeq (x:A) \wedge (Bx \wedge C);$$
$$(x:A_+)\wedge B_+x) \wedge C \simeq (x:A) \times Bx.$$
The proofs of these properties should be similar to the proofs in \autoref{sec:smash-monoidal}.
\item Therefore, for $X:\type^*$ and $Y:X\to\prespectrum$ we have a prespectrum $(x : X)\wedge Yx$. The maps are given by the above equivalence.
\item We can now define parametrized (reduced, generalized) homology as
$$H_n(X;\lam{x}Yx) \defeq \pi_n((x : X) \wedge Yx).$$
We can define unreduced homology by adding a point to $X$, in the same way as for cohomology.
\item As before, given $X:\type^*$ and $Y:X\to\spectrum$, we can again form the Postnikov tower of $Yx$ for any $x:X$. We now want to take the parametrized smash over $X$, but there is no hope to compute the fiber of this spectrum.
\item However, we should be able to do it when we work in spectra. The forgetful
functor $\spectrum \to \prespectrum$ has a left adjoint, called
\emph{spectrification}. The spectrification $LY$ of a prespectrum $Y$ can be
constructed either as a higher inductive family of types~\cite{shulman2011spectrification} or as the colimit
$$(LY)_n\defeq\colim_{k\to\infty}\Omega^k Y_{n+k}.$$
For neither definition a careful proof of the adjunction has been given.
\item We can now define the parametrized smash of spectra as the spectrification of the parametrized smash for prespectra. This should preserve cofiber sequences of spectra, in the sense that if 
$$Ax \to Bx \to Cx$$
is a family of cofiber sequences of spectra indexed by $x:\type^*$, the following sequence is also a cofiber sequence of spectra
$$((x:X)\wedge Ax)\to ((x:X)\wedge Bx)\to ((x:X)\wedge Cx)$$
\item A sequence of spectra should be a fiber sequence of spectra if and only if it is a cofiber sequence of spectra. This is true classically, and should also hold in HoTT.
\item Assuming that all the above properties have been proven, we can get the Atiyah-Hirzebruch spectral sequence for reduced homology. Suppose given a pointed type $X$ and $Y:X\to\spectrum$ a family of spectra. We can apply
\autoref{thm:spectral-sequence-spectrum} to the iterated fiber sequence
$$((x:X)\wedge \susp^nH)\to((x:X)\wedge \|Yx\|_s)\to ((x:X)\wedge
\|Yx\|_{s-1}).$$ To satisfy the conditions for that theorem we need to assume
some conditions on $X$ and/or $Y$. In particular it is sufficient if $Y$ is a family of 
truncated and connected spectra, but weaker conditions might also suffice.
Using homological indexing (where $p$ and $q$ have their sign reversed) we get 
$$E^2_{p,q}=\pi_{p+q}(B_q)\Rightarrow \pi_{p+q}(A_{\gamma_{p+q}}).$$
Now we compute
\begin{align*}
  \pi_{p+q}(B_q)&\simeq\pi_{p+q}((x:X) \wedge \susp^qH\pi_q(Yx))\\
  &\simeq\tilde H_{p+q}(X;\lam{x}\susp^qH\pi_q(Yx))\\
  &\simeq\tilde H_p(X;\lam{x}\pi_q(Yx))
  \intertext{and}
  \pi_{p+q}(A_k)&\simeq \pi_{p+q}((x:X)\wedge\|Yx\|_k)\\
  &\simeq\pi_{p+q}((x:X)\to Yx)\\
  &\simeq\tilde H_{p+q}(X;\lam{x}Yx).
\end{align*}
This gives the desired spectral sequence:
$$E^2_{p,q}=\tilde H_p(X;\lam{x}\pi_q(Yx))\Rightarrow \tilde H_{p+q}(X;\lam{x}Yx).$$
\item We get the Atiyah-Hirzebruch spectral sequence for \emph{unreduced} homology in the same way as for cohomology, by applying the version for reduced homology to $X_+$ and $Y_+$.
\item We get the Serre spectral sequence for homology also in the same way. Suppose given
$B:\type$ and $F:B\to\type$ and a truncated spectrum $Y$. Applying the Atiyah-Hirzebruch spectral sequence for unreduced homology to the type $B$ and the spectrum $\lam{b}Fb\wedge Y$ we get
$$E^2_{p,q}=H_p(B;\lam{b}\pi_q(Fb\wedge Y))\Rightarrow H_{p+q}(B;\lam{b}Fb\wedge Y).$$
The second page is what we want. For the $\infty$-page we compute
\begin{align*}
  H_{p+q}(B;\lam{b}Fb\to Y) &= \pi_{p+q}((b:B_+) \wedge (F_+b \wedge Y))\\
  &= \pi_{p+q}(((b:B_+) \wedge F_+b)\wedge Y)\\
  &= \pi_{p+q}(((b:B) \times Fb)\wedge Y)\\
  &= H_{p+q}((b:B)\times Fb;Y).
\end{align*}
This gives the Serre spectral sequence for homology:
$$E^2_{p,q}=H_p(B;\lam{b}H_q(Fb;Y))\Rightarrow H_{p+q}((b:B)\times Fb;Y).$$
\end{itemize}
\end{rmk}
\begin{rmk}
  We can also use the parametrized smash to get a spectral sequence for reduced
  homology and reduced cohomology. Suppose given $B:\type^*$ and a family of types $F:B\to\type^*$ and a spectrum $Y:\spectrum$ that is $k$-truncated. Then we get the following two spectral sequences
  $$E_2^{p,q}=\tilde H^p(B;\lam{b}\tilde H^q(Fb;Y))\Rightarrow \tilde H^{p+q}((b:B)\wedge Fb;Y);$$
  $$E^2_{p,q}=\tilde H_p(B;\lam{b}\tilde H_q(Fb;Y))\Rightarrow \tilde H_{p+q}((b:B)\wedge Fb;Y).$$
  For homology, the proof is the same as above. For cohomology, we apply the Atiyah-Hirzebruch spectral sequence for reduced cohomology to the pointed type $B$ and the family of spectra $\lam{b}Fb\to^* Y$. We get the desired spectral sequence by the adjunction between parametrized smash and dependent pointed maps.

  These spectral sequences generalize \autoref{thm:serre-spectral-sequence} and the corresponding version for homology: we get those versions back when we add a point to $B$ and $F$. Whether this extra generality is useful is unknown.
\end{rmk}

\section{Applications of Spectral Sequences}\label{sec:applications-spectral-sequences}

Classically, there are many applications of the Serre and Atiyah-Hirzebruch spectral sequences. 
Here we will list some of these applications, and give thoughts on how to translate these results in HoTT. 
The results in this section have not been formalized. Before we start, we compute the cohomology of spheres.

\begin{lem}\label{lem:cohomology-spheres}
  If $n\geq 1$, then 
\begin{equation}H^k(\S^n;A)=\begin{cases}A&\text{if $k\in\{0,n\}$}\\0&\text{otherwise.}\end{cases}\label{eq:cohomology-spheres}\end{equation}
\end{lem}
This is a special case of the universal coefficient theorem, which we do not have yet in HoTT. 
However, we can prove these equalities directly from the definition of cohomology. 
\begin{proof}
For $k=0$ we have 
$$H^0(\S^n;A)=\|\S^n\to A\|_0=(\S^n\to A) = A,$$
where we use that $\S^n$ is 0-connected. For $k\neq 0$ we have
$$H^k(\S^n;A)=\tilde H^k(\S^n+1;A)=\tilde H^k(\S^n;A)=\|\S^n\to^* K(A,k)\|_0=\|\Omega^n K(A,k)\|_0.$$
Now for $n<k$ the type $\Omega^n K(A,k)$ is 0-connected, hence the result is contractible. For $n=k$ the result is $A$, and for $n>k$ the type $\Omega^n K(A,k)$ itself is contractible. 
\end{proof}

The first application is the path fibration. Suppose given a simply connected pointed type $B$ we have a map $\unit\to B$ that has fiber $\Omega B$.\footnote{It is called the \emph{path fibration} because classically to get a Serre fibration we need to take the path space $PB$ instead of $\unit$.} In other words, we have the fiber sequence
$$\Omega B \to \unit \to B.$$
Now the Serre spectral sequence for cohomology gives (say, with integer coefficients)
$$E_2^{p,q}=H^p(B;H^q(\Omega B))\Rightarrow H^{p+q}(\unit).$$
Note that the $\infty$-page vanishes, except when $p+q=0$, when the coefficient is $\Z$. 
For ordinary cohomology $H^n$ is trivial for $n<0$, which means that the second page is only nontrivial in the first quadrant of the plane, 
hence this is true for all pages, including the $\infty$-page. 
Therefore, the $\infty$-page has one group $\Z$ at the origin, and trivial groups everywhere else, as shown in \autoref{fig:infty-path-fibration}.
\begin{figure}[ht]
  \begin{center}
  \begin{tikzpicture}
    \draw [thick] (0.5,0.5) -- (0.5,3.5);
    \draw [thick] (0.5,0.5) -- (5.5,0.5);
    \node[font=\normalsize] at (5.4,0.3) {$p$};
    \node[font=\normalsize] at (0.3,3.4) {$q$};
    \node at (0.1,1) {$0$};
    \node at (0.1,2) {$1$};
    \node at (0.1,3) {$2$};
    \node at (1,0.1) {$0$};
    \node at (2,0.1) {$1$};
    \node at (3,0.1) {$2$};
    \node at (4,0.1) {$3$};
    \node at (5,0.1) {$4$};
    \node (x11) at (1,1) {$\Z$};
    \node (x12) at (1,2) {$0$};
    \node (x13) at (1,3) {$0$};
    \node (x21) at (2,1) {$0$};
    \node (x22) at (2,2) {$0$};
    \node (x23) at (2,3) {$0$};
    \node (x31) at (3,1) {$0$};
    \node (x32) at (3,2) {$0$};
    \node (x33) at (3,3) {$0$};
    \node (x41) at (4,1) {$0$};
    \node (x42) at (4,2) {$0$};
    \node (x43) at (4,3) {$0$};
    \node (x51) at (5,1) {$0$};
    \node (x52) at (5,2) {$0$};
    \node (x53) at (5,3) {$0$};
  \end{tikzpicture}
\end{center}
  \caption{$E_\infty^{p,q}$ for the path fibration.}
  \label{fig:infty-path-fibration}
\end{figure}

This gives a relation between the cohomology of $B$ and the cohomology of $\Omega B$. 
If we know the cohomology for one of the spaces one of them, 
then we can sometimes compute the cohomology from the other using this. 
Using the Serre spectral sequence for homology, we have the same relationship between the homology of $B$ and the homology of $\Omega B$. 
The computations in the next example will work exactly the same for homology.

\begin{ex}\label{ex:cohomology-KZ2}
As an example, we can compute the cohomology groups of $B = K(\Z,2)$ 
(which is the complex projective space $\mathbf{CP}^\infty$). 
Its loop space is $\Omega K(\Z,2) = K(\Z,1) = \S^1$, and by \autoref{lem:cohomology-spheres} we have
$$H^n(\S^1)=\begin{cases}\Z&\text{if $n=0,1$}\\ 0&\text{otherwise.}\end{cases}$$
The resulting second page of the spectral sequence is shown in \autoref{fig:two-page-KZ2-path-fibration}, all other groups on the second page are trivial. 
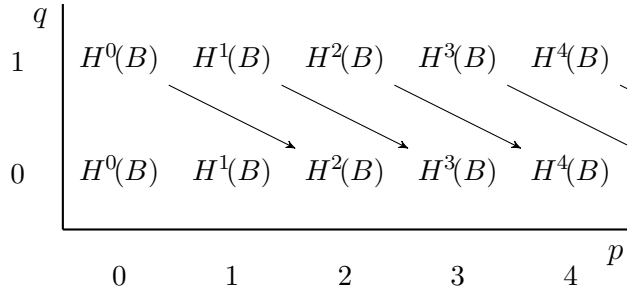
\begin{figure}[ht]
  \begin{center}
  \begin{tikzpicture}[every node/.style={font=\small},x=15mm,y=15mm]
    \draw [thick] (0.5,0.5) -- (0.5,2.5);
    \draw [thick] (0.5,0.5) -- (5.5,0.5);
    \node[font=\normalsize] at (5.4,0.3) {$p$};
    \node[font=\normalsize] at (0.3,2.4) {$q$};
    \node at (0.1,1) {$0$};
    \node at (0.1,2) {$1$};
    \node at (1,0.1) {$0$};
    \node at (2,0.1) {$1$};
    \node at (3,0.1) {$2$};
    \node at (4,0.1) {$3$};
    \node at (5,0.1) {$4$};
    \node (x11) at (1,1) {$H^0\!(B)$};
    \node (x12) at (1,2) {$H^0\!(B)$};
    \node (x21) at (2,1) {$H^1\!(B)$};
    \node (x22) at (2,2) {$H^1\!(B)$};
    \node (x31) at (3,1) {$H^2\!(B)$};
    \node (x32) at (3,2) {$H^2\!(B)$};
    \node (x41) at (4,1) {$H^3\!(B)$};
    \node (x42) at (4,2) {$H^3\!(B)$};
    \node (x51) at (5,1) {$H^4\!(B)$};
    \node (x52) at (5,2) {$H^4\!(B)$};
   \path
    (x12) edge[->] (x31)
    (x22) edge[->] (x41)
    (x32) edge[->] (x51)
    (x42) edge (5.5,1.25)
    (x52) edge (5.5,1.75);
  \end{tikzpicture}
  \end{center}
  \caption{$E_2^{p,q}$ for the path fibration of $K(\Z,2)$.}
  \label{fig:two-page-KZ2-path-fibration}
\end{figure}
Note that the shown differentials are the only nontrivial differentials on the second page, and all differentials on all later pages are also trivial. This means that $E_3=E_\infty$, depicted in \autoref{fig:infty-path-fibration}. Note that there are no nontrivial differentials going in or out of the $H^0(B)$ and $H^1(B)$ in the bottom line. This means that 
$$H^0(B)=E_\infty^{0,0}=\Z$$ 
and 
$$H^1(B)=E_\infty^{1,0}.$$
All other groups displayed on the second page vanish on the $\infty$-page. Therefore, all shown differentials must be isomorphisms. 
This means that $H^{n+2}(B)=H^n(B)$, which shows that $H^n(B)$ is $\Z$ for even $n$ and $0$ for odd $n$.
\end{ex}
Another simple application of the Serre spectral sequence is to compute the homology and cohomology groups of $\Omega \S^n$, given in~\cite[Example 1.5]{hatcher2004spectral}. 
In this case, we know the (co)homology of the base space $\S^n$, and from it we can deduce the (co)homology of the loop space $\Omega \S^n$. We will do the computation here for cohomology.

\begin{ex}\label{ex:cohomology-OSn}
  If we take the Serre spectral sequence for the path fibration of $B=S^n$ for $n\geq 2$, then the second page has entries
  $$E_2^{p,q}=H^p(S^n;H^q(\Omega S^n))=\begin{cases}H^q(\Omega S^n)&\text{if $p=0,n$}\\ 0&\text{otherwise.}\end{cases}$$
  using \autoref{lem:cohomology-spheres}. Therefore, the only nontrivial groups are in the columns $p=0$ and $p=n$. 
  This means that by looking at the degree of the differentials, the only nonzero differentials can occur in page $n$, as shown in \autoref{fig:n-page-path-fibration-sphere}.
\begin{figure}[ht]
  \begin{center}
  \begin{tikzpicture}[every node/.style={font=\small},x=25mm,y=12mm]
    \draw [thick] (0.5,0.5) -- (2.5,0.5);
    \draw [thick] (0.5,0.5) -- (0.5,4.5);
    \node[font=\normalsize] at (2.4,0.3) {$p$};
    \node[font=\normalsize] at (0.4,4.4) {$q$};
    \node at (1,0.1) {$0$};
    \node at (2,0.1) {$n$};
    \node at (0.1,1) {$0$};
    \node at (0.1,2) {$n-1$};
    \node at (0.1,3) {$2(n-1)$};
    \node at (0.1,4) {$3(n-1)$};
    \node (x11) at (1,1) {$H_0(\Omega S^n)$};
    \node (x12) at (1,2) {$H_{n-1}(\Omega S^n)$};
    \node (x13) at (1,3) {$H_{2(n-1)}(\Omega S^n)$};
    \node (x14) at (1,4) {$H_{3(n-1)}(\Omega S^n)$};
    \node (x21) at (2,1) {$H_0(\Omega S^n)$};
    \node (x22) at (2,2) {$H_{n-1}(\Omega S^n)$};
    \node (x23) at (2,3) {$H_{2(n-1)}(\Omega S^n)$};
    \node (x24) at (2,4) {$H_{3(n-1)}(\Omega S^n)$};
    \path
    (x12) edge[->] (x21)
    (x13) edge[->] (x22)
    (x14) edge[->] (x23)
    (1.5,4.5) edge[->] (x24);
  \end{tikzpicture}
  \end{center}
  \caption{$E_n^{p,q}$ for the path fibration of $S^n$.}
  \label{fig:n-page-path-fibration-sphere}
\end{figure}
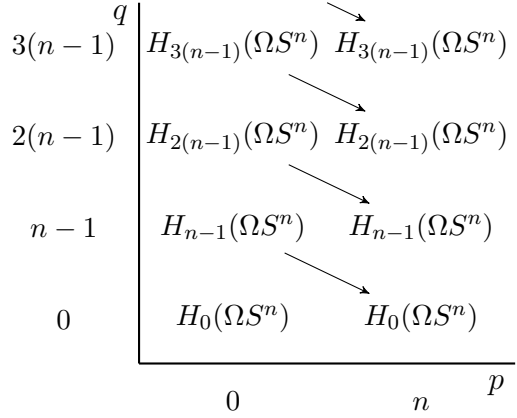  
  Because all later differentials are trivial, $E_{n+1}=E_\infty$, which is depicted in \autoref{fig:infty-path-fibration}. This means that all differentials on page $n$ from the $p=0$ column to the $p=n$ column must be isomorphisms, except for the differential from $(0,0)$ to $(n,-(n-1))$. Hence we can conclude by induction that
\begin{equation}H^k(\Omega\S^n)\begin{cases}\Z&\text{if $n-1\mid k$}\\0&\text{otherwise.}\end{cases}\label{eq:cohomology-loop-spheres}\end{equation}
\end{ex}

As a generalization of \autoref{ex:cohomology-KZ2}, we can construct the \emph{Gysin sequence} from the Serre spectral sequence~\cite[Theorem 3.3.3]{holmbergperoux2014models}. 
The Gysin sequence for homology states that if $f:E\to B$ is a pointed map with fiber $\S^{n-1}$ for $n\geq 2$ and if $B$ is simply connected, then there exists a long exact sequence
$$\cdots \to H_i(E)\to H_i(B)\to H_{i-n}(B)\to H_{i-1}(E) \to \cdots.$$
There is also an analogue for cohomology, which states that under the same assumptions there exists a long exact sequence of cohomology groups
$$\cdots \to H^{i-1}(E)\to H^{i-n}(B)\to H^i(B)\to H^i(E) \to \cdots.$$
The proof given in~\cite[Theorem 3.3.3]{holmbergperoux2014models} works the same in HoTT. 
An alternative construction of the Gysin sequence in HoTT is given in~\cite[Section 6.1]{brunerie2016spheres}, which was used as a main ingredient to compute $\pi_4(\S^3)$.

We can also generalize \autoref{ex:cohomology-OSn} to get the \emph{Wang sequence}. 
For homology this states that if $E\to \S^n$ is a pointed map for $n\geq 2$ with fiber $F$, then there exists a long exact sequence
$$\cdots \to H_i(F)\to H_i(E)\to H_{i-n}(F)\to H_{i-1}(F) \to \cdots.$$
Again, a similar long exact sequence holds for cohomology, and the proof given in~\cite[Theorem 3.3.6]{holmbergperoux2014models} works the same in HoTT.

As another application, we can prove the Hurewicz theorem from the Serre spectral sequence~\cite{holmbergperoux2014models}. 
The Hurewicz theorem only holds for homology, and requires the Serre spectral sequence for homology. 
The theorem states that if $X$ is a simply connected pointed type, $n\geq 2$ and $\pi_q(X)$ is trivial for $q<n$, then $H_q(X)=0$ for $q<n$ and $H_n(X)=\pi_n(X)$. 
For $n=1$ the Hurewicz theorem states that for a 0-connected pointed type $X$, the first homology group $H_1(X)$ is the abelianization of $\pi_1(X)$. 
In the proof given in the aforementioned reference, the case for $n=1$ needs to be proven separately, 
but then the case for $n\geq2$ follows from that using the Serre spectral sequence. 
Since the case for $n=1$ seems easier than the general case, this should be very helpful to prove the Hurewicz theorem in HoTT.

An application for the Atiyah-Hirzebruch spectral sequence would be to compute cohomology groups of generalized cohomology theories.
One such generalized cohomology theory is K-theory. Although K-theory has not been precisely defined yet in HoTT, 
one possible idea by Ulrik Buchholtz is to define it using Snaith's theorem~\cite{snaith1981ktheory}. 
If we have defined K-theory, we could try to compute its cohomology groups using the Atiyah-Hirzebruch spectral sequence. 
With the current machinery, we can compute the cohomology groups of all types that satisfy weak pointed choice (cf. \autoref{def:weak-pointed-choice}), 
which probably includes all finite CW-complexes. 

An application that is probably trickier in HoTT is the Serre class theorem. A \emph{Serre class} is a class $\mc C$ of abelian groups such that for every short exact sequence
$0\to A \to B \to C \to 0$ of abelian groups we have $B\in \mc C$ iff $A,C\in \mc C$. In particular, any Serre class is closed under taking subgroups and quotient groups. Classical examples of Serre classes include 
\begin{itemize}
\item finite abelian groups;
\item finitely generated abelian groups;
\item torsion abelian groups.
\end{itemize}
However, constructively, the first two classes are not closed under either taking subgroups or quotient groups (torsion abelian groups do form a Serre class constructively).

The Serre class theorem is a theorem about certain Serre classes that satisfy some extra properties. 
This include the three examples mentioned above. 
If $\mc C$ is such a Serre class and if $X$ is a simply connected 
type,\footnote{or path-connected and \emph{abelian}. A space $X$ is abelian if the action of $\pi_1(X)$ on $\pi_n(X)$ is trivial for all $n\geq 1$.} 
then the theorem states that $\pi_n(X)\in\mc C$ for all $n$ iff $H_n(X)\in\mc C$ for all $n$. 
More general is the Hurewicz theorem modulo a Serre class, which states that if $\pi_i(X)\in\mc C$ for all $i<n$, 
then the kernel and the cokernel of the Hurewicz homomorphism $h:\pi_n(X)\to H_n(X)$ belong to $\mc C$.

As a corollary of the Serre class theorem, we know that the homotopy groups of the spheres are finitely generated, since their homology groups are finitely generated. 
Moreover, the homotopy groups of simply connected finite CW-complexes are also finitely generated, using cellular cohomology~\cite{buchholtz2018cellular}.
A classical proof of the Serre class theorem can be found in~\cite[Section 1.1]{hatcher2004spectral}.

It is not straightforward to adapt the proof of the Serre class theorem to a proof in HoTT. One difficulty is that the classical proof uses the \emph{universal coefficient theorem} for homology. 
This theorem is not yet proven in HoTT. 
The universal coefficient theorem relates the homology group $H_n(X;A)$ with coefficients in any abelian group $A$ to the the homology group $H_n(X)$ with integer coefficients.
There is also a dual universal coefficient theorem for cohomology that relates the cohomology group $H^n(X;A)$ with the homology group $H_n(X)$.
It is not clear how to prove or even formulate the universal coefficient theorem in HoTT. 
The universal coefficient theorem for homology uses the $\mathsf{Tor}$ functor, whose definition requires projective resolutions. 
Similarly, the universal coefficient theorem for cohomology uses the $\mathsf{Ext}$ functor, whose definition requires injective resolutions. 
Basic properties of projective and injective resolutions are classically proven with the axiom of choice~\cite{blass1979injectivity}, 
so it is not clear whether we can prove the universal coefficient theorem without the axiom of choice.
Another problem with proving the universal coefficient theorem is that classically it is proven algebraically for chain complexes. 
Since homology and cohomology groups of spaces are defined as the (co)homology of a chain complex, the universal coefficient theorem then applies to spaces.
Even if we could solve the issues with the axiom of choice, and we could prove the universal coefficient theorem for chain complexes in HoTT, 
it does not directly follow that it is true for spaces, since these groups are \emph{not} defined as the (co)homology of chain 
complexes.\footnote{Cellular (co)homology~\cite{buchholtz2018cellular} is defined as the (co)homology of a chain complex, 
  and therefore we could prove it for finite CW complexes, but it would not follow for arbitrary types.}
Therefore, in HoTT, it seems fruitful to prove the universal coefficient theorem directly for spaces, 
using the definition in terms of Eilenberg-MacLane spaces, but this is an open problem as of now. 
A good first step might be to try to prove a special case of the universal coefficient theorem where the $\mathsf{Tor}$ and $\mathsf{Ext}$ functors vanish, 
although that will not be sufficient to prove the Serre class theorem.

Proving the Serre class theorem in HoTT will be tricky, and it might be necessary to reformulate or weaken some notions to get a usable result in HoTT. 
If we manage to prove these results in HoTT, we can get a lot of information about the homotopy groups of spheres. 
One additional ingredient that is required is the fact that the cup product structure of the cohomology groups respect the Serre spectral sequence. 
From these ingredients we can classically show the following:
\begin{itemize}
\item The groups $\pi_i(\S^n)$ are finite for $i>n$, except for $\pi_{4k-1}(\S^{2k})$, which are the direct sum of $\Z$ and a finite group~\cite[Theorem 1.21]{hatcher2004spectral}.
\item For a prime $p$ the $p$-torsion subgroup of $\pi_i(\S^3)$ is 0 for $i<2p$ and $\Z_p$ for $i=2p$~\cite[Example 1.18]{hatcher2004spectral}.
\item From the two above results we can immediately conclude that $\pi_4(\S^3)=\Z_2$.
\item Using additionally the localization of a space at a prime, we can show that for $p$ a prime, the $p$-torsion subgroup of $\pi_i(\S^{n+3})$ is 0 for $i<n+2p$ and $\Z_p$ for $i=n+2p$~\cite[Theorem 1.28]{hatcher2004spectral}.
\item We can compute more homotopy groups of spheres using significantly more machinery. 
For this we need the EHP sequence, Steenrod squares and Serre's theorem, which computes the cohomology rings of $K(\Z_2,n)$ and $K(\Z,n)$ and $K(\Z_{2^k},n)$ with coefficients in $\Z_2$.
If we have all these results, we can compute $\pi_{n+i}(\S^n)$ for all $n$ and $i\leq 3$~\cite[Theorem 1.40]{hatcher2004spectral}.
\end{itemize}

\chapter*{Conclusion}
\addcontentsline{toc}{chapter}{Conclusion}
\lstMakeShortInline"

In this dissertation I have shown that homotopy type theory is a practical language to prove involved theorems in homotopy theory, 
most notably the construction of two important spectral sequences: the Atiyah-Hirzebruch and the Serre spectral sequences for cohomology.
The discovery of these spectral sequences in classical homotopy theory was an important milestone, 
and we expect that the corresponding proof in HoTT will lead to many useful corollaries in synthetic homotopy theory.
That said, many applications of these spectral sequences require more machinery, such as the universal coefficient theorem, Serre classes and the Hurewicz theorem. 
The first two of these three results might be problematic to prove in HoTT, because of their dependence on the axiom of choice. 
I hope that an adapted or weaker version of these theorems can be found, which avoids the use of the axiom of choice, or alternatively, 
that when looking at their applications, we can avoid the use of choice. For example, we cannot prove constructively that finitely generated abelian groups form a Serre class, 
but it is conceivable that we can still prove that all homotopy groups of spheres are finitely generated without resorting to the axiom of choice.
That said, we could also assume the axiom of choice and continue proving results in synthetic homotopy theory using it. 
However, then the resulting theorem would not hold anymore in all models of HoTT.

It would be interesting to see other spectral sequences proven in HoTT, such as the Adams spectral sequence and the Eilenberg-Moore spectral sequences~\cite{hatcher2004spectral}.

\subsection*{Thoughts on formalization}

This dissertation also shows that homotopy type theory provides a good language for the computer formalization of results in homotopy theory.

Through the formal methods community there is a strong desire that formal methods will be adopted in a large scale by general mathematicians.
The main bottlenecks for this adoption are
\begin{enumerate}
  \item the necessary expertise of formalization in the proof assistant of choice;
  \item the vast number of proof assistant in existence;
  \item the amount of work it takes to formalize mathematics compared to writing it on paper. 
\end{enumerate}
It definitely takes time to learn a proof assistant, familiarize oneself with the library and get enough practice to use a proof assistant efficiently.
Moreover, in my experience, learning to use a proof assistant takes longer than learning to use other programs, like LaTeX or Mathematica. 
Still, I do not think this is the main bottleneck to the adoption of proof assistants. 
Various courses that integrate the use of proof assistants have been taught, and students taking those courses will get a level of proficiency of using that proof assistant.

The second concern is the number of proof assistants in existence, each with a separate library and the near-impossibility to translate theorems and proofs between two proof assistants. 
There are translation procedures between some proof assistants, such as~\cite{mclaughlin2006interpretation}, 
but such translations are often incomplete, and only specific to two proof assistants.

However, I think the main bottleneck is the amount of extra time it takes to formalize mathematics compared to writing a paper proof. 
Rough estimates for the formalization time is about one week to formalize a page of a mathematical paper or textbook~\cite{asperti2010itp}.
My experience with formalizing synthetic homotopy theory specifically is a little different. 
Paper proofs given in synthetic homotopy theory are often quite detailed, and the techniques used are often very close to the underlying type theory. 
I would argue that this is necessary; we do not have much experience with proving theorems in synthetic homotopy theory yet, and it is not always clear which results are hard to prove.
Some results turn out more difficult to prove than initially thought. 
For example, the ``basic property'' of the smash product that it forms a 1-coherent symmetric monoidal product (see \autoref{sec:smash-product}) 
was assumed with a vague proof sketch in~\cite{brunerie2016spheres} to prove $\pi_4(\S^3)=\Z_2$, but this result is still open as of now. 
Another example is \autoref{thm:colim_sm}, which was originally thought to be a basic result about colimits by Egbert Rijke and me, but the proof turned out to be much harder than expected.

Because paper proofs in synthetic homotopy theory are often proven with many details, in my experience, giving a fully formal proof is not much more work. 
In cases where the formal proof is a lot more work, the paper proof sometimes omitted showing the case of the path constructor when inducting over a HIT, 
which is the hardest --- but least enlightening --- part of the proof. 
In the paper proofs of this dissertation, I have also sometimes omitted these steps, because they are tedious to work through and not very enlightening. 
However, the formal proofs (of course) contain all the details.
For some theorems the formalization did take substantially more work. 
For the formalization of spectral sequences, a substantial algebra library had to be developed, 
consisting of basic group theory, ring theory, modules over a ring and graded modules.
This took many man-hours of work, which would have no counterpart in a paper proof. 

Another reason why formalization is more work, is that necessarily such proofs have to be encoded in the corresponding logic, intensional type theory. 
Most of the time, this is straightforward, but in some cases it takes a bit more work. 
Especially when dealing with sequences of types, in intensional type theory one has to work explicitly with transports (or its relatives, like pathovers or heterogenous equality), 
which is especially laborious in the proof-relevant setting of HoTT. 
Sometimes an ``encoding trick'' is useful when dealing with these dependent types. In this dissertation some of these tricks have been given. 
In \autoref{sec:les-homotopy} we defined a chain complex over an arbitrary successor structure, 
because we wanted to not only index chain complexes over $\N$ or $\Z$, but also over $\N\times\fin_3$ or similar types, to get a more convenient computational content.
These successor structures also turned out to be useful for spectra in \autoref{sec:spectra}, so that we can apply the same notion to spectra indexed over $\N$ and spectra indexed over $\Z$. 
The reason that spectra indexed over $\Z$ are useful (traditionally they are only indexed over $\N$) is that for certain definitions, 
such as the homotopy group of a spectrum, no case-splits are required when they are indexed over $\Z$. 
Another encoding trick was given in the definition of graded morphisms. In order to define the composition of graded morphisms more easily, 
we defined the degree of a graded morphism to be an automorphism of the indexing set. 
In order to avoid dealing with transports everywhere, we defined a graded morphism to act on a path in the indexing set, see \autoref{sec:spectral-sequences}.

Formalizing in Lean is a fun activity, and Lean is a good language for formalization. 
In Lean 2 one of the main annoyances when formalizing was the unpredictability of the elaborator, which was greatly improved in Lean 3.
Another issue was the ability to simplify expressions. There was a tactic "esimp" that simplified by evaluation, 
but it was quite slow, and would sometimes use $99\%$ of the elaboration time of a proof. 
I do not have enough experience with "dsimp" in Lean 3 to see whether it has similar issues. 

In 2016 Leonardo de Moura decided that he would stop supporting homotopy type theory in Lean~3. It was quite devastating to hear this. 
I am glad that Gabriel Ebner has found a method to do homotopy type theory in Lean~3 safely, by avoiding the use of "Prop".
Since then, I have been slowly working on porting the Lean 2 HoTT library to Lean~3, although the progress has been slow. 
The main reasons for this are:
\begin{itemize}
\item The elaborator in Lean~3 is weaker to make it more robust, which causes many proofs to break.
\item Some tactics do not work without using "Prop" or the "Prop"-valued equality. 
Gabriel Ebner has modified the "simp" and "rewrite" tactic to work in HoTT. 
I have written an "induction"-tactic, since the default induction tactic does not allow custom induction principles to eliminate to only non-"Prop" sorts.
\item The notation "!" has been removed in Lean~3. This was used in Lean~2 to turn (some) explicit arguments into implicit ones.
\item There are many small differences in Lean~2 and Lean~3 in syntax for tactics, proof styles, attributes, universe levels and declarations. 
None of these issues take much time to fix, but the sheer number of them add up.
\end{itemize}
Despite this, a significant part of the library has been ported, and I am planning to continue this so that Lean~3 (and later Lean~4) can be used to formalize results in homotopy type theory.

The current implementation in Lean is probably not the ultimate proof assistant for HoTT in the long-term. 
Many cubical type theories have been developed over the last few years, 
and a few proof assistants have been developed using a cubical type theory as their underlying logic. 
Cubical type theory offers many advantages when reasoning about higher inductive types and when doing higher path algebra, since more relations hold strictly. 
For example, the computation rule of the induction principle for a higher inductive type holds judgmentally in cubical type theory. 
This is very convenient when working with HITs, especially HITs with higher path constructors.
It is conceivable that a cubical type theory can be implemented in Lean, although it will require some hacking in the "C++" code, 
and many features of Lean will need to be modified to work well with the cubical structure. 
This will be a big project, and it is probably smart not do this project until the different variants of cubical type theory have been studied more.
In particular, current versions of cubical type theory do not satisfy regularity, which states that the induction principle for paths has judgmental computation rules.\footnote{It 
is possible to have two notions of paths: a path type with all the cubical structure, and an identity type with an induction principle and a judgmental computation rule. 
However, in current cubical type theories these cannot be the same type.} 
Some constructions in HoTT are done by doing a long string of path inductions, and such proofs will be harder to reason with in cubical type theory. That said, it would be interesting to perform some constructions of this dissertation in one of the cubical type theories to see whether the proof would significantly simplify. In particular the proofs in \autoref{sec:non-recursive-2} would simplify when the induction principle of higher inductive types reduces definitionally when applied to path constructors.


\chapter*{Acknowledgements}
\addcontentsline{toc}{chapter}{Acknowledgements}
First and foremost I would like to thank my advisor Jeremy Avigad, who was always ready to give useful feedback, proofread drafts of all my written work and provide support. 
Futhermore, I would like to thank Steve Awodey for always being ready to answer any questions I have about HoTT or category theory. 
I would like to thank Mike Shulman for many helpful remarks and insights whenever I show my work. 
I also want to thank Ulrik Buchholtz, Egbert Rijke, Jakob von Raumer, Stefano Piceghello and Kristina Sojakova for the collaborations and discussions. 
I am grateful towards Leonardo de Moura for all his help with getting me up to speed with Lean, and answering all my stupid questions and ideas I brought up early in the development of Lean.
I would like to thank Marc Bezem and Dan Christensen to invite me for academic visits. 
More generally, I would like to thank everyone in the HoTT community for maintaining such a good research community. 
It is very nice to be part of such a friendly and collaborative research community, 
where it is normal to have unfinished projects on Github or discuss half-baked ideas on a mailing list.

For moral support, I would like to thank my parents, Peter van Doorn and Judith van Wakeren, for supporting me during times when I was struggling. Dank jullie wel! 
Lastly I would like to thank Cecilia Hornberger for the moral support over the last months.

I gratefully acknowledge the support of the Air Force Office of Scientific Research through MURI
grant FA9550-15-1-0053. Any opinions, findings and conclusions or recommendations expressed in this
material are those of the authors and do not necessarily reflect the views of the AFOSR. 

The author would like to thank the Isaac Newton Institute for Mathematical Sciences, Cambridge,
for support and hospitality during the programme Big Proof where work on this paper
was undertaken. This work was supported by EPSRC grant no EP/K032208/1. 

This material is based upon work supported by the National Science Foundation under Grant Number DMS 1641020. 

\clearpage
\phantomsection
\addcontentsline{toc}{chapter}{Bibliography}
\Urlmuskip=0mu plus 1mu 
\bibliographystyle{amsalphaurl}
\bibliography{references}

\end{document}